\documentclass[reqno]{amsart}
\usepackage{amssymb}
\usepackage{graphicx}
\usepackage{hyperref}
\usepackage[mathscr]{euscript}
\usepackage{mathtools}
\usepackage{subfigure}
\usepackage{calc}
\usepackage{tikz}
\usetikzlibrary{cd}

\hypersetup{colorlinks=true,linkcolor=blue}

\numberwithin{equation}{section}

\theoremstyle{plain}
\newtheorem{theorem}{Theorem}[section]
\newtheorem{lemma}[theorem]{Lemma}
\newtheorem{proposition}[theorem]{Proposition}
\newtheorem{corollary}[theorem]{Corollary}
\newtheorem{cornot}[theorem]{Corollary/Notation}

\newtheorem{conjecture}[theorem]{Conjecture}
\newtheorem{bigtheorem}{Theorem}
\newtheorem{bigconjecture}[bigtheorem]{Conjecture}

\theoremstyle{definition}
\newtheorem{definition}[theorem]{Definition}
\newtheorem{example}[theorem]{Example}
\newtheorem{assumption}[theorem]{Assumption}
\newtheorem{notation}[theorem]{Notation}

\theoremstyle{remark}
\newtheorem{remark}[theorem]{Remark}

\DeclareMathOperator{\area}{Area}
\DeclareMathOperator{\crit}{Crit}
\DeclareMathOperator{\Cut}{Cut}
\DeclareMathOperator{\cone}{Cone}
\DeclareMathOperator{\core}{Core}
\DeclareMathOperator{\Diff}{Diff}
\DeclareMathOperator{\domain}{Domain}
\DeclareMathOperator{\Dr}{Dr}
\DeclareMathOperator{\forget}{Forget}
\DeclareMathOperator{\Gl}{Gl}
\DeclareMathOperator{\gr}{Gr}
\let\hom\relax\DeclareMathOperator{\hom}{Hom}
\DeclareMathOperator{\im}{Image}
\DeclareMathOperator{\Rol}{Rol}
\DeclareMathOperator{\Sh}{Sh}

\DeclareMathOperator{\St}{St}

\DeclareMathOperator{\closure}{cl}
\DeclareMathOperator{\rk}{rk}
\DeclareMathOperator{\rot}{rot}
\DeclareMathOperator{\sgn}{sgn}
\DeclareMathOperator{\val}{val}

\newcommand{\BB}{\mathbb{B}}
\newcommand{\CC}{\mathbb{C}}
\newcommand{\NN}{\mathbb{N}}
\newcommand{\RR}{\mathbb{R}}
\newcommand{\ZZ}{\mathbb{Z}}

\newcommand{\bE}{\mathbf{E}}
\newcommand{\bG}{\mathbf{G}}
\newcommand{\bL}{\mathbf{L}}
\newcommand{\bO}{\mathbf{O}}
\newcommand{\bP}{\mathbf{P}}
\newcommand{\bQ}{\mathbf{Q}}
\newcommand{\bU}{\mathbf{U}}
\newcommand{\bV}{\mathbf{V}}

\newcommand{\be}{\mathbf{e}}

\newcommand{\bh}{\mathbf{h}}
\newcommand{\bp}{\mathbf{p}}

\newcommand{\bv}{\mathbf{v}}
\newcommand{\bw}{\mathbf{w}}
\newcommand{\bx}{\mathbf{x}}
\newcommand{\by}{\mathbf{y}}

\newcommand{\cA}{\mathcal{A}}
\newcommand{\cB}{\mathcal{B}}
\newcommand{\cC}{\mathcal{C}}
\newcommand{\cD}{\mathcal{D}}
\newcommand{\cE}{\mathcal{E}}
\newcommand{\cF}{\mathcal{F}}
\newcommand{\cG}{\mathcal{G}}

\newcommand{\cH}{\mathcal{H}}
\newcommand{\cI}{\mathcal{I}}
\newcommand{\cK}{\mathcal{K}}
\newcommand{\cL}{\mathcal{L}}
\newcommand{\cLG}{\mathcal{LG}}

\newcommand{\cLT}{\mathcal{LT}}
\newcommand{\cM}{\mathcal{M}}
\newcommand{\cN}{\mathcal{N}}

\newcommand{\cP}{\mathcal{P}}
\newcommand{\cR}{\mathcal{R}}

\newcommand{\cT}{\mathcal{T}}
\newcommand{\cU}{\mathcal{U}}
\newcommand{\cV}{\mathcal{V}}
\newcommand{\cW}{\mathcal{W}}

\newcommand{\scrM}{\mathscr{M}}
\newcommand{\scrP}{\mathscr{P}}
\newcommand{\scrU}{\mathscr{U}}

\newcommand{\sC}{\mathsf{C}}
\newcommand{\sD}{\mathsf{D}}
\newcommand{\sE}{\mathsf{E}}
\newcommand{\sG}{\mathsf{G}}
\newcommand{\sH}{\mathsf{H}}
\newcommand{\sI}{\mathsf{I}}
\newcommand{\sL}{\mathsf{L}}

\newcommand{\sO}{\mathsf{O}}
\newcommand{\sR}{\mathsf{R}}
\newcommand{\sS}{\mathsf{S}}
\newcommand{\sT}{\mathsf{T}}
\newcommand{\sU}{\mathsf{U}}
\newcommand{\sV}{\mathsf{V}}

\DeclareSymbolFont{sfletters}{OT1}{cmss}{m}{n}
\DeclareMathSymbol{\sTheta}{\mathord}{sfletters}{"02}

\newcommand{\sfa}{\mathsf{a}}
\newcommand{\sfb}{\mathsf{b}}
\newcommand{\sfc}{\mathsf{c}}
\newcommand{\sfd}{\mathsf{d}}
\newcommand{\sfe}{\mathsf{e}}

\newcommand{\sfg}{\mathsf{g}}
\newcommand{\sfh}{\mathsf{h}}
\newcommand{\sfm}{\mathsf{m}}
\newcommand{\sfp}{\mathsf{p}}
\newcommand{\sfq}{\mathsf{q}}

\newcommand{\sft}{\mathsf{t}}
\newcommand{\sfv}{\mathsf{v}}

\newcommand{\sfw}{\mathsf{w}}
\newcommand{\sfx}{\mathsf{x}}

\newcommand{\ttE}{\mathtt{E}}
\newcommand{\ttH}{\mathtt{H}}
\newcommand{\ttV}{\mathtt{V}}

\newcommand{\tte}{\mathtt{e}}
\newcommand{\tth}{\mathtt{h}}
\newcommand{\ttv}{\mathtt{v}}

\newcommand{\fG}{\mathfrak{G}}

\newcommand{\fP}{\mathfrak{P}}
\newcommand{\fR}{\mathfrak{R}}
\newcommand{\fT}{\mathfrak{T}}
\newcommand{\fX}{\mathfrak{X}}

\newcommand{\fd}{\mathfrak{d}}

\newcommand{\fr}{\mathfrak{r}}
\newcommand{\ft}{\mathfrak{t}}
\newcommand{\fx}{\mathfrak{x}}
\newcommand{\fy}{\mathfrak{y}}
\newcommand{\fz}{\mathfrak{z}}

\renewcommand{\bar}{\overline}
\renewcommand{\emptyset}{\varnothing}
\renewcommand{\hat}{\widehat}
\renewcommand{\tilde}{\widetilde}

\newcommand{\origin}{\mathbf{0}}
\newcommand{\tameisom}{\stackrel{t}\simeq}
\newcommand{\stabletameisom}{\stackrel{s.t.}\simeq}
\newcommand{\Zmod}[1]{\ZZ/#1\ZZ}
\newcommand{\hybto}[1]{\stackrel{#1}{\Rightarrow}}
\newcommand{\cl}[1]{\closure(#1)}
\newcommand{\rPi}{\mathring{\Pi}}
\newcommand{\ePi}{{\bE\Pi}}
\newcommand{\vPi}{{\bV\Pi}}
\newcommand{\separator}{\diamond}

\newcommand{\CE}{\mathsf{CE}}
\newcommand{\EN}{\mathsf{EkN}}
\newcommand{\id}{\mathsf{Id}}
\newcommand{\Ng}{\mathsf{EtNS}}
\newcommand{\pair}{\mathsf{Pair}}

\newcommand{\co}{\mathsf{co}}
\newcommand{\ess}{\mathsf{ess}}
\newcommand{\ext}{\mathsf{ext}}
\newcommand{\hyb}{\mathsf{hyb}}
\newcommand{\sinf}{\mathsf{inf}}
\newcommand{\intr}{\mathsf{int}}

\newcommand{\reg}{\mathsf{reg}}
\newcommand{\rotation}{\mathsf{rot}}
\newcommand{\sm}{\mathsf{sm}}

\newcommand{\std}{\mathsf{std}}

\newcommand{\tangleReplacement}{\amalg_v}
\newcommand{\tangleReplacementLag}{\amalg_\sfv}
\newcommand{\concavevertex}{\mathsf{concave}}
\newcommand{\positivefolding}{\mathsf{fold^+}}
\newcommand{\infinitesimalmonogon}{\mathsf{I}}
\newcommand{\infinitesimaltriangle}{\mathsf{III}}
\newcommand{\infinitesimalbigonleft}{\mathsf{II,L}}
\newcommand{\infinitesimalbigonmiddle}{\mathsf{II,C}}
\newcommand{\infinitesimalbigonright}{\mathsf{II,R}}
\newcommand{\infinitesimalsquare}{\mathsf{IV}}
\newcommand{\negativefolding}{\mathsf{fold^-}}
\newcommand{\negativefoldingconcave}{\mathsf{fold^-,concave}}
\newcommand{\pmfolding}{\mathsf{fold^\pm}}
\newcommand{\rightmost}{\mathsf{v,most^R}}
\newcommand{\leftmost}{\mathsf{v,most^L}}
\newcommand{\rightrigid}{\mathsf{e,rigid^R}}
\newcommand{\leftrigid}{\mathsf{e,rigid^L}}
\newcommand{\flexible}{\mathsf{v,flex}}
\newcommand{\vertexleftrigid}{\mathsf{v,rigid^L}}
\newcommand{\vertexbirigid}{\mathsf{v,rigid}}

\renewcommand{\S}{Section}

\newcommand*{\sector}{{
	\begin{tikzpicture}[scale=0.15]
		\draw (0,0)--(45:1) arc (45:-45:1)--(0,0);
	\end{tikzpicture}}
}

\begin{document}
\title{A Chekanov-Eliashberg algebra for Legendrian graphs}
\author{Byung Hee An}
\address{Center for Geometry and Physics, Institute for Basic Science (IBS), Pohang 37673, Korea}
\email{anbyhee@ibs.re.kr}
\author{Youngjin Bae}
\address{Research Institute for Mathematical Sciences, Kyoto University, Kyoto 606-8317, Japan}
\email{ybae@kurims.kyoto-u.ac.jp}

\begin{abstract}
We define a differential graded algebra for Legendrian graphs and tangles in the standard contact Euclidean three space.
This invariant is defined combinatorially by using ideas from Legendrian contact homology.
The construction is distinguished from other versions of Legendrian contact algebra by the vertices of Legendrian graphs.
A set of countably many generators and a generalized notion of equivalence are introduced for invariance.
We show a van Kampen type theorem for the differential graded algebras under the tangle replacement. Our construction recovers many known algebraic constructions of Legendrian links via suitable operations at the vertices.
\end{abstract}

\maketitle

\tableofcontents

\section{Introduction}
The theory of Legendrian submanifolds plays an important role in the study of contact and symplectic topology.
Moreover, it has an interesting connection to low dimensional topology and knot theory.
A key breakthrough in contact topology and symplectic geometry was initiated by Gromov's pseudo-holomorphic disk counting technique.
Invariants of Legendrian submanifolds using this curve counting method were outlined in \cite{Eliashberg2000} and concretized by Chekanov \cite{Chekanov2002} for Legendrian links in the standard  contact $\RR^3$ in terms of combinatorial data.
Later the invariants, realized as a differential graded algebra(DGA), have been generalized in several directions including \cite{EES2005,EES2007,EN2015,LS2013,Ng2010,NT2004,Sabloff2003,Sivek2011}.

Legendrian graphs are used in the proof of the famous Giroux correspondence theorem and recently appeared in the study of arboreal singularities as 1-dimensional Legendrians with singularities. 
They have been studied by several groups including \cite{BI2009,OP2014,Prasolov2014} in their own right, especially in the spirit of classification.
The goal of this article is to extend the curve counting idea to Legendrian graphs in the standard contact $\RR^3$.

The main issue is how to deal with the singularities, i.e., the vertices of the Legendrian graph. 
The crucial feature of the construction of a DGA for Legendrian graphs is that we associate a set of countably many generators, Reeb chords, for each vertex of the Legendrian graph.
There is geometric motivation for such an assignment.
Instead of considering a Legendrian with singularities, let us consider a bordered manifold $S^3\setminus k\mathring B^3$, where $\mathring B^3$ is an open 3-ball and $k$ is the number of vertices of our Legendrian graph. 
Edges in a Legendrian graph are replaced by properly embedded Legendrian arcs in $S^3\setminus k\mathring B^3$.
By admitting a certain standard model near the boundary we have a Reeb orbit for each boundary component which yields infinitely many Reeb chords.
A DGA with infinitely many generators was discussed in \cite{EN2015} where the authors considered Legendrian links in the boundary of a subcritical Stein $4$-manifold.
Note that these two constructions are deeply related both geometrically and algebraically.

The second issue is about the grading of the DGA. For Legendrian knots, there is a canonical construction of a potential function, which is unique up to translation and induces a $\ZZ$ or $\Zmod{r}$-grading.
Similarly, the gradings on $n$-component links are given by componentwise potential functions which have $(n-1)$ degrees of freedom up to translation.
We generalize this construction further to Legendrian graphs by considering {\em edgewise} potential functions. Then each edge contributes one to the degree of freedom for grading and exactly one of them is reduced by the translation action as in the link case. To have a well-defined grading on our DGA, we consider Legendrian graphs with potential instead of Legendrian graphs alone.

The last important issue is about invariance with respect to Legendrian isotopy, or Reidemeister moves for Legendrian graphs.\footnote{See Figure~\ref{fig:RM}.}
The stable-tame isomorphism, a notion of equivalence between DGAs, works well when a pair of generators emerges or cancels out.
Such a phenomenon typically appear when we perform the Legendrian Reidemeister move $\rm(II)$ on Legendrian links in the standard contact $\RR^3$.
When there is a $m$-valent vertex in a Legendrian graph, however, the Legendrian Reidemeister move $\rm(IV_*)$ forces us to develop the notion of algebraic equivalence which cares about the birth and death of $m$ generators. To remedy this problem, we suggest the notion of {\em peripheral structures} and {\em generalized stabilizations}.
With this terminology, we have the following two theorems.

\begin{bigtheorem}[Theorem~\ref{theorem:DGA}]
Let $\cL=(\Lambda,\fP)$ be a Legendrian graph with potential. Then there is a pair $(\cA_\cL, \cP_\cL)$ consisting of a DGA $\cA_\cL\coloneqq(A_\Lambda,|\cdot|_\fP,\partial)$ and a canonical peripheral structure $\cP_\cL$.
\end{bigtheorem}

\begin{bigtheorem}[Theorem~\ref{theorem:invariance}]
The pair $(\cA_\cL, \cP_\cL)$ up to generalized stable-tame isomorphisms is an invariant for $\cL$ under the Legendrian Reidemeister moves for Legendrian graphs with potential. In particular the induced homology $H_*(\cA_\cL,\partial)$ is an invariant.
\end{bigtheorem}

The DGA construction can be generalized to {\em Legendrian tangles} and we consider the operation given by replacing a Darboux neighborhood of a vertex with a suitable Legendrian tangle, which yields a van Kampen type theorem for DGAs.

Legendrian links in a bordered manifold and their associated DGAs were first considered in \cite{Sivek2011} via combinatorial methods and later in \cite{HS2014} with geometric interpretation.
The main statement there was also a van Kampen type theorem for Legendrian links in the standard contact $\RR^3$. Note that their construction has at most two borders, and it can be interpreted as a Legendrian graph with one or two vertices in our terminology. Our construction of a DGA generalizes that of \cite{Sivek2011} as follows:

\begin{bigtheorem}[Theorems~\ref{thm:van Kampen} and \ref{thm:inclusion of DGA diagram}]
Let $\cL$ be a Legendrian graph(or tangle) with potential having a $m$-valent vertex $v$, $\cT$ be a Legendrian $m$-tangle with potential, and $\cL\tangleReplacement\cT$ be a tangle replacement of $\cL$ near $v$. Then we have a following commutative diagram of DGAs:
\[
\begin{tikzcd}
\cI_v\ar[r,"\bp_\infty"]\ar[d,"\bp_v"',hook]& \cA_\cT \ar[d,hook]\\
\cA_\cL\ar[r,"\bw_v"]& \cA_{\cL\tangleReplacement\cT}
\end{tikzcd}
\]
Here $\cI_v$ is a DGA for the $m$-valent vertex $v$ with peripheral structures $\bp_v$ and $\bp_\infty$, and ${\bw}_v$ is defined by evaluating $\bp_\infty$ for the image $\bp_v(\cI_m)$.

Moreover, there is a canonical inclusion from Sivek's DGA diagram to the above DGA diagram.
\end{bigtheorem}

On the other hand, Legendrian links can be considered as Legendrian graphs having bivalent vertices only which are smooth at each vertex.
Conversely, we define an operation, called {\em smoothing}, on a bivalent vertex of a given Legendrian graph, which can be used to define an associated DGA for the result. Then via this operation, we can recover Chekanov--Eliashberg's DGA \cite{Chekanov2002} and Etnyre--Ng--Sabloff's DGA \cite{ENS2002}.

\begin{bigtheorem}[Theorem~\ref{theorem:smoothing}]
Let $\cK=(\Lambda,\fP)$ be a Legendrian circle with $(\Zmod{2\rot(\cK)})$-valued potential consisting of one bivalent vertex $v$ and one edge.
Suppose that two half-edges are opposite and have the same potential. Then there is a DGA isomorphism
\[
\cA_\cK^{\sm}(v)\otimes_\ZZ(\Zmod{2})\to\cA^{\CE}_\cK,
\]
where $\cA^{\CE}_\cK$ is the Chekanov--Eliashberg DGA over $\Zmod{2}$ for the Legendrian knot obtained from $\cK$.
\end{bigtheorem}
\begin{bigtheorem}[Theorem~\ref{theorem:smoothing 2}]
Let $\cL=(\Lambda,\fP)$ be a Legendrian graph with $\ZZ$-valued potential whose underlying graph is a disjoint union of circles. Suppose that each component has only one bivalent vertex whose two half-edges are opposite and have the same potential.
Then there is a DGA isomorphism
\[
\cA_\cL^{\sm}\to\cA^{\Ng}_\cL,
\]
where $\cA^{\Ng}_\cL$ is the $\ZZ$-graded Chekanov--Eliashberg DGA over $\ZZ[\sft_1^{\pm1},\cdots, \sft_m^{\pm1}]$ for the Legendrian $m$-component link $\cL$ generalized by Etnyre, Ng, and Sabloff.
\end{bigtheorem}

Even though we only consider a combinatorial description for pseudo-holomorphic disks in the rest of the article, the main idea of the construction of our DGA and of the proof of the invariance come from the geometric picture sketched above. This model is inspired by the standard local model for Legendrians in boundaries of Weinstein 1-handles discussed by Ekholm and Ng in \cite[\S 4.2]{EN2015}. Indeed, we need half of that standard model. We have the following relation in this regard:

\begin{bigtheorem}[Theorem~\ref{theorem:relation to EN}]
Let $\cL$ be a Legendrian graph with $2m$ vertices and $\Phi$ be a $m$-pair of gluings of vertices such that the gluing $\cL_\Phi$ is a Legendrian link in $\#^m(S^1\times S^2)$. Then the DGA $\cA_\cL$ is generalized stable-tame isomorphic to the DGA $\cA_{\cL_\Phi}^{\EN}$ defined by Ekholm and Ng.
\end{bigtheorem}

For a given Legendrian link in $S^3$, there is a construction of a Weinstein domain obtained by attaching a cotangent cone (or Weinstein two handle) along the neighborhood of the given Legendrian link. To extend this construction to a Legendrian graph $\Lambda$, we need additional data on $\Lambda$, the {\em base points} $I$, which determine preferred Legendrian cycles $\{\Lambda_i\}_{i \in I}$.

Now it is possible to mimic the construction for the Legendrian graph, and note that the resulting Weinstein domain, say $\cW$, depends on a {\em Legendrian ribbon} $R$ under the ribbon isotopies, not the starting Legendrian $\Lambda$. So it is natural and important to consider algebraic invariants for the ribbon structure which can be extracted from $\Lambda$ equipped with the above additional data.

There are two distinguished Lagrangians in $\cW$. One is the Lagrangian skeleton of $\cW$ and the other is its symplectic dual, the union of the corresponding cocore disks $D_\Lambda=\{D_i\}_{i\in I}$. Along with \cite[Theorem~2 and Conjecture~3]{EL2017} we propose a relation between the partially wrapped Floer cohomology of the dual of the Lagrangian skeleton and the newly constructed chain complexes of the Legendrian graph $\Lambda$.

\begin{bigconjecture}[Conjecture~\ref{conjecture:relation with CW}]
There is an $A_\infty$ quasi-isomorphism between the partially wrapped $A_\infty$ algebra $CW^*(D_\Lambda,D_\Lambda)$ and $\cA_{\cL}(\Lambda,\Lambda)=\bigoplus_{i,j}\cA_{\cL}(\Lambda_i,\Lambda_j)$ which extends the quasi-isomorphisms between the chain complexes $CW^*(D_i,D_j)$ and $\cA_{\cL}(\Lambda_i,\Lambda_j)$.
\end{bigconjecture}

It seems interesting to make a comparison between our method and other approaches including the theory of microlocal sheaves, the infinitesimal Fukaya category, and the study of holonomic $\cD$-modules. At the end of this article we give explicit computations of the algebra $\cA_{\cL}(\Lambda,\Lambda)$ for Legendrian graphs which are related to certain arboreal singularities in Nadler's list \cite{Nadler2017}.

\subsection*{Summary of the Paper}
As preliminaries, we recall in \S~\ref{sec:Preliminaries} the basic notions for Legendrian graphs such as Lagrangian projection, Reidemeister moves, and Maslov potentials.

In \S~\ref{sec:Peripheral Structures and Generalized stable tame isomorphism for DGAs}, we introduce a new equivalence relation for DGAs which generalizes stable-tame isomorphism of DGAs. In particular, peripheral structures and generalized stabilizations for DGAs are introduced. 

A construction of a DGA for Legendrian graphs is given in \S~\ref{Differential graded algebra for Legendrian graphs with potentials}. 
Capping paths are introduced as the generators and the grading is induced from the Maslov potential.
Two types of admissible disks are introduced for the differential of generators coming from the crossings and the vertices, respectively.
In order to encode the data of vertices of a graph, we introduce a canonical peripheral structure which is a collection of DG-subalgebras at each vertex.

\S~\ref{section:Hybrid disks and the invariance theorem} is devoted to showing the invariance theorem. Hybrid disks and pairs are introduced as a combinatorial model for pseudo-holomorphic disks in the Lagrangian cobordism induced by the move~$\rm(IV_*)$. A map between two DGAs is given by counting such hybrid disks and pairs.

In \S~\ref{section:Legendrian tangle}, we extend the DGA construction to Legendrian tangles and show the van Kampen type theorem for DGAs.

Relations to the known DGA constructions are dealt with in \S~\ref{section:Applications}. In particular, the DGA construction in \cite{Chekanov2002,EN2015,Ng2003,Sivek2011} are discussed. 

We introduce in \S~\ref{section:partially wrapped Floer homology} the Legendrian ribbon of the Legendrian graph and a related construction of a Weinstein domain.
We also discuss the relation between the composable DGA of Legendrian graphs and the wrapped Floer homology of the newly constructed Weinstein domain. 

In Appendix~\ref{section:manipulation of disks} and \ref{appendix:hybrid admissible pairs}, we perform explicit manipulations of (hybrid) admissible pairs to show $\partial^2=0$ and the invariance property of our DGA.

\section{Preliminaries}\label{sec:Preliminaries}
\subsection{Conventions}
We will use different font styles to denote graphs and other objects in various categories as follows:
\begin{enumerate}
\item An abstract graph will be denoted by $\Gamma$, whose vertices, edges and half-edges are denoted by a typewriter font, such as $\ttv\in\ttV_\Gamma$, $\tte\in\ttE_\Gamma$ and $\tth\in\ttH_\ttv$.
\item A Legendrian graph will be denoted by $\Lambda$, whose vertices, edges and half-edges are denoted by italic letters, such as $v\in\cV_\Lambda$, $e\in\cE_\Lambda$ and $h\in\cH_v$.
\item A projection of a Legendrian graph will be denoted by $\sL$, whose vertices, edges and half-edges are denoted by sans serif letters, such as $\sfv\in\sV_\sL$, $\sfe\in\sE_\sL$ and $\sfh\in\sH_\sfv$.
\item For Maslov potentials and gradings, we will use Fraktur letters, such as $1_\fR\in\fR$, $\ft\in\fT_\sL\subset\fX_\sL$, $\fP\in\fX_\sL^*$ and $\fG(\sL;\fR)$.
\item On the other hand, we will consider the convex polygon $\Pi$, whose vertices, edges and generic points will be denoted by the boldface letters, such as $\bv\in\vPi$, $\be\in\ePi$ and $\bx\in\Pi$.
\end{enumerate}

We will use the notations 
\begin{align*}
\RR^3_0&\coloneqq(\RR^3,\xi_0),&
\xi_0&=\ker(dz-ydx),\\
\RR^3_{\rotation}&\coloneqq(\RR^3,\xi_{\rotation}),&
\xi_{\rotation}&=\ker(dz+xdy-ydx)
\end{align*}
to denote two standard tight contact structures on $\RR^3$, which are contactomorphic to each other via $\Xi:\RR_0^3\to\RR_{\rotation}^3$ defined as
\begin{align}
\setlength{\arraycolsep}{2pt}
\begin{array}{rccc}
\Xi:&\RR^3_0&\to&\RR^3_{\rotation},\\
&(x,y,z)&\mapsto&(x,y,2z-xy).
\end{array}
\end{align}
Notice that two contact forms $dz-ydx$ and $dz+xdy-ydx$ define the same Reeb vector field $\partial_z$.

\subsection{Legendrian graphs}\label{sec:Legendrian graphs}
Throughout this paper, we mean by a {\em directed graph} $\Gamma=(\ttV_\Gamma,\ttE_\Gamma)$ a finite regular CW complex of dimension 1, where $\ttV_\Gamma$ and $\ttE_\Gamma$ are the sets of 0-cells and {\em oriented} 1-cells called {\em vertices} and {\em edges}, respectively. 
Hence each edge $\tte\in \ttE_\Gamma$ can be regarded as a function
\begin{align*}
\tte:([0,1],\{0,1\})\to(\Gamma,\ttV_\Gamma).
\end{align*}

A {\em half-edge} $\tth$ is a restriction of an edge $\tte$, which is either $\tte|_{[0,\epsilon)}$ or $\tte|_{(1-\epsilon,1]}$ for $\epsilon\ll 1$ and whose vertex is defined as either $\tte(0)$ or $\tte(1)$, respectively. We say that $\tth$ is {\em adjacent to $\ttv$} if the vertex of $\tth$ is $\ttv$.
We denote the set of half-edges adjacent to $\ttv$ by $\ttH_\ttv$. Then the number $m$ of half-edges adjacent to $\ttv$ defines the {\em valency} of $\ttv$.
\begin{align*}
\ttH_\ttv&\coloneqq\{\tth_{\ttv,1},\cdots,\tth_{\ttv,m}\},&
\val(\ttv)&\coloneqq m.
\end{align*}
We say that a vertex $\ttv$ is {\em isolated} if it is of valency $0$.

\begin{assumption}[No isolated vertices]
We assume that there are no isolated vertices in $\Gamma$. 
\end{assumption}

A {\em spatial directed graph $\Lambda=(\cV_\Lambda, \cE_\Lambda)$} of a directed graph $\Gamma$ in $\RR^3$ is defined to be an embedding $\Lambda:\Gamma\to\RR^3$, where 
\begin{align*}
\cV_\Lambda&\coloneqq \Lambda(\ttV_\Gamma),&
\cE_\Lambda&\coloneqq\Lambda(\ttE_\Gamma),&
\cH_v&\coloneqq\Lambda(\ttH_\ttv)=\{h_{v,1},\cdots, h_{v,m}\mid h_{v,i}=\Lambda(\tth_{\ttv,i})\}
\end{align*}
for $v=\Lambda(\ttv)$. We use $\Lambda$ to denote the image $\Lambda(\Gamma)$ for convenience's sake.

\begin{definition}[Legendrian graphs, \cite{OP2014}]\label{definition:Legendrian graph}
A spatial directed graph $\Lambda=(\cV_\Lambda,\cE_\Lambda)$ of $\Gamma$ is a {\em Legendrian graph} if 
\begin{enumerate}
\item each edge $e\in\cE_\Lambda$ as an embedding
\[
e:([0,1],\{0,1\})\to(\Lambda,\cV_\Lambda)\subset \RR^3_{\rotation}
\]
is smooth Legendrian on its interior,
\item for each edge, the one-sided differentials at both ends are well-defined and non-vanishing, and therefore the angles between two half-edges at the same vertex are well-defined, and
\item the angle between any two half-edges adjacent to the same vertex is non-zero\footnote{The angle $\pi$ between two half-edges is allowed.}.
\end{enumerate}

We say that two Legendrian graphs $\Lambda$ and $\Lambda'$ are {\em equivalent} if 
they are isotopic, that is, there exists a family of Legendrian graphs
\begin{align*}
\Lambda_t&:\Gamma\times[0,1]\to\RR^3_{\rotation},&
\Lambda_0&=\Lambda,&
\Lambda_1&=\Lambda'.
\end{align*}
\end{definition}

\begin{theorem}[Darboux Theorem for Legendrian graphs]\label{theorem:Darboux}
Let $\Lambda=(\cV_\Lambda,\cE_\Lambda)$ be a Legendrian graph. Then for each vertex $v\in \cV_\Lambda$ of valency $m\ge 1$, there exist a neighborhood $\cU_v$ of $v$ and a contactomorphism $\phi_v$ preserving both orientation and co-orientation
\begin{align*}
\phi_v:(\cU_v\setminus\{v\}, (\cU_v\setminus\{v\})\cap\Lambda)\stackrel{\simeq}\longrightarrow (\cU_\origin\setminus\{\origin\}, (\cU_{\origin}\setminus\{\origin\})\cap T_\Theta),
\end{align*}
where $\cU_\origin$ is the unit ball of $\RR^3_{\rotation}$ and $T_\Theta$ is the union of Legendrian rays
\begin{align*}
T_\Theta&\coloneqq\Lambda_{\theta_1}\cup\cdots\cup\Lambda_{\theta_m},&
\Lambda_\theta&=\left\{ (r\cos\theta,r\sin\theta,0)\in \RR^3_{\rotation}\mid r\ge0\right\},
\end{align*}
for some $\Theta=(\theta_1,\cdots,\theta_m)$ with $0\le \theta_m<\cdots<\theta_1<2\pi$.

We call the neighborhood $\cU_v$ a {\em Darboux neighborhood} of $v$.
\end{theorem}
\begin{proof}
We omit the proof since it is essentially the same as the proof of \cite[Lemma~3.2]{ABK2018}.
\end{proof}

\begin{figure}[ht]
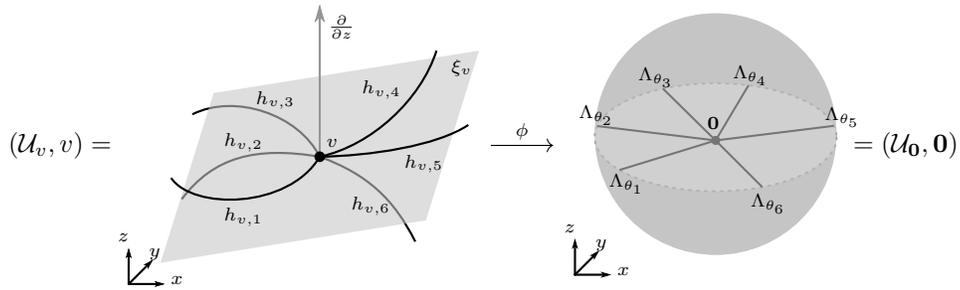

\[
\begin{tikzcd}
(\cU_v,v)=\vcenter{\hbox{\scriptsize\def\sfvgscale{0.8}\input{cyclicorder_input.tex}}} \ar[r,"\phi"] &
\vcenter{\hbox{\scriptsize\input{Darboux_input.tex}}}=(\cU_\origin,\origin)
\end{tikzcd}
\]
\caption{A Darboux neighborhood $\cU_v$ and half-edges of a vertex $v$}
\label{fig:Darboux chart}
\end{figure}

This is nothing but a generalization of the Darboux theorems, such as \cite[Theorem~2.5.1]{Geiges2008} for a single Legendrian and \cite[Lemma~3.2]{ABK2018} for two intersecting Legendrians, but we exclude the central vertex since one can not control how singular the arcs formed by two half-edges are.

However, the contactomorphism $\phi_v$ defined above extends to the vertex $v$ {\em up to equivalence}. In other words, for any $\Lambda$, $v\in\cV_\Lambda$ and its Darboux neighborhood $\cU_v$, there exists an equivalent triple $(\cU_{v'}, \cU_{v'}\cap \Lambda', v')$ via a Legendrian isotopy $\Lambda_t$ such that
\[
(\cU_v,\cU_v\cap\Lambda,v)\stackrel{\Lambda_t}\sim (\cU_{v'}, \cU_{v'}\cap \Lambda', v') \stackrel{\phi_{v'}}{\simeq} (\cU_\origin, \cU_\origin\cap T_\Theta,\origin).
\]
Indeed, one can find such $\Lambda_t$ whose support is arbitrarily small.

\begin{definition}[Tameness]\label{definition:tameness}
We say that a Legendrian graph $\Lambda$ is {\em tame} if the contactomorphism $\phi_v$ for a Darboux neighborhood extends to the vertex $v$, and we denote the set of all tame Legendrian graphs by $\bar\cLG$.
\end{definition}

\begin{assumption}
Unless mentioned otherwise, we assume that all Legendrian graphs are tame throughout this paper.
\end{assumption}

It is easy to see from Theorem~\ref{theorem:Darboux} that one can equip a cyclic order between half-edges at each vertex $v$ by using the co-orientation with respect to the Reeb direction $\partial_z$ for the contact form $dz+xdy-ydx$. This is observed in \cite[Remark~3.3]{OP2014} and also in \cite[\S3.2]{ABK2018}.

\begin{assumption}
For each $v\in\cV_\Lambda$, the intersection $\cU_v\cap \Lambda$ is the same as the union of half-edges at $v=\Lambda(\ttv)$,
\[
\cU_v\cap\Lambda=\bigcup_{h_{v,i}\in\cH_v} h_{v,i}=\bigcup_{\tth_{\ttv,i}\in\ttH_\ttv}\Lambda(\tth_{\ttv,i})
\]
and moreover, half-edges adjacent to $v$
\[
\cH_v=\{h_{v,1},\cdots,h_{v,m}\}
\]
are cyclically ordered as written.
\end{assumption}

Let $\Lambda_t$ be a Legendrian isotopy between Legendrian graphs. Then the angles between two consecutive half-edges at each vertex may vary but their relative positions keep unchanged during the isotopy. Hence this implies that there exist cyclic-order-preserving bijections 
\[
\cH_{v_0}\simeq\cH_{v_t}\simeq\cH_{v_1},\quad v_t\coloneqq\Lambda_t(\ttv)
\]
for any $\ttv\in \ttV_\Gamma$ and $t\in[0,1]$.

\subsection{Lagrangian projection}
Let $\pi_L:\RR^3\to\RR^2_{xy}$ be the projection onto the $xy$-plane $\RR^2_{xy}$.

For $\Lambda=(\cV_\Lambda,\cE_\Lambda)\in\bar\cLG$, the {\em (Lagrangian) projection $\sL=(\sV_\sL,\sE_\sL)$ of $\Lambda$} is defined as the composition 
\[
\begin{tikzcd}
\sL:\Gamma\ar[r,"\Lambda"]&\RR^3_{\rotation}\ar[r,"\pi_L"]&\RR^2_{xy}.
\end{tikzcd}
\]

We also define the sets of vertices $\sV_\sL$, edges $\sE_\sL$ and half-edges $\sH_\sfv$ at $\sfv$ as the images of the corresponding sets of $\Lambda$ under $\pi_L$. Indeed, for $\sfv=\pi_L(v)$,
\begin{align*}
\sV_\sL&\coloneqq\pi_L(\cV_\Lambda),&
\sE_\sL&\coloneqq\pi_L(\cE_\Lambda),&
\sH_\sfv&\coloneqq\pi_L(\cH_v)=\{\sfh_{\sfv,1},\cdots,\sfh_{\sfv,m}\mid\sfh_{\sfv,i}=\pi_L(h_{v,i})\}.
\end{align*}

Then as before, we will use $\sL$ to denote the image $\sL(\Gamma)$. Note that $\sH_\sfv$ inherits the cyclic order from $\cH_v$ via $\pi_L$.

\begin{definition}[General position and regularity]\label{defn:regular projection}
We say that $\Lambda$ is in a {\em general position}, or its projection $\sL=\pi_L(\Lambda)$ is {\em regular} if it satisfies the following.
\begin{enumerate}
\item There are only finitely many transverse double points in $\sL$ and no vertices of $\Lambda$ are double points.
\item No half-edges of $\Lambda$ are parallel to the $y$-axis at their vertices.
\end{enumerate}

We denote the sets of all Legendrian graphs in a general position by $\cLG$.
\end{definition}

\begin{figure}[ht]
\input{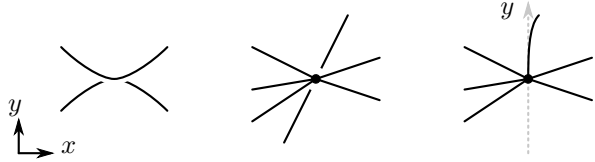}
\caption{Non-regular projections}
\end{figure}

\begin{remark}
Since each edge $e\in\cE_\Lambda$ is smooth Legendrian, $e$ cannot be parallel to the $z$-axis. That is, the restriction $\pi_L|_e$ always becomes a smooth immersion.

The second condition is equivalent to that the vertex is not cuspidal in the front projection.
Moreover, due to the second condition, regularity defined as above is more restrictive than the usual notion of regularity coming from smooth knot (or spatial graph) theory.
\end{remark}

It is obvious that for any $\Lambda\in\bar\cLG$, one can obtain a Legendrian graph $\Lambda'$ from $\Lambda$ by perturbing slightly so that $\Lambda'$ is equivalent to $\Lambda$ and is in general position. In this sense, we may say that $\cLG$ is {\em dense} in $\bar\cLG$.

Throughout this section, we assume that $\Lambda\in\cLG$ and $\sL=\pi_L(\Lambda)$.

\begin{definition}[Standard neighborhoods of projections]
For each vertex $\sfv\in\sV_\sL$, a {\em standard neighborhood $\sU_\sfv\subset\RR^2_{xy}$} is a small open ball neighborhood of $\sfv$ such that $\sU_\sfv$ contains no special points other than $\sfv$. 
We denote the boundary of the closure $\cl{\sU_\sfv}$ by $\sO_\sfv$.
\begin{align*}
\sfv&\in\sU_\sfv\subset\RR^2_{xy},&
\sO_\sfv&\coloneqq\partial(\cl{\sU_\sfv})
\end{align*}

We denote by $\sector_{\sfv,i}$ the $i$-th circular sector at $\sfv$ bounded by two half-edges $\sfh_{\sfv,i}$, $\sfh_{\sfv,i+1}$ and the reference circle $\sO_\sfv$.
\end{definition}

\begin{assumption}
We assume that the Darboux neighborhood $\cU_v$ intersects $\Lambda$ exactly as much as the inverse $\pi_L^{-1}(\sU_\sfv)$ does.
In other words, the restriction of $\pi_L|$ induces a homeomorphism
\[
\pi_L|:\cU_v\cap\Lambda \to\sU_\sfv\cap \sL.
\]
\end{assumption}

\begin{figure}[ht]
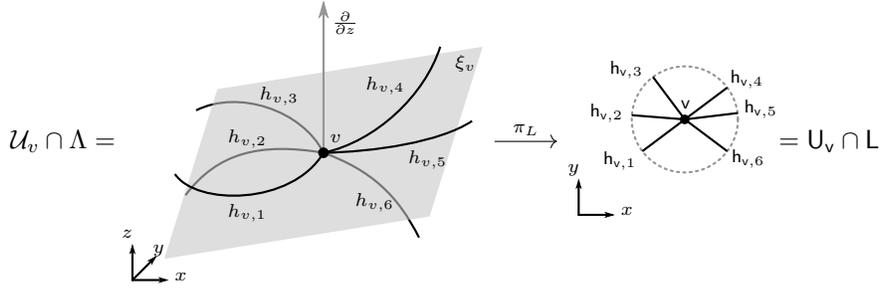

\[
\begin{tikzcd}
\cU_v\cap\Lambda=\vcenter{\hbox{\scriptsize\def\sfvgscale{0.8}\input{cyclicorder_input.tex}}} \ar[r,"\pi_L"] &
\vcenter{\hbox{\scriptsize\input{vertex_lag_input.tex}}}=\sU_\sfv\cap \sL
\end{tikzcd}
\]
\caption{A Lagrangian projection of a Darboux neighborhood $\cU_v$ and a standard neighborhood $\sU_\sfv$.}
\label{fig:vertex projection}
\end{figure}

However, the Lagrangian projection $\pi_L(\cU_v)$ of a Darboux chart is not necessarily the same as $\sU_\sfv$. Instead, one can choose $\cU_v$ and $\sU_\sfv$ such that $\pi_L(\cU_v)$ is very close to $\sU_\sfv$ in the sense that the Hausdorff distance $d_H(\pi_L(\cU_v),\sU_\sfv)$ is arbitrarily small.

\subsection{Reidemeister moves}
Let $\Lambda_t:\Gamma\times[0,1]\to\RR^3_{\rotation}$ be an isotopy between two Legendrian graphs $\Lambda_0, \Lambda_1\in\cLG$, and let $\sL_t\coloneqq\pi_L(\Lambda_t)$ be the Lagrangian projection of $\Lambda_t$.

Note that $\Lambda_t\in\bar\cLG$ for all $t$ but $\Lambda_t\not\in\cLG$ in general. Since $\cLG$ is dense in $\bar\cLG$, without loss of generality, we may assume that $\Lambda_t\in\cLG$ for all but finitely many $t_i$'s by perturbing $\Lambda_t$ slightly if necessary.
Then a bifurcation at each $t_i$ corresponds to a composition of {\em Reidemeister moves} in $\RR_{xy}^2$ depicted in Figure~\ref{fig:RM}. A dashed gray line is a line passing through a vertex which is parallel to the $y$-axis.

\begin{figure}[ht]
\[
\begin{tikzcd}[row sep=0pc]
\vcenter{\hbox{\input{RM_0_a_1_input.tex}}}\ar[r,"\rm(0_a)",leftrightarrow]\quad&\quad
\vcenter{\hbox{\input{RM_0_a_2_input.tex}}}\quad&\quad
\vcenter{\hbox{\input{RM_0_b_1_input.tex}}}\ar[r,"\rm(0_b)",leftrightarrow]\quad&\quad
\vcenter{\hbox{\input{RM_0_b_2_input.tex}}}\\
\vcenter{\hbox{\includegraphics{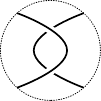}}}\ar[r,"\rm(II)",leftrightarrow]\quad&\quad
\vcenter{\hbox{\includegraphics{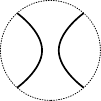}}}\quad&\quad
\vcenter{\hbox{\includegraphics{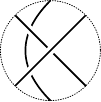}}}\ar[r,"\rm(III)",leftrightarrow]\quad&\quad
\vcenter{\hbox{\includegraphics{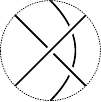}}}\\
\vcenter{\hbox{\includegraphics{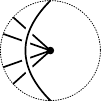}}}\ar[r,"\rm(IV_a)",leftrightarrow]\quad&\quad
\vcenter{\hbox{\includegraphics{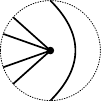}}}\quad&\quad
\vcenter{\hbox{\includegraphics{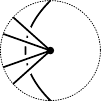}}}\ar[r,"\rm(IV_b)",leftrightarrow]\quad&\quad
\vcenter{\hbox{\includegraphics{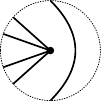}}}
\end{tikzcd}
\]
\caption{Reidemeister moves in Lagrangian projections}
\label{fig:RM}
\end{figure}

Any projection of local isotopy on $\RR^3_{\rotation}$ indeed preserves the {\em area} in the following sense:
the planar isotopy below is realizable as a projection of local isotopy on $\RR^3_{\rotation}$ if and only if two shaded regions have the same area.
\[
\begin{tikzcd}
\vcenter{\hbox{\includegraphics{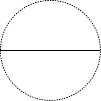}}}\ar[r,leftrightarrow]\quad&\quad
\vcenter{\hbox{\includegraphics{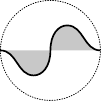}}}
\end{tikzcd}
\]

Therefore rigorously speaking, Reidemeister moves depicted in Figure~\ref{fig:RM} are only schematic diagrams but for convenience's sake, we may ignore this area-preserving condition unless mentioned otherwise.

\begin{remark}
The converse is not always true, that is, the diagram obtained by Reidemeister moves may not be realizable as (the schematic diagram of) the projection of any Legendrian graph.
\end{remark}

\begin{remark}
Reidemeister moves $\rm(0_a)$ and $\rm(0_b)$ are usually not considered since they are planar isotopies. However, since we are using a more restrictive notion of the regularity of the projection than usual, these two moves are necessary.
\end{remark}

\begin{definition}[Legendrian mirror]
The contact involution $\mu:\RR^3_{\rotation}\to\RR^3_{\rotation}$ defined by 
\[
\mu:(x,y,z)\mapsto(x,-y,-z)
\]
is called the {\em Legendrian mirror}.
\end{definition}

\begin{remark}\label{remark:Legendrian mirror}
It is obvious that Reidemeister moves $\rm(IV_a)$ are related with $\rm(IV_b)$ via the Legendrian mirror $\mu$.
That is, if two projections of Legendrian links $\sL_0$ and $\sL_1$ are related with $\rm(IV_a)$, then the projections of their mirrors $\mu(\sL_0)$ and $\mu(\sL_1)$ are related to $\rm(IV_b)$.
\[
\begin{tikzcd}[column sep=3pc]
\sL_0\ar[r,"\rm(IV_a)",leftrightarrow]\ar[d,"\mu"',leftrightarrow]\ar[rd,"\circlearrowright",phantom] &\sL_1\ar[d,"\mu",leftrightarrow]\\
\mu(\sL_0)\ar[r,"\rm(IV_b)",leftrightarrow] &\mu(\sL_1)
\end{tikzcd}
\]
\end{remark}

\begin{assumption}
We assume that the isotopy $\Lambda_t$ is fixed outside of a small neighborhood where each Reidemeister move occurs.
\end{assumption}

\subsection{Maslov potentials on Legendrian graphs}\label{sec:potential}
For a Legendrian link $\Lambda$, a {\em Maslov potential} is a function with values in a cyclic group $\fR$---$\ZZ$ or $\Zmod{m}$ in practice---whose domain is the set of subarcs obtained by removing cusps from the front projection of $\Lambda$, such that the difference between potentials of two adjacent arcs in the front projection is precisely $1_\fR\in\fR$.
Therefore the changes of the potential are concentrated at the cusps at where the tangent space of the Legendrian is parallel to the $y$-axis.

On the other hand, if $\Lambda$ is a knot, then the relations between two adjacent arcs determine a unique Maslov potential up to translation but if $\Lambda$ is a link of $n$ components, then the set of Maslov potential is an affine space modeled on $\fR^{n-1}$.
That is, if we regard circles as graphs consisting of one vertex and one edge, then this means that the choices of Maslov potentials on different edges are independent to each other.

In this section, we define $\fR$-valued Maslov potentials on the Lagrangian projections of Legendrian graphs by mimicking the above two ideas as follows:
\begin{enumerate}
\item Whenever one passes the tangent space parallel to $y$-axis, the Maslov potential changes;
\item Each edge possesses a Maslov potential independently.
\end{enumerate}

\begin{assumption}
Throughout this paper, the abelian group $\fR$ is assumed to be a cyclic group with chosen generator $1_\fR$. If $\fR$ is trivial, then we regard $1_\fR$ as the identity $0$.
\end{assumption}

Consider a {\em stretching map $f_M:\RR^2_{xy}\to\RR^2_{xy}$} which is a diffeomorphism defined as
\[
f_M:(x,y)\mapsto(Mx,y/M)\footnotemark
\]
and\footnotetext{This diffeomorphism lifts to a contactomorphism $\tilde f_M:(x,y,z)\mapsto(Mx, y/M,z)$ on $\RR^3_{\rotation}$ via $\pi_L$.}
let $\sL_M\coloneqq f_M(\sL)$. For large enough $M$, the slopes $dy/dx$ of half-edges of $\sL_M$ at each vertex become close to 0 since $\Lambda$ is in general position and no half-edges are parallel to the $y$-axis. 

Since our graph is directed, we have a parametrization $\sfe_M(t)$ for each edge $\sfe\in \sE_\sL$ and $\sfe_M\coloneqq f_M(\sfe)\in \sE_{\sL_M}$.
Hence its initial and terminal half-edges are well-defined and denoted by $\sfh_{s(\sfe_M)}$ and $\sfh_{t(\sfe_M)}$, respectively.
Then by perturbing both ends slightly, we may assume that two half-edges are parallel to the $x$-axis. Then the tangent space $D\sfe_M(t)$ defines a loop in $\gr(1,T_\mathbf{0}\RR^2_{xy})$ by the canonical identification $T_{\sfe_M(t)}\RR^2_{xy}\simeq T_\mathbf{0}\RR^2_{xy}$ for each $t$. Since $\pi_1(\gr(1,T_\mathbf{0}\RR^2_{xy}))\simeq \ZZ\langle 1_{\gr}\rangle$ generated by the counterclockwise $\pi$-rotation $1_{\gr}$, the homotopy class $[D\sfe_M(t)]$ gives us a number $n_\sfe$ such that 
\[
[D\sfe_M(t)]=n_\sfe\cdot 1_{\gr}.
\]

Let $\fX_\sL$ be the free $\fR$-module generated by $1$ and all half-edges of $\sL$, and $\fX_\sL^*$ the $\fR$-dual of $\fX_\sL$, and let $\fX_\sL'$ be the affine subspace of $\fX_\sL^*$ consisting of $\fR$-linear functions sending $1\in \fX_\sL$ to $1_\fR\in \fR$.
\begin{align*}
\fX_\sL&\coloneqq\fR\langle 1\rangle\oplus \bigoplus_{\sfv\in\sV_\sL}\fR\langle \sH_\sfv\rangle,&
\fX_\sL^*&\coloneqq\hom_\fR(\fX_\sL,\fR),&
\fX_\sL'&\coloneqq\{\fP\in\fX_\sL^*\mid\fP(1)=1_\fR\}.
\end{align*}
Then we have
\begin{align*}
\fX_\sL&\simeq\fR^{2E+1}\simeq\fX_\sL^*,&
\fX'_\sL&\stackrel{\textit{aff}}\simeq\fR^{2E},
\end{align*}
where $E=\#\sE_\sL(=\#\cE_\Lambda)$ is the number of edges of $\sL$ and ``$\stackrel{\textit{aff}}\simeq$'' denotes an affine isomorphism.

We define an element $\ft_\sfe\in\fX_\sL$ for each $\sfe\in\sE_\sL$ as
\begin{align*}
\ft_\sfe&\coloneqq\sfh_{s(\sfe)}-\sfh_{t(\sfe)}-n_\sfe\cdot 1
\end{align*}
and let $\fT_\sL\coloneqq\fR\langle \ft_\sfe\mid\sfe\in \sE_\sL\rangle$ be the submodule of $\fX_\sL$ generated by all $\ft_\sfe$'s.
\begin{remark}\label{remark:potentials_not_depending_on_orientation}
The submodue $\fT_\sL$ does not depend on the choice of orientations on edges. That is, if we change the orientation of $\sfe$, then 
\begin{align*}
\sfh_{s(-\sfe)}&=\sfh_{t(\sfe)},&
\sfh_{t(-\sfe)}&=\sfh_{s(\sfe)},&
n_{-\sfe}=-n_\sfe,
\end{align*}
since $[D(-\sfe_M(t))] = -[D\sfe_M(t)] = -n_\sfe\cdot 1_{\gr}$.
Therefore, 
\begin{align*}
\ft_{-\sfe}&=\sfh_{s(-\sfe)}-\sfh_{t(-\sfe)}-n_{-\sfe}\cdot 1=\sfh_{t(\sfe)}-\sfh_{s(\sfe)}+n_\sfe\cdot 1=-\ft_\sfe,
\end{align*}
and this implies that $\fT_\sL$ is well-defined.
\end{remark}

\begin{definition}[Maslov potential]
We call $\fP\in \fX_\sL'$ an {\em $\fR$-valued Maslov potential} if $\fP$ vanishes on $\fT_\sL$, and we say that two potentials $\fP$ and $\fP'$ are {\em equivalent} if they are the same up to translation, i.e., 
\[
\fP-\fP'\equiv\fr
\]
for some $\fr\in\fR$. We write as $\fP\sim\fP'$.
\end{definition}

\begin{notation}
We denote the set of all Maslov potentials by $\tilde\fG(\sL;\fR)$ and the set of all equivalence classes of Maslov potentials by $\fG(\sL;\fR)$.
\begin{align*}
\tilde\fG(\sL;\fR)&\coloneqq \{\fP\in \fX_\sL'\mid\fT_\sL\subset\ker\fP\},&
\fG(\sL;\fR)&\coloneqq\tilde\fG(\sL;\fR)/\sim
\end{align*}
\end{notation}

Since $\ft_\sfe$'s are linearly independent, $\fT_\sL$ is of rank $E$ and moreover, it is a direct summand of $\fX_\sL$ whose complement is an $\fR$-submodule of rank $(E+1)$ isomorphic to $\fX_\sL\big/\fT_\sL$.
\begin{align*}
\fT_\sL\oplus \left(\fX_\sL\big/\fT_\sL\right)&\simeq \fX_\sL,&
\fT_\sL&\simeq\fR^E,&
\fX_\sL\big/\fT_\sL&\simeq\fR^{E+1}
\end{align*}
Therefore
\begin{align*}
\tilde\fG(\sL;\fR)
&=\left\{\bar\fP\in\hom_\fR(\fX_\sL\big/\fT_\sL,\fR)\mid\bar\fP([1])=1_\fR\in \fR\right\}
\stackrel{\textit{aff}}\simeq \fR^E.
\end{align*}

Now we consider an action of $\fR$ on $\tilde\fG(\sL;\fR)$ defined as for $\fr\in\fR$,
\[
(\fr\cdot\fP)(\sfh) \coloneqq \fP(\sfh)+\fr.
\]
Then $\fG(\sL;\fR)$ is the same as the orbit space $\tilde\fG(\sL;\fR)\big/\fR$, and therefore
\[
\fG(\sL;\fR)= \tilde\fG(\sL;\fR)\big/ \fR\stackrel{\textit{aff}}\simeq \fR^{E-1}.
\]

\begin{definition}[Legendrian graphs with potential]
For a Legendrian graph $\Lambda\in\cLG$ with projection $\sL$, the set of Maslov potentials $\fG(\Lambda;\fR)$ for $\Lambda$ is defined as
\[
\fG(\Lambda;\fR)\coloneqq\fG(\sL;\fR).
\]

We denote by $\cLG_\fR$ the set of all Legendrian graphs in general position with $\fR$-valued Maslov potentials. That is,
\[
\cLG_\fR\coloneqq\{\cL=(\Lambda,\fP)\mid \Lambda\in\cLG, \fP\in\fG(\Lambda;\fR)\}.
\]
\end{definition}

Notice that for the trivial group $\mathbf{1}$, we have $\fG(\Lambda;\mathbf{1})=\{*\}$ and so $\cLG_{\mathbf{1}}\simeq\cLG$.

Indeed, the construction $\fG(\Lambda;-)$ is functorial. Namely, if we have a homomorphism $\mathfrak{f}:\fR\to\fR'$, then it induces a canonical morphism $\mathfrak{f}_*$
\[
\begin{tikzcd}
\tilde \fG(\Lambda;\fR)\ar[r,"\tilde{\mathfrak{f}}_*"]\ar[d,"/\fR"']& \tilde \fG(\Lambda;\fR)\otimes_{\fR} \fR' \ar[r,hook]& \tilde\fG(\Lambda;\fR')\ar[d,"/\fR'"]\\
\fG(\Lambda;\fR)\ar[rr,"\mathfrak{f}_*"]&& \fG(\Lambda;\fR').
\end{tikzcd}
\]

Furthermore if $\mathfrak{f}$ is epic, then so is $\tilde{\mathfrak{f}}_*$ and we have a natural isomorphism
\[
\tilde\fG(\Lambda;\fR)\otimes_\fR \fR'\stackrel{\textit{aff}}\simeq \tilde\fG(\Lambda;\fR')
\]
between affine spaces. Therefore we have an epimorphism
\begin{align}\label{eq:coefficient change}
\mathfrak{f}_*:\fG(\Lambda;\fR)\to \fG(\Lambda;\fR').
\end{align}
In particular, the induced map $\cLG_\fR\to\cLG_{\mathbf{1}}\simeq\cLG$ from $\fR\to\mathbf{1}$ is the same as the map forgetting about potentials which will be denoted by $\forget$. Then
\[
\forget^{-1}(\Lambda)=\fG(\Lambda;\fR)\stackrel{\textit{aff}}\simeq \fR^{\#\cE_\Lambda-1}.
\]
and the natural question is whether the equivalence on $\cLG$ given by Reidemeister moves lifts to $\cLG_\fR$ via $\forget$ or not. This is equivalent to whether the following implication is true or not:
\[
\Lambda\sim \Lambda'\Longrightarrow\cG(\Lambda;\fR)\stackrel{\textit{aff}}\simeq\cG(\Lambda';\fR).
\]
The answer is yes.
\begin{theorem}\label{theorem:equivalence on Legendrian graphs with potentials}
The equivalence on $\cLG$ lifts to $\cLG_\fR$. 
\end{theorem}
\begin{proof}
Let $\sL$ and $\sL'$ be the projections of $\Lambda$ and $\Lambda'$, respectively. 
Then it suffices to show that $\fG(\sL;\fR)$ and $\fG(\sL';\fR)$ are canonically affine isomorphic.

Without loss of generality, we may assume that $\sL'$ is obtained from $\sL$ by only one Reidemeister move, say ${\rm (F)}$, depicted in Figure~\ref{fig:RM}.
Then the move $\rm(F)$ induces bijections 
\begin{align*}
{\rm (F)}_\sE&:\sE_\sL\to\sE_{\sL'},&
{\rm (F)}_\sH&:\sH_\sL\to\sH_{\sL'}.
\end{align*}

Suppose that $\fP\in\tilde\fG(\sL;\fR)$. Unless $\rm(F)=\rm(0_a)$ or $\rm(0_b)$, we define 
$\fP'\in\fX_{\sL'}'$ as 
\begin{align*}
\fP'(\sfh')&\coloneqq\fP({\rm (F)}_\sH^{-1}(\sfh')).
\end{align*}
Then it is obvious that $\fP'\in\tilde\fG(\sL';\fR)$.

On the other hand, suppose that $\rm(F)=\rm(0_a)$ or $\rm(0_b)$. Indeed, we assume that $\sfh_i$ and $\sfh'_i={\rm (F)}_\sH(\sfh_i)$ are lying on the different side with respect to a line parallel to the $y$-axis.
Let $\sfe\in\sE_\sL$ and $\sfe'\in\sE_{\sL'}$ be edges containing $\sfh_i$ and $\sfh'_i$, respectively.

If we apply the stretching map $f_M$, the angle between two half-edges $f_M(\sfh_i)$ and $f_M(\sfh_i')$ gets closer to $\pi$, which contributes $\pm1_{\gr}\in\pi_1(\gr(1,\RR^2))$.
Hence the elements $\ft_\sfe$ and $\ft_{\sfe'}$ are as follows:
\[
\begin{tikzcd}
\ft_\sfe\coloneqq\sfh_i-\sfh_j-n_\sfe\cdot 1\quad \ar[r,"\rm(F)",leftrightarrow]&\quad \ft_{\sfe'}\coloneqq\sfh_i'-\sfh_j'-n_{\sfe'}\cdot1,
\end{tikzcd}
\]
where $\sfh_j$ and $\sfh_j'={\rm (F)}_\sH(\sfh_j)$ are half-edges of $\sfe$ and $\sfe'$ different from $\sfh_i$ and $\sfh_i'$, respectively, and
\[
n_{\sfe'}=n_\sfe\pm1.
\]
Therefore if we define $\fP'$ as 
\[
\fP'(\sfh')=\begin{cases}
\fP(\sfh) & \sfh'={\rm (F)}_\sH(\sfh)\neq \sfh_i';\\
\fP(\sfh)\mp1_\fR & \sfh'={\rm (F)}_\sH(\sfh)=\sfh_i',
\end{cases}
\]
then $\fP'$ becomes a Maslov potential for $\sL'$ as desired.
\end{proof}

\section{Generalized stable tame isomorphism for DGAs}\label{sec:Peripheral Structures and Generalized stable tame isomorphism for DGAs}
For algebraic preliminaries, we introduce the notions of peripheral structures and generalized stable tame isomorphisms for DGAs which are essential in the invariance theorem for Legendrian graphs. 
Throughout this section, a DGA $\cA$ is a triple $(A,|\cdot|,\partial)$, where $A\coloneqq\ZZ\langle \sG\rangle$ is a free unital associative algebra over $\ZZ$ with a chosen set $\sG$ of generators, $|\cdot|$ is a $\ZZ$-grading on $A$, and $\partial:A\to A$ is a graded linear map of degree $-1$ satisfying the Leibniz rule and $\partial^2=0$.

Let $\cA'=(A'=\ZZ\langle\sG'\rangle,|\cdot|',\partial')$ and $\cA''=(A''=\ZZ\langle\sG''\rangle,|\cdot|'',\partial'')$ be DGAs.
Suppose that there are two DGA morphisms $\phi':\cI\to\cA'$ and $\phi'':\cI\to\cA''$ from $\cI=(I=\ZZ\langle \sI\rangle, |\cdot|_I, \partial_I)$.
Then the {\em amalgamated free product} $\cA'\ast_{\cI}\cA''$ of $\cA'$ and $\cA''$ over $\cI$ is defined as follows:
\begin{align*}
A&\coloneqq\ZZ\left\langle\sG',\sG''~\middle|~ \phi'(\rho)=\phi''(\rho), \forall\rho\in\sI\right\rangle.
\end{align*}
Especially, if $\cI=(\ZZ,0,0)$ is the trivial algebra, then we have the {\em free product} $\cA'\ast \cA''$.

\begin{remark}
The amalgamated free product may not be the same as the homotopy push-out of dg category, but if $\cI$ acts freely on one of $\cA'$ or $\cA''$ via $\phi'$ or $\phi''$, then it is the same as (up to chain equivalence) the homotopy push-out of $\phi'$ and $\phi''$. 
In this case, we will say that it is the {\em push-out} rather than the amalgamated free product.
\end{remark}

\subsection{Peripheral structures}

\begin{definition}[{\cite[\S2.3]{EN2015}}]For $m\ge 1$, let $\cI_m=(I_m,|\cdot|, \partial_m)$ be the free associative unital $\ZZ$-graded algebra defined as
\begin{align*}
I_m&\coloneqq\ZZ\langle\rho_{i,\ell}\mid i\in \Zmod{m}, \ell\ge 1\rangle,\\
\partial_m \rho_{i,\ell}&\coloneqq\delta_{\ell,m}+\sum_{j=1}^{\ell-1} (-1)^{|\rho_{i,j}|-1}\rho_{i,j} \rho_{i+j,\ell-j},\\
\partial_m(ab)&\coloneqq(\partial_m a) b+(-1)^{|a|}a(\partial_m b)
\end{align*}
for any homogeneous elements $a$ and $b$ in $I_m$.
\end{definition}

\begin{remark}
The DGA $\cI_m$ is not uniquely determined since $I_m$ may possess multiple different gradings.
\end{remark}

\begin{definition}[Peripheral structure]\label{def:peripheral structure}
A {\em peripheral structure $\bp$} of $\cA$ is a DGA morphism
\[
\bp:\cI_m\to \cA
\]
where $\cI_m\coloneqq(I_m,|\cdot|,\partial_m)$ is a DGA for some $m\ge 1$ and grading $|\cdot|$ on $I_m$.
\end{definition}

\begin{example}[Distinguished peripheral structure]\label{ex:distinguished}
We consider the DGA $\cI_\emptyset\coloneqq(I_2,|\cdot|_\emptyset,\partial_2)$ and the map $\bp_\emptyset:\cI_\emptyset\to \cA$ defined as follows: 
\begin{align*}
|\rho_{i,\ell}|_\emptyset&\coloneqq(\ell-1),&
\bp_\emptyset(\rho_{i,\ell})&\coloneqq\begin{cases}
1 & \ell=1;\\
0 & \ell\neq 1.
\end{cases}
\end{align*}
Then $\bp_\emptyset$ becomes a DGA morphism, called the {\em distinguished peripheral structure}.
\end{example}

From now on, we will consider a pair $(\cA,\cP)$, where $\cP$ is a collection of peripheral structures containing $\bp_\emptyset$.
Recall that a DGA homomorphism $f:\cA\to \cA'$ is called {\em elementary} if $f$ maps all generators of $\cA$ to generators of $\cA'$ except for only one generator $g\in\cA$ whose image is the sum of a generator $g'\in\cA$ and words not involving $g'$.

\begin{definition}[Tame isomorphisms]
Let $(\cA,\cP)$ and $(\cA',\cP')$ be two DGAs with peripheral structures. 
\begin{enumerate}
\item A {\em homomorphism} $f:(\cA,\cP)\to(\cA',\cP')$ is a DGA homomorphism $f:\cA\to\cA'$ satisfying that the function
\[
\setlength{\arraycolsep}{2pt}
\begin{array}{rccc}
f_*:&\cP&\to&\cP'\\
&\bp&\mapsto& f\circ \bp
\end{array}
\]
induced by post-composition with $f$ is well-defined.

\item A homomorphism $f$ is an {\em elementary isomorphism} if $f:\cA\to\cA'$ is an elementary isomorphism between DGAs, see \cite[\S 2.2]{Chekanov2002}, and $f_*:\cP\to\cP'$ is a bijection.

\item A finite composition of elementary isomorphisms is called a {\em tame isomorphism}, and we denote it by
\[
(\cA, \cP)\tameisom (\cA',\cP')
\]
if $(\cA,\cP)$ and $(\cA',\cP')$ are tame isomorphic.
\end{enumerate}
\end{definition}

Any homomorphism $f$ preserves the distinguished peripheral structure $\bp_\emptyset$ since $\im(\bp_\emptyset)$ is the coefficient ring $\ZZ$. Therefore the above definition for pairs is equivalent to the usual definition for DGAs if we set 
\[
\cP=\cP'=\{\bp_\emptyset\}.
\]

\subsection{Generalized stabilizations}
Let $\bp:\cI_m=(I_m,|\cdot|,\partial_m)\to \cA$ be a peripheral structure.
For each $\fd\in \fR$, we define $\cF^{\fd \pm}_\bp\coloneqq(F_m^\pm,|\cdot|^{\fd},\bar\partial_m)$ as follows:
\begin{enumerate}
\item The algebra $F_m^+$ (resp. $F_m^-$) is the free product of $I_m$ and the unital free associative algebra generated by the $\bar\sfc_i$'s (resp. $\bar\sfc_{-i}$'s)
\begin{align*}
F_m^+& \coloneqq \ZZ\langle\{\bar\rho_{i,\ell}\}\amalg \{\bar\sfc_i\}\rangle\simeq I_m\ast\ZZ\langle \bar\sfc_1,\dots, \bar\sfc_m\rangle;\\
F_m^-& \coloneqq \ZZ\langle\{\bar\rho_{i,\ell}\}\amalg \{\bar\sfc_{-i}\}\rangle\simeq I_m\ast\ZZ\langle \bar\sfc_{-1},\dots, \bar\sfc_{-m}\rangle.
\end{align*}
\item The grading $|\cdot|^{\fd}$ is defined as
\begin{align*}
|\bar\rho_{i,\ell}|^{\fd}&\coloneqq|\rho_{i,\ell}|, &i\in\Zmod{m}, \ell\ge 1;\\
|\bar\sfc_1|^{\fd}&\coloneqq \fd;\\
|\bar\sfc_k|^{\fd}&\coloneqq \fd+|\rho_{1,k-1}|+1,  &1<k\le m; \\
|\bar\sfc_{-1}|^{\fd}&\coloneqq \fd;\\
|\bar\sfc_{-k}|^{\fd}&\coloneqq \fd+|\rho_{-k,k-1}|+1,  &1<k\le m.
\end{align*}
\item The differential $\bar\partial_m$ is defined as
\begin{align*}
\bar\partial_m \bar\rho_{i,\ell}&\coloneqq
\delta_{\ell,m}+\sum_{j=1}^{\ell-1} (-1)^{|\bar\rho_{i,j}|^{\fd}-1}\bar\rho_{i,j} \bar\rho_{i+j,\ell-j},& i\in \Zmod{m}, \ell\ge 1;\\
\bar\partial_m \bar\sfc_1&\coloneqq0;\\
\bar\partial_m \bar\sfc_{k}&\coloneqq (-1)^{|\bar\sfc_k|^\fd-1}\sum_{i=1}^{k-1}\bar\sfc_i \bar\rho_{i,k-i},& 1<k\le m;\\
\bar\partial_m \bar\sfc_{-1}&\coloneqq0;\\
\bar\partial_m \bar\sfc_{-k}&\coloneqq (-1)^{|\bar\sfc_{-k}|^\fd-1}\sum_{i=1}^{k-1}\bar\rho_{-k,k-i}\bar\sfc_{-i},& 1< k\le m,
\end{align*}
where $\bp^{\fd \pm}:I_m\to F_m^\pm$ is the canonical inclusion defined as $\rho_{i,\ell}\mapsto \bar\rho_{i,\ell}$.
\end{enumerate}

Then by definition, $\bp^{\fd \pm}$ induces a DGA morphism
\[
\bp^{\fd \pm}:\cI_m=(I_m,|\cdot|,\partial_m)\to \cF_\bp^{\fd \pm}=(F_m^\pm, |\cdot|^{\fd},\bar\partial_m),
\]
or equivalently, a peripheral structure for $\cF_\bp^{\fd \pm}$.

\begin{definition}[Generalized stabilization]
Let $(\cA,\cP)$ be a DGA with a collection of peripheral structures. Suppose that $(\bp:\cI_m\to\cA)\in\cP$ and $\fd\in \fR$ are given.
The {\em $\fd$-th generalized $(\pm)$-stabilization} $S_\bp^{\fd \pm}(\cA,\cP)$ of $(\cA,\cP)$ with respect to $\bp$ is the pair
\[
S_\bp^{\fd \pm}(\cA,\cP)\coloneqq(\cA_\bp^{\fd \pm}, \cP_\bp^{\fd \pm})
\]
defined as follows:
\begin{enumerate}
\item The DGA $\cA_\bp^{\fd \pm}\coloneqq(A_\bp^{\fd \pm}, |\cdot|_\bp^{\fd}, \partial_\bp^{\fd})$ is the push-out of $\bp$ and $\bp^{\fd \pm}$
\[
\begin{tikzcd}[column sep=4pc]
\cI_m=(I_m,|\cdot|,\partial_m)\ar[r,"\bp^{\fd \pm}"]\ar[d,"\bp"']& \cF_\bp^{\fd \pm}=(F_m^\pm,|\cdot|^{\fd},\bar\partial_m)\ar[d,"\exists\iota_\bp^\pm"]\\
\cA=(A,|\cdot|,\partial)\ar[r,"\exists\iota_\bp^{\fd \pm}"]&\cA_\bp^{\fd \pm}=(A_\bp^{\fd \pm}, |\cdot|^{\fd}, \bar\partial).
\end{tikzcd}
\]
\item The collection $\cP_\bp^{\fd \pm}$ of induced peripheral structures is defined by post-composition with the canonical map 
$\iota_\bp^{\fd \pm}:\cA\to \cA_\bp^{\fd \pm}$ so that
\[
\cP_\bp^{\fd \pm}\coloneqq \{ \iota_\bp^{\fd \pm}\circ \bp':\cI_{\bp'}\to \cA\to \cA_\bp^{\fd \pm}\mid \bp'\in \cP\}.
\]
\end{enumerate}
\end{definition}

\begin{remark}\label{remark:EN_stabilizations}
Each generalized stabilization defined above is a composition of {\em stabilization} and {\em destabilization} up to tame isomorphisms in the sense of Ekholm and Ng \cite[Definition~2.16]{EN2015}. 
Notice that infinitely many pairs of generators are necessary to make our generalized stabilizations.
\end{remark}

\begin{remark}\label{remark:anti-isomorphism and stabilizations}
The two DGA constructions $\cF_\bp^{\fd\pm}$ are related via the DGA antiisomorphism, which is an isomorphism as abelian groups, preserves gradings and differentials but reverses the order of multiplication,
\begin{align*}
\mu_*&:\cF_\bp^{\fd+}\to\cF_{\mu^*(\bp)}^{\fd-},&
\mu_*(\bar\sfc_i)&=\bar\sfc_{-i},&
\mu_*(\bar\rho_{i,\ell})&=\bar\rho_{-i-\ell,\ell},
\end{align*}
where $\mu^*(\bp):\mu^*(\cI_m)\to\mu^*(\cA)$ is the induced peripheral structure of $\bp$ and $\mu^*(-)$ is the DGA obtained by reversing the order of multiplication so that $\mu$ becomes an antiisomorphism.

Furthermore, so are the positive and negative $\fd$-th generalized stabilizations. That is, there is an anti-isomorphism of DGAs
\[
\mu_*:S_{\mu^*(\bp)}^{\fd+}(\mu^*(\cA),\mu^*(\cP))\to S_\bp^{\fd-}(\cA,\cP).
\]
\end{remark}

\begin{example}[Generalized stabilization with $\bp_\emptyset$]
We consider the generalized stabilization $\cA_{\bp_\emptyset}^{\fd+}$ with respect to the distinguished peripheral structure $\bp_\emptyset$.
Then since $|\rho_{1,1}|=|\rho_{2,1}|=0$, the grading and differential on $\{\bar\sfc_1,\bar\sfc_2\}$ become
\begin{align*}
|\bar\sfc_1|^{\fd}&\coloneqq \fd,& \bar\partial_2 \bar\sfc_1&\coloneqq 0;\\
|\bar\sfc_2|^{\fd}&\coloneqq \fd+|\rho_{1,1}|+1=\fd+1, &
\bar\partial_2 \bar\sfc_2&\coloneqq (-1)^{|\bar\sfc_1|}\bar\sfc_1 \bar\rho_{1,1}.
\end{align*}

Then the push-out of $\bp_\emptyset:I_2\to A$ and $\bp^{\fd+}:I_2\to F_2^+$ is isomorphic to
\[
A\ast_{I_2} F_2^+ \simeq A \ast \ZZ\langle \bar\sfc_1, \bar\sfc_2\rangle
\]
by the definition of $\bp_\emptyset:\rho_{i,\ell}\mapsto\delta_{\ell,1}$. Moreover, the DGA $\cA_{\bp_\emptyset}^{\fd+}$ is isomorphic to the free product
\[
\cA_{\bp_\emptyset}^{\fd+}\simeq \cA\ast E({\fd}),
\]
where the DGA $E({\fd})\coloneqq (\ZZ\langle \bar\sfc_1,\bar\sfc_2\rangle,|\cdot|_{E({\fd})},\partial_{E({\fd})})$
is defined as
\begin{align*}
|\bar\sfc_1|_{E({\fd})}&=\fd, & \partial_{E({\fd})}(\bar\sfc_1)&=0;\\
|\bar\sfc_2|_{E({\fd})}&=\fd+1, & \partial_{E({\fd})}(\bar\sfc_2)&=(-1)^{|\bar\sfc_1|}\bar\sfc_1.
\end{align*}

Moreover, one can ignore the sign by using $\bar\sfc_1'\coloneqq(-1)^{|\bar\sfc_1|}\bar\sfc_1$.
Therefore $\cA_{\bp_\emptyset}^{\fd+}$ is nothing but what is called the {\em $\fd$-th stabilization} of $\cA$ in the literature.
\end{example}

\begin{definition}[Generalized stable-tame isomorphism]
We say that two DGAs $(\cA,\cP)$ and $(\cA', \cP')$ with collections of peripheral structures are {\em generalized stable-tame isomorphic} if $(\cA,\cP)$ and $(\cA',\cP')$ are tame isomorphic up to generalized stabilizations, and denote this by
\[
(\cA,\cP)\stabletameisom(\cA',\cP'),
\]
i.e.,
\[
S_{\bp_k}^{\fd_k \pm}(\cdots S_{\bp_1}^{\fd_1 \pm}(\cA,\cP)\cdots)
\tameisom
S_{\bp'_\ell}^{\fd'_\ell \pm}(\cdots S_{\bp'_1}^{\fd_1 \pm}(\cA',\cP')\cdots)
\]
for some $\fd_i, \fd'_i\in \fR$ and $\bp_j\in (\cdots(\cP_{\bp_1}^{\fd_1 \pm})\cdots)_{\bp_{j-1}}^{\fd_{j-1} \pm}, \bp'_j\in (\cdots(\cP')_{\bp'_1}^{\fd_1 \pm})\cdots)_{\bp'_{j-1}}^{\fd'_{j-1} \pm}$.
\end{definition}

\begin{proposition}\label{prop:isom_p_md}
The DGA morphisms $\bp^{\fd \pm}$ are chain equivalences. In particular, $\bp^{\fd \pm}$ induce the isomorphisms between homology groups
\[
(\bp^{\fd \pm})_*:H_*(\cI_m)\to H_*(\cF_\bp^{\fd \pm}).
\]
\end{proposition}
\begin{proof}
As discussed in Remark~\ref{remark:anti-isomorphism and stabilizations}, it suffices to prove the claim only for the positive stabilization.

Let $\pi^+:\cF_\bp^{\fd +}\to\cI$ be a DGA morphism left inverse to $\bp^{\fd +}$, defined as 
\begin{align*}
\pi^+(\bar\rho_{i,\ell})&\coloneqq\rho_{i,\ell},&
\pi^+(\bar\sfc_i)&\coloneqq 0.
\end{align*}

We claim that $\pi^+$ is a chain homotopy inverse of $\bp^{\fd +}$, and it suffices to prove that there exists $h^+:\cF_\bp^{\fd +}\to\cF_\bp^{\fd +}$ of degree 1 such that
\[
\bar\partial_m \circ h^+ + h^+\circ\bar\partial_m = \id-(\bp^{\fd +}\circ\pi^+).
\]

We need the following lemma.
\begin{lemma}\label{lem:tildec}
For each $1\le k\le m$, let us define 
\[
\tilde\sfc_k\coloneqq (-1)^{|\bar\sfc_k|^\fd-1}\sum_{i=1}^m \bar\sfc_i \bar\rho_{i,m-i+k}\in F_m^+.
\]
Then the following holds:
\[
\bar\partial_m^+\tilde c_k=\bar\sfc_k+(-1)^{|\bar\sfc_k|^\fd-1}\sum_{i=1}^{k-1} \tilde \sfc_i\bar\rho_{i,k-i}.
\]
\end{lemma}
\begin{proof}
By direct computation, we have
\begin{align*}
(-1)^{|\bar\sfc_k|^\fd-1}\bar\partial_m\tilde \sfc_k&=\sum_{i=1}^m (\bar\partial_m \bar\sfc_i) \bar\rho_{i,m+k-i}
+ (-1)^{|\bar\sfc_i|^\fd}\bar\sfc_i(\bar\partial_m\bar\rho_{i,m+k-i})\\
&=\sum_{i=1}^m (-1)^{|\bar\sfc_i|^\fd-1}\left(\sum_{j=1}^{i-1}
\bar\sfc_j\bar\rho_{j,i-j}\right)\bar\rho_{i,m+k-i}+
(-1)^{|\bar\sfc_i|^\fd}\bar\sfc_i(\bar\partial_m\bar\rho_{i,m+k-i})\\
&=\sum_{j=1}^{m} (-1)^{|\bar\sfc_j|^\fd-1}\bar\sfc_j
\left(
\sum_{i=j+1}^{m}
(-1)^{|\bar\sfc_i|^\fd-|\bar\sfc_j|^\fd}\bar\rho_{j,i-j}\bar\rho_{i,m+k-i} - \partial_m\bar\rho_{j,m+k-j}
\right).
\end{align*}

Since $(-1)^{|\bar\sfc_i|^\fd-|\bar\sfc_j|^\fd}=(-1)^{|\bar\rho_{i,j-i}|^\fd-1}$ and
\[
\partial_m\bar\rho_{j,m+k-j} = \delta_{m+k-j,m} + \sum_{i=j+1}^{m+k-1} (-1)^{|\bar\rho_{j,i-j}|^\fd-1}\bar\rho_{j,i-j}\bar\rho_{i,m+k-i},
\]
the above equation is the same as
\begin{align*}
&\sum_{j=1}^m (-1)^{|\bar\sfc_j|^\fd-1}\bar\sfc_j \left(\delta_{m+k-j,m}+\sum_{i=m+1}^{m+k-1}(-1)^{|\bar\rho_{j,i-j}|^\fd-1}\bar\rho_{j,i-j}\bar\rho_{i,m+k-i}\right)\notag\\
&=(-1)^{|\bar\sfc_k|^\fd-1}\bar\sfc_k+
\sum_{\ell=1}^{k-1}\left(\sum_{j=1}^m (-1)^{|\bar\sfc_j|^\fd+|\bar\rho_{j,m+\ell-j}|^\fd}\bar\sfc_j \bar\rho_{j,m+\ell-j}\right)\bar\rho_{\ell,k-\ell}\notag\\
&=(-1)^{|\bar\sfc_k|^\fd-1}\bar\sfc_k+
\sum_{\ell=1}^{k-1}(-1)^{|\bar\sfc_\ell|^\fd+1}\left(\sum_{j=1}^m\bar\sfc_j \bar\rho_{j,i-j}\right)\bar\rho_{\ell,k-\ell}\notag\\
&=(-1)^{|\bar\sfc_k|^\fd-1}\bar\sfc_k+\sum_{i=1}^{k-1}\tilde\sfc_i\bar\rho_{i,k-i}.\qedhere
\end{align*}
\end{proof}

Let $w\in\cF_\bp^{\fd +}$ be a monomial. We define $h^+(w)$ inductively as follows:
\begin{align*}
h^+(1)&\coloneqq 0,&
h^+(w)&\coloneqq\begin{cases}
(-1)^{|x\bar\sfc_k|^\fd-1}x\tilde\sfc_k y & w=x\bar\sfc_k y, h^+(x)=0;\\
0 & \text{otherwise}.
\end{cases}
\end{align*}
Then $h^+(w)=0$ if and only if $w$ does not contain any $\bar\sfc_k$'s, and moreover, if $h^+(x)=0$,
\[
h^+(x w y)=(-1)^{|x|^\fd}x h^+(w) y
\]
for any $w, y\in\cF_\bp^{\fd+}$.

If $w$ contains no $\bar\sfc_\ell$'s, then neither does $\bar\partial_m(w)$ and therefore
\[
(\id-\bp^{\fd +}\circ\pi^+)(w)=w-w=0=\bar\partial_m \circ h^+(w)+h^+\circ\bar\partial_m(w).
\]

Otherwise, if $w=x \bar\sfc_k y$ with no $\bar\sfc_k$'s in $x$, then 
it is easy to check that 
\begin{align*}
&\mathrel{\hphantom{=}}(\bar\partial_m\circ h^+ +h^+\circ\bar\partial_m)\left((-1)^{|x\bar\sfc_k|^\fd-1}w\right)\\
&=\bar\partial_m(x\tilde \sfc_k y)
+(-1)^{|x\bar\sfc_k|^\fd-1}h^+( \bar\partial_m(x) \bar\sfc_k y + (-1)^{|x|^\fd}x\bar\partial_m(\bar\sfc_k)y + (-1)^{|x\bar\sfc_k|^\fd}x \bar\sfc_k\bar\partial_m(y))\\
&=\bar\partial_m(x\tilde \sfc_k y)
-\bar\partial_m(x)\tilde \sfc_k y + x(h^+\circ\bar\partial_m)(\bar\sfc_k) y + (-1)^{|x\bar\sfc_k|^\fd}x\tilde \sfc_k\bar\partial_m(y))\\
&=x(\bar\partial_m(\tilde \sfc_k) + (h^+\circ\bar\partial_m)(\bar\sfc_k)) y\\
&=x\bar\sfc_k y = w
\end{align*}
by the previous lemma. Since $\pi^+(w)=0$, we have 
\[
(\id-\bp^{\fd +}\circ \pi^+)(w)=w-0=w,
\]
which shows that $(\bp^{\fd +})_*:H_*(\cI_\bp)\to H_*(\cF_\bp^{\fd +})$ is an isomorphism.
\end{proof}

\begin{proposition}\label{proposition:stabilization preserves homology}
Let $(\cA,\cP)$ be a DGA with a collection of peripheral structures, and let $\cA_\bp^{\fd \pm}$ be the $\fd$-th generalized $(\pm)$-stabilization with respect to $\bp\in \cP$ defined as above. Then the DGA morphism $\iota_\bp^{\fd \pm}$ induces a chain equivalence, whose chain homotopy inverse is precisely the projection 
\[
\pi_\bp^{\fd \pm}:\cA_\bp^{\fd \pm}\to\cA.
\]

In particular, it induces an isomorphism between homology groups
\[
(\iota_\bp^{\fd \pm})_*:H_*(\cA)\to H_*(\cA_\bp^{\fd \pm}).
\]
\end{proposition}
\begin{proof}
This follows by exactly the same argument as Proposition~\ref{prop:isom_p_md}.
\end{proof}

\begin{remark}
The proofs of Propositions~\ref{prop:isom_p_md} and \ref{proposition:stabilization preserves homology} can be deduced from the fact that generalized stabilizations are compositions of (de)stabilizations involving infinitely many cancelling pairs of generators up to tame isomorphisms as mentioned in Remark~\ref{remark:EN_stabilizations}.
\end{remark}

\section{DGA for Legendrian graphs with potentials}\label{Differential graded algebra for Legendrian graphs with potentials}
By extending the idea of disk counting, we construct a DGA for Legendrian graphs with potential.
In addition to oriented admissible disks in the Lagrangian projection, we take into account {\em infinitesimal} disks near each vertex for the differential of generators from vertices.
We also assign a canonical peripheral structure, a collection of DG-subalgebras, in order to encode the data of vertices of Legendrian graphs as follows:

\begin{theorem}[Associated DGA with peripheral structures]\label{theorem:DGA}
Let $\cL=(\Lambda,\fP)$ be a Legendrian graph with potential. Then there is a pair $(\cA_\cL, \cP_\cL)$ consisting of a DGA $\cA_\cL\coloneqq(A_\Lambda,|\cdot|_\fP,\partial)$ and a canonical peripheral structure $\cP_\cL$.
\end{theorem}

\subsection{Capping paths}\label{sec:canonical_path}
Let us consider the set of double points of $\sL$ by $\sC_\sL$, and define the set $\sS_\sL$ of {\em special points} as the disjoint union 
\[
\sS_\sL\coloneqq\sC_\sL\amalg \sV_\sL.
\]

Now introduce capping paths as generators of our DGA.
Consider the sets
\begin{align*}
\tilde \sV_\sfv&\coloneqq\{\sfv_{i,\ell}\mid i\in\Zmod{\val(\sfv)}, \ell\ge 1\},&
\tilde\sV_\sL&\coloneqq\coprod_{\sfv\in\sV_\sL} \tilde\sV_\sfv.
\end{align*}

\begin{notation}
We denote the union of $\sC_\sL$ and $\tilde\sV_\sL$ by $\sG_\sL$.
\[
\sG_\sL\coloneqq\sC_\sL\amalg\tilde\sV_\sL
\]
\end{notation}

The projection $\tilde\sV_\sL\to\sV_\sL$ defined as $\sfv_{i,\ell}\mapsto\sfv$ extends to 
\[
\sG_\sL=\sC_\sL\amalg\tilde\sV_\sL\to\sC_\sL\amalg\sV_\sL=\sS_\sL
\]
which maps $\sfc\in\sC_\sL$ to $\sfc$.
Moreover, the assignment $\sfv_{i,\ell}\mapsto(\sfh_{\sfv,i},\ell)$ gives us a bijection
$\tilde\sV_\sfv\to\sH_\sfv\times\ZZ_{\ge 1}.$
Therefore one can interpret $\sfv_{i,\ell}$ as the region covered by clockwise $\ell$ sectors of $\sU_\sfv$ starting with the half-edge $\sfh_{\sfv,i}$.

\subsubsection{Capping paths for $\sC_\sL$}
Suppose that $\sfg=\sfc\in\sC_\sL$ is a double point. 
Let $c_\pm\in\Lambda$ be the two lifts of $\sfc$ chosen so that $c_+$ has a greater $z$-value than $c_-$. That is,
\begin{align*}
\pi_L(c_+)&=\pi_L(c_-)=\sfc\in\sC_\sL,& z(c_+)&>z(c_-).
\end{align*}
We denote the two edges of $\Lambda$ containing $c_+$ and $c_-$ by $e_{c}^+$ and $e_{c}^-$. Their projections are denoted by $\sfe_{\sfc}^+$ and $\sfe_{\sfc}^-$, and are called the {\em upper} and {\em lower edges} at $\sfc$, respectively.
\[
\begin{tikzcd}
c_\pm\in e_{c}^\pm\in \cE_\Lambda
\quad\ar[r,"\pi_L"]&\quad
\sfc\in\sfe_{\sfc}^\pm\in\sE_\sL.
\end{tikzcd}
\]
Then the tangent space $T_\sfc \sL\subset\RR_{xy}^2$ is the union $T_\sfc \sfe_{\sfc}^{+}\cup T_\sfc \sfe_{\sfc}^{-}$.
We define a path $\sfc(t)$ such that 
\begin{align*}
\sfc(t)&:[0,1]\to \gr(1,T_\sfc\RR^2_{xy}),&
\sfc(0)&=T_\sfc \sfe_{\sfc}^{+},&
\sfc(1)&=T_\sfc \sfe_{\sfc}^{-}
\end{align*}
and $\sfc(t)$ rotates {\em counterclockwise} as $t$ increases.

Pictorially, if we assign a $(+)$ or $(-)$ sign to each region near $\sfc$ separated by $\sL$ as depicted in Figure~\ref{fig:signs of regions}, then the path $\sfc(t)$ sweeps through both positive regions counterclockwise. See Figure~\ref{fig:capping path of crossing}.
\begin{figure}[ht]
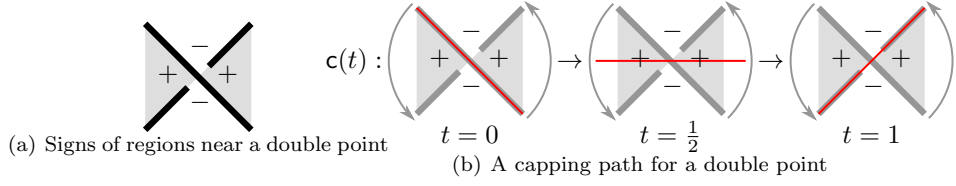

\subfigure[Signs of regions near a double point\label{fig:signs of regions}]{\makebox[0.4\textwidth]{
$\vcenter{\hbox{\input{crossing_sign_input.tex}}}$
}}
\subfigure[A capping path for a double point\label{fig:capping path of crossing}]{\makebox[0.5\textwidth]{
$\setlength\arraycolsep{1pt}
\begin{array}{rccccc}
\sfc(t):&\vcenter{\hbox{\input{canonicalpath_crossing_0_input.tex}}}&\to&
\vcenter{\hbox{\input{canonicalpath_crossing_0-5_input.tex}}}&\to&
\vcenter{\hbox{\input{canonicalpath_crossing_1_input.tex}}}\\
& t=0 & & t=\frac12 & & t=1
\end{array}
$}
}
\caption{Signs of regions and a capping path for a double point}
\label{figure:Reeb sign of crossing}
\end{figure}

\subsubsection{Capping paths for $\tilde\sV_\sL$}
Now suppose that $\sfg=\sfv_{i,\ell}\in\tilde\sV_\sL$. We define a path $\sfv_i:[0,1]\to \gr(1,T_\sfv\RR^2_{xy})$ in advance for each $\sfv\in\sV_\sL$ and $i\in\Zmod{\val(\sfv)}$ as a loop based at $T_\sfv \sfh_{\sfv,i}$ which turns {\em counterclockwise} $\pi$ radians. Then the (free) homotopy class $[\sfv_i(t)]$ corresponds to a generator $1_{\gr}\in\pi_1(\gr(1,T_\sfv\RR^2_{xy}))$.

We use induction on $\ell$ to define $\sfv_{i,\ell}(t)$.
For $\ell>1$, we define $\sfv_{i,\ell+1}$ as the concatenation
\begin{equation}\label{eqn:concatenation}
\sfv_{i,\ell+1}(t)\coloneqq \sfv_{i,1}(t)\cdot \sfv_{i+1}^{-1}(t)\cdot \sfv_{i+1,\ell}(t),
\end{equation}
where $\sfv_{i+1}(t)$ is the loop defined previously.

For $\ell=1$, the path $\sfv_{i,1}(t)$ is the same as the path defined by the tangent line along the boundary of the sector $\sector_{\sfv,i}$ after smoothing.
See Figure~\ref{figure:sectors and capping paths} for examples.
\begin{figure}[ht]
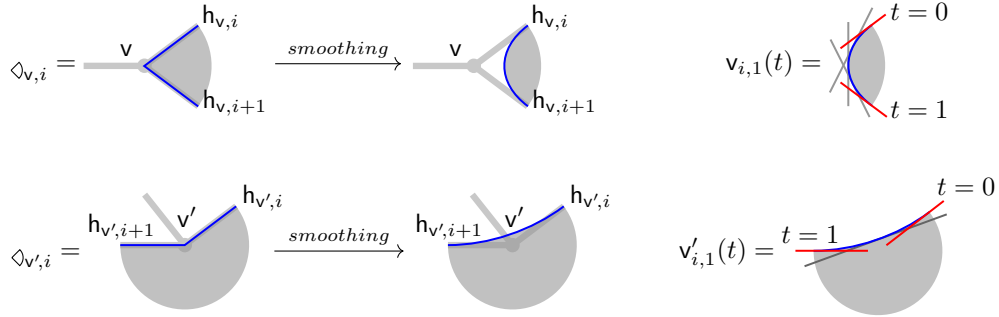

\begin{tikzcd}
\sector_{\sfv,i}=\vcenter{\hbox{\input{sector_vertex_0_input.tex}}}\arrow[rr,"smoothing"] & & \vcenter{\hbox{\input{sector_vertex_0_smoothing_input.tex}}} &
\sfv_{i,1}(t)=\vcenter{\hbox{\input{sector_vertex_0_capping_path_input.tex}}}\\
\sector_{\sfv'\!,i}=\vcenter{\hbox{\input{sector_vertex_1_input.tex}}}\arrow[rr,"smoothing"] & & \vcenter{\hbox{\input{sector_vertex_1_smoothing_input.tex}}} &
\sfv'_{i,1}(t)=\vcenter{\hbox{\input{sector_vertex_1_capping_path_input.tex}}}\\
\end{tikzcd}
\caption{Sectors and capping paths}
\label{figure:sectors and capping paths}
\end{figure}

More precisely, let $0<\theta\le 2\pi$ be the angle between two half-edges $\sfh_{\sfv,i}$ and $\sfh_{\sfv,i+1}$ given by clockwise rotation.\footnote{The angle $\theta$ is $2\pi$ if and only if $\sfh_{\sfv,i}=\sfh_{\sfv,i+1}$, or equivalently, if $\sfv$ is univalent.}
A path $\sfv_{i,1}:[0,1]\to \gr(1,T_\sfv\RR^2_{xy})$ is defined by
\begin{align*}
\sfv_{i,1}(0)&\coloneqq T_\sfv \sfh_{\sfv,i},&
\sfv_{i,1}(1)&\coloneqq T_\sfv \sfh_{\sfv,i+1}
\end{align*}
and as $t$ increases $\sfv_{i,1}(t)$ rotates by $(\pi-\theta)$ {\em counterclockwise} if $\theta<\pi$ and rotates by $(\theta-\pi)$ {\em clockwise} if $\pi<\theta$.

Note that the path $\sfv_{i,m}(t)$ for $m=\val(\sfv)$ becomes a loop which is a clockwise rotation by $\pi$. In general, $\sfv_{i,km}(t)$ corresponds to a {\em clockwise} rotation by $(2k-1)\pi$ for $k\ge 0$.\footnote{When $k=0$, then the clockwise $-\pi$ rotation means a counterclockwise rotation by $\pi$, which is the same as $\sfv_i(t)$.}

\begin{notation}
The free associative unital algebra over $\ZZ$ generated by $\sG_\sL$ is denoted by $A_\Lambda$.
\[
A_\Lambda\coloneqq\ZZ\langle \sG_\sL \rangle
\]
\end{notation}

\subsubsection{Geometric interpretation of generators}\label{sec:geometric model}
We interpret generators as Reeb chords. In particular we focus on a local model of each vertex by considering the standard model of a Weinstein one-handle.

Let us start with a function $F:\CC^2\to \RR$ defined by
\begin{align*}
F(x_1+iy_1, x_2+iy_2)=\frac{1}{2}(x_1^2+y_1^2)+2x_2^2-y_2^2.
\end{align*}
Then $F^{-1}(-\delta)$ and $F^{-1}(\delta)$ are homeomorphic to $\RR^3\times S^0$ and $S^2\times \RR$, respectively.
Now consider the {\em Liouville vector field}
\[
Z=\frac{1}{2}(x_1\partial _{x_1}+y_1\partial_{y_1})+2x_2\partial x_2-y_2\partial y_2
\]
of $F^{-1}(\pm\delta)$ with respect to the {\em standard symplectic form}
$\omega_{\rm std}=dx_1\wedge dy_1+dx_2\wedge dy_2.$ This defines a {\em contact form} on each $F^{-1}(\pm\delta)$ by restricting the following 1-form
\[
\omega_{\rm std}(Z,\cdot )=\frac{1}{2}(x_1dy_1-y_1dx_1)+2x_2 dy_2+y_2 dx_2,
\]
whose {\em Reeb vector field} $R$ is proportional to
\[
\frac{1}{2}(x_1\partial_{y_1}-y_1\partial_{x_1})+2x_2\partial_{y_2}+y_2\partial_{x_2},
\]
on each level set of $F$. We especially consider the following half of $F^{-1}(\delta)$
\[
N(\delta)\coloneqq F^{-1}(\delta)\cap \{y_2\geq 0\},
\]
which is homeomorphic to $S^2\times[0,+\infty)$. It is straightforward to check that there exists a unique closed {\em Reeb orbit} $R$ on $\partial N(\delta)=F^{-1}(\delta)\cap \{y_2= 0\}$ parametrized as 
\[
t \mapsto \left(
\sqrt{2\delta}\cos{t}, \sqrt{2\delta}\sin{t}, 0, 0
\right).
\]

For a Legendrian graph $\Lambda\subset \RR^3_{\rotation}$ and $v\in\cV_{\Lambda}$, by Theorem~\ref{theorem:Darboux}, we identify the standard neighborhood $\cU_v\subset\RR_{\rotation}^3$ with the open unit ball at the origin
\[
(\cU_v,\cU_v\cap\Lambda)\stackrel{\phi_v}{\simeq}(\cU_\origin,\cU_\origin\cap T_\Theta)
\]
and for sufficiently small $\delta_v>0$, let us consider
\begin{align*}
\cN_v&\coloneqq\{(x,y,z)\in\cU_\origin\mid x^2+y^2+z^2\ge \delta_v\}.
\end{align*}

Then there is a diffeomorphism $\bar\phi_v:N(\delta_v)\to\cN_v$ defined as $\bar \phi_v(x_1,y_1,x_2) \coloneqq (x, y, z)$ such that
\begin{align*}
x&\coloneqq x_1 \frac{y_2+\sqrt{\delta_v}}{\sqrt{2(y_2^2+\delta_v)}}, &
y&\coloneqq y_1 \frac{y_2+\sqrt{\delta_v}}{\sqrt{2(y_2^2+\delta_v)}}, &
z&\coloneqq x_2 \frac{\sqrt{2}(y_2+\sqrt{\delta_v})}{\sqrt{y_2^2+\delta_v}},
\end{align*}
where 
\[
y_2=\sqrt{\frac12(x_1^2+y_1^2)+2x_2^2-\delta_v}.
\]

We now equip $\cN_v$ with the contact form $\alpha_v$ which is given by interpolating the contact form on $\RR^3$ and the one obtained from $N(\delta_v)$ via $\bar\phi_v$. Then the periodic Reeb orbit $R_v$ is the equator of the sphere $\partial\cN_v=\{(x,y,z)\mid x^2+y^2+z^2=\delta_v\}$ parametrized as
\[
R_v(t)\coloneqq\{(\sqrt{\delta_v}\cos t, \sqrt{\delta_v}\sin t, 0)\in \partial\cN_v\}.
\]

Suppose that $m=\val(v)$ pair of Legendrian arcs in $\cN_v$ intersects $R_v$ at $\theta=\frac{2\pi}{m}$.
Then we have infinitely many {Reeb chords} on each $\cN_v$ which correspond the generators $\sfv_{i,\ell}$'s
\[
\sfv_{i,\ell}\longleftrightarrow\left\{R_v(t)~\middle|~ \frac{2\pi}{m}i \leq t \leq \frac{2\pi}{m}(i+\ell)\right\}.
\]
We call $\cN_v$ the {\em local geometric model} for $v\in\cV_{\Lambda}$.

As a result, we have a pair consisting of a bordered manifold with bordered Legendrian arcs 
\begin{align*}
\cB\Lambda&\coloneqq \RR^3\setminus \cU_\cV\cup \cN_\cV,& \cE\Lambda&\coloneqq \Lambda \cap \cB\Lambda,
\end{align*}
where $\cU_\cV\coloneqq \coprod_{v\in\cV}\cU_v$ and $\cN_\cV\coloneqq \coprod_{v\in\cV}\cN_v$.

\subsection{Gradings from Maslov potentials}\label{section:grading_singularpoint}
We will use the given potential $\fP$ to define an $\fR$-grading on $A_\Lambda$ by assigning an element $|\sfg|_\fP\in\fR$ for each $\sfg\in\sG_\sL$ as follows.

\subsubsection{Grading on $\sC_\sL$}\label{subsubsection:grading on c}
For $\sfc\in \sC_\sL$, let $\sfe_{\sfc}^{+}$ and $\sfe_{\sfc}^{-}$ be the upper and lower edges, respectively, which are oriented.
We choose two oriented subarcs $\gamma_{\sfc}^{\pm}$ of $\sfe_{\sfc}^{\pm}$ from $\sfc$ to their terminal vertices, say $\sfv\coloneqq \sfe_{\sfc}^+(1)$ and $\sfw\coloneqq \sfe_{\sfc}^-(1)$, respectively. 
\begin{align*}
\gamma_{\sfc}^{\pm}&:[0,1]\to \sfe_{\sfc}^{\pm},&
\gamma_{\sfc}^{\pm}(0)&=\sfc,&
\gamma_{\sfc}^{+}(1)&=\sfv,&
\gamma_{\sfc}^{-}(1)&=\sfw.
\end{align*}
See Figure~\ref{fig:paths for grading}.

\begin{figure}[ht]
\def\sfvgscale{1.2}\input{canonicalpath_grading_input.tex}
\caption{The two subarcs $\gamma_{\sfc}^{+}$ and $\gamma_{\sfc}^{-}$ of $\sL$ for a double point $\sfc\in \sC_\sL$}
\label{fig:paths for grading}
\end{figure}

We denote by $\sfh_{\sfv,i}$ and $\sfh_{\sfw,j}$ the two half-edges contained in $\gamma_{\sfc}^{+}$ and $\gamma_{\sfc}^{-}$, respectively.
As before, by using the stretching map $f_M$ and perturbing slightly, we may assume that $\sfh_{\sfv,i}$ and $\sfh_{\sfw,j}$ are parallel to the $x$-axis.
Then the loop $\hat \sfc(t)$ is defined to be the concatenation 
\[
\hat \sfc(t)\coloneqq D\gamma_{\sfc}^{+}(t)\cdot D\gamma_{\sfc}^{-}(t)^{-1}\cdot \sfc(t)^{-1}\in \gr(1,\RR^2),
\]
which travels along $\sfe_{\sfc}^{+}$ from $T_\sfc \sfe_{\sfc}^{+}$ to $T_\sfv \sfh_{\sfv,i}\approx \partial_x \approx T_\sfw \sfh_{\sfw,j}\in \gr(1,\RR^2)$ and then to $T_\sfc \sfe_{\sfc}^{-}$ along $\sfe_{\sfc}^{-}$.
We define $|\sfc|_\fP$ as
\begin{align*}\label{eq:Gr1}\tag{Gr1}
|\sfc|_\fP\coloneqq \fP(\sfh_{\sfv,i})-\fP(\sfh_{\sfw,j})+n_\sfc\cdot1_\fR\in \fR,
\end{align*}
where $[\hat\sfc] = n_\sfc\cdot 1_{\gr}\in\pi_1(\gr(1,\RR^2))$.

If $\sfc$ is a self-intersecting point of an edge $\sfe\in \sE_\sL$, then we may have that $\sfh_{\sfv,i}=\sfh_{\sfw,j}$ and one of $D\gamma_{\sfc}^{+}$ or $(D\gamma_{\sfc}^{-})^{-1}$ is completely canceled by the other up to homotopy. In this case $|\sfc|_\fP$ recovers the usual definition of the grading in the literature.

\subsubsection{Grading on $\tilde\sV_\sL$}
Suppose that $\sH_\sfv$ is decomposed into two parts 
\[
\sH_\sfv=\{\sfh_{\sfv,1},\dots, \sfh_{\sfv,a}\}\amalg \{\sfh_{\sfv,a+1},\dots, \sfh_{\sfv,a+b}\},\quad
a+b=\val(\sfv)
\]
so that $\sfh_{\sfv,1},\dots, \sfh_{\sfv,a}$ are on the left of $\sfv$ and $\sfh_{\sfv,a+1},\dots, \sfh_{\sfv,a+b}$ are on the right of $\sfv$, respectively.

We use the same strategy to define the grading of $\sfv_{i,\ell}$ by regarding $\sfh_{\sfv,\ell}$ and $\sfh_{\sfv,i+\ell}$ as the upper and lower strand for $\sfv_{i,\ell}$ together with two constant paths $D\gamma_{\sfv_{i,\ell}}^+$, $D\gamma_{\sfv_{i,\ell}}^-$ at $\sfh_{\sfv,\ell}$, $\sfh_{\sfv,i+\ell}$ in $\gr(1,\RR^2)$, respectively.

Then as before, two tangent spaces of half-edges $\sfh_{\sfv,i}$ and $\sfh_{\sfv,i+\ell}$ become close enough via the streching map $f_M$ to be identified and so we can compose paths to define the loop
\[
\hat\sfv_{i,\ell}\coloneqq D\gamma_{\sfv_{i,\ell}}^+(t)\cdot D\gamma_{\sfv_{i,\ell}}^-(t)\cdot \sfv_{i,\ell}^{-1}(t)=\sfv_{i,\ell}^{-1}(t).
\]

Suppose that this path turns $n_{\sfv_{i,\ell}}\pi$ for some $n_{\sfv_{i,\ell}}\in\ZZ$. Namely,
\[
[\hat\sfv_{i,\ell}]=n_{\sfv_{i,\ell}}\cdot 1_\fR.
\]
Then we define the grading $|\sfv_{i,\ell}|_\fP$ as the same as before.
\[
|\sfv_{i,\ell}|_\fP\coloneqq \fP(\sfh_{\sfv,i})-\fP(\sfh_{\sfv,i+\ell})+n_{\sfv_{i,\ell}}\cdot 1_\fR
\]

By the concatenation formula (\ref{eqn:concatenation}) for the capping paths we have the relation
\[
n_{\sfv_{i,\ell+r}} = n_{\sfv_{i,\ell}}+n_{\sfv_{i+\ell,r}}+1,
\]
where the last 1 comes from the path $\sfv_{i+\ell}^{-1}(t)$ in $\sfv_{i,\ell+r}(t)$.
Therefore the following grading relation holds:
For $\ell,r\ge 1$,
\begin{align*}\label{eq:Gr3}\tag{Gr3}
|\sfv_{i,\ell+r}|_\fP\coloneqq|\sfv_{i,\ell}|_\fP+|\sfv_{i+\ell,r}|_\fP+1_\fR\in \fR.
\end{align*}
It is easy to check from the definition that for $m=\val(\sfv)$,
\begin{align}
\sum_{j=1}^m (|\sfv_{j,1}|_\fP+1_\fR)&=2\cdot 1_\fR,&
|\sfv_{i,m}|_\fP &= 1_\fR,\quad\forall i\in\Zmod{m}.
\end{align}

Explicitly, $|\sfv_{i,1}|_\fP\in \fR$ is given as follows:
\begin{enumerate}
\item If $a,b\ge 1$, then for $1\le i\le a+b$, 
\begin{align*}\tag{Gr$2_1$}\label{eq:Gr2_1}
|\sfv_{i,1}|_\fP\coloneqq\begin{cases}
\fP(\sfh_{\sfv,i})-\fP(\sfh_{\sfv,i+1})-1_\fR & i\neq a, a+b;\\
\fP(\sfh_{\sfv,a})-\fP(\sfh_{\sfv,a+1}) & i=a;\\
\fP(\sfh_{\sfv,a+b})-\fP(\sfh_{\sfv,1}) & i=a+b.
\end{cases}
\end{align*}
\item If $a=0$, then for $1\leq i\leq b$,
\begin{align*}\tag{Gr$2_2$}\label{eq:Gr2_2}
|\sfv_{i,1}|_\fP\coloneqq
\begin{cases}
\fP(\sfh_{\sfv,i})-\fP(\sfh_{\sfv,i+1})-1_\fR & i\neq b;\\
\fP(\sfh_{\sfv,b})-\fP(\sfh_{\sfv,1})+1_\fR & i=b.
\end{cases}
\end{align*}
\item If $b=0$, then for $1\leq i\leq a$,
\begin{align*}\tag{Gr$2_3$}\label{eq:Gr2_3}
|\sfv_{i,1}|_\fP\coloneqq
\begin{cases}
\fP(\sfh_{\sfv,i})-\fP(\sfh_{\sfv,i+1})-1_\fR & i\neq a;\\
\fP(\sfh_{\sfv,a})-\fP(\sfh_{\sfv,1})+1_\fR & i=a.
\end{cases}
\end{align*}
\end{enumerate}

\begin{proposition}\label{proposition:degree well defined}
The function
\[
|\cdot|_\fP:\sG_\sL\to\fR
\]
determined by the conditions {\rm(\ref{eq:Gr1})}, {\rm(\ref{eq:Gr2_1})}, {\rm(\ref{eq:Gr2_2})}, {\rm(\ref{eq:Gr2_3})} and {\rm(\ref{eq:Gr3})} is well-defined and independent of the choices of orientations on edges.
\end{proposition}
\begin{proof}
The well-definedness is obvious. For $\sfv_{i,\ell}\in\tilde\sV_\sL$, $|\sfv_{i,\ell}|_\fP$ depends only on the given Maslov potential $\fP$ which is orientation independent as seen in Remark~\ref{remark:potentials_not_depending_on_orientation}. 
For $\sfc\in\sC_\sL$, if we change the orientation of an edge, then subarcs $\gamma_\sfc^\pm$ and so their terminal vertices may change simultaneously. However, from the vanishing condition for the Maslov potential $\fP$ on $\sft_\sfe$ described in \S~\ref{sec:potential}, their effects on $|\sfc|_\fP$ should be cancelled and this completes the proof.
\end{proof}

Now we are ready to define graded algebras associated to Legendrian graphs with Maslov potentials.
\begin{definition}[Graded algebra]\label{definition:algebra}
Let $\cL=(\Lambda,\fP)\in\cLG_\fR$ be a Legendrian graph with an $\fR$-valued Maslov potential.
We define a graded algebra
\[
\cA_\cL=(A_\Lambda,|\cdot|_\fP),
\]
where $A_\Lambda$ is the free unital associative algebra defined previously and $|\cdot|_\fP$ is the grading as above.
\end{definition}

\subsection{Admissible disks}\label{section:admissible disks}

Let $\Pi_t\subset\RR^2$ be an embedded $(t+1)$-gon which admits a canonical decomposition into open cells as follows:
\begin{align*}
\Pi_t&=\vPi_t\cup \ePi_t \cup \rPi_t,&
\ePi_t&=\{\be_0,\cdots, \be_t\},&
\vPi_t&=\{\bv_0,\cdots,\bv_t\};\\
\partial \Pi_t&=\vPi_t \cup \ePi_t,&
\partial\be_r &= \{\bv_r, \bv_{r+1}\},
\end{align*}
where the index $r$ is assumed to be in $\Zmod{(t+1)}$.

As usual, we call $\vPi_t$, $\ePi_t$ and $\rPi_t$ the sets of {\em vertices}, {\em edges} and {\em interior points} of $\Pi_t$, respectively.
We assume that vertices are labeled counterclockwise so that the labeling convention induces the orientation on $\Pi_t$.

For convenience's sake, we denote the half-edges at $\bv\in\vPi_t$ by $\bh_\bv^-$ and $\bh_\bv^+$ as shown in Figure~\ref{fig:Pi}, respectively, and suppress $t$ and simply use $\Pi$ unless any ambiguity occurs.
\begin{figure}[ht]
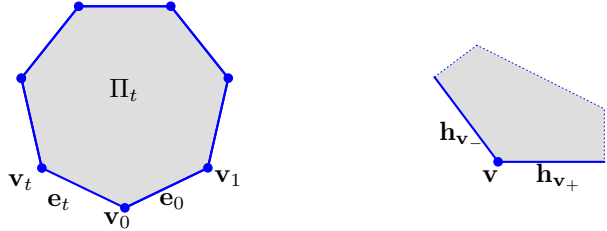

\begin{align*}
\vcenter{\hbox{\input{polygon_input.tex}}}& &&
\vcenter{\hbox{\input{polygon_neighborhood_input.tex}}}
\end{align*}
\caption{A $(t+1)$-gon $\Pi_t$ with vertices $\{\bv_0,\cdots, \bv_t\}$ and half-edges $\bh_\bv^\pm$ at $\bv$}
\label{fig:Pi}
\end{figure}

Recall the set $\sS_\sL=\sC_\sL\amalg \sV_\sL$ of special points.
Let $f$ denote a map of triples
\[
f:(\Pi,\partial\Pi,\vPi)\to(\RR^2,\sL,\sS_\sL).
\]

\begin{definition}[Standard neighborhoods and half-edges]
For each $\bv\in\vPi$ with $f(\bv)=\sfg\in\sS_\sL$, the {\em standard neighborhood} $\bU_\bv$ of $\bv$ is defined to be the component of the inverse $f^{-1}(\sU_\sfg)$ of the standard neighborhood $\sU_\sfg$ of $\sfg$ which contains $\bv$
\begin{align*}
\bv&\in\bU_\bv\subset f^{-1}(\sU_\sfg),&
f(\bv)&=\sfg\in\sS_\sL.
\end{align*}
\end{definition}

Recall that a point $\bx$ is a critical point of $f$ if the rank of the Jacobian $Df$ at $\bx$ is not maximal, and we denote the set of critical points by $\crit(f)$, which can be decomposed (or stratified) as 
\begin{align*}
\crit(f)&=\crit^0(f)\amalg\crit^1(f),&
\crit^i(f)&\coloneqq\{\bx\in\crit(f)\mid \rk Df|_\bx=i\}.
\end{align*}
For convenience's sake, we regard vertices as regular points.

\begin{assumption}[Differentiable disks]
From now on, any differentiable disk $f$ is assumed to satisfy the following:
\begin{enumerate}
\item $f$ is differentiable on $\rPi$ which extends to $\ePi$;
\item $\crit(f)$ is of measure-zero; and 
\item $f$ is smooth on $\bU_\bv\setminus\{\bv\}$ and $\bh_\bv^+\cup \bh_\bv^-=(\be_\bv^+\cup\be_\bv^-)\cap\bU_\bv$ for each $\bv\in\vPi$, where $\be_\bv^+$ and $\be_\bv^-$ are right and left edges adjacent to $\bv$, respectively.
\end{enumerate}
\end{assumption}

\begin{definition}[Inverse and critical graphs]
Let $f$ be a differentiable disk. We define the {\em inverse graph} $\bL_f$ and the {\em critical graph} $\bO_f$ as the inverse image of $\sL$ and the set of critical points of $f$, respectively.
\begin{align*}
\bL_f&\coloneqq f^{-1}(\sL),&
\bO_f&\coloneqq \crit(f).
\end{align*}
\end{definition}

We will indicate $\bL_f$ and $\bO_f$ by thick black lines and thin red lines as shown in Figure~\ref{fig:folding degree}, respectively.

\begin{definition}[$m$-folding edges]\label{def:folding edges}
Let $f$ be a differentiable disk. Suppose that an edge $\be\in\ePi$ has a unique critical point $\bx\in\be$.
We say that an edge $\be$ is either an {\em $m$-folding edge} or an {\em $(m/2)$-folding edge} for some $m\in\ZZ_{\ge1}$ if there exists a neighborhood $\bU_\bx$ of $\bx$ such that $f$ on $\bU_\bx\simeq\CC_{\ge0}\coloneqq\{x+iy\mid y\ge0\}$\footnote{This identification does not require any condition about orientations.} is modeled as either
\[
\setlength{\arraycolsep}{2pt}
\begin{array}{ccccc}
f|_{\bU_\bx}&:&
(\CC_{\ge0},0)&\to& (\CC,0)\\
&&z&\mapsto& z^m
\end{array}
\quad\text{or}\quad
\begin{array}{ccccccccc}
f|_{\bU_\bx}&:&
(\CC_{\ge0},0)&\to & (\CC,0)&\simeq& (\RR^2,\origin)&\to & (\RR^2,\origin) \\
&&z&\mapsto& z^m &\simeq& (a,b)&\to & (a^2,b),
\end{array}
\]
respectively.
\end{definition}

\begin{figure}[ht]
\begin{tikzcd}[column sep=0pt]
m & 1 & 1/2 & 2 & 2/2& 3 & 3/2& \cdots \\
\bU_\bx\ar[dd] & \vcenter{\hbox{\includegraphics[scale=0.9]{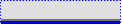}}} & 
\vcenter{\hbox{\includegraphics[scale=0.9]{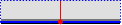}}} &
\vcenter{\hbox{\includegraphics[scale=0.9]{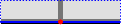}}} & 
\vcenter{\hbox{\includegraphics[scale=0.9]{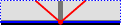}}} &
\vcenter{\hbox{\includegraphics[scale=0.9]{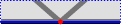}}} & 
\vcenter{\hbox{\includegraphics[scale=0.9]{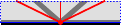}}} & 
\cdots \\ \\
f(\bU_\bx) & 
\vcenter{\hbox{\includegraphics{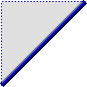}}} &
\vcenter{\hbox{\includegraphics{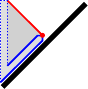}}} &
\vcenter{\hbox{\includegraphics{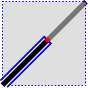}}} &
\vcenter{\hbox{\includegraphics{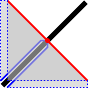}}} &
\vcenter{\hbox{\includegraphics{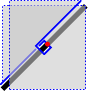}}} & 
\vcenter{\hbox{\includegraphics{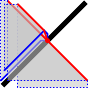}}} &
\cdots
\end{tikzcd}
\caption{$m$ and $(m/2)$-folding edges}\label{fig:folding degree}
\end{figure}

Let $\ell_f(\bv)$ be the number of regions separated by $\sL$ near $f(\bv)$ that $f$ covers, or equivalently, the number of components of $\bU_\bv\setminus\bL_f$
\begin{equation}\label{eq:ell_on_vertex}
\ell_f(\bv)\coloneqq \#(\pi_0(\bU_\bv\setminus\bL_f)).
\end{equation}
Since $\partial\Pi\subset\bL_f$, it is obvious that $\ell_f(\bv)=\val_{\bL_f}(\bv)-1$.

\begin{figure}[ht]
\renewcommand{\arraystretch}{1.5}
\begin{tabular}{c|c|c|c|c|c}
$f|_{\bU_\bv}$&
$\vcenter{\hbox{\includegraphics{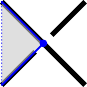}}}$&
$\vcenter{\hbox{\includegraphics{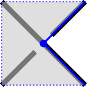}}}$&
$\vcenter{\hbox{\includegraphics{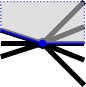}}}$&
$\vcenter{\hbox{\includegraphics{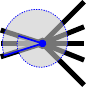}}}$&$\cdots$
\\
\hline
$\ell_f(\bv)$&1 & 3 & 3 & 10 & $\cdots$
\end{tabular}
\caption{The number $\ell_f(\bv)$}
\label{figure:ell_f}
\end{figure}

\begin{definition}[Canonical label]\label{def:Canonical label}
Let $f$ be a differentiable disk.
We define a function $\tilde f:\vPi\to \sG_\sL$, called the {\em canonical label} of $f$, as
\[
\tilde f(\bv)\coloneqq\begin{cases}
\sfc & f(\bv)=\sfc\in\sC_\sL;\\
\sfv_{i,\ell} & f(\bv)=\sfv\in\sV_\sL,
\end{cases}
\]
where $\ell\coloneqq\ell_f(\bv)$ and the index $i\in\Zmod{\val(\sfv)}$ is determined 
as follows:
\[
\begin{cases}
f(\bh_{\bv_+}\cap \bU_\bv)\subset \sfh_{\sfv,i} & f\text{ is orientation-reversing on }\bU_\bv;\\
f(\bh_{\bv_-}\cap \bU_\bv)\subset \sfh_{\sfv,i} & f\text{ is orientation-preserving on }\bU_\bv.
\end{cases}
\]
\end{definition}

The canonical label $\tilde f$ is a lift of the restriction $f|_\vPi:\vPi\to\sS_\sL$ with respect to the obvious projection $\sG_\sL\to\sS_\sL$ defined by $\sfv_{i,\ell}\mapsto \sfv$ and $\sfc\mapsto\sfc$.
\[
\begin{tikzcd}[column sep=3pc]
& \sG_\sL\ar[d]\\
\vPi\ar[r,"f|_\vPi"']\ar[ru,dashed,"\exists\tilde f"] & \sS_\sL
\end{tikzcd}
\]

For convenience's sake, we will assume that $\tilde f$ is freely extended to words on $\vPi$ so that
\[
\tilde f(\bv_{i_1}\cdots\bv_{i_r}) \coloneqq \tilde f(\bv_{i_1})\cdots\tilde f(\bv_{i_r}) \in A_\Lambda.
\]

\begin{definition}[Convexities of vertices]
We say that a vertex $\bv\in\vPi$ is
\begin{enumerate}
\item {\em convex} if $f(\bv)\in\sC_\sL$ and $\ell_f(\bv)=1$;
\item {\em concave} if $f(\bv)\in\sC_\sL$ and $\ell_f(\bv)=3$;
\item {\em neutral} if $f(\bv)\in\sV_\sL$.
\end{enumerate}
\end{definition}
\begin{remark}
Convexity, concavity, and neutrality are mutually exclusive but do not give a trichotomy since $\ell_f(\bv)$ can be arbitrarily large even if $f(\bv)\in\sC_\sL$.
\end{remark}

\begin{definition}[Signs of vertices]\label{definition:signs_of_vertices}
Let $f(\bv)\in \sC_\sL$. A convex vertex $\bv$ has {\em positive} or {\em negative} {\em Reeb sign} if the sign of the region covered by $f$ near $\bv$ is positive or negative, respectively.
Similarly, a concave vertex $\bv$ has {\em positive} or {\em negative} {\em Reeb sign} if two of the three regions covered by $f$ near $\bv$ is positive or negative, respectively.

Any vertex $\bv$ has the {\em orientation sign} $\sgn(f,\bv)$, which is either $+1$ or $(-1)^{|\tilde f(\bv)|-1}$ according to the orientation signs of quadrants or sectors assigned previously as depicted in Figure~\ref{figure:convex and concave vertices}. Here we make a certain choice at each crossing $\sC_\sL$ for the orientation sign (see Remark~\ref{remark:orientation sign}). We denote
\[
\sgn(f, \bv_{i_1}\cdots\bv_{i_k})\coloneqq\sgn(f,\bv_{i_1})\cdots\sgn(f,\bv_{i_k}).
\]
\end{definition}

\begin{figure}[ht]
\subfigure[Reeb and orientation signs of convex vertices]{
\renewcommand{\arraystretch}{1.5}
\begin{tabular}{c|c|c|c|c}
Convex vertices at $\sfc$ $\vcenter{\hbox{\rule{0pt}{5pc}}}$&
$\vcenter{\hbox{\includegraphics{crossing_sign_convex_positive.pdf}}}$ &
$\vcenter{\hbox{\scalebox{-1}{\includegraphics{crossing_sign_convex_positive.pdf}}}}$ &
$\vcenter{\hbox{\includegraphics{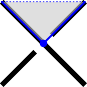}}}$ &
$\vcenter{\hbox{\scalebox{-1}[-1]{\includegraphics{crossing_sign_convex_negative.pdf}}}}$\\
\hline
Reeb sign & \multicolumn{2}{c|}{Positive} & \multicolumn{2}{c}{Negative}\\
\hline
Orientation sign & $+1$ & $(-1)^{|\sfc|-1}$ & $+1$ & $(-1)^{|\sfc|-1}$
\end{tabular}
}

\subfigure[Reeb and orientation signs of concave vertices]{
\renewcommand{\arraystretch}{1.5}
\begin{tabular}{c|c|c|c|c}
Concave vertices at $\sfc$ $\vcenter{\hbox{\rule{0pt}{5pc}}}$&
$\vcenter{\hbox{\scalebox{-1}[-1]{\includegraphics{crossing_sign_concave_negative.pdf}}}}$ &
$\vcenter{\hbox{\includegraphics{crossing_sign_concave_negative.pdf}}}$ &
$\vcenter{\hbox{\scalebox{-1}{\includegraphics{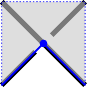}}}}$ &
$\vcenter{\hbox{\includegraphics{crossing_sign_concave_positive.pdf}}}$ \\
\hline
Reeb sign & \multicolumn{2}{c|}{Negative} & \multicolumn{2}{c}{Positive}\\
\hline
Orientation sign\footnotemark & $+1$ & $(-1)^{|\sfc|-1}$ & $+1$ & $(-1)^{|\sfc|-1}$
\end{tabular}
}

\subfigure[Neutral vertices and orientation signs]{
\renewcommand{\arraystretch}{1.5}
\begin{tabular}{c|c|c|c}
Neutral vertices at $\sfv$ $\vcenter{\hbox{\rule{0pt}{5pc}}}$&
$\vcenter{\hbox{\includegraphics{vertex_neutral_1.pdf}}}$&
$\vcenter{\hbox{\includegraphics{vertex_neutral_2.pdf}}}$ & $\cdots$ \\
\hline
Orientation sign & \multicolumn{3}{c}{$+1$ or $(-1)^{|\sfv_{i,\ell}|-1}$}
\end{tabular}
}
\caption{Convexities and signs of vertices}
\label{figure:convex and concave vertices}
\end{figure}
\footnotetext{These values are the products of orientation signs for convex regions inside.}

\begin{remark}\label{remark:orientation sign}
We are following the convention of Ekholm and Ng. As mentioned in \cite[\S2.4]{EN2015}, the roles of $+1$ and $(-1)^{|\tilde f(\bv)|-1}$ are interchangable by the automorphism $\sfc\mapsto (-1)^{|\sfc|-1}\sfc$ or $\sfv_{i,\ell}\mapsto (-1)^{|\sfv_{i,\ell}|-1}\sfv_{i,\ell}$ for all $i,\ell$.
\end{remark}

\begin{assumption}\label{assumption:orientation sign of vertex}
We assign the orientation sign for all $\sfv_{i,\ell}$ to be $+1$ unless we mention otherwise. That is, for any $f$ with $f(\bv)=\sfv$,
\[
\sgn(f,\bv)=1.
\]
\end{assumption}

\begin{definition}[Regular admissible disks]\label{def:Regular admissible disks}
A differentiable disk $f$ is {\em regular admissible} if 
\begin{enumerate}
\item $f$ is smooth and orientation-preserving on $\rPi$;
\item All vertices are either convex, concave or neutral;
\item All edges are either smooth or 2-folding;
\item $\bv_0$ is the only vertex having a positive Reeb sign.
\end{enumerate}
\end{definition}

Remark that the critical graph $\bO_f$ for a regular admissible disk $f$ is discrete and indicating positions of 2-folding edges.

\begin{definition}[Infinitesimal admissible disks]\label{definition:infinitesimal}
A differentiable disk $f$ is {\em infinitesimal admissible} if
\begin{enumerate}
\item $f$ maps a pair $(\Pi,\bO_f)$ into $(\cl{\sU_\sfv}, \sO_\sfv)$ for some $\sfv\in\sV_\sL$;
\item For each $\bx\in\rPi\setminus\bO_f$, $f$ is orientation-reversing at $\bx$ if and only if $\bx\in\bU_{\bv_0}$;
\item $\bO_f$ is connected.
\end{enumerate}
\end{definition}

\begin{remark}
This definition is not optimal in the sense that the third condition follows from the first two. However, we regard the third statement not as a consequence but as a condition for convenience's sake.
\end{remark}

\begin{example}
An infinitesimal triangle $f$ looks as follows:
\[
\vcenter{\hbox{
\begingroup%
  \makeatletter%
  \providecommand\color[2][]{%
    \errmessage{(Inkscape) Color is used for the text in Inkscape, but the package 'color.sty' is not loaded}%
    \renewcommand\color[2][]{}%
  }%
  \providecommand\transparent[1]{%
    \errmessage{(Inkscape) Transparency is used (non-zero) for the text in Inkscape, but the package 'transparent.sty' is not loaded}%
    \renewcommand\transparent[1]{}%
  }%
  \providecommand\rotatebox[2]{#2}%
  \ifx\svgwidth\undefined%
    \setlength{\unitlength}{75.25807728bp}%
    \ifx\svgscale\undefined%
      \relax%
    \else%
      \setlength{\unitlength}{\unitlength * \real{\svgscale}}%
    \fi%
  \else%
    \setlength{\unitlength}{\svgwidth}%
  \fi%
  \global\let\svgwidth\undefined%
  \global\let\svgscale\undefined%
  \makeatother%
  \begin{picture}(1,0.63321851)%
    \put(0,0){\includegraphics[width=\unitlength,page=1]{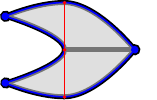}}%
    \put(0.83400868,0.42292462){\color[rgb]{0,0,0}\makebox(0,0)[lb]{\smash{$\sfv$}}}%
  \end{picture}%
\endgroup%
}}\!\!\!=\left(
\begin{tikzcd}[column sep=1pc]
\vcenter{\hbox{\def\sfvgscale{1}
\begingroup%
  \makeatletter%
  \providecommand\color[2][]{%
    \errmessage{(Inkscape) Color is used for the text in Inkscape, but the package 'color.sty' is not loaded}%
    \renewcommand\color[2][]{}%
  }%
  \providecommand\transparent[1]{%
    \errmessage{(Inkscape) Transparency is used (non-zero) for the text in Inkscape, but the package 'transparent.sty' is not loaded}%
    \renewcommand\transparent[1]{}%
  }%
  \providecommand\rotatebox[2]{#2}%
  \ifx\svgwidth\undefined%
    \setlength{\unitlength}{69.4574425bp}%
    \ifx\svgscale\undefined%
      \relax%
    \else%
      \setlength{\unitlength}{\unitlength * \real{\svgscale}}%
    \fi%
  \else%
    \setlength{\unitlength}{\svgwidth}%
  \fi%
  \global\let\svgwidth\undefined%
  \global\let\svgscale\undefined%
  \makeatother%
  \begin{picture}(1,1.01865929)%
    \put(0,0){\includegraphics[width=\unitlength,page=1]{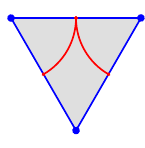}}%
    \put(0.45054523,0.01358183){\color[rgb]{0,0,0}\makebox(0,0)[lb]{\smash{$\bv_0$}}}%
    \put(0.92054877,0.97490499){\color[rgb]{0,0,0}\makebox(0,0)[lb]{\smash{$\bv_1$}}}%
    \put(-0.00478036,0.97199666){\color[rgb]{0,0,0}\makebox(0,0)[lb]{\smash{$\bv_2$}}}%
  \end{picture}%
\endgroup%
}}\ar[r,"\simeq"]&
\vcenter{\hbox{\def\sfvgscale{1}
\begingroup%
  \makeatletter%
  \providecommand\color[2][]{%
    \errmessage{(Inkscape) Color is used for the text in Inkscape, but the package 'color.sty' is not loaded}%
    \renewcommand\color[2][]{}%
  }%
  \providecommand\transparent[1]{%
    \errmessage{(Inkscape) Transparency is used (non-zero) for the text in Inkscape, but the package 'transparent.sty' is not loaded}%
    \renewcommand\transparent[1]{}%
  }%
  \providecommand\rotatebox[2]{#2}%
  \ifx\svgwidth\undefined%
    \setlength{\unitlength}{82.027897bp}%
    \ifx\svgscale\undefined%
      \relax%
    \else%
      \setlength{\unitlength}{\unitlength * \real{\svgscale}}%
    \fi%
  \else%
    \setlength{\unitlength}{\svgwidth}%
  \fi%
  \global\let\svgwidth\undefined%
  \global\let\svgscale\undefined%
  \makeatother%
  \begin{picture}(1,0.84338818)%
    \put(0,0){\includegraphics[width=\unitlength,page=1]{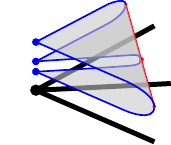}}%
    \put(0.11137751,0.66046517){\color[rgb]{0,0,0}\makebox(0,0)[lb]{\smash{$\bv_0$}}}%
    \put(0.05472964,0.49052191){\color[rgb]{0,0,0}\makebox(0,0)[lb]{\smash{$\bv_1$}}}%
    \put(0.05226676,0.40678218){\color[rgb]{0,0,0}\makebox(0,0)[lb]{\smash{$\bv_2$}}}%
    \put(0.15950347,0.16541219){\color[rgb]{0,0,0}\makebox(0,0)[lb]{\smash{$\sfv$}}}%
  \end{picture}%
\endgroup%
}}\ar[r,"f"]&
\vcenter{\hbox{\def\sfvgscale{1}
\begingroup%
  \makeatletter%
  \providecommand\color[2][]{%
    \errmessage{(Inkscape) Color is used for the text in Inkscape, but the package 'color.sty' is not loaded}%
    \renewcommand\color[2][]{}%
  }%
  \providecommand\transparent[1]{%
    \errmessage{(Inkscape) Transparency is used (non-zero) for the text in Inkscape, but the package 'transparent.sty' is not loaded}%
    \renewcommand\transparent[1]{}%
  }%
  \providecommand\rotatebox[2]{#2}%
  \ifx\svgwidth\undefined%
    \setlength{\unitlength}{55.6679686bp}%
    \ifx\svgscale\undefined%
      \relax%
    \else%
      \setlength{\unitlength}{\unitlength * \real{\svgscale}}%
    \fi%
  \else%
    \setlength{\unitlength}{\svgwidth}%
  \fi%
  \global\let\svgwidth\undefined%
  \global\let\svgscale\undefined%
  \makeatother%
  \begin{picture}(1,1.1855457)%
    \put(0,0){\includegraphics[width=\unitlength,page=1]{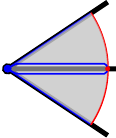}}%
    \put(-0.00596449,0.37720928){\color[rgb]{0,0,0}\makebox(0,0)[lb]{\smash{$\sfv$}}}%
  \end{picture}%
\endgroup%
}}
\end{tikzcd}\right)
\]

Then $f$ is orientation-preserving on two neighborhoods $\bU_{\sfv_1}$ and $\bU_{\sfv_2}$ of vertices $\sfv_1$ and $\sfv_2$ and orientation-reversing on $\bU_{\sfv_0}$. The critical graph $\bO_f$ depicted in red is connected and intersects each and every edge exactly once, and so edges are $(1/2)$, $(2/2)$ and $(1/2)$-folding (see Figure~\ref{fig:folding degree}).
\end{example}

\begin{example}[Infinitesimal bigons]\label{example:infinitesimal bigons} Another examples are infinitesimal bigons as follows:
\begin{align*}
&\vcenter{\hbox{
\begingroup%
  \makeatletter%
  \providecommand\color[2][]{%
    \errmessage{(Inkscape) Color is used for the text in Inkscape, but the package 'color.sty' is not loaded}%
    \renewcommand\color[2][]{}%
  }%
  \providecommand\transparent[1]{%
    \errmessage{(Inkscape) Transparency is used (non-zero) for the text in Inkscape, but the package 'transparent.sty' is not loaded}%
    \renewcommand\transparent[1]{}%
  }%
  \providecommand\rotatebox[2]{#2}%
  \ifx\svgwidth\undefined%
    \setlength{\unitlength}{58.40840098bp}%
    \ifx\svgscale\undefined%
      \relax%
    \else%
      \setlength{\unitlength}{\unitlength * \real{\svgscale}}%
    \fi%
  \else%
    \setlength{\unitlength}{\svgwidth}%
  \fi%
  \global\let\svgwidth\undefined%
  \global\let\svgscale\undefined%
  \makeatother%
  \begin{picture}(1,1.23854202)%
    \put(0,0){\includegraphics[width=\unitlength,page=1]{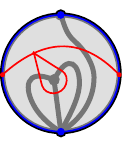}}%
    \put(0.41120876,0.01207152){\color[rgb]{0,0,0}\makebox(0,0)[lb]{\smash{$\bv_0$}}}%
    \put(0.41120876,1.18584198){\color[rgb]{0,0,0}\makebox(0,0)[lb]{\smash{$\bv_1$}}}%
  \end{picture}%
\endgroup%
}}&
&\vcenter{\hbox{
\begingroup%
  \makeatletter%
  \providecommand\color[2][]{%
    \errmessage{(Inkscape) Color is used for the text in Inkscape, but the package 'color.sty' is not loaded}%
    \renewcommand\color[2][]{}%
  }%
  \providecommand\transparent[1]{%
    \errmessage{(Inkscape) Transparency is used (non-zero) for the text in Inkscape, but the package 'transparent.sty' is not loaded}%
    \renewcommand\transparent[1]{}%
  }%
  \providecommand\rotatebox[2]{#2}%
  \ifx\svgwidth\undefined%
    \setlength{\unitlength}{58.40840098bp}%
    \ifx\svgscale\undefined%
      \relax%
    \else%
      \setlength{\unitlength}{\unitlength * \real{\svgscale}}%
    \fi%
  \else%
    \setlength{\unitlength}{\svgwidth}%
  \fi%
  \global\let\svgwidth\undefined%
  \global\let\svgscale\undefined%
  \makeatother%
  \begin{picture}(1,1.23854202)%
    \put(0,0){\includegraphics[width=\unitlength,page=1]{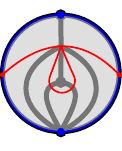}}%
    \put(0.41120876,0.01207152){\color[rgb]{0,0,0}\makebox(0,0)[lb]{\smash{$\bv_0$}}}%
    \put(0.41120876,1.18584198){\color[rgb]{0,0,0}\makebox(0,0)[lb]{\smash{$\bv_1$}}}%
  \end{picture}%
\endgroup%
}}&
&\vcenter{\hbox{
\begingroup%
  \makeatletter%
  \providecommand\color[2][]{%
    \errmessage{(Inkscape) Color is used for the text in Inkscape, but the package 'color.sty' is not loaded}%
    \renewcommand\color[2][]{}%
  }%
  \providecommand\transparent[1]{%
    \errmessage{(Inkscape) Transparency is used (non-zero) for the text in Inkscape, but the package 'transparent.sty' is not loaded}%
    \renewcommand\transparent[1]{}%
  }%
  \providecommand\rotatebox[2]{#2}%
  \ifx\svgwidth\undefined%
    \setlength{\unitlength}{58.40839883bp}%
    \ifx\svgscale\undefined%
      \relax%
    \else%
      \setlength{\unitlength}{\unitlength * \real{\svgscale}}%
    \fi%
  \else%
    \setlength{\unitlength}{\svgwidth}%
  \fi%
  \global\let\svgwidth\undefined%
  \global\let\svgscale\undefined%
  \makeatother%
  \begin{picture}(1,1.23854206)%
    \put(0,0){\includegraphics[width=\unitlength,page=1]{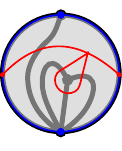}}%
    \put(0.41120877,0.01207152){\color[rgb]{0,0,0}\makebox(0,0)[lb]{\smash{$\bv_0$}}}%
    \put(0.41120877,1.18584202){\color[rgb]{0,0,0}\makebox(0,0)[lb]{\smash{$\bv_1$}}}%
  \end{picture}%
\endgroup%
}}
\end{align*}
where the vertex $f(\bv_0)$ is trivalent, and $f$ covers 5 sectors on $\bU_{\bv_0}$ and 2 sectors on $\bU_{\bv_1}$ and both edges are $(1/2)$-folding.

Notice that these disks are distinct even up to the action of the diffeomorphism group on the domain.
\end{example}

Suppose that $f$ is infinitesimal admissible. 
Since there is a unique special point $\sfv$ in $\cl{\sU_\sfv}$,
\[
f(\bv_0)=f(\bv_1)=\cdots=f(\bv_t)=\sfv
\]
and so every edge contains a critical point, which is unique and $(m/2)$-folding for some $m$ by (1)\footnote{See, figures in \S~\ref{subsubsection:infinitesimal disks} for examples of the infinitesimal disks. One can see that every edge of such disks contains a unique critical point of $(m/2)$-folding type.}.
Moreover, each component of $\Pi\setminus (\bL_f\cup\bO_f)$ is diffeomorphically mapped onto the interior of a sector $\sector_{\sfv,i}$ for some $i$ by $f$.
Therefore the following conditions hold:
\begin{enumerate}
\item [(4)] All vertices are neutral;
\item [(5)] All edges are $(m/2)$-folding for some $m\ge 1$;
\item [(6)] All $\bU_\bv$'s are disjoint.
\end{enumerate}

Condition (2) implies that the orientation is reversed an odd number of times on neighborhoods of $\be_i$ if and only if $t>0$ and $i$ is either $0$ or $t$. Therefore
\begin{enumerate}
\item [($\rm5'$)] $m$ is odd if and only if $t>0$ and $\be$ is either $\be_0$ or $\be_t$.
\end{enumerate}

On the other hand, by condition (1), the graph $\bL_f$ defines a singular foliation $\mathbf{FL}_f$, where the critical point on a $(m/2)$-folding edge corresponds to a $(m+1)$-prong hyperbolic singularity while points $\bx$ with $f(\bx)=\sfv$ including vertices correspond to elliptic singularities.
Furthermore, one can regard $\bx$ is a source or a sink if $f$ is orientation-reversing or orientation-preserving, respectively. Then condition (2) implies uniqueness of the source which is precisely $\bv_0$.
\begin{align*}
\text{Sources}&=\{\bv_0\},&
\text{Sinks}&=\{\bv_1,\cdots,\bv_t\}\cup\{\bx\in\rPi\mid f(\bx)=\sfv\}
\end{align*}
Notice that the graph $\bO_f$ defines a dual foliation $\mathbf{FO}_f$ and so the sinks of $\mathbf{FL}_f$ in $\rPi$ have one-to-one correspondence with minimal circuits of $\bO_f$. See Figure~\ref{fig:singular foliation}.

\begin{figure}[ht]
\begin{align*}
f&=\vcenter{\hbox{\input{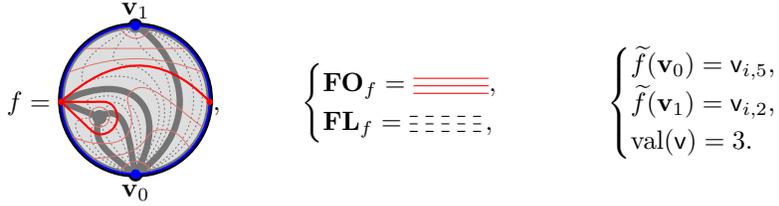}}},&
&\begin{cases}
\mathbf{FO}_f=
\begin{tikzpicture}[baseline=-.5ex]
\draw[color=red](0,0)--(1,0);
\draw[color=red](0,-.1)--(1,-.1);
\draw[color=red](0,.1)--(1,.1);
\end{tikzpicture},\\
\mathbf{FL}_f=\begin{tikzpicture}[baseline=-.5ex]
\draw[dashed](0,0)--(1,0);
\draw[dashed](0,-.1)--(1,-.1);
\draw[dashed](0,.1)--(1,.1);
\end{tikzpicture},
\end{cases}
&
&\begin{cases}
\tilde f(\bv_0)=\sfv_{i,5},\\
\tilde f(\bv_1)=\sfv_{i,2},\\
\val(\sfv)=3.
\end{cases}
\end{align*}
\caption{The singular foliations $\mathbf{FL}_f$ and $\mathbf{FO}_f$ on an infinitesimal admissible disk}
\label{fig:singular foliation}
\end{figure}

\begin{remark}
Let us consider a geometric realization of infinitesimal disks near $\sfv$ as a pseudo-holomorphic map from a punctured disk to the symplectization of the geometric model $(\cB\Lambda,\cE\Lambda)$ in \S~\ref{sec:geometric model}. Note that some of the images lie on the plane generated by the Reeb orbit $R_v$ on the boundary $\phi_v^{-1}(\cN_v)$ and the symplectization coordinate. So its corresponding image in $\RR^2_{xy}$ is degenerate and has area zero. To visualize this disk in $\RR_{xy}^2$, we need to introduce the orientation-reversing part which cancels out the orientation-preserving region and makes the total area zero.
\end{remark}

\begin{remark}\label{remark:Legendrian mirror on disk}
The Legendrian mirror $\mu$ gives us a bijection between admissible disks of a given Legendrian $\sL$ and its mirror $\mu(\sL)$ as follows: Let $f$ be an admissible disk on $\Pi_t$ for $\sL$. Since $\mu$ on $\RR^2_{xy}$ is orientation-reversing, by pre-composition with the orientation-reversing involution $\tau$ on $\Pi_t$ sending $\bv_i\mapsto\bv_{t+1-i}$, we have an admissible disk for $\mu(\sL)$. Therefore by identifying $\sG_{\mu(\sL)}$ with $\mu(\sG_\sL)$ in the canonical way, we have
\begin{align}\label{equation:Legendrian mirror on labels}
\tilde {\mu(f)}(\bv_i)=\mu(\tilde f(\bv_{t+1-i}))\in\sG_{\mu(\sL)}.
\end{align}
\end{remark}

\subsubsection{Degree of admissible disks}
Let $f$ be a regular or infinitesimal admissible disk on $\Pi_t$. 
For each edge $\be\in \ePi$ parametrized by $u\in[0,1]$, we define a path $D\be(u)$ in $\gr(1,\RR^2)$ of derivatives of $\be(u)$ by setting $D\be(u)$ to be $T_{f(\be(u))} \sL$ when $D\be(u)$ vanishes. This can be done in the obvious way.

On the other hand, for two adjacent edges $\be, \be'$ sharing common vertex $\bv$, the end points $D\be(1)$ and $D\be'(0)$ are exactly the same as end points of a path $\tilde f(\bv)^{-1}$ if $\bv=\bv_0$ or $\tilde f(\bv)$ otherwise.

Hence by concatenating all vertex paths and edge paths in an alternating manner, we obtain the loop in $\gr(1,\RR^2)$ as follows:
\begin{align*}
D(\partial f)&\coloneqq \tilde f(\bv_0)^{-1}\cdot D\be_0\cdot \tilde f(\bv_1)\cdot D\be_1 \cdot \quad\cdots\quad \cdot D\be_{t-1}\cdot \tilde f(\bv_t)\cdot D\be_t\in \pi_1(\gr(1,\RR^2)).
\end{align*}

\begin{definition}[Degree of an admissible disk]\label{definition:degree of admissible disk}
The {\em $\ZZ$-degree $|f|_\ZZ$ of $f$} is defined as the integer determined by the homotopy class $[D(\partial f)]$
\[
[D(\partial f)]=|f|_\ZZ\cdot 1_{\gr}\in \pi_1(\gr(1,\RR^2)),
\]
and the {\em ($\fR$-)degree $|f|$} is defined as
\[
|f|\coloneqq|f|_\ZZ\cdot 1_\fR\in \fR.
\]
\end{definition}

Note that the degree $|f|_\ZZ$ does not depend on the choice of the Maslov potential but only on the immersion $f$.
However, the following proposition shows how the two degrees $|f|_\ZZ$ and $|\cdot|_\fP$ are related. 

\begin{proposition}\label{proposition:degree formula}
The following holds:
\[
|f|=\left|\tilde f(\bv_0)\right|_\fP-\sum_{r=1}^t\left|\tilde f(\bv_r)\right|_\fP.
\]
\end{proposition}

\begin{proof}
For convenience's sake, let us denote $\tilde f(\bv_r)$ by $\sfg_r$ for $0\le r\le t$. Suppose that $f$ is regular.
We start by choosing reference curves $\gamma_{\sfg_{r}}^\pm:[0,1]\to\sL$ from $f(\bv_r)$ to the vertices\footnote{If $\sfg_r$ is a vertex, then we may choose a constant path on $\sfg_r$.} which are used in \S~\ref{subsubsection:grading on c} to define $|\sfg_r|_\fP$ which satisfy the following conditions:
\begin{align}\label{eqn:reference curve}
\gamma^+_{\sfg_0}&=f|_{\be_0} \cdot \gamma^+_{\sfg_1},&
\gamma^-_{\sfg_t}&=f|_{\be_t}\cdot \gamma^-_{\sfg_0},&
\gamma^-_{\sfg_i}&=f|_{\be_i}\cdot \gamma^+_{\sfg_{i+1}}, 0\le i<t
\end{align}
up to homotopy.

It is easy to check that such a set of reference curves exists.
Since $\gamma_{\sfg_r}^\pm(1)\in \sV(\sL)$, for a sufficiently small $\epsilon>0$, the image $\im(\gamma_{\sfg_r}^\pm|_{(1-\epsilon,1]})$ determines an half edge of the vertex $\gamma_{\sfg_r}^\pm(1)$. We simply denote the induced half edge by $\sfh(\gamma_{\sfg_r}^\pm)$,
then (\ref{eqn:reference curve}) implies that
\begin{align}\label{eqn:half edges}
\sfh(\gamma^+_{\sfg_0})&=\sfh(\gamma^+_{\sfg_1}),& \sfh(\gamma^-_{\sfg_t})&=\sfh(\gamma^-_{\sfg_0}),& \sfh(\gamma^-_{\sfg_i})&=\sfh(\gamma^+_{\sfg_{i+1}}),
\end{align}
for $0\le i<t$.

Now recall from (\ref{eq:Gr1}) in \S~\ref{section:grading_singularpoint} that
\[
|\sfg_r|_\fP=\fP(\sfh(\gamma^+_{\sfg_r}))-\fP(\sfh(\gamma^-_{\sfg_r}))+n_{\sfg_r}\cdot 1_{\fR},
\]
where $n_{\sfg_r}$ is defined to satisfy $[\hat\sfg_r]=n_{\sfg_r}\cdot 1_{\gr}\in\pi_1(\gr(1,\RR^2))$ for
\[
\hat \sfg_r(t) = D\gamma^+_{\sfg_r}(t)\cdot D\gamma^-_{\sfg_r}(t)^{-1}\cdot \sfg_r(t)^{-1} \in \gr(1,\RR^2).
\]
Then the right hand side of the statement becomes
\begin{align*}
|\sfg_0|_\fP-\sum_{r=1}^t\left| \sfg_r \right|_\fP
&= \fP(\sfh(\gamma^+_{\sfg_0}))-\fP(\sfh(\gamma^-_{\sfg_0}))+n_{\sfg_0}\cdot 1_{\fR}\\
&\phantom{=} -\sum_{r=1}^{t}\fP(\sfh(\gamma^+_{\sfg_r}))+\sum_{r=1}^{t}\fP(\sfh(\gamma^-_{\sfg_r}))-\sum_{r=1}^t n_{\sfg_r}\cdot 1_{\fR}\\
&=n_{\sfg_0}\cdot 1_{\fR} - \sum_{r=1}^t n_{\sfg_r}\cdot 1_{\fR},
\end{align*}
where we use (\ref{eqn:half edges}) for the second equality.

So it suffices to show that
\[
[D(\partial f)]=n_{\sfg_0}\cdot 1_{\gr} - \sum_{r=1}^t n_{\sfg_r}\cdot 1_{\gr} \in \pi_1(\gr(1,\RR^2)).
\]
By definition, $n_{\sfg_0}\cdot 1_{\gr} - \sum_{r=1}^t n_{\sfg_r}\cdot 1_{\gr}$ is equal to
\begin{align*}
[D\gamma^+_{\sfg_0}(t)\cdot &D\gamma^-_{\sfg_0}(t)^{-1}\cdot \sfg_0(t)^{-1}\cdot \prod_{r=1}^t\left( D\gamma^+_{\sfg_r}(t)\cdot D\gamma^-_{\sfg_r}(t)^{-1}\cdot \sfg_r(t)^{-1}\right)^{-1} ]\\
&=[D\gamma^+_{\sfg_0}(t)\cdot D\gamma^-_{\sfg_0}(t)^{-1}\cdot \sfg_0(t)^{-1}\cdot \prod_{r=1}^t\left(\sfg_r(t) \cdot D\gamma^-_{\sfg_r}(t)\cdot D\gamma^+_{\sfg_r}(t)^{-1}\right) ]\\
&=[ D\gamma^-_{\sfg_0}(t)^{-1}\cdot \sfg_0(t)^{-1}\cdot D\gamma^+_{\sfg_0}(t)\cdot \prod_{r=1}^t\left( D\gamma^+_{\sfg_r}(t)^{-1}\cdot \sfg_r(t)\cdot D\gamma^-_{\sfg_r}(t)\right) ]\\
&=[ D\gamma^-_{\sfg_0}(t)^{-1}\cdot \sfg_0(t)^{-1}\cdot Df|_{\be_0} \cdot \sfg_1(t)\cdots Df|_{\be_{t-1}}\cdot \sfg_t(t)\cdot D\gamma^-_{\sfg_t}(t)]\\
&=[D(\partial f)],
\end{align*}
where we use (\ref{eqn:reference curve}) for the third and fourth equalities above.

Suppose that $f$ is infinitesimal with $\tilde f(\bv_j)=\sfv_{i_j,\ell_j}$ for some $i_j\in\Zmod{\val(\sfv)}$. Then 
$\ell_0=\ell_1+\cdots+\ell_t+m\cdot\val(\sfv)$ 
and by (\ref{eq:Gr3}), we have
\begin{align*}
& &|\sfv_{i_0,\ell_0}|_\fP=\sum_{j=1}^t |\sfv_{i_j,\ell_j}|_\fP + ((t-1)+2m)\cdot 1_\fR\notag\\
&\Longleftrightarrow&
|\sfv_{i_0,\ell_0}|_\fP-\sum_{j=1}^t|\sfv_{i_j,\ell_j}|_\fP = |f|_\ZZ\cdot 1_\fR = |f|.\qedhere
\end{align*}
\end{proof}

\subsubsection{Equivalences of admissible disks}
We may regard two admissible disks as the same whenever they are {\em combinatorially} the same, or equivalently, their critical sets are {\em isomorphic} via one-parameter family of admissible disks in the following sense.

\begin{definition}[Equivalent classes of admissible disks]\label{definition:equivalence of admissible disks}
We say that two admissible disks $f$ and $f'$ are {\em equivalent} and denote this by $f\sim f'$ if there is a family of admissible disks $\{f_u\mid u\in[0,1]\}\subset\tilde\cM_t(\sfg)$ such that
\begin{align*}
f_0&=f,&
f_1&=f',&
(\Pi, \bO_{f_0})&\cong (\Pi, \bO_{f_u})\cong (\Pi, \bO_{f_1}),
\end{align*}
where the isomorphism $\cong$ means the diffeomorphism inducing an isomorphism between properly embedded graphs (or stratified spaces).
\end{definition}

Notice that if two disks are the same up to the action of $\Diff^+(\Pi,\partial\Pi, \vPi)$ by the pre-composition, then they are equivalent. Therefore the usual equivalence relation coming from the $\Diff^+(\Pi,\partial\Pi, \vPi)$ action is finer than our equivalence relation.

\begin{example}[Infinitesimal bigons again]
Recall infinitesimal bigons in Example~\ref{example:infinitesimal bigons}.
As mentioned there, they are all distinct but can be connected by one-parameter family of admissible disks.
Roughly speaking this family can be obtained by {\em rolling} the image of the region containing a sink.
\end{example}

\begin{notation}\label{notation:admissible disks}
From now on, we will focus on admissible disks of degree 1 or 2, and define the following sets:
\begin{align*}
\cM&\coloneqq\{f\mid f\text{ is admissible}, |f|_\ZZ=1, 2\}\big/\sim, &\cM_t&\coloneqq\{f\in\cM\mid \text{ Domain of }f=\Pi_t \};\\
\cM^\reg&\coloneqq\{f\in\cM\mid f\text{ is regular}\},&\cM_{(d)}&\coloneqq\{f\in\cM\mid |f|_\ZZ=d\}; \\
\cM^{\sf{inf}}&\coloneqq\{f\in\cM\mid f\text{ is infinitesimal}\},&\cM(\sfg)&\coloneqq\{f\in\cM\mid \tilde f(\bv_0)=\sfg\}.
\end{align*}
We also introduce notation for their intersections by decorating with sub- and superscripts. For example,
\begin{align*}
\cM^\reg_t(\sfc)&\coloneqq\cM^\reg \cap \cM_t \cap \cM(\sfc),& \cM(\sfv_{i,\ell})_{(d)}&\coloneqq \cM(\sfv_{i,\ell}) \cap \cM_{(d)}.
\end{align*}
\end{notation}

\subsection{Classifications of admissible disks}
We classify admissible disks in $\cM$, which are of degree 1 or 2.

\begin{proposition}\label{proposition:degree admissible disk}Let $f$ be an admissible disk. Then the $\ZZ$-degree $|f|_\ZZ$ of $f$ can be computed as follows:
\begin{enumerate}
\item If $f$ is regular,
\[
|f|_\ZZ=1+\#(2\text{-folding edges})+\#(\text{concave vertices}).
\]
\item If $f$ is infinitesimal,
\[
|f|_\ZZ=-2\chi(\bO_f)+(t+1).
\]
\end{enumerate}
\end{proposition}
\begin{proof}
\noindent~(1)~Let $\gamma_f$ be the curve in $\Pi$ which is the inner boundary of a collar neighborhood of $\partial\Pi$. Then $D\gamma_f$ defines a loop in $\gr(1,\RR^2)$ corresponding to $2\cdot 1_{\gr}$ since $\gamma_f$ is smooth and one full rotation. Note that $D\gamma_f$ and $D(\partial f)$ can be identified away from $\bU_{\bv}$ for $\bv\in \bV\Pi$ and $\crit(f)$, i.e., the singular points in the edge.
We need to compare the local path $D\gamma_f|_{\bU_{\bv}}$ and the induced capping path $D(\partial f)|_{\bU_{\bv}}$ in $\gr(1,2)$. When $\bv\in\vPi$ is
\begin{enumerate}
\item convex negative or neutral, then $D\gamma_f|_{\bU_{\bv}}$ and $D(\partial f)|_{\bU_{\bv}}$ are homotopic;
\item concave negative, $D(\partial f)|_{\bU_{\bv}}$ contributes a $\pi$-rotation compared to $D\gamma_f|_{\bU_{\bv}}$;
\item convex positive, then $D(\partial f)|_{\bU_{\bv}}$ is the inverse of the capping path $-\tilde f(\bv)$. So it contributes a $(-\pi)$-rotation followed by $D\gamma_f|_{\bU_{\bv}}$;
\item concave positive, the two paths on $\pi_1(\gr(1,2))$ agree.
\end{enumerate}

Now suppose that any edges $\be$ containing a singular point $\bp$ are 2-folding. Then $D\gamma_f|_{\bU_{\bp}}$ contributes a $(-\pi)$-rotation, while $D(\partial f)|_{\bU_{\bp}}$ is constant.
Therefore by summing up all these contributions,
\[
|f|_\ZZ=[D\gamma_f]-\#(\text{positive vertices})+\#(2\text{-folding edges})+\#(\text{concave vertices}).
\]

\noindent~(2)~Suppose that for each $0\le j\le t$, $\tilde f(\bv_j)=\sfv_{i_j,\ell_j}$ for some $i_j\in\Zmod{\val(\sfv)}$ with $\ell_j\ge 1$ such that $i_j=i_0+\ell_0+\cdots+\ell_{j-1}$.
Then it is obvious that $\ell_0 = \ell_1+\cdots +\ell_t + m\cdot \val(\sfv)$ for some $m\ge 0$ since $f$ is a disk.
Moreover, by considering the singular foliation defined by $\bL_f$ mentioned earlier, $m$ is the same as the number of sinks in the interior, or the number of minimal circuits of $\bO_f$, which can be formulated as $m=1-\chi(\bO_f)$.
This implies that near $\bv_0$, $f$ makes $m$ additional full turns counterclockwise, which contributes $2m$ to $|f|_\ZZ\in\pi_1(\gr(1,2))$.

On the other hand, all edges other than $\be_0$ and $\be_t$ play exactly the same role of folding edges and contribute $1$ to $|f|_\ZZ$. Here we exclude $\be_0$ and $\be_t$ since $f$ changes orientations.

Finally, the turns corresponding to $\ell_1, \cdots, \ell_t$ are cancelled with $\ell_0-m\cdot\val(\sfv)$ and so there are no more contributions to $|f|_\ZZ$. Therefore we have
\[
|f|_\ZZ = 2(1-\chi(\bO_f))+(t+1-2)=-2\chi(\bO_f)+(t+1)
\]
as claimed.
\end{proof}

\subsubsection{Regular admissible disks}
Let $f$ be a regular admissible disk.
Then
\[
|f|_\ZZ=1+\#(2\text{-folding edges})+\#(\text{concave vertices}).
\]

\begin{enumerate}
\item Degree 1 disks ($|f|_\ZZ=1$) : All edges are smooth and all vertices are convex.

\item Degree 2 disks ($|f|_\ZZ=2$) : Either all edges are smooth and there is only one concave (not necessarily negative) vertex $\bv$,
\[
\begin{tikzcd}
\vcenter{\hbox{\includegraphics{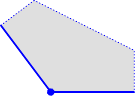}}}\ar[r,"f"]&
\vcenter{\hbox{\includegraphics{crossing_sign_concave_positive.pdf}}}\quad
\vcenter{\hbox{\scalebox{-1}{\includegraphics{crossing_sign_concave_positive.pdf}}}}\quad
\vcenter{\hbox{\includegraphics{crossing_sign_concave_negative.pdf}}}\quad
\vcenter{\hbox{\includegraphics{crossing_sign_concave_negative.pdf}}}
\in\cM^\concavevertex
\end{tikzcd}
\]
or $f$ has exactly one 2-folding edge $\be$ with all convex vertices, which gives us a one-parameter family of admissible disks as follows:
\[
\begin{tikzcd}
\vcenter{\hbox{\includegraphics{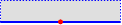}}}\ar[r,"f"]&
\vcenter{\hbox{\includegraphics{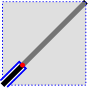}}}=
\vcenter{\hbox{\includegraphics{folding_edge_2.pdf}}}=
\vcenter{\hbox{\includegraphics{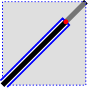}}}
\in\cM^\positivefolding
\end{tikzcd}
\]
Obviously, these disks are all equivalent by Definition~\ref{definition:equivalence of admissible disks}. We denote the sets of regular admissible disks of degree 2 with a folding edge and with a concave vertex by $\cM^\positivefolding$ and $\cM^\concavevertex$, respectively.
\[
\cM^\reg_{(2)}=\cM^\positivefolding\amalg\cM^\concavevertex
\]
\end{enumerate}

\subsubsection{Infinitesimal disks}\label{subsubsection:infinitesimal disks}
Let $f$ be an infinitesimal admissible disk on $\Pi_t$ with $\tilde f(\bv_0)=\sfv_{i,\ell}\in\tilde\sV_\sL$. We use the formula in Proposition~\ref{proposition:degree admissible disk}(2)
\[
|f|_\ZZ=-2\chi(\bO_f)+(t+1).
\]
Since $\bO_f$ is connected and $|f|_\ZZ\le 2$, either $\chi(\bO_f)=0$ or $\chi(\bO_f)=1$.

\begin{enumerate}
\item Degree 1 ($|f|_\ZZ=1$) : There are two possibilities,
\begin{align*}
\begin{cases}
t=0;\\
\chi(\bO_f)=0,
\end{cases}&\Longleftrightarrow
\begin{cases}
\Pi \text{ is a monogon};\\
\bO_f=S^1,
\end{cases}
&
\begin{cases}
t=2;\\
\chi(\bO_f)=1,
\end{cases}&\Longleftrightarrow
\begin{cases}
\Pi \text{ is a triangle};\\
\bO_f \text{ is a tree},
\end{cases}
\end{align*}
where each corresponds to either $\tilde f(\bv_0)=\sfv_{i,\val(\sfv)}$ and so
\[
\vcenter{\hbox{\def\svgscale{0.9}
\begingroup%
  \makeatletter%
  \providecommand\color[2][]{%
    \errmessage{(Inkscape) Color is used for the text in Inkscape, but the package 'color.sty' is not loaded}%
    \renewcommand\color[2][]{}%
  }%
  \providecommand\transparent[1]{%
    \errmessage{(Inkscape) Transparency is used (non-zero) for the text in Inkscape, but the package 'transparent.sty' is not loaded}%
    \renewcommand\transparent[1]{}%
  }%
  \providecommand\rotatebox[2]{#2}%
  \ifx\svgwidth\undefined%
    \setlength{\unitlength}{71.22725826bp}%
    \ifx\svgscale\undefined%
      \relax%
    \else%
      \setlength{\unitlength}{\unitlength * \real{\svgscale}}%
    \fi%
  \else%
    \setlength{\unitlength}{\svgwidth}%
  \fi%
  \global\let\svgwidth\undefined%
  \global\let\svgscale\undefined%
  \makeatother%
  \begin{picture}(1,0.81994741)%
    \put(0,0){\includegraphics[width=\unitlength,page=1]{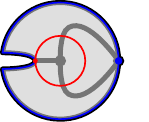}}%
    \put(0.82461508,0.52505124){\color[rgb]{0,0,0}\makebox(0,0)[lb]{\smash{$\sfv$}}}%
  \end{picture}%
\endgroup%
}}\!\!\!\!\!=\left(
\begin{tikzcd}[column sep=1pc]
\vcenter{\hbox{\def\svgscale{0.9}
\begingroup%
  \makeatletter%
  \providecommand\color[2][]{%
    \errmessage{(Inkscape) Color is used for the text in Inkscape, but the package 'color.sty' is not loaded}%
    \renewcommand\color[2][]{}%
  }%
  \providecommand\transparent[1]{%
    \errmessage{(Inkscape) Transparency is used (non-zero) for the text in Inkscape, but the package 'transparent.sty' is not loaded}%
    \renewcommand\transparent[1]{}%
  }%
  \providecommand\rotatebox[2]{#2}%
  \ifx\svgwidth\undefined%
    \setlength{\unitlength}{56.79996323bp}%
    \ifx\svgscale\undefined%
      \relax%
    \else%
      \setlength{\unitlength}{\unitlength * \real{\svgscale}}%
    \fi%
  \else%
    \setlength{\unitlength}{\svgwidth}%
  \fi%
  \global\let\svgwidth\undefined%
  \global\let\svgscale\undefined%
  \makeatother%
  \begin{picture}(1,1.14991942)%
    \put(0,0){\includegraphics[width=\unitlength,page=1]{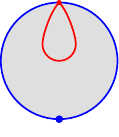}}%
    \put(0.40862016,0.01660845){\color[rgb]{0,0,0}\makebox(0,0)[lb]{\smash{$\bv_0$}}}%
  \end{picture}%
\endgroup%
}}\ar[r,"\simeq"]&
\vcenter{\hbox{\def\svgscale{0.9}
\begingroup%
  \makeatletter%
  \providecommand\color[2][]{%
    \errmessage{(Inkscape) Color is used for the text in Inkscape, but the package 'color.sty' is not loaded}%
    \renewcommand\color[2][]{}%
  }%
  \providecommand\transparent[1]{%
    \errmessage{(Inkscape) Transparency is used (non-zero) for the text in Inkscape, but the package 'transparent.sty' is not loaded}%
    \renewcommand\transparent[1]{}%
  }%
  \providecommand\rotatebox[2]{#2}%
  \ifx\svgwidth\undefined%
    \setlength{\unitlength}{92.39999976bp}%
    \ifx\svgscale\undefined%
      \relax%
    \else%
      \setlength{\unitlength}{\unitlength * \real{\svgscale}}%
    \fi%
  \else%
    \setlength{\unitlength}{\svgwidth}%
  \fi%
  \global\let\svgwidth\undefined%
  \global\let\svgscale\undefined%
  \makeatother%
  \begin{picture}(1,0.85056873)%
    \put(0,0){\includegraphics[width=\unitlength,page=1]{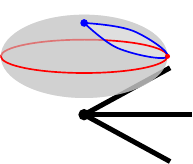}}%
    \put(0.39393939,0.81767845){\color[rgb]{0,0,0}\makebox(0,0)[lb]{\smash{$\bv_0$}}}%
    \put(0.39393939,0.12503776){\color[rgb]{0,0,0}\makebox(0,0)[lb]{\smash{$\sfv$}}}%
  \end{picture}%
\endgroup%
}}\ar[r,"f"]&
\vcenter{\hbox{\def\svgscale{0.9}
\begingroup%
  \makeatletter%
  \providecommand\color[2][]{%
    \errmessage{(Inkscape) Color is used for the text in Inkscape, but the package 'color.sty' is not loaded}%
    \renewcommand\color[2][]{}%
  }%
  \providecommand\transparent[1]{%
    \errmessage{(Inkscape) Transparency is used (non-zero) for the text in Inkscape, but the package 'transparent.sty' is not loaded}%
    \renewcommand\transparent[1]{}%
  }%
  \providecommand\rotatebox[2]{#2}%
  \ifx\svgwidth\undefined%
    \setlength{\unitlength}{76.3999998bp}%
    \ifx\svgscale\undefined%
      \relax%
    \else%
      \setlength{\unitlength}{\unitlength * \real{\svgscale}}%
    \fi%
  \else%
    \setlength{\unitlength}{\svgwidth}%
  \fi%
  \global\let\svgwidth\undefined%
  \global\let\svgscale\undefined%
  \makeatother%
  \begin{picture}(1,0.84816754)%
    \put(0,0){\includegraphics[width=\unitlength,page=1]{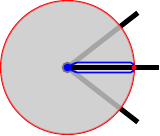}}%
    \put(0.37172775,0.26701571){\color[rgb]{0,0,0}\makebox(0,0)[lb]{\smash{$\sfv$}}}%
  \end{picture}%
\endgroup%
}}
\end{tikzcd}\right)\in\cM^\infinitesimalmonogon
\]
or $\tilde f(\bv_0)=\sfv_{i,\ell_0}$, $\tilde f(\bv_1)=\sfv_{i,\ell_1}$ and $\tilde f(\bv_2)=\sfv_{i+\ell_1,\ell_2}$ with $\ell_0=\ell_1+\ell_2$.
\[
\vcenter{\hbox{\def\svgscale{0.9}
\begingroup%
  \makeatletter%
  \providecommand\color[2][]{%
    \errmessage{(Inkscape) Color is used for the text in Inkscape, but the package 'color.sty' is not loaded}%
    \renewcommand\color[2][]{}%
  }%
  \providecommand\transparent[1]{%
    \errmessage{(Inkscape) Transparency is used (non-zero) for the text in Inkscape, but the package 'transparent.sty' is not loaded}%
    \renewcommand\transparent[1]{}%
  }%
  \providecommand\rotatebox[2]{#2}%
  \ifx\svgwidth\undefined%
    \setlength{\unitlength}{75.25807728bp}%
    \ifx\svgscale\undefined%
      \relax%
    \else%
      \setlength{\unitlength}{\unitlength * \real{\svgscale}}%
    \fi%
  \else%
    \setlength{\unitlength}{\svgwidth}%
  \fi%
  \global\let\svgwidth\undefined%
  \global\let\svgscale\undefined%
  \makeatother%
  \begin{picture}(1,0.63321851)%
    \put(0,0){\includegraphics[width=\unitlength,page=1]{triangle_degree_1_2.pdf}}%
    \put(0.83400868,0.42292462){\color[rgb]{0,0,0}\makebox(0,0)[lb]{\smash{$\sfv$}}}%
  \end{picture}%
\endgroup%
}}\!\!\!=\left(
\begin{tikzcd}[column sep=1pc]
\vcenter{\hbox{\def\svgscale{0.9}}}\ar[r,"\simeq"]&
\vcenter{\hbox{\def\svgscale{0.9}}}\ar[r,"f"]&
\vcenter{\hbox{\def\svgscale{0.9}
\begingroup%
  \makeatletter%
  \providecommand\color[2][]{%
    \errmessage{(Inkscape) Color is used for the text in Inkscape, but the package 'color.sty' is not loaded}%
    \renewcommand\color[2][]{}%
  }%
  \providecommand\transparent[1]{%
    \errmessage{(Inkscape) Transparency is used (non-zero) for the text in Inkscape, but the package 'transparent.sty' is not loaded}%
    \renewcommand\transparent[1]{}%
  }%
  \providecommand\rotatebox[2]{#2}%
  \ifx\svgwidth\undefined%
    \setlength{\unitlength}{55.6679686bp}%
    \ifx\svgscale\undefined%
      \relax%
    \else%
      \setlength{\unitlength}{\unitlength * \real{\svgscale}}%
    \fi%
  \else%
    \setlength{\unitlength}{\svgwidth}%
  \fi%
  \global\let\svgwidth\undefined%
  \global\let\svgscale\undefined%
  \makeatother%
  \begin{picture}(1,1.1855457)%
    \put(0,0){\includegraphics[width=\unitlength,page=1]{triangle_degree_1.pdf}}%
    \put(-0.00596449,0.37720928){\color[rgb]{0,0,0}\makebox(0,0)[lb]{\smash{$\sfv$}}}%
  \end{picture}%
\endgroup%
}}
\end{tikzcd}\right)\in\cM^\infinitesimaltriangle
\]

We denote the sets of infinitesimal admissible monogons and triangles of degree 1 by $\cM^\infinitesimalmonogon$ and $\cM^\infinitesimaltriangle$, respectively.
\[
\cM^\sinf_{(1)}=\cM^\infinitesimalmonogon\amalg\cM^\infinitesimaltriangle
\]
\end{enumerate}

In summary, we have the following corollary.
\begin{corollary}\label{corollary:classification_of_infinitesimal_degree_1}
For each $\sfv_{i,\ell}$, $\cM^\infinitesimalmonogon(\sfv_{i,\ell})$ contains at most one element, which happens if and only if $\ell=\val(\sfv)$.
Moreover, the elements in $\cM^{\infinitesimaltriangle}(\sfv_{i,\ell})$ have one-to-one correspondence with the indices $1\le \ell_1<\ell$ as above.
\end{corollary}

\begin{enumerate}
\item[(2)] Degree 2 ($|f|_\ZZ=2$) : There are two possibilities as before.
\begin{align}
\begin{cases}\label{eqn:degree2infinitesimal-1}
t=1;\\
\chi(\bO_f)=0,
\end{cases}&\Longleftrightarrow
\begin{cases}
\Pi \text{ is a bigon};\\
\bO_f \text{ contains a loop}.
\end{cases}\\
\begin{cases}\label{eqn:degree2infinitesimal-2}
t=3;\\
\chi(\bO_f)=1,
\end{cases}&\Longleftrightarrow
\begin{cases}
\Pi \text{ is a quadrilateral};\\
\bO_f \text{ is a tree}.
\end{cases}
\end{align}

The first solution gives us a one-parameter family $\{f_u\mid u\in[0,1]\}$ of admissible disks depicted in Figure~\ref{figure:admissible infinitesimal bigons},
where the internal vertex of $\bO_f$ moves along $\sO_\sfv$ via $f_u$.
By Definition~\ref{definition:equivalence of admissible disks} again, all disks $f_u$ for $u\in(0,1)$ are equivalent.

\begin{figure}[ht]
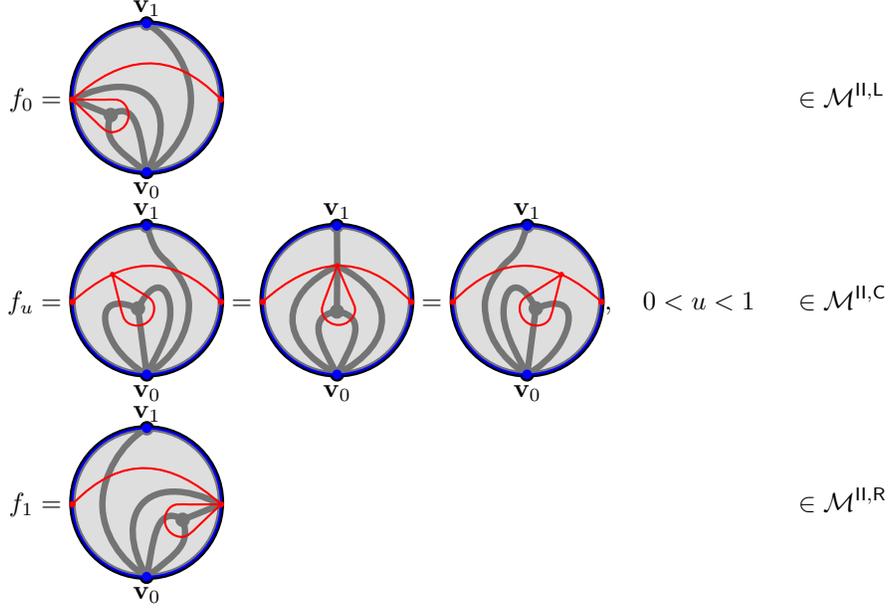

\begin{align*}\label{eq:family of disks with interior vertex}
f_0&=\vcenter{\hbox{\input{bigon_degree_2_L_0_input.tex}}}&&\in\cM^{\infinitesimalbigonleft}\\
f_u&=\vcenter{\hbox{\input{bigon_degree_2_L_1_input.tex}}}=\vcenter{\hbox{\input{bigon_degree_2_input.tex}}}=\vcenter{\hbox{\input{bigon_degree_2_R_1_input.tex}}},\quad 0<u<1&&\in\cM^{\infinitesimalbigonmiddle}\\
f_1&=\vcenter{\hbox{\input{bigon_degree_2_R_0_input.tex}}}&&\in\cM^{\infinitesimalbigonright}
\end{align*}
\caption{A family of infinitesemial admissible bigons}
\label{figure:admissible infinitesimal bigons}
\end{figure}

We denote the sets of infinitesimal admissible bigons whose critical graphs $\bO_f$ look like the three equivalence classes $f_0, f_u$ and $f_1$ above for some $u\in(0,1)$ by $\cM^{\infinitesimalbigonleft}, \cM^{\infinitesimalbigonmiddle}$ and $\cM^{\infinitesimalbigonright}$, respectively.

Moreover, the difference between the numbers of sectors $\ell_f(\bv_0)$ and $\ell_f(\bv_1)$ that $f$ covers at $\bv_0$ and $\bv_1$, respectively, is exactly the same as $\val(\sfv)$
\[
\ell_f(\bv_0)-\ell_f(\bv_1)=\val(\sfv)
\]
since the inner circular region corresponding to a sink consumes exactly one full turn ($=\val(\sfv)$ sectors). This is equivalent to
\begin{align*}
\ell_0-\ell_1&=\val(\sfv),&
\tilde f(\bv_0)&=\sfv_{i,\ell_0},&
\tilde f(\bv_1)&=\sfv_{i,\ell_1}.
\end{align*}

On the other hand, the second solution corresponds to the following admissible disk:
\[
\vcenter{\hbox{\def\svgscale{0.9}
\begingroup%
  \makeatletter%
  \providecommand\color[2][]{%
    \errmessage{(Inkscape) Color is used for the text in Inkscape, but the package 'color.sty' is not loaded}%
    \renewcommand\color[2][]{}%
  }%
  \providecommand\transparent[1]{%
    \errmessage{(Inkscape) Transparency is used (non-zero) for the text in Inkscape, but the package 'transparent.sty' is not loaded}%
    \renewcommand\transparent[1]{}%
  }%
  \providecommand\rotatebox[2]{#2}%
  \ifx\svgwidth\undefined%
    \setlength{\unitlength}{75.26018286bp}%
    \ifx\svgscale\undefined%
      \relax%
    \else%
      \setlength{\unitlength}{\unitlength * \real{\svgscale}}%
    \fi%
  \else%
    \setlength{\unitlength}{\svgwidth}%
  \fi%
  \global\let\svgwidth\undefined%
  \global\let\svgscale\undefined%
  \makeatother%
  \begin{picture}(1,0.6332008)%
    \put(0,0){\includegraphics[width=\unitlength,page=1]{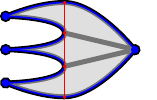}}%
    \put(0.83401333,0.42291279){\color[rgb]{0,0,0}\makebox(0,0)[lb]{\smash{$\sfv$}}}%
  \end{picture}%
\endgroup%
}}\!\!\!=\left(
\begin{tikzcd}[column sep=1pc]
\vcenter{\hbox{\def\svgscale{0.9}
\begingroup%
  \makeatletter%
  \providecommand\color[2][]{%
    \errmessage{(Inkscape) Color is used for the text in Inkscape, but the package 'color.sty' is not loaded}%
    \renewcommand\color[2][]{}%
  }%
  \providecommand\transparent[1]{%
    \errmessage{(Inkscape) Transparency is used (non-zero) for the text in Inkscape, but the package 'transparent.sty' is not loaded}%
    \renewcommand\transparent[1]{}%
  }%
  \providecommand\rotatebox[2]{#2}%
  \ifx\svgwidth\undefined%
    \setlength{\unitlength}{82.35441873bp}%
    \ifx\svgscale\undefined%
      \relax%
    \else%
      \setlength{\unitlength}{\unitlength * \real{\svgscale}}%
    \fi%
  \else%
    \setlength{\unitlength}{\svgwidth}%
  \fi%
  \global\let\svgwidth\undefined%
  \global\let\svgscale\undefined%
  \makeatother%
  \begin{picture}(1,0.86296456)%
    \put(0,0){\includegraphics[width=\unitlength,page=1]{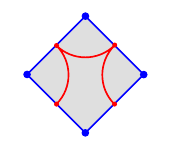}}%
    \put(0.42864895,0.00856151){\color[rgb]{0,0,0}\makebox(0,0)[lb]{\smash{$\bv_0$}}}%
    \put(0.42864895,0.825588){\color[rgb]{0,0,0}\makebox(0,0)[lb]{\smash{$\bv_2$}}}%
    \put(-0.001838,0.40008336){\color[rgb]{0,0,0}\makebox(0,0)[lb]{\smash{$\bv_3$}}}%
    \put(0.89930147,0.40008336){\color[rgb]{0,0,0}\makebox(0,0)[lb]{\smash{$\bv_1$}}}%
  \end{picture}%
\endgroup%
}}\ar[r,"\simeq"]&
\vcenter{\hbox{\def\svgscale{0.9}
\begingroup%
  \makeatletter%
  \providecommand\color[2][]{%
    \errmessage{(Inkscape) Color is used for the text in Inkscape, but the package 'color.sty' is not loaded}%
    \renewcommand\color[2][]{}%
  }%
  \providecommand\transparent[1]{%
    \errmessage{(Inkscape) Transparency is used (non-zero) for the text in Inkscape, but the package 'transparent.sty' is not loaded}%
    \renewcommand\transparent[1]{}%
  }%
  \providecommand\rotatebox[2]{#2}%
  \ifx\svgwidth\undefined%
    \setlength{\unitlength}{80.15336973bp}%
    \ifx\svgscale\undefined%
      \relax%
    \else%
      \setlength{\unitlength}{\unitlength * \real{\svgscale}}%
    \fi%
  \else%
    \setlength{\unitlength}{\svgwidth}%
  \fi%
  \global\let\svgwidth\undefined%
  \global\let\svgscale\undefined%
  \makeatother%
  \begin{picture}(1,0.9324274)%
    \put(0,0){\includegraphics[width=\unitlength,page=1]{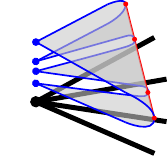}}%
    \put(0.16323374,0.16928065){\color[rgb]{0,0,0}\makebox(0,0)[lb]{\smash{$\sfv$}}}%
    \put(0.11398227,0.73136408){\color[rgb]{0,0,0}\makebox(0,0)[lb]{\smash{$\bv_0$}}}%
    \put(0.0526276,0.57256877){\color[rgb]{0,0,0}\makebox(0,0)[lb]{\smash{$\bv_1$}}}%
    \put(0.05010713,0.49317204){\color[rgb]{0,0,0}\makebox(0,0)[lb]{\smash{$\bv_2$}}}%
    \put(0.05010709,0.4196058){\color[rgb]{0,0,0}\makebox(0,0)[lb]{\smash{$\bv_3$}}}%
  \end{picture}%
\endgroup%
}}\ar[r,"f"]&
\vcenter{\hbox{\def\svgscale{0.9}
\begingroup%
  \makeatletter%
  \providecommand\color[2][]{%
    \errmessage{(Inkscape) Color is used for the text in Inkscape, but the package 'color.sty' is not loaded}%
    \renewcommand\color[2][]{}%
  }%
  \providecommand\transparent[1]{%
    \errmessage{(Inkscape) Transparency is used (non-zero) for the text in Inkscape, but the package 'transparent.sty' is not loaded}%
    \renewcommand\transparent[1]{}%
  }%
  \providecommand\rotatebox[2]{#2}%
  \ifx\svgwidth\undefined%
    \setlength{\unitlength}{54.55568249bp}%
    \ifx\svgscale\undefined%
      \relax%
    \else%
      \setlength{\unitlength}{\unitlength * \real{\svgscale}}%
    \fi%
  \else%
    \setlength{\unitlength}{\svgwidth}%
  \fi%
  \global\let\svgwidth\undefined%
  \global\let\svgscale\undefined%
  \makeatother%
  \begin{picture}(1,1.20971671)%
    \put(0,0){\includegraphics[width=\unitlength,page=1]{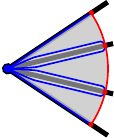}}%
    \put(-0.00554909,0.38489985){\color[rgb]{0,0,0}\makebox(0,0)[lb]{\smash{$\sfv$}}}%
  \end{picture}%
\endgroup%
}}
\end{tikzcd}\right)\in\cM^{\infinitesimalsquare}
\]
In this case, if $\tilde f(\bv_j)=\sfv_{i_j, \ell_j}$, then $\ell_0=\ell_1+\ell_2+\ell_3$ and $i_j=i_0+\ell_1+\cdots+\ell_{j-1}$.

We denote the set of infinitesimal admissible quadrilaterals of degree 2 by $\cM^{\infinitesimalsquare}$. Then
\[
\cM^\sinf_{(2)}=\cM^{\infinitesimalbigonleft}\amalg\cM^{\infinitesimalbigonmiddle}\amalg\cM^{\infinitesimalbigonright}\amalg\cM^{\infinitesimalsquare}.
\]
As before, for each $\sfv_{i,\ell}$, each of sets $\cM^{\infinitesimalbigonleft}(\sfv_{i,\ell})$, $\cM^{\infinitesimalbigonmiddle}(\sfv_{i,\ell})$ and $\cM^{\infinitesimalbigonright}(\sfv_{i,\ell})$ has at most one element, which happens if and only if $\ell>\val(\sfv)$.
\end{enumerate}

\subsection{DGA and canonical peripheral structures}

Recall the signed label $\tilde f$ defined in Definition~\ref{definition:signs_of_vertices}.

\begin{definition}\label{definition:total label}
We define a function $\cP:\cM\to A_\Lambda$ by for $f\in\cM_t$,
\[
\cP(f)\coloneqq \sgn(f)\tilde f(\bv_1\cdots \bv_t),
\]
where $\sgn(f)$ is defined as 
\[
\sgn(f)\coloneqq
\sgn(f,\bv_0\cdots\bv_t)(-1)^{|\tilde f(\bv_1\cdots\bv_{r_1})|-1}\cdots(-1)^{|\tilde f(\bv_1\cdots\bv_{r_k})|-1}.
\]
Here, the indices $r_i$ are the positions of the 2-folding edges or the concave vertices. 
\end{definition}

More precisely, since we are assuming that $\sgn(f,\bv)=1$ for any $f$ with $\tilde f(\bv)\in \tilde\sV_\sL$ by Assumption~\ref{assumption:orientation sign of vertex}, we have
\[
\sgn(f)=
\begin{cases}
\sgn(f,\bv_0\cdots\bv_t) & f\in\cM^{\reg}_{(1)};\\
\sgn(f,\bv_0\cdots\bv_t)(-1)^{|\tilde f(\bv_1\cdots\bv_{r})|-1} & f\in \cM^{\reg}_{(2)}; \\ 
+1 & f\in \cM^{\infinitesimalmonogon}\cup \cM^{\infinitesimalbigonleft}\cup \cM^{\infinitesimalbigonmiddle}\cup \cM^{\infinitesimalbigonright};\\
(-1)^{|\tilde f(\bv_1)|-1} & f\in\cM^{\infinitesimaltriangle};\\
(-1)^{|\tilde f(\bv_2)|} & f\in\cM^{\infinitesimalsquare}.
\end{cases}
\]

Now we equip $\cA_\cL$ with a differential $\partial$ by using admissible disks and the function $\cP$ as follows.

\begin{definition}\label{definition:differential}
Let $\cL=(\Lambda,\fP)\in\cLG_\fR$ and let $\cA_\cL=(A_\Lambda,|\cdot|)$ be the graded algebra defined in Definition~\ref{definition:algebra}. 
A {\em differential $\partial: A_\Lambda\to A_\Lambda$} is defined as follows: For any $\sfg\in \sG_\sL$,
\begin{align}\label{equation:differential another form}
\partial \sfg\coloneqq
\sum_{f\in\cM(\sfg)_{(1)}} \cP(f)
\end{align}
and $\partial$ is linearly extended from $\sG_\sL$ to $A_\Lambda$ via the graded Leibniz rule with $\partial 1=0$. That is,
\begin{align*}
\partial(\sfg\sfg') &\coloneqq (\partial \sfg)\sfg' + (-1)^{|\sfg|}\sfg(\partial \sfg'),& \partial 1 &\coloneqq0.
\end{align*}
\end{definition}

\begin{proposition}\label{prop:boundary of vertex}
For any $\sfv_{i,\ell}\in \sG_\sL$, the following holds:
\begin{equation}\label{eqn:differential_vertex}
\partial \sfv_{i,\ell}=\delta_{\ell,\val(\sfv)}+\sum_{\ell_1+\ell_2=\ell}(-1)^{|\sfv_{i,\ell_1}|-1}\sfv_{i,\ell_1} \sfv_{i+\ell_1,\ell_2}.
\end{equation}
Here $\delta_{k,\ell}$ is the Kronecker delta.
\end{proposition}
\begin{proof}
This is a direct consequence of Corollary~\ref{corollary:classification_of_infinitesimal_degree_1} and Definitions~\ref{definition:total label} and \ref{definition:differential}.
\end{proof}

\begin{remark}
In \cite{EN2015}, the authors considered Legendrian links in $\#_k (S^1\times S^2)$ with the tight contact structure and constructed a DGA invariant. The presence of a closed Reeb orbit in each boundary component of a Weinstein handle forces them to consider infinitely many Reeb chords as seen in \S~\ref{sec:geometric model}. Their definition of the {\em internal DGA} in \cite[\S 2.3]{EN2015} on induced generators exactly coincides with (\ref{eqn:differential_vertex}).
\end{remark}

\begin{proposition}\label{proposition:finiteness}
The map $\partial$ is well-defined for $\sfg\in\sC_\sL$, i.e., $\partial \sfg$ has only finitely many terms.
\end{proposition}
\begin{proof}
Let us start by defining the {\em action} $z(\sfg)\in\RR_{\ge0}$ for each $\sfg\in\sG_\sL$ as follows:
\begin{enumerate}
\item If $\sfg=\sfc\in\sC_\sL$, then
\[
z(\sfc)\coloneqq |z(c_+)-z(c_-)|>0,
\]
where $c_+$ and $c_-$ are the two lifts of $\sfc$ to $\Lambda$.
\item If $\sfg=\sfv_{i,\ell}\in\tilde\sV_\sL$, then
\[
z(\sfv_{i,\ell})\coloneqq0.
\]
\end{enumerate}
Since $\sfg\in\sC_\sL$, an admissible disk $f\in \cM_t(\sfg)_{(1)}$ is regular.
Then by Stokes' theorem we have
\begin{align}\label{eqn:action and area}
z(\sfg)= \sum_{r=1}^t z\left(\tilde f(\bv_r)\right)+\int_{\Pi_t}f^*(dx\wedge dy).
\end{align}
Let $\{\sD_j\}_{i\in J}$ be a set of bounded connected components of $\RR^2\setminus \sL$.
When $\Pi_t \cap f^{-1}(\sD_j)$ is nonempty, $\int_{\Pi_t \cap f^{-1}(\sD_j)}f^*(dx\wedge dy)$ is positive and in particular greater or equal than $\area(\sD_j)$. Note that
\[
\int_{\Pi_t}f^*(dx\wedge dy)=\sum_{j\in J}\int_{\Pi_t \cap f^{-1}(\sD_j)}f^*(dx\wedge dy).
\]

Suppose that $\tilde f(\bv_r)=\sfc$ for some $r\neq0$ and generator $\sfc\in \sC_\sL$. There exist two regions $\sD_j$ and $\sD_j'$ containing a negative corner of $f(\bv_r)$.
One of them, say $\sD_j$, satisfies $\Pi_t \cap f^{-1}(\sD_j)\neq \emptyset$.
Recall that the action of a generator $\sC_\sL$ is positive. So $\tilde f(\bv_r)$ contributes at least two positive terms $z(\tilde f(\bv_r))+\area(\sD_j)$ in the right hand side of (\ref{eqn:action and area}).

On the other hand, assume that $\tilde f(\bv_r)=\sfw_{i,\ell}$ for some generator in $\sV_\sL$.
Then there exist $\sD_{j_1},\dots,\sD_{j_\ell}$ such that $\sD_{j_k}$ has a corner which corresponds to $\sfw_{i+k-1,1}$, for $k=1,\dots,\ell$.
Definition~\ref{def:Canonical label} implies that $\Pi_t \cap f^{-1}(\sD_{j_k})\neq \emptyset$, for $k=1,\dots,\ell$. Thus $\tilde f(\bv_r)$ contributes at least $\ell$ positive terms $\sum_{k=1}^{\ell}\area(\sD_{j_k})$ in the right hand side of (\ref{eqn:action and area}).

Note that each term of $z(\tilde f(\bv_r))+\min\{\area(\sD_j),\area(\sD_j') \}$ and $\sum_{k=1}^{\ell}\area(\sD_{j_k})$ are determined by the configuration of $\sL$.
Thus $\cM(\sfg)$ is finite for a fixed $\sfg\in \sC_\sL$ because of the constraint (\ref{eqn:action and area}). This proves the proposition.
\end{proof}

\begin{theorem}\label{thm:differential}
The map $\partial$ is a differential. That is, $|\partial|=-1_\fR$ and $\partial^2=0$.
\end{theorem}
It is obvious to see that $|\partial|=-1_\fR$ by Proposition~\ref{proposition:degree formula}, and we will prove that $\partial^2=0$ later in Appendix~\ref{section:manipulation of disks}.

\begin{definition}[Differential graded algebra]\label{definition:DGA}
For $\cL=(\Lambda,\fP)\in\cLG_\fR$, the {\em differential graded algebra (DGA) $\cA_\cL$} for $\cL$ is the triple
\[
\cA_\cL\coloneqq(A_\Lambda,|\cdot|,\partial),
\]
where
\begin{enumerate}
\item $A_\Lambda$ is the free associative unital algebra over $\ZZ$ generated by $\sG_\sL$
\[
A_\Lambda\coloneqq\ZZ\langle \sG_\sL\rangle,
\]
\item $|\cdot|$ is a degree function on $A_\Lambda$ defined by the potential $\fP$ as in \S~\ref{section:grading_singularpoint}
\[
|\cdot|\coloneqq|\cdot|_\fP:A_\Lambda\to\fR,
\]
\item $\partial$ is the differential on $A_\Lambda$ defined in Definition~\ref{definition:differential}
\[
\partial:A_\Lambda\to A_\Lambda.
\]
\end{enumerate}
\end{definition}

\begin{definition}[Canonical peripheral structures]\label{definition:Canonical peripheral structures}
For each vertex $v\in \cV_\cL$, we consider a DGA 
\[
\cI_v\coloneqq(I_\sfv\coloneqq I_{\val(\sfv)},|\cdot|_\sfv,\partial_\sfv\coloneqq \partial_{\val(\sfv)})
\]
which is isomorphic to the DG-subalgebra of $\cA_\cL$ generated by $\sfv_{i,\ell}$.
Then a {\em canonical peripheral structure} of $\cA_\cL$ at $\sfv$ is a canonical injective DGA morphism
\[
\bp_v:\cI_v\to\cA.
\]

The collection $\cP_\cL$ of {\em canonical peripheral structures} is denoted by
\[
\cP_\cL\coloneqq \{\bp_\emptyset\}\cup \{\bp_v\mid v\in \cV_\Lambda\}.
\]
\end{definition}

\begin{proof}[Proof of Theorem~{\rm\ref{theorem:DGA}}]
Definition~\ref{definition:DGA} and Definition~\ref{definition:Canonical peripheral structures} provide a DGA and a collection of canonical peripheral structures for $\cL=(\Lambda,\fP)$.
\end{proof}

Before closing this section, we consider the Legendrian mirror $\mu$ and the induced map on DGAs.
\begin{lemma}\label{lemma:Legendrian mirror anti-isomorphism}
Let $\mu(\cL)=(\mu(\Lambda),\mu(\fP))$ be the Legendrian mirror of a Legendrian graph with potential $\cL$ such that $\mu(\fP)$ is given by
\[
\mu(\fP)(\mu(\sfh))\coloneqq \fP(\sfh)
\]
for any half-edge $\sfh$ of $\sL=\pi_L(\Lambda)$.
Then there exists an anti-isomorphism 
\[
\mu_*:\cA_{\cL}\to\cA_{\mu(\cL)}
\]
between DGAs.
\end{lemma}
\begin{proof}
Let $\cA_{\mu(\cL)}\coloneqq(A_{\mu(\Lambda)},|\cdot|_\mu,\partial_\mu)$. Then $A_{\mu(\Lambda)}$ is the free associative unital algebra generated by $\sG_{\mu(\sL)}$ which is identified with $\mu(\sG_\sL)$ in the canonical way.
Moreover, by the definition of $\mu(\fP)$, it is obvious that for any $\sfg\in\sG_\sL$,
\[
|\mu(\sfg)|_\mu=|\sfg|.
\]
Therefore the map $\mu_*$ defined by
\begin{align*}
\mu_*(\sfg)&=\mu(\sfg),&
\mu_*(\sfg\sfg')&=\mu(\sfg')\mu(\sfg)
\end{align*}
becomes an anti-isomorphism between graded algebras.

Finally, the equation (\ref{equation:Legendrian mirror on labels}) in Remark~\ref{remark:Legendrian mirror on disk} implies that for any admissible disk $f$,
\[
\cP(\mu(f))=\mu_*(\cP(f))
\]
and therefore by the equation (\ref{equation:differential another form}), we have
\[
\partial_\mu\circ\mu_* = \mu_*\circ\partial.\qedhere
\]
\end{proof}

\subsection{Examples}
In this section, we compute several examples of DGAs.

\subsubsection{Legendrian Theta graphs}
Let $\Theta$ be a Legendrian graph which consists of two vertices $\cV_\Theta=\{v,w\}$ and three edges $\cE_\Theta=\{e_1, e_2, e_3\}$ joining $v$ and $w$, and has the Lagrangian projection $\sTheta=\pi_L(\Theta)$ as follows:
\begin{align*}
\sTheta&=\vcenter{\hbox{
\begingroup%
  \makeatletter%
  \providecommand\color[2][]{%
    \errmessage{(Inkscape) Color is used for the text in Inkscape, but the package 'color.sty' is not loaded}%
    \renewcommand\color[2][]{}%
  }%
  \providecommand\transparent[1]{%
    \errmessage{(Inkscape) Transparency is used (non-zero) for the text in Inkscape, but the package 'transparent.sty' is not loaded}%
    \renewcommand\transparent[1]{}%
  }%
  \providecommand\rotatebox[2]{#2}%
  \ifx\svgwidth\undefined%
    \setlength{\unitlength}{133.87495375bp}%
    \ifx\svgscale\undefined%
      \relax%
    \else%
      \setlength{\unitlength}{\unitlength * \real{\svgscale}}%
    \fi%
  \else%
    \setlength{\unitlength}{\svgwidth}%
  \fi%
  \global\let\svgwidth\undefined%
  \global\let\svgscale\undefined%
  \makeatother%
  \begin{picture}(1,0.36551885)%
    \put(0,0){\includegraphics[width=\unitlength,page=1]{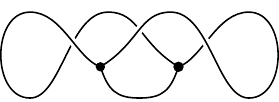}}%
    \put(0.31175915,0.07152867){\color[rgb]{0,0,0}\makebox(0,0)[lb]{\smash{$\sfv$}}}%
    \put(0.6625651,0.07152867){\color[rgb]{0,0,0}\makebox(0,0)[lb]{\smash{$\sfw$}}}%
    \put(0.48018445,0.31148281){\color[rgb]{0,0,0}\makebox(0,0)[lb]{\smash{$\sfa$}}}%
    \put(0.21688665,0.27811115){\color[rgb]{0,0,0}\makebox(0,0)[lb]{\smash{$\sfb_1$}}}%
    \put(0.7070829,0.27811115){\color[rgb]{0,0,0}\makebox(0,0)[lb]{\smash{$\sfb_2$}}}%
    \put(0.47148857,0.03501402){\color[rgb]{0,0,0}\makebox(0,0)[lb]{\smash{$\sfe_1$}}}%
    \put(0.07341501,0.34252631){\color[rgb]{0,0,0}\makebox(0,0)[lb]{\smash{$\sfe_2$}}}%
    \put(0.84565314,0.34252631){\color[rgb]{0,0,0}\makebox(0,0)[lb]{\smash{$\sfe_3$}}}%
  \end{picture}%
\endgroup%
}}&
\sH_\sTheta&=\vcenter{\hbox{
\begingroup%
  \makeatletter%
  \providecommand\color[2][]{%
    \errmessage{(Inkscape) Color is used for the text in Inkscape, but the package 'color.sty' is not loaded}%
    \renewcommand\color[2][]{}%
  }%
  \providecommand\transparent[1]{%
    \errmessage{(Inkscape) Transparency is used (non-zero) for the text in Inkscape, but the package 'transparent.sty' is not loaded}%
    \renewcommand\transparent[1]{}%
  }%
  \providecommand\rotatebox[2]{#2}%
  \ifx\svgwidth\undefined%
    \setlength{\unitlength}{122.92041927bp}%
    \ifx\svgscale\undefined%
      \relax%
    \else%
      \setlength{\unitlength}{\unitlength * \real{\svgscale}}%
    \fi%
  \else%
    \setlength{\unitlength}{\svgwidth}%
  \fi%
  \global\let\svgwidth\undefined%
  \global\let\svgscale\undefined%
  \makeatother%
  \begin{picture}(1,0.26548239)%
    \put(0,0){\includegraphics[width=\unitlength,page=1]{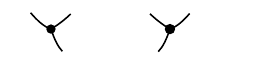}}%
    \put(0.14658603,0.09260536){\color[rgb]{0,0,0}\makebox(0,0)[lb]{\smash{$\sfv$}}}%
    \put(0.68813479,0.09260536){\color[rgb]{0,0,0}\makebox(0,0)[lb]{\smash{$\sfw$}}}%
    \put(0,0){\includegraphics[width=\unitlength,page=2]{Theta_halfedges.pdf}}%
    \put(-0.00246285,0.23829009){\color[rgb]{0,0,0}\makebox(0,0)[lb]{\smash{$\sfh_{\sfv,2}$}}}%
    \put(0.27374669,0.23504716){\color[rgb]{0,0,0}\makebox(0,0)[lb]{\smash{$\sfh_{\sfv,3}$}}}%
    \put(0.22983211,0.00622933){\color[rgb]{0,0,0}\makebox(0,0)[lb]{\smash{$\sfh_{\sfv,1}$}}}%
    \put(0.52028486,0.00622862){\color[rgb]{0,0,0}\makebox(0,0)[lb]{\smash{$\sfh_{\sfw,1}$}}}%
    \put(0.73369456,0.23504716){\color[rgb]{0,0,0}\makebox(0,0)[lb]{\smash{$\sfh_{\sfw,3}$}}}%
    \put(0.45865032,0.24044078){\color[rgb]{0,0,0}\makebox(0,0)[lb]{\smash{$\sfh_{\sfw,2}$}}}%
  \end{picture}%
\endgroup%
}}
\end{align*}

We define a potential $\fP:\sH_\sTheta\to\ZZ$ as follows:
\begin{align*}
\fP(\sfh_{\sfv,1})&=\fx,&\fP(\sfh_{\sfv,2})&=\fy,&\fP(\sfh_{\sfv,3})&=\fz,\\
\fP(\sfh_{\sfw,1})&=\fx,&\fP(\sfh_{\sfw,2})&=\fy-1,&\fP(\sfh_{\sfw,3})&=\fz-1.
\end{align*}

The DGA $\cA_{(\Theta, \fP)}$ is generated by the set
\[
\sG_{\sTheta}=\{\sfv_{i,\ell}, \sfw_{i,\ell}\mid i\in\Zmod{3}, \ell\ge1\}\amalg\{\sfa, \sfb_1, \sfb_2\},
\]
whose gradings are given by
\begin{align*}
|\sfv_{i,1}|&=|\sfw_{i,1}|=\begin{cases}
\fx-\fy& i=1;\\
\fy-\fz& i=2;\\
\fz-\fx-1& i=3,
\end{cases}&
|\sfa|&=\fz-\fy,&
|\sfb_1|&=|\sfb_2|=1,
\end{align*}
and the differentials for generators $\sfa, \sfb_1$ and $\sfb_2$ are as follows:
\begin{align*}
\partial \sfa &= \sfv_{3,1}\sfw_{1,1},\\
\partial \sfb_1 &= 1+(-1)^{|\sfb_1|-1}(\sfv_{2,1}\sfa+\sfv_{2,2}\sfw_{1,1})\\
&=1+\sfv_{2,1}\sfa+\sfv_{2,2}\sfw_{1,1},\\
\partial \sfb_2 &= 1+(-1)^{|\sfb_2|-1}((-1)^{|\sfa|-1}\sfa\sfw_{2,1}+\sfv_{3,1}\sfw_{1,2})\\
&=1+(-1)^{|\sfa|-1}\sfa\sfw_{2,1}+\sfv_{3,1}\sfw_{1,2}.
\end{align*}

Then it is not hard to check that $\partial^2=0$. For example,
\begin{align*}
\partial^2 \sfb_1 &= 0+(-1)^{|\sfv_{2,1}|}\sfv_{2,1}\partial(\sfa) + \partial(\sfv_{2,2})\sfw_{1,1}\\
&=(-1)^{|\sfv_{2,1}|}\sfv_{2,1}\sfv_{3,1}\sfw_{1,1}+(-1)^{|\sfv_{2,1}|-1}\sfv_{2,1}\sfv_{3,1}\sfw_{1,1}=0,\\
\partial^2 \sfb_2 &= 0+ (-1)^{|\sfa|-1}\partial(\sfa)\sfw_{2,1}+(-1)^{|\sfv_{3,1}|}\sfv_{3,1}\partial(\sfw_{1,2})\\
&=(-1)^{\fz-\fy-1}\sfv_{3,1}\sfw_{1,1}\sfw_{2,1}+(-1)^{\fz-\fx-1}\sfv_{3,1}(-1)^{|\sfw_{1,1}|-1}\sfw_{1,1}\sfw_{2,1}=0.
\end{align*}

\subsubsection{4-valent graphs and Legendrian singular links}
Let $\Lambda$ be a 4-valent Legendrian graph.
Then it can be naturally regarded as a Legendrian singular link by making two pairs of opposite half-edges at each vertex smooth.
Furthermore, one can equip orientations on edges which are consistent at every vertex, and so can choose the half-edge at each vertex which corresponds to $+x$-axis in the standard model as observed in \cite[Lemma~3.2]{ABK2018}.
This is well-defined and so we label half-edges with indices $\{1,2,3,4\}$ clockwise starting from the chosen half-edge.
See Figure~\ref{figure:local model 4-valent vertex}.

\begin{figure}[ht]
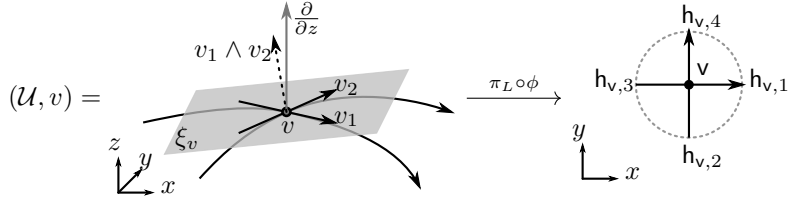

\[
\begin{tikzcd}[column sep=3pc]
(\cU, v)=\vcenter{\hbox{\input{localmodel_input.tex}}}\ar[r,"\pi_L\circ\phi"]&
\vcenter{\hbox{\input{vertex_lag_singular_oriented_input.tex}}}
\end{tikzcd}
\]
\caption{A local picture of a 4-valent vertex $v$ and the projection of the standard model}
\label{figure:local model 4-valent vertex}
\end{figure}

One of the direct consequence is as follows: let $\Lambda$ and $\Lambda'$ be two equivalent oriented Legendrian singular links via the isotopy $\phi_t$ and let $\bp_v^\circlearrowright:\cI_v\to\cA_\Lambda$ and $\bp_{v'}^\circlearrowright:\cI_{v'}\to\cA_{\Lambda'}$ be two peripheral structures for $v\in \cV_\Lambda$ and $v'\in\cV_{\Lambda'}$ with $\phi_1(v)=v'$.
Then we have the following commutative diagram of DGAs.
\[
\begin{tikzcd}
\cA_\Lambda \ar[r,"\phi_*","\stabletameisom"'] & \cA_{\Lambda'}\\
\cI_4\ar[u,"\bp_v^\circlearrowright"]\ar[r,"\phi_*","\simeq"']&\cI_4'\ar[u,"\bp_{v'}^\circlearrowright"']
\end{tikzcd}
\]
Here we abuse the notation of the above map, which is in fact a sequence of (zig-zags of) generalized stable tame isomorphisms.
However, the bottom map is always a DGA isomorphism which maps $\sfv_{i,\ell}$ to $\sfv'_{i,\ell}$. That is, $\phi$ preserves both indices $i$ and $\ell$.

Let us see a concrete example.
The {\em oriented Legendrian pinched figure-eight knot} and its reverse, denoted by $4_1^\times$ and $-4_1^\times$, are Legendrian graphs, each of which consists of one 4-valent vertex with two edges as depicted below.
\begin{align*}
4_1^\times&=\vcenter{\hbox{
\begingroup%
  \makeatletter%
  \providecommand\color[2][]{%
    \errmessage{(Inkscape) Color is used for the text in Inkscape, but the package 'color.sty' is not loaded}%
    \renewcommand\color[2][]{}%
  }%
  \providecommand\transparent[1]{%
    \errmessage{(Inkscape) Transparency is used (non-zero) for the text in Inkscape, but the package 'transparent.sty' is not loaded}%
    \renewcommand\transparent[1]{}%
  }%
  \providecommand\rotatebox[2]{#2}%
  \ifx\svgwidth\undefined%
    \setlength{\unitlength}{160.62711321bp}%
    \ifx\svgscale\undefined%
      \relax%
    \else%
      \setlength{\unitlength}{\unitlength * \real{\svgscale}}%
    \fi%
  \else%
    \setlength{\unitlength}{\svgwidth}%
  \fi%
  \global\let\svgwidth\undefined%
  \global\let\svgscale\undefined%
  \makeatother%
  \begin{picture}(1,0.3899927)%
    \put(0,0){\includegraphics[width=\unitlength,page=1]{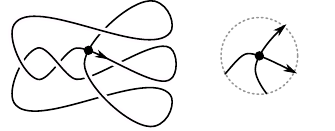}}%
    \put(0.25918759,0.32643636){\color[rgb]{0,0,0}\makebox(0,0)[lb]{\smash{$\sfb_1$}}}%
    \put(0.25918759,0.03848075){\color[rgb]{0,0,0}\makebox(0,0)[lb]{\smash{$\sfb_2$}}}%
    \put(0.33360212,0.22304674){\color[rgb]{0,0,0}\makebox(0,0)[lb]{\smash{$\sfa$}}}%
    \put(-0.0018847,0.1905288){\color[rgb]{0,0,0}\makebox(0,0)[lb]{\smash{$\sfc_1$}}}%
    \put(0.14766792,0.24222361){\color[rgb]{0,0,0}\makebox(0,0)[lb]{\smash{$\sfc_2$}}}%
    \put(0.23313093,0.11101978){\color[rgb]{0,0,0}\makebox(0,0)[lb]{\smash{$\sfd$}}}%
    \put(0.230422,0.26098417){\color[rgb]{0,0,0}\makebox(0,0)[lb]{\smash{$\sfv$}}}%
    \put(0.74122808,0.244499){\color[rgb]{0,0,0}\makebox(0,0)[lb]{\smash{$\sfv$}}}%
    \put(0.84291917,0.32913075){\color[rgb]{0,0,0}\makebox(0,0)[lb]{\smash{$\sfh_4$}}}%
    \put(0.88949562,0.13928692){\color[rgb]{0,0,0}\makebox(0,0)[lb]{\smash{$\sfh_1$}}}%
    \put(0.78337209,0.06155782){\color[rgb]{0,0,0}\makebox(0,0)[lb]{\smash{$\sfh_2$}}}%
    \put(0.60826813,0.13751776){\color[rgb]{0,0,0}\makebox(0,0)[lb]{\smash{$\sfh_3$}}}%
  \end{picture}%
\endgroup%
}}&
-4_1^\times&=\vcenter{\hbox{
\begingroup%
  \makeatletter%
  \providecommand\color[2][]{%
    \errmessage{(Inkscape) Color is used for the text in Inkscape, but the package 'color.sty' is not loaded}%
    \renewcommand\color[2][]{}%
  }%
  \providecommand\transparent[1]{%
    \errmessage{(Inkscape) Transparency is used (non-zero) for the text in Inkscape, but the package 'transparent.sty' is not loaded}%
    \renewcommand\transparent[1]{}%
  }%
  \providecommand\rotatebox[2]{#2}%
  \ifx\svgwidth\undefined%
    \setlength{\unitlength}{161.22314314bp}%
    \ifx\svgscale\undefined%
      \relax%
    \else%
      \setlength{\unitlength}{\unitlength * \real{\svgscale}}%
    \fi%
  \else%
    \setlength{\unitlength}{\svgwidth}%
  \fi%
  \global\let\svgwidth\undefined%
  \global\let\svgscale\undefined%
  \makeatother%
  \begin{picture}(1,0.38855093)%
    \put(0,0){\includegraphics[width=\unitlength,page=1]{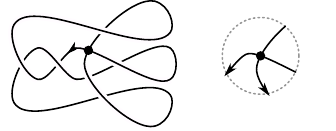}}%
    \put(0.2582294,0.32522955){\color[rgb]{0,0,0}\makebox(0,0)[lb]{\smash{$\sfb_1$}}}%
    \put(0.2582294,0.03833849){\color[rgb]{0,0,0}\makebox(0,0)[lb]{\smash{$\sfb_2$}}}%
    \put(0.33236881,0.22222216){\color[rgb]{0,0,0}\makebox(0,0)[lb]{\smash{$\sfa$}}}%
    \put(-0.00187774,0.18982443){\color[rgb]{0,0,0}\makebox(0,0)[lb]{\smash{$\sfc_1$}}}%
    \put(0.147122,0.24132813){\color[rgb]{0,0,0}\makebox(0,0)[lb]{\smash{$\sfc_2$}}}%
    \put(0.23226906,0.11060935){\color[rgb]{0,0,0}\makebox(0,0)[lb]{\smash{$\sfd$}}}%
    \put(0.22957014,0.26001933){\color[rgb]{0,0,0}\makebox(0,0)[lb]{\smash{$\sfv$}}}%
    \put(0.74218474,0.24359511){\color[rgb]{0,0,0}\makebox(0,0)[lb]{\smash{$\sfv$}}}%
    \put(0.84349988,0.32791398){\color[rgb]{0,0,0}\makebox(0,0)[lb]{\smash{$\sfh_2$}}}%
    \put(0.88990414,0.13877198){\color[rgb]{0,0,0}\makebox(0,0)[lb]{\smash{$\sfh_3$}}}%
    \put(0.78417295,0.06133025){\color[rgb]{0,0,0}\makebox(0,0)[lb]{\smash{$\sfh_4$}}}%
    \put(0.60971634,0.13700937){\color[rgb]{0,0,0}\makebox(0,0)[lb]{\smash{$\sfh_1$}}}%
  \end{picture}%
\endgroup%
}}
\end{align*}
Here, we are following the labeling convention for half-edges as given in Figure~\ref{figure:local model 4-valent vertex}.

Notice that for each edge $\sfe$, the number $n_\sfe$ used to define a potential in \S~\ref{sec:potential} is zero. 
For simplicity, we assume that all half-edges have the same potential.

Then the DGAs $\cA_{4_1^\times}$ and $\cA_{-4_1^\times}$ are generated by
\[
\sG_{4_1^\times}=\sG_{-4_1^\times}=\{\sfv_{i,\ell}\mid i\in\Zmod{4}, \ell\ge 1\}\amalg \{ \sfa, \sfb_1,\sfb_2, \sfc_1,\sfc_2,\sfd\},
\]
whose gradings are as follows: 
\begin{align*}
|\sfv_{1,1}|&=|\sfv_{3,1}|=0,& |\sfb_1|&=1,&|\sfc_1|&=1,& |\sfd|&=0,\\
|\sfv_{2,1}|&=|\sfv_{4,1}|=-1,& |\sfb_2|&=1,&|\sfc_2|&=-1.
\end{align*}

The differentials $\partial$ and $\partial'$ for each generator in $\cA_{4_1^\times}$ and $\cA_{-4_1^\times}$ are as follows:
\begin{align*}
\partial \sfa&=1+(-1)^{|\sfd|-1}\sfv_{1,1}\sfd=1-\sfv_{1,1}\sfd,&\partial' \sfa&=1+(-1)^{|\sfd|-1}\sfv_{3,1}\sfd=1-\sfv_{3,1}\sfd,\\
\partial \sfb_1&= 1+\sfv_{3,1}+(-1)^{|\sfc_1|+|\sfc_2|-2}\sfc_1\sfc_2\sfv_{3,1}&\partial' \sfb_1&= 1+\sfv_{1,1}+(-1)^{|\sfc_1|+|\sfc_2|-2}\sfc_1\sfc_2\sfv_{1,1}\\
&=1+\sfv_{3,1}+\sfc_1\sfc_2\sfv_{3,1},& &=1+\sfv_{1,1}+\sfc_1\sfc_2\sfv_{1,1},\\
\partial \sfb_2&= 1+\sfd+\sfd\sfc_2\sfc_1+\sfv_{2,1}\sfc_1,&\partial' \sfb_2&= 1+\sfd+\sfd\sfc_2\sfc_1+\sfv_{4,1}\sfc_1,\\
\partial \sfc_1&=\partial\sfc_2=\partial\sfd=0,&
\partial' \sfc_1&=\partial'\sfc_2=\partial'\sfd=0.
\end{align*}

Therefore two DGAs $\cA_{4_1^\times}$ and $\cA_{-4_1^\times}$ are obviously isomorphic via $\sfv_{i,\ell}\mapsto\sfv_{i+2,\ell}$. 
However, this isomorphism is not consistent with the choice of the orientation since $\sfh_{\sfv,1}$ in $4_1^\times$ is outward but $\sfh_{\sfv,3}$ in $-4_1^\times$ is inward.
Indeed, there are no orientation-preserving isotopies between $4_1^\times$ and $-4_1^\times$ as seen in \cite{ABK2018}, and therefore $4_1^\times$ is irreversible.

In terms of DGAs and peripheral structures, one can prove that generalized stable-tame isomorphism classes of pairs $(\cA_{4_1^\times},\bp_v^\circlearrowright)$ and $(\cA_{-4_1^\times},\bp_v^\circlearrowright)$ are not the same.
\[
(\cA_{4_1^\times},\bp_v^\circlearrowright)\not\stabletameisom(\cA_{-4_1^\times},\bp_v^\circlearrowright)
\]

In this sense, one can conclude that a DGA together with peripheral structures has more information than the DGA alone has.

\subsubsection{More rotations at vertices}
Compared with Legendrian links, one of the unique characteristics of DGAs for Legendrian graphs is that each vertex contributes infinitely many generators.
However, for all above examples, the differential of each crossing generator does not involve any vertex generator which rotates more than a full turn, and therefore one may wonder about the necessity of generators having more rotations. 
Here we present a very simple example that we must consider vertex generators rotating more than a full turn.

Let $\sL$ be a Lagrangian projection of Legendrian graph with one bivalent vertex $\sfv$ and one (oriented) edge as follows:
\[
\sL=\vcenter{\hbox{
\begingroup%
  \makeatletter%
  \providecommand\color[2][]{%
    \errmessage{(Inkscape) Color is used for the text in Inkscape, but the package 'color.sty' is not loaded}%
    \renewcommand\color[2][]{}%
  }%
  \providecommand\transparent[1]{%
    \errmessage{(Inkscape) Transparency is used (non-zero) for the text in Inkscape, but the package 'transparent.sty' is not loaded}%
    \renewcommand\transparent[1]{}%
  }%
  \providecommand\rotatebox[2]{#2}%
  \ifx\svgwidth\undefined%
    \setlength{\unitlength}{91.48807982bp}%
    \ifx\svgscale\undefined%
      \relax%
    \else%
      \setlength{\unitlength}{\unitlength * \real{\svgscale}}%
    \fi%
  \else%
    \setlength{\unitlength}{\svgwidth}%
  \fi%
  \global\let\svgwidth\undefined%
  \global\let\svgscale\undefined%
  \makeatother%
  \begin{picture}(1,0.51865652)%
    \put(0,0){\includegraphics[width=\unitlength,page=1]{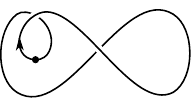}}%
    \put(0.49298785,0.32544708){\color[rgb]{0,0,0}\makebox(0,0)[lb]{\smash{$\sfa$}}}%
    \put(0.13287069,0.48501142){\color[rgb]{0,0,0}\makebox(0,0)[lb]{\smash{$\sfb$}}}%
    \put(0.15629294,0.12343041){\color[rgb]{0,0,0}\makebox(0,0)[lb]{\smash{$\sfv$}}}%
  \end{picture}%
\endgroup%
}}
\]
We equip an order on half-edges such that the outgoing half-edge is indexed by $1$.
Then the generators and degrees for $\cA(\sL)$ are given as
\begin{align*}
\sG_\sL&=\{\sfv_{1,\ell},\sfv_{2,\ell}\mid \ell\ge 1\}\amalg\{\sfa, \sfb\},\\
|\sfv_{1,\ell}|&=\begin{cases}
2r-1 & \ell=2r;\\
2r+2 & \ell=2r+1,
\end{cases}&
|\sfv_{2,\ell}|&=\begin{cases}
2r-1 & \ell=2r;\\
2r-2 & \ell=2r+1,
\end{cases}&
|\sfa|&=1,&
|\sfb|&=-1.
\end{align*}

The differentials for $\sfa$ and $\sfb$ are 
\begin{align*}
\partial \sfa &= 1+(-1)^{|\sfb|-1}\sfb\sfv_{2,1}\sfb+(-1)^{|\sfb|-1}\sfb\sfv_{2,2}+\sfv_{1,2}\sfb+\sfv_{1,3}\\&=1+\sfb\sfv_{2,1}\sfb+\sfb\sfv_{2,2}+\sfv_{1,2}\sfb+\sfv_{1,3},\\
\partial \sfb &= \sfv_{1,1}.
\end{align*}
Notice that the monomial $\sfv_{1,3}$ corresponds to an immersed bigon which is folded twice on the inner circular region. See Figure~\ref{figure:more turns}.
\begin{figure}[ht]
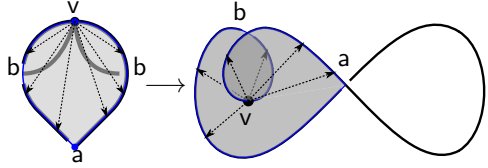

\[
\vcenter{\hbox{\def\svgscale{0.7}\input{more_turns_domain_input.tex}}}\longrightarrow
\vcenter{\hbox{\def\svgscale{1.2}\input{more_turns_image_input.tex}}}
\]
\caption{The immersed bigon corresponding to $\sfv_{1,3}$ in $\partial \sfa$}
\label{figure:more turns}
\end{figure}

Moreover, we have
\begin{align*}
\partial^2 \sfa &= 0+\left(\sfv_{1,1}\sfv_{2,1}\sfb+(-1)^{|\sfb\sfv_{2,1}|}\sfb\sfv_{2,1}\sfv_{1,1}\right)\\
&\mathrel{\hphantom{=}}+\left(\sfv_{1,1}\sfv_{2,2}+(-1)^{|\sfb|}\sfb(1+(-1)^{|\sfv_{2,1}|-1}\sfv_{2,1}\sfv_{1,1})\right)\\
&\mathrel{\hphantom{=}}+\left( (1+(-1)^{|\sfv_{1,1}|-1}\sfv_{1,1}\sfv_{2,1})\sfb+(-1)^{|\sfv_{1,2}|}\sfv_{1,2}\sfv_{1,1}\right)\\
&\mathrel{\hphantom{=}}+\left( (-1)^{|\sfv_{1,1}|-1}\sfv_{1,1}\sfv_{2,2}+(-1)^{|\sfv_{1,2}|-1}\sfv_{1,2}\sfv_{1,1}\right)\\
&=0.
\end{align*}

\section{Hybrid disks and the invariance theorem}\label{section:Hybrid disks and the invariance theorem}
This section is devoted to show that the pair of a DGA and the collection of canonical peripheral structures is an invariant for Legendrian graphs with potential. 
More precisely, we show the following.

\begin{theorem}[Invariance theorem]\label{theorem:invariance}
The pair $(\cA_\cL, \cP_\cL)$ up to generalized stable-tame isomorphisms is an invariant for $\cL$ under Legendrian Reidemeister moves for Legendrian graphs with potential.
\end{theorem}

In other words, whenever we have two equivalent Legendrian graphs $\cL$ and $\cL'$ with potentials, then the two DGAs with canonical peripheral structures $(\cA_\cL,\cP_\cL)$ and $(\cA_{\cL'},\cP_{\cL'})$ are generalized stable-tame isomorphic.
\[
\cL\simeq \cL'\Longrightarrow (\cA_\cL,\cP_\cL)\stabletameisom(\cA_{\cL'},\cP_{\cL'})
\]

\subsection{Strategy of the proof}\label{sec:strategy of the proof}
Recall the Reidemeister moves for Legendrian graphs with potentials shown in Figure~\ref{fig:RM}. Then without loss of any generality, we may assume that $\cL$ and $\cL'$ are different by exactly one Reidemeister move. The moves $\rm(0_a)$ and $\rm(0_b)$ only change the potentials and induce an isomorphism not only between DGAs but also the canonical peripheral structures, and there is nothing to prove.

The moves $\rm(II)$ and $\rm(III)$ are well-studied in the literature. See \cite{Chekanov2002}.
Note that the proofs of invariance under $\rm(II)$ and $\rm(III)$ are still valid for Legendrian {\em graphs}, and therefore we only need to prove invariance under the moves $\rm(IV_a)$ and $\rm(IV_b)$.

We consider two Legendrian graphs $\cL$ and $\cL'$ whose Lagrangian projections, denoted by $\sL$ and $\sL'$, only differ by Reidemeister move $\rm(IV_a)$. 
For convenience's sake, we will use the notation $(\cdot)$ and $(\cdot)'$ to denote $(\cdot)_\sL$ and $(\cdot)_{\sL'}$. For example,
the DGAs for $\cL$ and $\cL'$ with corresponding collections of canonical peripheral structures are denoted by $(\cA,\cP)$ and $(\cA',\cP')$, respectively.
\begin{align*}
\sC&\coloneqq \sC_\sL,&\sV&\coloneqq \sV_\sL,&\sG&\coloneqq \sG_\sL,&\cdots\\
\sC'&\coloneqq \sC_{\sL'},&\sV'&\coloneqq \sV_{\sL'},&\sG'&\coloneqq \sG_{\sL'}, &\cdots
\end{align*}

Let $\cA_\sfv^{\fd +}$ denote the $\fd$-th generalized ${\rm (+)}$-stabilization of $\cA$ with respect to the canonical peripheral structure 
$\bp_v:\cI_v \to \cA$, and we simply use bar notation $\bar{(\cdot)}$ for this stabilized algebra. For example,
\begin{align*}
\bar\cA&\coloneqq \cA_\sfv^{\fd +}=(\bar A, |\bar{\ \cdot\ }|, \bar\partial),&
\bar A&\coloneqq\ZZ\langle \bar \sG\rangle.
\end{align*}

Let $\sU$ and $\sU'$ be two neighborhoods of the vertices $\sfv\in\sV$ and $\sfv'\in\sV'$, respectively.
Henceforth, the outside of $\sU$ and $\sU'$ will be regarded as the same. We label the double points in $\sU'$ as $\{\sfc_1',\dots,\sfc_m'\}$ as shown in Figure~\ref{fig:RMIV}.

\begin{figure}[ht]
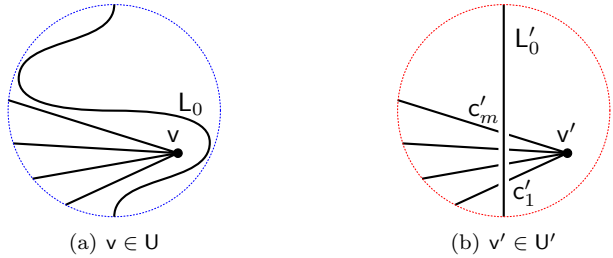

\subfigure[$\sfv\in\sU$]{\makebox[0.4\textwidth]{\input{nbhd_v_input.tex}}}
\subfigure[$\sfv'\in\sU'$]{\makebox[0.4\textwidth]{\input{nbhd_v_prime_input.tex}}}
\caption{Neighborhoods $\sU$ and $\sU'$ of $\sfv$ and $\sfv'$}
\label{fig:RMIV}
\end{figure}

We denote by $C_i'$ the triangular region in $\sU'$ determined by $\sfc_i', \sfc_{i+1}'$ and $\sfv'$.
Then by Stokes' theorem, there are obvious relations between $z(\sfc_i')$ and $\area(C_j')$ given by
\[
z(\sfc_m')=z(\sfc_{m-1}')+\area(C_{m-1})>\cdots>z(\sfc_2')=z(\sfc_1')+\area(C_1')>z(\sfc_1')>0.
\]
\begin{assumption}\label{assump:action}
By moving an arc $\sL_0'$ slightly in $\sU'$, we may assume that $\area(C_i')$'s are sufficiently small so that for any $\sfg'\in\sG'\setminus\{\sfc_1',\dots,\sfc_m'\}$ with $z(\sfg')\neq 0$,
\[
z(\sfc_m')-z(\sfc_1')=\sum_{i=1}^{m-1}\area(C_i') < z(\sfg')
\]
and either
\[
z(\sfg')<z(\sfc_1')\quad\text{ or }\quad z(\sfg')>z(\sfc_m').
\]
\end{assumption}

The canonical identifications between $\sG$ and a subset of $\sG'$ and between $\sG'$ and $\bar\sG$ will be denoted by $\iota$ and $\iota'$, and their left inverses will be denoted by $\pi$ and $\pi'$, respectively. 
\begin{align*}
&\begin{tikzcd}[ampersand replacement=\&]
\sG\cup\{0\}\ar[r,"\iota",shift left=0.5ex]\&\sG'\cup\{0\}\ar[r,"\iota'",shift left=0.5ex]\ar[l,"\pi",shift left=0.5ex]\&\bar\sG\cup\{0\},\ar[l,"\pi'",shift left=0.5ex]
\end{tikzcd}&
z(\sfg)&=z(\iota(\sfg)),&
\pi(\sfg')&\coloneqq\begin{cases}
0 & \sfg'=\sfc_i';\\
\sfg & \sfg'=\iota(\sfg).
\end{cases}
\end{align*}

Notice that the compositions $\iota'\circ\iota$ and $\pi\circ\pi'$ denoted by $\bar\iota$ and $\bar\pi$ induce two canonical DGA morphisms $\bar\iota_*=\iota_\sfv^{\fd+}$ and $\bar\pi_*=\pi_\sfv^{\fd+}$, respectively.
\begin{equation}\label{eq:bar iota}
\begin{tikzcd}[column sep=4pc]
\cA\ar[r, "\bar\iota_*=\iota_\sfv^{\fd +}",shift left=.5ex]&
\bar\cA\coloneqq\cA_\sfv^{\fd+}\ar[l,"\bar\pi_*=\pi_\sfv^{\fd +}", shift left=.5ex].
\end{tikzcd}
\end{equation}

\begin{notation}
Let us denote the generators having action at least $\sfc_1'$ by $\sfg_i'$ so that
\begin{align*}
\sfg_i'&\coloneqq \sfc_i'\quad\forall i\le m,&
z(\sfg_1')&<\cdots<z(\sfg_m')< z(\sfg_{m+1}')\le z(\sfg_{m+2}')\le\cdots\le z(\sfg_{m+n}')
\end{align*}
and let $\sG'_{[i]}$ be the subset of $\sG'$ defined as
\begin{align*}
\sG'_{[0]}&\coloneqq\{\sfg'\in\sG'\mid z(\sfg')<z(\sfc_1')\},&
\sG'_{[i]}&\coloneqq \sG'_{[0]}\amalg\{\sfg_1',\cdots, \sfg_i'\}.
\end{align*}
\end{notation}

We define $\sfg_i\coloneqq\pi(\sfg_i')$ and $\sG_{[i]}\coloneqq\pi(\sG_{[i]}')$ as the images under the partial function $\pi$ of $\sfg_i'$ and $\sG_{[i]}'$, respectively, so that 
\begin{align*}
\sG_{[0]}&=\sG_{[1]}=\cdots=\sG_{[m]},&
\sG_{[m+j]}&=\sG_{[0]}\amalg\{\sfg_{m+1},\cdots,\sfg_{m+j}\}.
\end{align*}

\begin{definition}
We define the filtered DGAs $\cA_\bullet$ and $\cA_\bullet'$ as filtrations of DG-subalgebras of $\cA$ and $\cA'$ generated by $\sG_{[i]}$ and $\sG'_{[i]}$, respectively.
\begin{align*}
\cA_\bullet&\coloneqq \{\cA_{[i]}=(A_{[i]},|\cdot|,\partial)\},&
A_{[i]}&\coloneqq \Zmod{2}\langle \sG_{[i]}\rangle;\\
\cA'_\bullet&\coloneqq \{\cA'_{[i]}=(A_{[i]}',|\cdot'|,\partial')\},& 
A_{[i]}'&\coloneqq \Zmod{2}\langle \sG'_{[i]}\rangle;
\end{align*}
\end{definition}

\begin{assumption}[hybrid admissible disks]
For each $\bar\sfg\in\bar\sG$, we will define the finite set $\scrM(\bar\sfg)_{(0)}$ such that it decomposed into 
\[
\scrM(\bar\sfg)_{(0)}=\coprod_{t\ge s\ge 0}\scrM_{t,s}(\bar\sfg)_{(0)},
\]
and for each $f\in\scrM_{t,s}(\bar\sfg_i)_{(0)}$, there are $\sgn(f)=\pm1$ and a function $\tilde f:\{\bv_1,\dots, \bv_t\}\to \left(\sG_{[i-1]}\amalg \sG'_{[i-1]}\right)$ such that 
\begin{align*}
\tilde f(\bv_1),\cdots,\tilde f(\bv_s)&\in \sG'_{[i-1]},&
\tilde f(\bv_{s+1}),\cdots,\tilde f(\bv_t)&\in \sG_{[i-1]}.
\end{align*}

The geometric construction of $\scrM_{t,s}(\bar\sfg)_{(0)}$ will be given in Definition~\ref{def:hybrid admissible disk} and Notation~\ref{not:hybrid disks}.
We call elements in $\scrM(\bar\sfg)_{(0)}$ {\em hybrid admissible disks} of degree $0$ and $\tilde f$ the {\em canonical label} for $f$ as before.
\end{assumption}

\begin{definition}\label{def:filtered graded algebra}
We define a morphism $\Psi_\bullet:A'_\bullet\to A_\bullet$ between filtered graded algebras inductively as follows:
\begin{enumerate}
\item $\Psi_{[0]}\coloneqq \pi_*$ and $\Psi^\hyb(\sfg')\coloneqq0$ for $\sfg'\in\sG'_{[0]}$;
\item for each $i\ge 1$ and $\sfg'\in\sG'_{[i]}$,
\begin{align*}
\Psi_{[i]}(\sfg')&\coloneqq\begin{cases}
\Psi_{[i-1]}(\sfg')& \sfg'\neq\sfg'_i;\\
\pi_*(\sfg'_i)+\Psi^\hyb(\sfg'_i)
&\sfg'=\sfg'_i,
\end{cases}\\
\Psi^\hyb(\sfg'_i)&\coloneqq
\sum_{\substack{t\ge r\ge 0\\f\in\scrM_{t,r}(\bar\sfg_i)_{(0)}}} \sgn(f)\Psi_{[i-1]}\left(\tilde f(\bv_1 \cdots \bv_r)\right) \tilde f(\bv_{r+1}\cdots \bv_t)\in A_{[i-1]}.
\end{align*}
\end{enumerate}
\end{definition}

The soundness of this definition will be shown in Remark~\ref{remark:soundness of definition of Psi}.
We have the following propositon which will be proved in \S~\ref{section:obstruction} and \S~\ref{section:hybrid disks}.
\begin{proposition}\label{proposition:commutativity}
The following holds:
\begin{enumerate}
\item The restrictions of $\iota_*$ and $\pi_*$ on $\cA_{[0]}$ and $\cA_{[0]}'$ are isomorphisms between the DGAs
\[
\begin{tikzcd}
\cA_{[0]}\ar[r,"\iota_*",shift left=.5ex] & \cA_{[0]}'\ar[l,"\pi_*", shift left=.5ex]
\end{tikzcd}
\]
which are inverse to each other. In particular, $\Psi_{[0]}=\pi_*$ commutes with differentials.
\item Suppose that $\Psi_{[i-1]}$ commutes with differentials. Then so does $\Psi_{[i]}$, i.e.,
\begin{align}\label{eqn:Psi commutes differential}
(\partial\circ \Psi_{[i]})(\sfg_i')=(\Psi_{[i]}\circ\partial')(\sfg_i').
\end{align}
\end{enumerate}

Therefore $\Psi_\bullet$ becomes a morphism between filtered DGAs.
\end{proposition}
\begin{remark}
Neither $\iota$ nor $\pi$ induces a DGA morphism between $\cA$ and $\cA'$ since $\iota_*$ and $\pi_*$ do not commute with differentials in general.
\end{remark}

Then there is an essential corollary of this proposition as follows:
\begin{corollary}\label{corollary:dga isomorphism}
There is a tame DGA isomorphism 
\[
\Phi:\cA'\to \bar\cA.
\]
\end{corollary}
\begin{proof}
Let us define $\Phi:\cA'\to\bar\cA$ as
\[
\Phi(\sfg')\coloneqq(\iota'_*+(\bar\iota_*\circ\Psi^\hyb))(\sfg')=\bar\sfg+(\bar\iota_*\circ\Psi^\hyb)(\sfg').
\]
Then it is obvious that $\Phi$ is a tame isomorphism and therefore it suffices to prove that $\Phi$ commutes with differentials.

If $\sfg'\in\sG'$ different from any $\sfc_j'$, then $(\bar\iota_*\circ\pi_*)(\sfg')=\bar\sfg=\iota'_*(\sfg')$ and so
\begin{align}\label{eq:commutativity small action}
\Phi(\sfg')\coloneqq(\iota'_*+(\bar\iota_*\circ(\Psi_{[i]}-\pi_*))(\sfg')=(\bar\iota_*\circ\Psi_{[i]})(\sfg')
\end{align}
for some $i\ge0$ with $\sfg'\in\sG'_{[i]}$. Since both $\bar\iota_*$ and $\Psi_{[i]}$ commute with differentials by (\ref{eq:bar iota}) and Proposition~\ref{proposition:commutativity}(2), we are done.

Suppose that $\sfg'=\sfc'_i$. Then $\Psi^\hyb(\sfc'_i)=\Psi_{[i]}(\sfc'_i)$ by definition of $\Psi_{[i]}$ since $\pi_*(\sfc_i')=0$.
Therefore since $(\bar\iota_*\circ\Psi_{[i]})$ commutes with differentials, 
\[
(\bar\partial\circ\Phi)(\sfc_i')=\bar\partial(\iota'_*+(\bar\iota_*\circ\Psi_{[i]}))(\sfc_i')
=\bar\partial(\bar\sfc_i)+(\bar\iota_*\circ\Psi_{[i]})(\partial'\sfc_i').
\]
For further manipulations, we need to know $\partial'\sfc_i'$ more carefully.

By the area condition (\ref{eqn:action and area}) of Proposition~\ref{proposition:finiteness}, for a monomial $w$ of $\partial' \sfc_i'$, the actions of all generators involved in $w$ is strictly less than $\sfc_i'$. That is, if $w$ involves $\sfc_j'$ then $j<i$, and moreover, by Assumption~\ref{assump:action} that the difference $z(\sfc_i')-z(\sfc_j')$ is sufficiently small, all other generators come from vertices. Furthermore, these vertices can not be far from $\sfc_i'$ since the area of the immersed disk defining the given monomial $w$ is at most $z(\sfc_i')-z(\sfc_j')$. Therefore the only possibilities correspond to the triangular regions $C_j'\cup\cdots\cup C_{i-1}'$ for $1\le j<i$ which contribute the monomial $\sfc_j' \sfv'_{j,i-j}$ in $\partial' \sfc'_i$. 

If $w$ does not involve any $\sfc'_j$, then $w\in A'_{[0]}$ by the area condition again. Let us collect these monomials and denote the result by $d'_i$, then we have
\[
\partial' \sfc_i'=(-1)^{|\sfc'_i|-1}\left(\sfc_1' \sfv'_{1,i-1}+\cdots+\sfc'_{i-1}\sfv'_{i-1,1}\right)+d'_i
\]
for some $d'_i\in A'_{[0]}$. 
Note that
\[
(\bar\iota_*\circ\Psi_{[i]})(d_i')=(\bar\iota_*\circ\Psi_{[0]})(d_i')=\Phi(d_i'),
\]
where the last equality comes from (\ref{eq:commutativity small action}). Therefore
\begin{align*}
(\bar\partial\circ\Phi)(\sfc'_i)
&=\bar\partial(\bar\sfc_i)+(\bar\iota_*\circ\Psi_{[i]})(\partial'\sfc'_i)\\
&=(-1)^{|\bar\sfc_i|-1}(\bar\sfc_1\bar\sfv_{1,i-1}+\cdots+\bar\sfc_{i-1}\bar\sfv_{i-1,1})\\
&\mathrel{\hphantom{=}}+(-1)^{|\sfc'_i|-1}(\bar\iota_*\circ\Psi_{[i]})(\sfc'_1\sfv'_{1,i-1}+\cdots+\sfc'_{i-1}\sfv'_{i-1,1})+(\bar\iota_*\circ\Psi_{[i]})(d'_i)\\
&=(-1)^{|\bar\sfc_i|-1}(\iota'_*+(\bar\iota_*\circ\Psi_{[i]}))(\sfc_1'\sfv_{1,i-1}'+\cdots+\sfc_{i-1}'\sfv_{i-1,1})+(\bar\iota_*\circ\Psi_{[i]})(d_i')\\
&=\Phi((-1)^{|\sfc'_i|-1}(\sfc_1'\sfv_{1,i-1}'+\cdots+\sfc_{i-1}'\sfv_{i-1,1})+d_i')=\Phi(\partial'\sfc_i')
\end{align*}
as desired and this completes the proof.
\end{proof}

With the aid of this corollary, we can prove Theorem~\ref{theorem:invariance}.
\begin{proof}[Proof of Theorem~{\rm\ref{theorem:invariance}}]
Suppose that $\cL$ and $\cL'$ are related by $\rm(IV_a)$. Then we are done by Corollary~\ref{corollary:dga isomorphism}.

On the other hand, suppose that two Legendrian graphs with potentials $\cL$ and $\cL''$ only differ by Reidemeister move $\rm(IV_b)$. 
Let us denote their Lagrangian projections by $\sL$ and $\sL''$.
Then by Remark~\ref{remark:Legendrian mirror}, $\mu(\sL)$ and $\mu(\sL'')$ are related by Reidemeister move $\rm(IV_a)$ which implies that $\cA_{\mu(\cL'')}$ and $S_{\mu^*(\bp)}^{\fd+}\left(\mu^*(\cA_\cL)\right)$ are tame isomorphic. Now we conclude
\[
\begin{tikzcd}
\cA_{\cL''}\ar[d,"\mu_*"',"\simeq"]\ar[r,"\exists\tameisom"] & S_\bp^{\fd-}(\cA_{\cL})\ar[d,"\mu_*","\simeq"']\\
\mu_*(\cA_{\mu(\cL'')})\ar[r,"\tameisom"] &\mu_*(S_{\mu^*(\bp)}^{\fd+}(\mu^*(\cA_\cL))),
\end{tikzcd}
\]
where Lemma~\ref{lemma:Legendrian mirror anti-isomorphism} implies the left isomorphism, the right isomorphism is guaranteed by Remark~\ref{remark:anti-isomorphism and stabilizations} and
the bottom tame isomorphism comes from Corollary~\ref{corollary:dga isomorphism}. 

It is obvious that the composition of these three isomorphisms is a tame isomorphism, and this finally proves the theorem. 
\end{proof}

\subsection{Realizability on the other side}\label{section:obstruction}
We discuss the realizability of admissible disks for $\sL$ as admissible disks for $\sL'$, or {\it vice versa}. 
For convenience's sake, we assume the following:
\begin{align*}
f&\in\cM(\Lambda),&
f'&\in\cM(\Lambda'),&
\iota(\sfg)&=\sfg',&
\iota(\sfv_{i,\ell})=\sfv'_{i,\ell},
\end{align*}
for any $\sfg\in\sC_\sL$, $\sfv\in\sV_\sL$, $i\in\Zmod{\val(\sfv)}$ and $\ell\ge 1$.

\begin{lemma}\label{lemma:infinitesimal correspondence}
For any $\sfv_{i,\ell}\in\tilde\sV_\sL$ and $\sfv'_{i,\ell}\in\tilde\sV_{\sL'}$ with $\iota(\sfv_{i,\ell})=\sfv'_{i,\ell}$, there is a one-to-one correspondence 
$\iota_\cM:\cM(\sfv_{i,\ell})\to\cM(\sfv'_{i,\ell})$ such that
\[
\iota_*(\cP(f))=\cP(\iota_\cM(f)),
\]
where $\iota_*:A\to A'$ is a graded algebra morphism.
\end{lemma}
\begin{proof}
Since the set $\cM(\sfv_{i,\ell})$ is determined only by local information such as the valency and two indices $i$ and $\ell$, we have a canonical bijection $\iota_\cM:\cM(\sfv_{i,\ell})\to\cM(\sfv'_{i,\ell})$.
Moreover, it is obvious that $\iota_\cM$ preserves degrees and labels via $\iota_*:A\to A'$.
\end{proof}

Suppose that $f$ is regular. That is, $f\in\cM(\sfg)$ for some $\sfg\in\sC$.
Then the boundary $\partial f\coloneqq f|_{\partial\Pi}\subset \sL$ is a loop based at $\sfg$ which can be lifted {\em piecewise} to $\tilde{\partial f}\subset \Lambda_0\subset\RR^3$ via $\pi_L$.
Notice that vertical jumps occur only at double points of the projection.

Now let $\phi_t$ be a Legendrian isotopy from $\Lambda_0=\Lambda$ to $\Lambda_1=\Lambda'$ realizing Reidemeister move $\rm(IV_a)$ which is fixed outside of $\pi_L^{-1}(\sU)$. Then for each $0\le t\le 1$, the projection $\tilde{\partial f}_t\coloneqq\pi_L(\phi_t(\tilde{\partial f}))$ can be understood as a loop in $\sL_t\coloneqq\pi_L(\Lambda_t)$ because all jumps of $\tilde{\partial f}$ lie outside $\pi_L^{-1}(\sU)$ and remain fixed during the isotopy.

Conversely, let $f'\in\cM(\sfg')$ and assume that $\tilde f'$ misses all $\sfc'_j$'s.
Then all jumps of the piecewise lift $\tilde{\partial f'}$ avoid $\pi_L^{-1}(\sU')$ and therefore the one-parameter family of loops $\{\tilde{\partial f'}_t\subset \sL_t\mid 0\le t\le 1\}$ is well-defined.

The natural question is whether $\tilde{\partial f}_1$ or $\tilde{\partial f'}_0$ can be realized as the boundary of a regular admissible disk in $\cM(\Lambda')$ or $\cM(\Lambda)$.
We will give a complete answer for this question as follows:

Let $\beta$ and $\beta'$ be arcs from $\sfv$ and $\sfv'$ to points in $\sL_0$ and $\sL_0'$, respectively, such that they are disjoint from $\sL$ and $\sL'$ in their interior. See Figure~\ref{figure:gamma curves}.

\begin{figure}[ht]
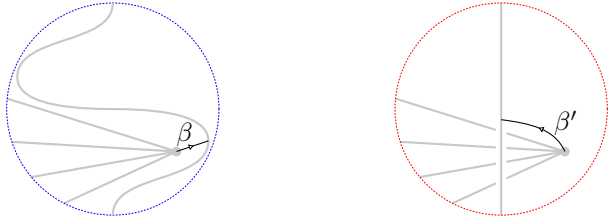

\begin{align*}
\vcenter{\hbox{\input{beta_input.tex}}}
&&
\vcenter{\hbox{\input{beta_prime_input.tex}}}
\end{align*}
\caption{Curves $\beta$ and $\beta'$ in $\sU$ and $\sU'$}
\label{figure:gamma curves}
\end{figure}

\begin{lemma}\label{lem:adm_obstruction}
Let $f\in\cM(\sfg)$ and $f'\in\cM(\sfg')$ for some $\sfg\in\sC$ and $\sfg'\in\sC'$ with $\iota(\sfg)=\sfg'$. Suppose that $\tilde f'$ misses all $\sfc_i'$'s.
Then the following holds:
\begin{enumerate}
\item $\tilde{\partial f}_1$ extends to an admissible disk $\iota_\cM(f)\in\cM(\sfg')$ if and only if there are no properly embedded components of $f^{-1}(\beta)$ in $\Pi$.
\item $\tilde{\partial f'}_0$ extends to an admissible disk $\pi_\cM(f')\in\cM(\sfg)$ if and only if there are no properly embedded components of $f'^{-1}(\beta')$ in $\Pi'$.
\end{enumerate}

In both cases, the labels are related via $\iota_*$ and $\pi_*$, respectively.
\begin{align*}
\iota_*(\cP(f))&=\cP(\iota_\cM(f)),&
\cP(\pi_\cM(f'))&=\pi_*(\cP(f')).
\end{align*}
\end{lemma}
\begin{proof}
\noindent(1)~Let $\{\beta_1,\dots,\beta_r\}$ be the connected components of $f^{-1}(\beta)$, where each component $\beta_j$ is contained in a component $\bU(\beta_j)$ of $f^{-1}(\sU)$.
Then there are four cases according to whether the end points of $\beta_j$ are lying in the boundary $\partial \Pi$ or not, which correspond to the pictures on the left of Figures~\ref{figure:interior-interior}, \ref{figure:boundary-interior}, \ref{figure:interior-boundary A} and \ref{figure:improper overlap}, respectively.

Unless there is a component $\beta_j$ with $\partial\beta_j\subset\partial\Pi$, or equivalently, unless Figure~\ref{figure:improper overlap} happens, we can modify $\bO_f\coloneqq f^{-1}(\sL)$ as shown in the pictures on the right of the figures to obtain a new complex $\bO'\subset\Pi$. Then we can define an admissible disk $\iota_\cM(f)\in\cM(\sfg')$ on $\Pi$ by the condition
\[
\bO_{\iota_\cM(f)}=\bO'.
\]

Conversely, if there is a properly embedded component, then there are no corresponding admissible disks since the orientation-preserving property will be broken during the isotopy.
More precisely, as $t\to1$, neighborhoods of the two end points $\beta_j(0)$ and $\beta_j(1)$ will overlap via $\phi_t$ and an orientation-reversing region, the shaded triangle in Figure~\ref{figure:improper overlap}, is created and so $\tilde{\partial f}_1$ will never extend to a disk. 

\noindent(2)~The proof of (2) is the essentially the same as (1). There are four possibilities for configurations of the inverse image $\beta_j'$ depicted in Figures~\ref{figure:interior-interior}, \ref{figure:boundary-interior}, \ref{figure:interior-boundary B} and \ref{figure:improper overlap 2}. Among these configurations, the one and only obstruction arises when $\beta_j'$ is properly embedded and an orientation-reversing (shaded) region as seen in Figure~\ref{figure:improper overlap 2} is created again.
\end{proof}

\begin{figure}[ht]
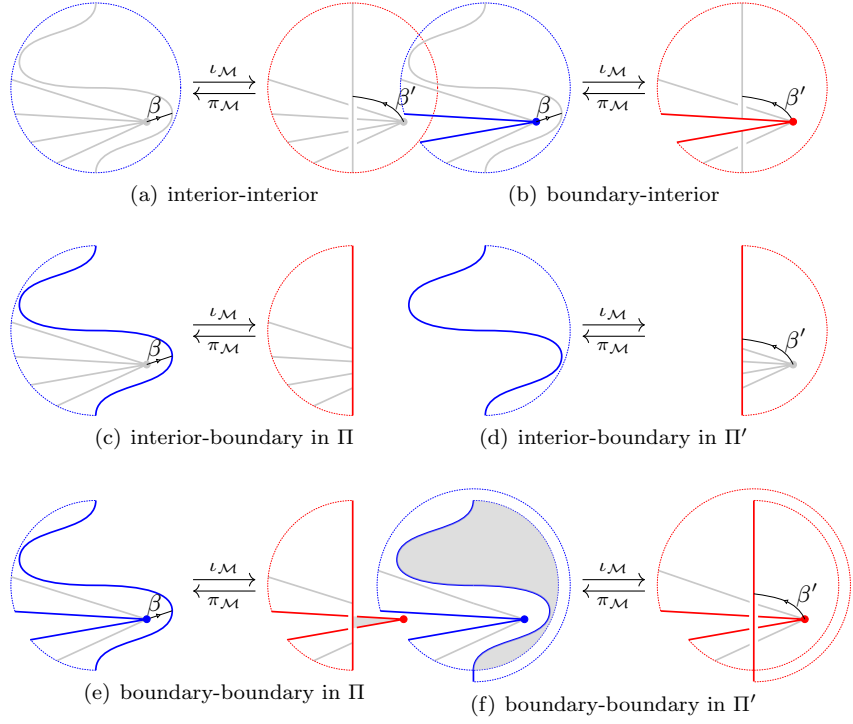

\subfigure[\label{figure:interior-interior}interior-interior]{\makebox[0.4\textwidth]{
$
\begin{tikzcd}[ampersand replacement=\&]
\vcenter{\hbox{\small\def\svgscale{0.8}\input{beta_input.tex}}}\ar[r,"\iota_\cM",shift left=.5ex]\&
\vcenter{\hbox{\small\def\svgscale{0.8}\input{beta_prime_input.tex}}}\ar[l,"\pi_\cM",shift left=.5ex]
\end{tikzcd}
$
}}
\subfigure[\label{figure:boundary-interior}boundary-interior]{\makebox[0.4\textwidth]{
$
\begin{tikzcd}[ampersand replacement=\&]
\vcenter{\hbox{\small\def\svgscale{0.8}\input{beta_3_input.tex}}}\ar[r,"\iota_\cM",shift left=.5ex]\&
\vcenter{\hbox{\small\def\svgscale{0.8}\input{beta_3_prime_input.tex}}}\ar[l,"\pi_\cM",shift left=.5ex]
\end{tikzcd}
$
}}
\subfigure[\label{figure:interior-boundary A}interior-boundary in $\Pi$]{\makebox[0.4\textwidth]{
$
\begin{tikzcd}[ampersand replacement=\&]
\vcenter{\hbox{\small\def\svgscale{0.8}\input{beta_2_input.tex}}}\ar[r,"\iota_\cM",shift left=.5ex]\&
\vcenter{\hbox{\small\def\svgscale{0.8}\input{beta_2_prime_input.tex}}}\ar[l,"\pi_\cM",shift left=.5ex]
\end{tikzcd}
$
}}
\subfigure[\label{figure:interior-boundary B}interior-boundary in $\Pi'$]{\makebox[0.4\textwidth]{
$
\begin{tikzcd}[ampersand replacement=\&]
\vcenter{\hbox{\small\def\svgscale{0.8}\input{beta_2-1_input.tex}}}\ar[r,"\iota_\cM",shift left=.5ex]\&
\vcenter{\hbox{\small\def\svgscale{0.8}\input{beta_2-1_prime_input.tex}}}\ar[l,"\pi_\cM",shift left=.5ex]
\end{tikzcd}
$
}}
\subfigure[\label{figure:improper overlap}boundary-boundary in $\Pi$]{\makebox[0.4\textwidth]{
$
\begin{tikzcd}[ampersand replacement=\&]
\vcenter{\hbox{\small\def\svgscale{0.8}\input{beta_4_input.tex}}}\ar[r,"\iota_\cM",shift left=.5ex]\&
\vcenter{\hbox{\small\def\svgscale{0.8}\input{beta_4_prime_input.tex}}}\ar[l,"\pi_\cM",shift left=.5ex]
\end{tikzcd}
$
}}
\subfigure[\label{figure:improper overlap 2}boundary-boundary in $\Pi'$]{\makebox[0.4\textwidth]{
$
\begin{tikzcd}[ampersand replacement=\&]
\vcenter{\hbox{\small\def\svgscale{0.8}\input{beta_5_input.tex}}}\ar[r,"\iota_\cM",shift left=.5ex]\&
\vcenter{\hbox{\small\def\svgscale{0.8}\input{beta_5_prime_input.tex}}}\ar[l,"\pi_\cM",shift left=.5ex]
\end{tikzcd}
$
}}
\caption{Inverse images $f^{-1}(\beta\cup \sL)$ and $f'^{-1}(\beta'\cup \sL')$ in components of $f^{-1}(\sU)\subset\Pi$ and $f'^{-1}(\sU')\subset\Pi'$}
\label{figure:beta_0}
\end{figure}

\begin{remark}\label{remark:small implies no obstructions}
The existence of a properly embedded component in $f'^{-1}(\beta')$ implies that $z(\sfg')>z(\sfc_m')$ and therefore $\sfg'\in\sG'_{>\sfc}$ by the area condition.

A similar result holds for $f$. That is, if there is a properly embedded component in $f^{-1}(\beta)$, then $\sfg\in\sG_{>\sfc}$.
Therefore, for $\sfg\in\sG_{<\sfc}$ and $\sfg'\in\sG_{<\sfc}'$, all $f\in\cM(\sfg)$ and $f'\in\cM(\sfg')$ satisfy the hypothesis of the above lemma.
\end{remark}

As a corollary of Lemma~\ref{lem:adm_obstruction}, we can prove Proposition~\ref{proposition:commutativity}(1).

\begin{proof}[Proof of Proposition~{\rm\ref{proposition:commutativity}(1)}]
Since $\iota_*$ and $\pi_*$ are isomorphisms between graded algebras and inverse to each other, it suffices to prove the commutativity with differentials only for $\iota_*$.

By Lemmas~\ref{lemma:infinitesimal correspondence} and \ref{lem:adm_obstruction}, there exists a bijection $\iota_\cM:\cM(\sfg)\to\cM(\sfg')$ preserving $\cP$. Therefore by the definition of the differential in (\ref{equation:differential another form})
\[
\iota_*(\partial(\sfg))=\sum_{f\in\cM(\sfg)}\iota_*(\cP(f))=\sum_{f\in\cM(\sfg)}\cP(\iota_\cM(f))=\sum_{f'\in\cM(\iota(\sfg))}\cP(f')=\partial'(\iota(\sfg)).\qedhere
\]
\end{proof}

\subsection{Hybrid admissible disks}\label{section:hybrid disks}
Let $\bar\Lambda\coloneqq \Lambda\cup\Lambda'$ be the union of $\Lambda$ and $\Lambda'$, and let its projection be $\bar\sL$.
Then the corresponding neighborhoods $\sU, \bar\sU$ and $\sU'$ of $\sL, \bar\sL$ and $\sL'$ look as follows:
\begin{figure}[ht]
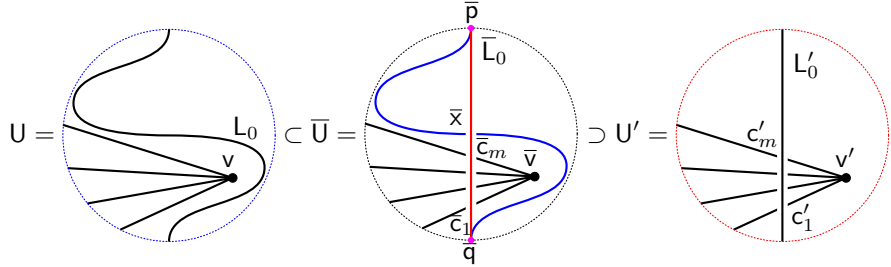
\label{fig:nbhd}
\[
\sU=\vcenter{\hbox{\input{nbhd_v_input.tex}}}\subset
\bar\sU=\vcenter{\hbox{\input{nbhd_bar_v_input.tex}}}
\supset\sU'=\vcenter{\hbox{\input{nbhd_v_prime_input.tex}}}
\]
\caption{The neighborhoods $\sU, \bar\sU$ and $\sU'$}\label{fig:hybrid disks}
\end{figure}

Notice that $\bar\sL$ in $\bar\sU$ has the following information (see Figure~\ref{fig:hybrid disks}):
\begin{enumerate}
\item double points $\{\bar\sfc_1,\dots,\bar\sfc_m\}$ coming from $\sL'$;
\item one additional double point $\bar\sfx$ at the center, which corresponds to the intersection of $\sL_0$ and $\sL_0'$;
\item two points $\bar \sfp$ (upper) and $\bar \sfq$ (lower) where $\sL_0$ and $\sL_0'$ are tangent.
\end{enumerate}

\begin{remark}
By declaring $\bar\sfp$ and $\bar\sfq$ to be trivalent vertices, we may regard $\bar\Lambda$ as a spatial graph whose edges are all Legendrians.
However, it is {\em not} a Legendrian graph in our definition since it violates the non-tangency condition of half-edges.

Instead, one may regard $\bar\sL$ as a Legendrian graph together with certain smoothing information.
\end{remark}

We denote the sets of vertices, double points, special points and generators by 
\begin{align*}
&\bar\sV,&
\bar\sC_{\bar\sfx}&\coloneqq\bar\sC\amalg\{\bar\sfx\},&
\bar\sS_{\bar\sfx}&\coloneqq\bar\sS\amalg\{\bar\sfx\},&
\bar\sG_{\bar\sfx}&\coloneqq\bar\sG\amalg\{\bar\sfx\}.
\end{align*}
Then $\bar\sV, \bar\sC, \bar\sS$ and $\bar\sG$ can be identified with $\sV',\sC', \sS'$ and $\sG'$.
Moreover, we may regard $\bar\sG$ as the generating set of the stabilized DGA $\bar\cA$.

Let $\bar A_{\bar\sfx}$ be the free associative unital graded algebra over $\ZZ$ generated by $\bar\sG_{\bar\sfx}$.
We define a grading $|\cdot|:\bar A_{\bar\sfx}\to\fR$, a differential $\bar\partial_{\bar\sfx}:\bar A_{\bar\sfx}\to \bar A_{\bar\sfx}$ and a projection $\pi_{\bar\sfx}:\bar A_{\bar\sfx}\to \bar A$ by
\begin{align*}
|\bar\sfg|&\coloneqq \begin{cases}
|\sfg'|&\bar\sfg\neq\bar\sfx;\\
0 & \bar\sfg=\bar\sfx,
\end{cases}
&
\bar\partial_{\bar\sfx}(\bar\sfg)&\coloneqq\begin{cases}
\bar\partial(\bar\sfg)&\bar\sfg\neq\bar\sfx;\\
0 & \bar\sfg=\bar\sfx,
\end{cases}
&
\pi_{\bar\sfx}(\bar\sfg)&\coloneqq\begin{cases}
\bar\sfg & \bar\sfg\neq\bar\sfx;\\
0 & \bar\sfg=\bar\sfx.
\end{cases}
\end{align*}
Then both the inclusion $\bar\cA\to\bar\cA_{\bar\sfx}$ and the projection $\pi_{\bar\sfx}:\bar\cA_{\bar\sfx}\to\bar\cA$ are DGA morphisms.

\begin{definition}
Let $f:(\Pi,\partial\Pi,\vPi)\to(\RR^2,\bar\sL,\bar\sS_{\bar\sfx})$ be a differentiable disk. We call $\be\in\ePi$
\begin{enumerate}
\item {\em $(-)$-folding} if $f|_{\be}$ has a unique critical point $\bx\in\be$ such that $f$ maps a neighborhood of $\bx$ as follows:
\[
\begin{tikzcd}
\vcenter{\hbox{\includegraphics{polygon_edge_singular_0.pdf}}}\ar[r,"f"]&
\vcenter{\hbox{\includegraphics{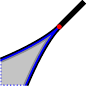}}}
\end{tikzcd}
\]
\item {\em $(+)$-folding} if $\be$ is $2$-folding in Definition~\ref{def:folding edges};
\item {\em $(\pm)$-folding} if $f|_{\be}$ has two critical points which are locally $(+)$-folding and $(-)$-folding, respectively.
\end{enumerate}
\end{definition}

A $(-)$-folding edge never occurs in our usual situation since there are no corners of angle 0 bounded by a Legendrian graph.
However, in our current situation, there are two such corners in $\bar\sL$, which are $\bar \sfp$ and $\bar \sfq$.
In other words, if $f$ has a $(-)$-folding edge $\be$ with critical point $\bx\in\be$, then $f(\bx)=\bar \sfp$ or $\bar \sfq$.

\begin{notation}
Let us identify the boundary $\partial \Pi_t$ of a $(t+1)$-gon with the quotient space of the interval $[0,t+1]$.
\[
\partial \Pi_t\simeq [0,t+1]/0\sim (t+1),
\]
and assume that the integers in $[0,t+1]$ correspond to the vertices.
Then for $a,b\in[0,t+1]$, we denote the open or closed subarc in $\partial\Pi_t$ by $(a,b)$ or $[a,b]$. Note that 
\begin{align*}
[0,t+1]&=\partial \Pi_t,&
[t+1,t+1]&=\bv_0.
\end{align*}
\end{notation}

\begin{definition}[hybrid admissible disk]\label{def:hybrid admissible disk}
We say that $f$ is a {\em hybrid admissible} disk if it satisfies the following:

\begin{enumerate}
\item $f$ is smooth and orientation-preserving on $\rPi$ which extends to $\ePi$;
\item All vertices are either convex, concave, or neutral;
\item All edges are either smooth, $(+)$-folding, $(-)$-folding, or $(\pm)$-folding; 
\item $\bv_0$ is positive;
\item There is a distinguished point $\separator\in[0,t+1]$, called the {\em separator}, such that 
\begin{align*}
f([\separator,t+1])&\subset\sL\subset\bar\sL,&
f([0,\separator])&\subset\sL'\subset\bar\sL.
\end{align*}
We call $f$ either
\begin{enumerate}
\item {\em vertex-separated} if $\separator\in\vPi$; or 
\item {\em edge-separated} if $\separator\in\ePi$ and $f(\separator)\in\{\bar\sfp,\bar\sfq\}$.
\end{enumerate}
\end{enumerate}
\end{definition}
\begin{figure}[ht]
\[
\begin{tikzcd}
\sL&\ar[l,"f"'] [\separator,t+1]=\left\{\vcenter{\hbox{\input{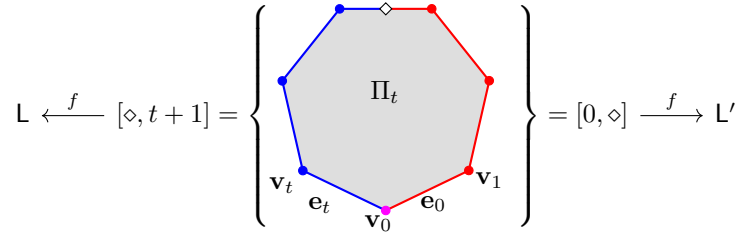}}}\right\}=[0,\separator]\ar[r,"f"]& \sL'
\end{tikzcd}
\]
\caption{A schematic picture of an edge-separated hybrid disk}
\end{figure}

\begin{remark}\label{rmk:separator}
Suppose that $f$ has a $(-)$-folding edge $\be$. Then its critical point $\bx\in\be$ must be sent to either $\bar \sfp$ or $\bar \sfq$ via $f$ and $\bx=\separator$. Therefore $f$ should be edge-separated.
\end{remark}

As before, the canonical label $\tilde f$ is well-defined so that it fits into the commutative diagram as follows:
\[
\begin{tikzcd}[column sep=3pc]
& \bar\sG_{\bar\sfx}\ar[d]\\
\vPi \ar[r,"f|_\vPi"']\ar[ru,"\exists\tilde f", dashed]&\bar\sS_{\bar\sfx}.
\end{tikzcd}
\]
Notice that $\tilde f$ may hit $\bar\sfx$ only if $f$ is vertex-separated. Otherwise one can always regard $\tilde f(\bv_i)$ as an element in $\sG'$ if $i\in(0,\separator)$ or in $\sG$ if $i\in(\separator,t+1)$, respectively.
As before, we assume that $\tilde f$ is extended freely to words on $\vPi$.

The $\ZZ$ and $\fR$-degrees for $f$ are defined as before,
\begin{align*}
[D(\partial f)]&=|f|_\ZZ\cdot 1_{\gr}\in\pi_1(\gr(1,\RR^2))\simeq\ZZ,&
|f|&\coloneqq |f|_\ZZ\cdot 1_\fR\in\fR.
\end{align*}

\begin{proposition}\label{prop:degree_hybrid}
The degree of a hybrid admissible disk $f$ is given as follows:
\[
|f|_\ZZ=1-\#((-)\text{-folding edges})+\#((+)\text{-folding edges})+\#(\text{concave vertices}).
\]
Moreover, 
\[
|f|=\left|\tilde f(\bv_0)\right|_\fP-\sum_{r=1}^t \left|\tilde f(\bv_r)\right|_\fP.
\]
\end{proposition}
\begin{proof}
The proof is essentially the same as the proofs of Propositions~\ref{proposition:degree formula} and \ref{proposition:degree admissible disk}. 
A new effect arises from $(-)$-folding edges which contribute an additional $(-\pi)$-rotation on $[D(\partial f)]$. Note that a $(\pm)$-folding edge has no contribution to $D(\partial f)$, since the effects from the two critical points on the edge cancel out. This completes the proof.
\end{proof}

\begin{remark}\label{remark:role of negative folding edge}
As seen in this formula, we may think that the negative folding edge plays the role of the element of degree $-1$.
\end{remark}

\begin{definition}[Equivalence of hybrid admissible disks]
Two hybrid admissible disks are {\em equivalent} if they are the same up to a one-parameter family preserving the combinatorial structure in the sense of Definition~\ref{definition:equivalence of admissible disks}.
\end{definition}

\begin{notation}\label{not:hybrid disks}
We denote the set of equivalence classes of all hybrid admissible disks of degree at most 1 by $\scrM$ and define the following sets:
\begin{align*}
\scrM_t&\coloneqq\{f\in\scrM\mid \domain(f)=\Pi_t\},&
\scrM_{t,s}&\coloneqq\{f\in\scrM_t\mid \separator\in[s,s+1)\};\\
\scrM_{(i)}&\coloneqq\{f\in\scrM\mid |f|_\ZZ=i\},&
\scrM(\bar\sfg)&\coloneqq\{f\in\scrM\mid \tilde f(\bv_0)=\bar\sfg\}.
\end{align*}

As before, their intersections will be denoted by using several decorations.
\end{notation}

\begin{definition}
A hybrid admissible disk $f\in\scrM_t$ is said to be
\begin{enumerate}
\item {\em leftmost} or {\em rightmost} if $\separator$ is $0$ or $(t+1)$, respectively\footnote{Here we regard $0$ and $(t+1)$ as different points. In other words, $0$ is the left-end and $(t+1)$ is the right-end of the boundary $\partial\Pi_t=[0,t+1]$.}; 
\item {\em left-rigid} if $f((\separator-\epsilon, \separator))\subset \sL_0'$ for some $\epsilon>0$;
\item {\em right-rigid} if $f((\separator, \separator+\epsilon))\subset \sL_0$ for some $\epsilon>0$;
\item {\em bi-rigid} if $f$ is both left- and right-rigid;
\item {\em flexible} if $\separator \neq 0$ and $f((\separator-\epsilon,\separator+\epsilon))\subset \sL\cap \sL'$ for some $\epsilon>0$.
\end{enumerate}
We also say that a left-rigid or right-rigid disk is {\em strict} if it is not bi-rigid.
\end{definition}

One can easily see that being -most, -rigid, and flexible form a trichotomy, and moreover, all edge-separated disks are left- or right-rigid. By using this observation, we can decompose $\scrM$ as follows.
\begin{cornot}[Decomposition of hybrid admissible disks]\label{corollary:classification of hyb disks}
The following holds:
\begin{enumerate}
\item $\scrM_{(0)}=\scrM^{\negativefolding}\coloneqq\{f\in\scrM\mid f\text{\rm\ has one $(-)$-folding edge only}\}$;
\item $\scrM_{(1)}$ can be decomposed into three subsets:
\begin{enumerate}
\item $\scrM^{\pmfolding}\coloneqq\{f\in\scrM\mid f\text{\rm\ has $(-)$- and $(+)$-folding edges or a $(\pm)$-folding edge}\}$;
\item $\scrM^{\negativefoldingconcave}\coloneqq\{f\in\scrM\mid f\text{\rm\ has a $(-)$-folding edge and a concave vertex}\}$;
\item $\scrM^{\emptyset}\coloneqq\{f\in\scrM\mid f\text{\rm\ has no folding edges and no concave vertices}\}$ which can be de\-com\-posed further into the following seven subsets:
\begin{enumerate}
\item $\scrM^{\leftrigid}\coloneqq\{f\in\scrM\mid f\text{\rm\ is edge-separated and strictly left-rigid}\}$;
\item $\scrM^{\rightrigid}\coloneqq\{f\in\scrM\mid f\text{\rm\ is edge-separated and strictly right-rigid}\}$;
\item $\scrM^{\leftmost}\coloneqq\{f\in\scrM\mid f\text{\rm\ is vertex-separated and leftmost}\}$;
\item $\scrM^{\rightmost}\coloneqq\{f\in\scrM\mid f\text{\rm\ is vertex-separated and rightmost}\}$;
\item $\scrM^{\vertexleftrigid}\coloneqq\{f\in\scrM\mid f\text{\rm\ is vertex-separated and strictly left-rigid}\}$;
\item $\scrM^{\vertexbirigid}\coloneqq\{f\in\scrM\mid f\text{\rm\ is vertex-separated and bi-rigid}\}$;
\item $\scrM^{\flexible}\coloneqq\{f\in\scrM\mid f\text{\rm\ is vertex-separated and flexible}\}$.
\end{enumerate} 
\end{enumerate}
\end{enumerate}
\end{cornot}
\begin{proof}
It is easy to see that the following decompositions hold:
\begin{align*}
\scrM&=\scrM_{(0)} \amalg \scrM_{(1)},&
\scrM_{(1)}&=\scrM^{\pmfolding} \amalg \scrM^{\negativefoldingconcave} \amalg \scrM^{\emptyset}.
\end{align*}

It is also obvious that (i)-(vii) are mutually exclusive subsets of $\scrM^{\emptyset}$. In order to verify that their union is $\scrM^{\emptyset}$, it remains to prove that
\begin{enumerate}
\item there are no edge-separated bi-rigid admissible disks in $\scrM^\emptyset$;
\item there are no vertex-separated strictly right-rigid admissible disks in $\scrM^\emptyset$.
\end{enumerate}

The first claim follows since an admissible disk is edge-separated and bi-rigid admissible if and only if it has either a $(-)$ or a $(\pm)$-folding edge and so is not contained in $\scrM^\emptyset$.

On the other hand, vertex-separated right-rigidity requires a special point involving $\bar\sL_0$ and the only such special point is precisely $\bar\sfx$. However, any disk hitting $\bar\sfx$ at the separator is bi-rigid and therefore there are no such special points which completes the proof of the second claim.
\end{proof}

\begin{remark}
Note that $f\in \scrM^{\vertexleftrigid}$ satisfies $f(\separator)=\bar\sfc_i$ for some $i=1,\dots,m$, while $f\in\scrM^{\vertexbirigid}$ satisfies $f(\separator)=\bar\sfx$.
\end{remark}

\begin{table}[ht]
\renewcommand{\arraystretch}{1.5}
\setlength{\tabcolsep}{6pt}
\begin{tabular}{c||c|c||c|c|c||c}
$\scrM$ &\multicolumn{2}{c||}{-most}& \multicolumn{3}{c||}{-rigid} & flexible\\
\cline{2-6}
& Left & Right & Str. left- & bi- & Str. right-\\
\hline\hline
Edge-sep. & \multicolumn{2}{c||}{$\emptyset$} & $\scrM^{\leftrigid}$ & $\scrM^{\negativefolding}\amalg\scrM^{\pmfolding}\amalg\scrM^{\negativefoldingconcave}$ & $\scrM^{\rightrigid}$ & $\emptyset$\\
\hline\hline
Vertex-sep. & $\scrM^{\leftmost}$ & $\scrM^{\rightmost}$ & $\scrM^{\vertexleftrigid}$ & $\scrM^{\vertexbirigid}$ & $\emptyset$ & $\scrM^{\flexible}$
\end{tabular}
\caption{A decomposition of $\scrM$}
\label{table:decomposition of hyb disks}
\end{table}

\begin{figure}[ht]
\subfigure[Local configurations of $\scrM^{\leftrigid}$\label{figure:left-rigid}]{\makebox[0.45\textwidth]{
$
\vcenter{\hbox{\includegraphics[scale=0.8]{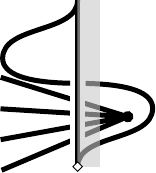}}}\quad
\vcenter{\hbox{\includegraphics[scale=0.8]{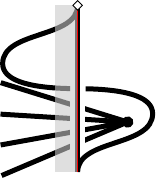}}}
$
}
}
\subfigure[Local configurations of $\scrM^{\rightrigid}$\label{figure:right-rigid}]{\makebox[0.45\textwidth]{
$
\vcenter{\hbox{\includegraphics[scale=0.8]{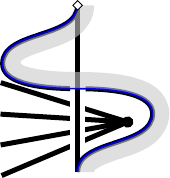}}}\quad
\vcenter{\hbox{\includegraphics[scale=0.8]{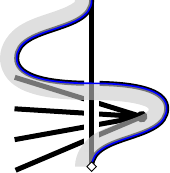}}}
$
}
}
\caption{Strict left-right or strict right-rigid disks}
\label{figure:half-rigid disks}
\end{figure}

\begin{lemma}
For each $\bar\sfg$, the set $\scrM(\bar\sfg)$ is finite.
\end{lemma}
\begin{proof}
As seen above, the image of the separator $f(\separator)$ is either a special point or in $\{\bar\sfp,\bar\sfq\}$.
Therefore up to one-parameter families preserving combinatorial structures, $\scrM(\bar\sfg)$ is discrete.
Finally, finiteness follows from the exactly same argument as for Proposition~\ref{proposition:finiteness}.
\end{proof}

\begin{remark}[Soundness of Definition~\ref{def:filtered graded algebra}]\label{remark:soundness of definition of Psi}
For each $\bar\sfg\in\bar\sG$, the set $\scrM(\bar\sfg)_{(0)}\subset\scrM(\bar\sfg)$ is finite by the previous lemma and can be decomposed into subsets $\scrM_{t,s}(\bar\sfg)_{(0)}$.
Since any $f\in\scrM_{t,s}(\bar\sfg)_{(0)}$ are edge-separated as discussed in Remark~\ref{rmk:separator}, its labels $\tilde f(\bv_i)$ miss $\bar\sfx$ and can be regarded as elements in $\sG'$ and $\sG$ as mentioned earlier.
\end{remark}

\begin{example}\label{ex:disk near x}
Let $\sfg=\bar\sfx$. Then $\scrM\left(\bar\sfx\right)$ consists of two elements $f_{\bar\sfx,\bar \sfp}$ and $f_{\bar\sfx,\bar \sfq}$, which are monogons $(t=0)$ of degree 0 contained in $\bar\sU$ having negative folding edges touching $\bar p$ and $\bar q$, respectively. See Figure~\ref{figure:hybrid disks of bar x}.
\[
\scrM(\bar\sfx)=\scrM_{0,0}(\bar\sfx)_{(0)}=\{f_{\bar\sfx,\bar \sfp},f_{\bar\sfx,\bar \sfq}\}.
\]
\end{example}
\begin{figure}[ht]
\[
\begin{tikzcd}
f_{\bar\sfx,\bar\sfp}=\vcenter{\hbox{\input{f_x_p_input.tex}}}
&\ar[l,"f_{\bar\sfx,\bar\sfp}"']\vcenter{\hbox{\input{monogon_hyb_input.tex}}}\ar[r,"f_{\bar\sfx,\bar\sfq}"] &
\vcenter{\hbox{\input{f_x_q_input.tex}}}=f_{\bar\sfx,\bar\sfq}
\end{tikzcd}
\]
\caption{Two hybrid disks $f_{\bar\sfx,\bar\sfp}$ and $f_{\bar\sfx,\bar\sfp}$ in $\scrM(\bar\sfx)$}
\label{figure:hybrid disks of bar x}
\end{figure}

\begin{lemma}\label{lemma:AB_into_hybrid}
Let $\sfg\in\sG$ and $\sfg'\in\sG'$. Then there are canonical bijections
\begin{align*}
\cM(\sfg)_{(1)}&\simeq \scrM^{\leftmost}(\bar\sfg),&
\cM(\sfg')_{(1)}&\simeq \scrM^{\rightmost}(\bar\sfg).
\end{align*}
\end{lemma}
\begin{proof}
We first regard both $\cM(\sfg)_{(1)}$ and $\cM(\sfg')_{(1)}$ as subsets of $\scrM$ by choosing the separator $\separator$ as follows:
For $f\in\cM_t(\sfg)$ and $f'\in\cM_{t'}(\sfg')$, we choose separators $\separator$ as $0$ and $(t'+1)$, respectively. That is,
\[
\begin{tikzcd}[row sep=1pc]
\cM(\sfg)_{(1)}\ni f=\vcenter{\hbox{
\begingroup%
  \makeatletter%
  \providecommand\color[2][]{%
    \errmessage{(Inkscape) Color is used for the text in Inkscape, but the package 'color.sty' is not loaded}%
    \renewcommand\color[2][]{}%
  }%
  \providecommand\transparent[1]{%
    \errmessage{(Inkscape) Transparency is used (non-zero) for the text in Inkscape, but the package 'transparent.sty' is not loaded}%
    \renewcommand\transparent[1]{}%
  }%
  \providecommand\rotatebox[2]{#2}%
  \ifx\svgwidth\undefined%
    \setlength{\unitlength}{63.35678299bp}%
    \ifx\svgscale\undefined%
      \relax%
    \else%
      \setlength{\unitlength}{\unitlength * \real{\svgscale}}%
    \fi%
  \else%
    \setlength{\unitlength}{\svgwidth}%
  \fi%
  \global\let\svgwidth\undefined%
  \global\let\svgscale\undefined%
  \makeatother%
  \begin{picture}(1,0.35516263)%
    \put(0,0){\includegraphics[width=\unitlength,page=1]{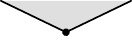}}%
    \put(0.37363852,0.01488963){\color[rgb]{0,0,0}\makebox(0,0)[lb]{\smash{$\bv_0$}}}%
  \end{picture}%
\endgroup%
}} \ar[r,mapsto]&
\vcenter{\hbox{
\begingroup%
  \makeatletter%
  \providecommand\color[2][]{%
    \errmessage{(Inkscape) Color is used for the text in Inkscape, but the package 'color.sty' is not loaded}%
    \renewcommand\color[2][]{}%
  }%
  \providecommand\transparent[1]{%
    \errmessage{(Inkscape) Transparency is used (non-zero) for the text in Inkscape, but the package 'transparent.sty' is not loaded}%
    \renewcommand\transparent[1]{}%
  }%
  \providecommand\rotatebox[2]{#2}%
  \ifx\svgwidth\undefined%
    \setlength{\unitlength}{63.35678299bp}%
    \ifx\svgscale\undefined%
      \relax%
    \else%
      \setlength{\unitlength}{\unitlength * \real{\svgscale}}%
    \fi%
  \else%
    \setlength{\unitlength}{\svgwidth}%
  \fi%
  \global\let\svgwidth\undefined%
  \global\let\svgscale\undefined%
  \makeatother%
  \begin{picture}(1,0.3551626)%
    \put(0,0){\includegraphics[width=\unitlength,page=1]{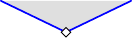}}%
    \put(0.37363851,0.01488963){\color[rgb]{0,0,0}\makebox(0,0)[lb]{\smash{$\bv_0$}}}%
  \end{picture}%
\endgroup%
}}=\bar f\in\scrM^{\leftmost}(\bar\sfg);\\
\cM(\sfg')_{(1)}\ni f'=\vcenter{\hbox{
\begingroup%
  \makeatletter%
  \providecommand\color[2][]{%
    \errmessage{(Inkscape) Color is used for the text in Inkscape, but the package 'color.sty' is not loaded}%
    \renewcommand\color[2][]{}%
  }%
  \providecommand\transparent[1]{%
    \errmessage{(Inkscape) Transparency is used (non-zero) for the text in Inkscape, but the package 'transparent.sty' is not loaded}%
    \renewcommand\transparent[1]{}%
  }%
  \providecommand\rotatebox[2]{#2}%
  \ifx\svgwidth\undefined%
    \setlength{\unitlength}{63.35678299bp}%
    \ifx\svgscale\undefined%
      \relax%
    \else%
      \setlength{\unitlength}{\unitlength * \real{\svgscale}}%
    \fi%
  \else%
    \setlength{\unitlength}{\svgwidth}%
  \fi%
  \global\let\svgwidth\undefined%
  \global\let\svgscale\undefined%
  \makeatother%
  \begin{picture}(1,0.35516263)%
    \put(0,0){\includegraphics[width=\unitlength,page=1]{hyb_pure.pdf}}%
    \put(0.37363852,0.01488963){\color[rgb]{0,0,0}\makebox(0,0)[lb]{\smash{$\bv_0$}}}%
  \end{picture}%
\endgroup%
}} \ar[r,mapsto]&
\vcenter{\hbox{
\begingroup%
  \makeatletter%
  \providecommand\color[2][]{%
    \errmessage{(Inkscape) Color is used for the text in Inkscape, but the package 'color.sty' is not loaded}%
    \renewcommand\color[2][]{}%
  }%
  \providecommand\transparent[1]{%
    \errmessage{(Inkscape) Transparency is used (non-zero) for the text in Inkscape, but the package 'transparent.sty' is not loaded}%
    \renewcommand\transparent[1]{}%
  }%
  \providecommand\rotatebox[2]{#2}%
  \ifx\svgwidth\undefined%
    \setlength{\unitlength}{63.35678299bp}%
    \ifx\svgscale\undefined%
      \relax%
    \else%
      \setlength{\unitlength}{\unitlength * \real{\svgscale}}%
    \fi%
  \else%
    \setlength{\unitlength}{\svgwidth}%
  \fi%
  \global\let\svgwidth\undefined%
  \global\let\svgscale\undefined%
  \makeatother%
  \begin{picture}(1,0.35516263)%
    \put(0,0){\includegraphics[width=\unitlength,page=1]{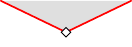}}%
    \put(0.37363852,0.01488963){\color[rgb]{0,0,0}\makebox(0,0)[lb]{\smash{$\bv_0$}}}%
  \end{picture}%
\endgroup%
}}=\bar f'\in\scrM^{\rightmost}(\bar\sfg).
\end{tikzcd}
\]

Conversely, it is obvious that we can regard all hybrid disks in $\scrM^{\leftmost}(\bar\sfg)$ and $\scrM^{\rightmost}(\bar\sfg)$ as admissible disks in $\cM(\sfg)_{(1)}$ and $\cM(\sfg')_{(1)}$, respectively. 
\end{proof}

\begin{remark}\label{remark:leftmost c_i empty}
The set $\scrM^{\leftmost}(\bar\sfc_i)=\emptyset$ for any $i$ since $\bar\sfc_i$ comes only from $\sL'$.
\end{remark}

\begin{remark}
We may reformulate Proposition~\ref{proposition:commutativity}(1) as follows: For $\sfg\in\sG_{[0]}$ and $\sfg'\in\sG'_{[0]}$ with $\sfg'=\iota(\sfg)$, we have bijections
\begin{equation}
\cM(\sfg)_{(1)}\simeq \scrM^{\leftmost}(\bar\sfg)\simeq\scrM^{\rightmost}(\bar\sfg)\simeq \cM'(\sfg')_{(1)}.
\end{equation}
Note that first and third bijections are established already by Lemma~\ref{lemma:AB_into_hybrid}.
The second bijection comes from the forward or backward move of the separator.
Roughly speaking, since there are no obstructions in the sense of Lemma~\ref{lem:adm_obstruction}, one can freely move a separator along the boundary $\partial\Pi$.
\end{remark}

Let us introduce a polynomial $\scrP(f)$ associated to a hybrid admissible disk $f$ as follows.
\begin{definition}\label{definition:polynomial of hybrid disk}
A map $\scrP:\scrM_{(0)}\to A$ is defined as for $f\in\scrM_{t,s}(\bar\sfg)$,
\begin{align*}
\scrP(f)&\coloneqq\sgn(f)\Psi\left(\tilde f(\bv_1\cdots\bv_s)\right) \tilde f(\bv_{s+1}\cdots \bv_t),&
\sgn(f)&\coloneqq\sgn(f,\bv_0\cdots\bv_t),
\end{align*}
where $\Psi=\Psi_{[m+n]}$ in Definition~\ref{def:filtered graded algebra}.
\end{definition}

Then it is straightforward that 
\begin{align}\label{eq:sum of hybrids}
\Psi^\hyb(\bar\sfg)=\sum_{f\in\scrM(\bar\sfg)_{(0)}} \scrP(f),
\end{align}
and the proof of Proposition~\ref{proposition:commutativity}(2) will be given in Appendix~\ref{appendix:hybrid admissible pairs}.

\section{A DGA for Legendrian tangles}\label{section:Legendrian tangle}
In this section, we define a DGA for {\em Legendrian tangles} and consider the operation given by replacing a Darboux neighborhood of a vertex with suitable Legendrian tangle, which yields a van Kampen type theorem for DGAs. To this end, we first consider Legendrian graphs embedded in $S^3$ rather than in $\RR^3$.

Let us assume that $S^3$ has the standard contact structure
\begin{align*}
S^3_{\std}&\coloneqq(S^3,\xi_{\std}\coloneqq \ker \alpha_\std),&
\alpha_\std&\coloneqq x_1dy_1-y_1dx_1+x_2dy_2-y_2dx_2,
\end{align*}
where $S^3$ is the unit sphere in $\CC^2$ with two coordinates $z_1=x_1+iy_1$ and $z_2=x_2+iy_2$. Then for a designated point $\infty\in S^3$, there is a well-known contactomorphism from $S^3_{\std}\setminus\{\infty\}$ to $\RR^3_{\rotation}$ given by
\begin{align}\label{eq:stereographic projection}
\setlength{\arraycolsep}{2pt}
\begin{array}{rccc}
\pi_\infty:&S^3_{\std}\setminus\{\infty\}&\to& \RR^3_{\rotation}\\
&(z_1,z_2)&\mapsto&\left(
\frac{z_1}{1+z_2}, \frac{y_2}{|1+z_2|}\right)
\end{array}
\end{align}
Here, we regard $\infty$ as $(0,-1)\in S^3$ and identify $\RR^3$ with $\CC\times\RR$.

\subsection{Legendrian tangles and their projections}
A Legendrian graph in $S^3_{\std}$ is a spatial graph in $S^3$, whose edges are Legendrian submanifolds in $S^3_{\std}$ which are non-tangent to each other at their ends.
Then a spatial graph $\Lambda\subset S^3_{\std}$ missing $\infty$ is Legendrian if and only if the image $\pi_{\infty}(\Lambda)$ is a Legendrian graph in the sense of Definition~\ref{definition:Legendrian graph}.

\begin{remark}
Indeed, any Legendrian graph in $S^3_{\std}$ can be isotoped to a Legendrian graph which avoids $\infty$, and two Legendrian graphs avoiding $\infty$ are isotopic in $S^3_{\std}$ if and only if their images under $\pi_{\infty}$ are isotopic in $\RR^3_{\rotation}$.

Therefore the sets of isotopy classes of Legendrian graphs in $S^3_{\std}$ and in $\RR^3_{\rotation}$ coincide. In this sense, there are no differences between Legendrian graph theory in $S^3_{\std}$ and in $\RR^3_{\rotation}$.
\end{remark}

On the other hand, if a Legendrian graph $\Lambda$ hits the point $\infty\in S^3$ especially at its vertex of valency $m$, then its projection defines a Legendrian $m$-tangle as follows:
\begin{definition}[Legendrian tangles]\label{definition:Legendrian tangle}
A {\em Legendrian $m$-tangle} $T$ in $\RR^3_{\rotation}$ is an image of a Legendrian graph $\Lambda$ under the stereographic projection $\pi_\infty$ for some $\Lambda\subset S^3_{\std}$ such that $\infty$ is a vertex of $\Lambda$ of valency $m$, called the {\em vertex at infinity},
\begin{align*}
T&\coloneqq \pi_\infty(\Lambda)\subset\RR^3_{\rotation},&
\Lambda&\subset S^3_{\std},&
\infty\in \cV_\Lambda.
\end{align*}

We call $\Lambda$ the {\em closure} of $T$, denote it by $\hat T$, and say that two Legendrian tangles $T_0$ and $T_1$ are {\em equivalent} if the pairs $(\hat T_0,\infty)$ and $(\hat T_1,\infty)$ are isotopic in $S^3_{\std}$.
\end{definition}

By definition, for any Legendrian $m$-tangle $T\subset \RR^3_{\rotation}$, its closure $\hat T\subset S^3_{\std}$ is well-defined.
We denote the sets of vertices and edges of $T$ by $\cV_T$ and $\cE_T$ which are the images of $\cV_{\hat T}$ and $\cE_{\hat T}$ under $\pi_\infty$ after removing $\infty$ from vertices and edges if they hit $\infty$ in ${\hat T}$.
However, the set of half-edges $\cH_T$ is the image of half-edges in $\cH_{\hat T}\setminus \cH_\infty$ which are non-adjacent to $\infty$.
\begin{align*}
\cV_T&\coloneqq\pi_\infty(\cV_{\hat T}\setminus\{\infty\}),&
\cE_T&\coloneqq\{\pi_\infty(e\setminus\{\infty\})\mid e\in \cE_{\hat T}\},&
\cH_T&\coloneqq\{\pi_\infty(h)\mid h\in\cH_{\hat T}\setminus \cH_\infty\}.
\end{align*}
We call half-edges in $\cH_\infty$ {\em half-edges at infinity} or simply {\em ends} of $T$.

Let $\cU_{\origin}\subset\RR^3_{\rotation}$ be the unit ball as before. We define $\scrU_{\origin}$ and $\scrU_\infty$ as the subsets of $S^3_{\std}$ 
\begin{align*}
\scrU_{\origin}&\coloneqq \pi_\infty^{-1}(\cU_{\origin}),&
\scrU_\infty&\coloneqq S^3\setminus \bar\scrU_{\origin}.
\end{align*}

Notice that the two subsets $\scrU_{\origin}$ and $\scrU_\infty$ are contactomorphic via the involution $\tau:(z_1,z_2)\mapsto(z_1,-z_2)$ on $\CC^2$.
We denote the image of $\scrU_\infty$ under the composition $(\pi_\infty\circ\tau)$ by $\cU_\infty\subset\RR_{\rotation}^3$
\[
(\pi_\infty\circ\tau):\scrU_\infty\subset S^3_\std \stackrel{\simeq}\longrightarrow \cU_\infty\subset\RR^3_{\rotation}.
\]
The subset $\cU_\infty$ is the same as $\cU_{\origin}$ but we use different notation to clarify where they are coming.

\begin{example}\label{example:symmetry of T_Theta}
Let $\Theta=\{\theta_1,\cdots,\theta_m\}$ be a subset of $[0,2\pi)$. Then the union $T_\Theta$ of Legendrian rays described in Theorem~\ref{theorem:Darboux} becomes a Legendrian tangle whose closure $\hat T_\Theta$ is invariant under the involution $\tau$. That is,
\[
\begin{tikzcd}
(\scrU_\origin,\scrU_\origin\cap \hat T_\Theta)\ar[r,"\simeq","\tau"'] & (\scrU_\infty,\scrU_\infty\cap\hat T_\Theta).
\end{tikzcd}
\]
\end{example}

Furthermore, for any Legendrian tangle $T$, the closure $\hat T$ is contactomorphic to $T_{\Theta(T)}$ near $\infty$ for some $\Theta(T)=\{\theta_1,\cdots,\theta_m\}$ due to Theorem~\ref{theorem:Darboux} at $\infty$.
Indeed, we assume the following.
\begin{assumption}\label{assumption:tangle}
For any Legendrian tangle $T$, the closure $\hat T$ coincides with $\hat T_{\Theta(T)}$ on $\scrU_\infty\subset S^3_{\std}$ for some $\Theta(T)$
\[
(\scrU_\infty, \scrU_\infty\cap \hat T) =
(\scrU_\infty, \scrU_\infty\cap \hat T_{\Theta(T)}),
\]
and we denote $T_{\Theta(T)}$ by $T_\infty$.
\end{assumption}

Now one can decompose $\hat T$ into two pieces via $\pi_\infty\coprod (\pi_\infty\circ\tau)$
\begin{align}\label{eq:decomposition of tangle closure}
(S^3_{\std},\hat T)&= 
(\bar\scrU_\origin,\bar\scrU_\origin\cap\hat T)\coprod_{\partial}(\bar\scrU_\infty,\bar\scrU_\infty\cap\hat T_\infty)
\simeq (\bar\cU_\origin,\bar\cU_\origin\cap T)\coprod_{\partial} (\bar\cU_\infty,\bar\cU_\infty\cap T_\infty)
\end{align}

Recall the Lagrangian projection $\pi_L:\RR^3\to\RR_{xy}^2$.
Then the composition $\pi_L\circ\pi_\infty$ extends to all of $S^3$ as follows:
\[
\begin{tikzcd}[column sep=3pc]
S^3\setminus\{\infty\}\ar[r,"\pi_L\circ\pi_\infty"]\ar[d,hook]&\RR_{xy}^2\ar[d,hook]\\
S^3\ar[r,"\pi_L"]& S_{xy}^2
\end{tikzcd}
\]
Here we use the same notation $\pi_L$ for the extension.\footnote{The map $\pi_L$ is the restriction of the canonical map $\CC^2\setminus\{\origin\}\to\CC P^1\simeq S^2_{xy}$.}
We denote the images of $T$ and $\hat T$ under $\pi_L$ by $\sT$ and $\hat \sT$, called the {\em Lagrangian projections}, 
\begin{align*}
\sT&\coloneqq\pi_L(T)\subset \RR_{xy}^2,&
\hat\sT&\coloneqq\pi_L(\hat T)\subset S_{xy}^2,&
(\RR_{xy}^2, \sT)&\simeq (S_{xy}^2\setminus\{\infty\},\hat\sT\setminus\{\infty\}),
\end{align*}
and define sets of vertices $\sV_{\bullet}$, edges $\sE_{\bullet}$ and half-edges $\sH_{\bullet}$ of $\sT$ and $\hat \sT$ in a similar manner
\begin{align*}
\sV_{\hat \sT}&=\sV_\sT\cup\{\infty\},&
\sE_{\hat \sT}&\simeq \sE_\sT,&
\sH_{\hat \sT}&=\sH_\sT\cup\sH_{\infty}=\sH_\sT\cup\{\sfh_{\infty,1},\cdots,\sfh_{\infty,m}\}.
\end{align*}

Since $\pi_L=\pi_L\circ\Phi$ as mentioned earlier, by (\ref{eq:decomposition of tangle closure})
\[
(S_{xy}^2,\hat\sT)\simeq(\bar\sU_\origin,\bar\sU_\origin\cap \sT)\coprod_{\partial}(\bar\sU_\infty,\bar\sU_\infty\cap \sT_\infty),
\]
where $\sT_\infty\coloneqq\pi_L(T_\infty)$ and $\sU_*=\pi_L(\cU_*)$ is the unit ball in $\RR_{xy}^2$ for each $*\in\{\origin,\infty\}$.

\begin{figure}[ht]
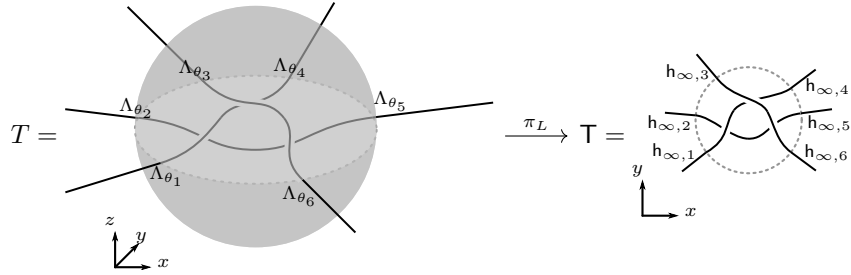

\[
\begin{tikzcd}
T=\vcenter{\hbox{\scriptsize\input{tangle_input.tex}}}\ar[r,"\pi_L"]& 
\sT=\vcenter{\hbox{\scriptsize\input{tangle_lag_input.tex}}}
\end{tikzcd}
\]
\caption{A Legendrian tangle $T$ and its Lagrangian projection $\sT$}
\label{fig:Legendrian tangle}
\end{figure}

\begin{definition}
Let $T$ be a Legendrian tangle. We say that $T$ is {\em in general position} if the conditions in Definition~\ref{defn:regular projection} hold for both $\sT$ and $\sT_\infty$.

We denote the set of all Legendrian tangles in general position by $\cLT$.
\end{definition}

Let $\sT_0$ and $\sT_1$ be the Lagrangian projections of $T_0$ and $T_1$, respectively. Then roughly speaking, $T_0$ and $T_1$  are equivalent only if there exists a sequence of Reidemeister moves as depicted in Figure~\ref{fig:RM} that transforms $\sT_0$ to $\sT_1$ up to planar isotopy on $\overline\sU_\origin$ and rotation in $\RR^2_{xy}$.
It is not hard to check that this definition agrees with the equivalence in Definition~\ref{definition:Legendrian tangle}, and we omit the proof.

\subsection{Capping paths and gradings}
One can equip a Legendrian $m$-tangle $T\in\cLT$ with an $\fR$-valued potential $\fP$ on $T$ which is a function
\[
\fP:\sH_\sT\to\fR
\]
such that for each edge $\sfe\in \sE_\sT\simeq \sE_\sT$, the function $\fP$ satisfies the vanishing condition
\[
\fP(\sfh_{s(\sfe)})-\fP(\sfh_{t(\sfe)})-n_\sfe\cdot 1=0,
\]
where $n_\sfe$ is the number defined in \S~\ref{sec:potential}.

\begin{corollary}
For any Legendrian $m$-tangle with potential $\cT=(T,\fP)$, there is a Legendrian $m$-tangle with the inherited potential $\cT_\infty=(T_\infty,\fP_\infty)$. Moreover, $(\cT_\infty)_\infty=\cT_\infty$.
\end{corollary}
\begin{proof}
Since the number $n_\sfe$ is defined even for edges hitting $\infty$, the potential $\fP$ extends to $\hat\fP$ on $\sH_{\hat\sT}$ naturally by using the same vanishing condition.
Then the restriction of $\hat\fP$ to $\sH_\infty\subset \sH_{\hat\sT}$ defines a potential $\fP_\infty$ for $T_\infty$.

Since $T_\infty=T_\Theta$ is the union of Legendrian rays for some $\Theta=\Theta(T)$ and is invariant under the involution $\tau$, we have $(T_\infty)_\infty=T_\infty$. Finally we have $(\fP_\infty)_\infty=\fP_\infty$ since $n_\sfe=0$ for all edges $\sfe\in \sE_{\sT_\infty}$.
\end{proof}

As before, we denote the set of all Legendrian tangles in general position with $\fR$-valued potentials by $\cLT_\fR$. 
\begin{theorem}
The Legendrian tangle equivalence on $\cLT$ lifts to $\cLT_\fR$.
\end{theorem}
\begin{proof}
This is essentially the same as Theorem~\ref{theorem:equivalence on Legendrian graphs with potentials} and we omit the proof.
\end{proof}

From now on, we fix a Legendrian $m$-tangle with $\fR$-valued potential $\cT=(T,\fP)\in\cLT_\fR$ whose Lagrangian projections are $\sT$ and $\hat\sT$.

We define the sets $\sC_\sT, \sS_\sT, \tilde \sV_\sT$ and $\sG_\sT$ as before
\begin{align*}
\sC_\sT&\coloneqq\{\sfc\in \sT\mid \sfc\text{ is a double point}\},&
\sS_\sT&\coloneqq \sC_\sT\amalg \sV_\sT,\\
\tilde\sV_\sT&\coloneqq\{\sfv_{i,\ell}\mid \sfv\in\sV_\sT, i\in\Zmod{\val(\sfv)}, \ell\ge 1\},&
\sG_\sT&\coloneqq \sC_\sT\amalg \tilde\sV_\sT,
\end{align*}
and we define $(\cdot)_{\hat \sT}\coloneqq(\cdot)_\sT\cup (\cdot)_{\sT_\infty}$. Then $\sC_{\sT_\infty}=\emptyset$ and $\tilde\sV_{\sT_\infty}=\{\infty_{i,\ell}\mid i\in\Zmod{m},\ell\ge 1\}=\sG_{\sT_\infty}$ and therefore
\begin{align*}
\sC_{\hat\sT}&\coloneqq\sC_\sT\cup\sC_{\sT_\infty}=\sC_\sT\\
\tilde\sV_{\hat\sT}&\coloneqq\tilde\sV_\sT\cup\tilde\sV_{\sT_\infty}=\tilde\sV_\sT\cup\{\infty_{i,\ell}\mid i\in\Zmod{m},\ell\ge 1\}\\
\sG_{\hat\sT}&\coloneqq \sG_\sT\cup\sG_{\sT_\infty}=\sC_{\hat\sT}\cup\tilde\sV_{\hat\sT}.
\end{align*}

As seen in Proposition~\ref{proposition:degree well defined}, there exists a {\em $\fR$-valued grading}
\[
|\cdot|_\fP:\sG_\sT\to\fR,
\]
which is uniquely determined by the conditions {\rm(\ref{eq:Gr1})}, {\rm(\ref{eq:Gr2_1})}, {\rm(\ref{eq:Gr2_2})}, {\rm(\ref{eq:Gr2_3})} and {\rm(\ref{eq:Gr3})}.

\subsection{Admissible disks and DGAs}\label{sec:Admissible disks hitting the vertex at infinity}
We also define the admissibility for a differentiable disk between triples
\[
f:(\Pi,\partial\Pi,\vPi)\to(\RR_{xy}^2,\sT,\sS_\sT)
\]
as before. For an admissible disk, the canonical label $\tilde f:\vPi\to\sG_\sT$, the notion of being regular or infinitesimal, the equivalence relation, and degree $|f|\in\fR$ are well-defined as in \S~\ref{section:admissible disks}.
In particular, the same formulas as in Propositions~\ref{proposition:degree formula} and \ref{proposition:degree admissible disk} hold.

We will use similar notation as in Notation~\ref{notation:admissible disks} to denote admissible disks. For each $\sfg\in\sG_\sT$,
\begin{align*}
\cM(\sfg)_{(d)}\coloneqq\left\{f:(\Pi,\partial\Pi,\vPi)\to(\RR_{xy}^2,\sT,\sS_\sT)\mid f\text{ is admissible}, \tilde f(\bv_0)=\sfg, |f|_\ZZ=d\right\}.
\end{align*}

\begin{proposition}
For each Legendrian tangle with $\fR$-valued potential $\cT=(T,\fR)\in\cLT_\fR$, there is a free, unital, associative and $\fR$-graded DGA $\cA_\cT=(A_T, |\cdot|_\fP, \partial)$, where $A_T=\ZZ\langle\sG_\sT \rangle$ and the differential $\partial$ is given as follows: for each $\sfg\in\sG_\sT$,
\[
\partial \sfg\coloneqq 
\sum_{f\in\cM(\sfg)_{(1)}} \cP(f),
\]
where $\cP(f)=\sgn(f)\tilde f(\bv_1)\cdots \tilde f(\bv_t)\in A_T$ if $f$ is defined on $\Pi_t$.
\end{proposition}
\begin{proof}
Exactly the same proofs as Proposition~\ref{proposition:degree well defined} and Theorem~\ref{thm:differential} hold and therefore the DGA $\cA_\cT$ is well-defined.
\end{proof}

On the other hand, we also consider an admissible disk hitting the vertex at infinity
\begin{align*}
f:(\Pi,\partial\Pi,\vPi)&\to(S_{xy}^2,\hat\sT, \sS_{\hat\sT}),
\end{align*}
which satisfies the following conditions:
\begin{enumerate}
\item the disk $f$ is smooth and orientation-preserving on $\rPi$ and extends to $\partial\Pi$;
\item all edges are smooth or 2-folding;
\item the disk $f$ hits $\infty$ at $\bx\in\Pi$ if and only if $\bx=\bv_0$;
\item all vertices $\bv_i$ with $i\ge 1$ are either convex negative, concave negative, or neutral.
\end{enumerate}

Notice that these conditions are similar to the conditions for a regular admissible disk in Definition~\ref{def:Regular admissible disks}. However, this can be also thought of as the conditions for an infinitesimal admissible disk as follows:
let $p:S^2_{xy}\simeq \bar\sU_0\cup\bar\sU_\infty\to \bar\sU_\infty\subset\RR^2_{xy}$ be the map identifying $\bar\sU_0$ and $\bar\sU_\infty$. Then we can see the alternative conditions for $p\circ f$ below.
\begin{enumerate}
\item [($1'$)] the image of $p\circ f$ is contained in $\bar\sU_\infty$;
\item [($2'$)] the composition $p\circ f$ is orientation-reversing only on $\bU_{\bv_0}\coloneqq f^{-1}(\sU_\infty)$;
\item [($3'$)] the critical graph $\bO_{p\circ f}$ is the union of $f^{-1}(\sO_\infty)$ and points in the middle of 2-folding edges of $f$.
\end{enumerate}

In this viewpoint, admissibility for $f$ is almost the same as infinitesimal admissibility for $p\circ f$.
The only condition that might be violated is the connectedness of the critical graph $\bO_f$, which is condition (3) in Definition~\ref{definition:infinitesimal}. Hence $f$ has both characteristics of regular and infinitesimal admissible disks simultaneously.

\begin{figure}[ht]
\[
\includegraphics{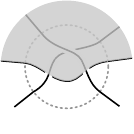}\qquad\qquad
\includegraphics{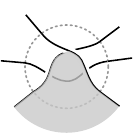}
\]
\caption{Admissible disks hitting the vertex at infinity}
\label{figure:admissible disk at infinity}
\end{figure}

Then as before, the canonical label $\tilde f:\vPi\to\sG_{\hat\sT}$ is well-defined so that $\tilde f(\bv_0)=\infty_{i,\ell}$ for some $i\in\Zmod{m}$ and $\ell\ge 1$. We can define equivalence between two such disks in a similar manner as in Definition~\ref{definition:equivalence of admissible disks}, and the degree $|f|_\ZZ$ as $|p\circ f|_\ZZ$ in the sense of Definition~\ref{definition:degree of admissible disk}. That is,
\[
[D\partial(p\circ f)] = |f|_\ZZ\cdot 1_{\gr} \in \pi_1(\gr(1,\RR_{xy}^2)).
\]
Then the analog of Proposition~\ref{proposition:degree formula} holds as well. However, we have the following proposition corresponding to Proposition~\ref{proposition:degree admissible disk}.
\begin{proposition}\label{prop:degree hitting the infinity}
The $\ZZ$-degree $|f|_\ZZ$ of an admissible disk $f$ onto $S^2_{xy}$ can be computed as follows:
\begin{align*}
|f|_\ZZ=\#(2\text{-folding edges})+\#(\text{concave vertices}).
\end{align*}

In particular, $f$ is of degree 0 if and only if there are neither 2-folding edges nor concave vertices.
\end{proposition}
\begin{proof}
The proof is essentially the same as the proof of Proposition~\ref{proposition:degree admissible disk}(1) except for the following:
\begin{enumerate}
\item the capping path $\tilde f(\bv_0)$ has no contributions to $|f|_\ZZ$ since it is neutral,
\item the composition $p\circ f$ has two $1/2$-folding edges $\be_0$ and $\be_t$, which contribute $-2$ to $|f|_\ZZ$.
\end{enumerate}
Therefore
\[
|f|_\ZZ=2+\#(2\text{-folding edges})+\#(\text{concave vertices})-2,
\]
where the first $2$ corresponds to the case when a differentiable disk is smooth and orientation-preserving, and the last $-2$ comes from $\be_0$ and $\be_t$ as above.
\end{proof}

For any $\infty_{i,\ell}\in \sG_{\sT_\infty}$, we denote the set of admissible disks by $\cM(\infty_{i,\ell})_{(d)}$
\[
\cM(\infty_{i,\ell})_{(d)}\coloneqq\left\{
f:(\Pi,\partial\Pi,\vPi)\to(S_{xy}^2,\hat\sT, \sS_{\hat\sT})\mid f\text{ is admisible}, \tilde f(\bv_0)=\infty_{i,\ell}, |f|_\ZZ=d
\right\}.
\]

\subsection{Canonical peripheral structures for Legendrian tangles}
Similar to the Legendrian graph cases, one can consider the collection of canonical peripheral structures $\cP_\cT$ which consists of DGA morphisms
\[
\cP_\cT\coloneqq\{\bp_\emptyset,\bp_\infty\}\cup \{\bp_v:\cI_v\to \cA_\cT\mid v\in \cV_\cT\},
\]
where
\[
\bp_\infty:\cI_\infty\to \cA_\cT
\]
is called the {\em peripheral structure at infinity} and defined as follows:
\begin{enumerate}
\item $\cI_\infty\coloneqq \cA_{\cT_\infty}$ for $\cT_\infty=(T_\infty,\fP_\infty)$,
\item $\bp_\infty(\infty_{i,\ell})$ is defined as
\[
\bp_\infty(\infty_{i,\ell})\coloneqq
\sum_{f\in\cM(\infty_{i,\ell})_{(0)}} \cP(f).
\]
\end{enumerate}

\begin{lemma}
The map $\bp_\infty$ is a DGA morphism.
\end{lemma}
\begin{proof}
This can be shown by introducing suitable admissible pairs and their manipulations as described in Appendix~\ref{section:manipulation of disks}. The only possible manipulations occurring only in Legendrian tangle cases are either 
\begin{enumerate}
\item to glue two admissible disks of degree 0 to obtain an admissible disk of degree 1 having only one 2-folding edge, or
\item to cut a disk of degree 1 having only one 2-folding edge in which the image of the critical point lies on a half-edge at infinity.
\end{enumerate}
These two operations are inverses to each other, and one can build the moduli graph $\bG_{\infty_{i,\ell}}$, which is bivalent at every internal vertex and has the same label at every edge. Then the handshaking lemma completes the proof as before.
\end{proof}

\begin{theorem}
Let $\cT=(T,\fP)\in\cLT_\fR$ be a Legendrian tangle with potential. Then there is a pair $(\cA_\cT, \cP_\cT)$ consisting of a DGA $\cA_\cT\coloneqq(A_T,|\cdot|_\fP,\partial)$ and a canonical peripheral structure $\cP_\cT$, which is invariant under Reidemeister moves and tail rotations up to generalized stable-tame isomorphisms.
In particular the induced homology $H_*(\cA_\cT,\partial)$ is an invariant.
\end{theorem}
\begin{proof}
The well-definedness of $(\cA_\cT,\cP_\cT)$ has been shown, and the invariance under Reidemeister moves up to generalized stable-tame isomorphisms is the same as in Theorem~\ref{theorem:invariance}.
Since tail rotations do not change any data used to define the pair $(\cA_\cT,\cP_\cT)$, we are done.
\end{proof}

\subsection{Tangle replacements and van Kampen theorem for DGAs}
Let $\cL=(\Lambda,\fP_\Lambda)$ be a Legendrian graph (or tangle) with potential having a vertex $v$ of $val(v)=m$, 
and $\cT=(T,\fP_T)$ be a Legendrian $m$-tangle with potential.
Then by Theorem~\ref{theorem:Darboux}, one can identify $(\bar\cU_\infty, \bar\cU_\infty\cap T_\infty)$ with $(\bar\cU_v, \bar\cU_v\cap\Lambda)$ via $\bar\phi_v$.
Then it induces a diffeomorphism 
\[
\partial\bar\phi_v:(\partial\bar\cU_\infty,\partial\bar\cU_\infty\cap T_\infty) \to (\partial\bar\cU_v,\partial\bar\cU_v\cap \Lambda)
\]
on pairs preserving the orientation of spheres, the cyclic orders at the two vertices, and the potential on corresponding Legendrian arcs.

\begin{definition}\label{def:tangle_replacement}
Let $((\cL=(\Lambda,\fP),v),\cT=(T,\fP_T))$ be a pair of a Legendrian graph $\cL$ with $m$-valent vertex $v$ and a Legendrian $m$-tangle $\cT$. A {\em tangle replacement} of $((\cL,v),\cT)$ is a Legendrian graph with potential
\[
\cL\tangleReplacement\cT\coloneqq(\Lambda\tangleReplacement T, \fP\tangleReplacement \fP_T)
\]
defined to be
\begin{align*}
\Lambda\tangleReplacement T&\coloneqq(\RR^3\setminus \cU_v,\Lambda\cap(\RR^3\setminus \cU_v))\coprod_{\partial\bar\phi_v} (S^3\setminus \cU_\infty,T_\infty\cap(S^3\setminus \cU_\infty))\\
\fP\tangleReplacement\fP_T&\coloneqq \begin{cases}
\fP(\sfh) \in \fR' & \sfh\in\sH_\sL;\\
\fP_T(\sfh) \in\fR' & \sfh\in\sH_\sT,
\end{cases}
\end{align*}
where 
\[
\fR'=\fR/\langle \fP(\sfh_{\sfv,i})-\fP_T(\sfh_{\infty,i})\mid i\in\Zmod{\val(v)} \rangle.
\]
\end{definition}

\begin{lemma}\label{lem:tangle_Lagrangian_projection}
The tangle replacement induces a replacement in the Lagrangian projection.
\end{lemma}

\begin{proof}
Let $\sL$ and $\sT$ be the Lagrangian projections of $\Lambda$ and $T$, respectively.
Here $T$ is the corresponding Legendrian $m$-tangle in $\RR^3_{\rotation}$.
Then it is straightforward to check that $\partial\bar\phi_v$ induces a cyclic order preserving diffeomorphism $\partial\bar\phi_\sfv:(\sS_\infty,\sT \cap \sS_\infty) \to (\sS_\sfv,\sL\cap \sS_\sfv)$ on pairs of circles with points, where $\sS_\infty$ and $\sS_\sfv$ are the $xy$-projections of the equator of $\partial\cU_\infty$ and $\partial\cU_v$, respectively.
Then the resulting Lagrangian projection of the tangle replacement is
\[
(\RR^2_{xy}\setminus \sU_\sfv,\sL\cap (\RR^2_{xy}\setminus \sU_\sfv))\coprod_{\partial\bar\phi_\sfv}(\RR^2_{xy}\setminus \sU_\infty,\sT\cap (\RR^2_{xy}\setminus \sU_\infty)).\qedhere
\]
\end{proof}

Let us denote the Lagrangian projection of the tangle replacement by $\sL\tangleReplacementLag\sT$. We denote DGAs for Legendrian graph $\Lambda$ and tangle $\cT$ by
\begin{align*}
\cA_\Lambda&=(A_\sL=\ZZ\langle\sG_\sL\rangle, |\cdot|_\sL, \partial_\sL),&
\cA_\cT&=(A_T=\ZZ\langle\sG_\sT\rangle, |\cdot|_\sT, \partial_\sT),
\end{align*}
respectively.
Then the DGA $\cA_{\cL\tangleReplacement\cT}=(A_{\sL\tangleReplacementLag\sT},|\cdot|_{\sL\tangleReplacementLag\sT},\partial_{\sL\tangleReplacementLag\sT})$ for $\cL\tangleReplacement\cT$ is generated by
\begin{align*}
\sG_{\sL\tangleReplacementLag\sT}&=(\sG_\sL\setminus\{\sfv_{i,\ell}\})\amalg \sG_\sT,&
A_{\sL\tangleReplacementLag\sT}&=\ZZ\langle\sG_{\sL\tangleReplacementLag\sT}\rangle,
\end{align*}
and the gradings are inherited from those for $\sL$ and $\sT$
\[
|\sfg|_{\sL\tangleReplacementLag\sT}\coloneqq\begin{cases}
|\sfg|_\sL & \sfg\in\sG_\sL;\\
|\sfg|_\sT & \sfg\in\sG_\sT.
\end{cases}
\]

For differentials, we need to argue that the admissible disks in $\sL$, $\sT$, and $\sL\#_\sfv\sT$ are compatible with tangle replacement. We first define a {\em replacement} for a certain pair of admissible disks $f^1$ and $f^0$ in $\sL$ and $\sT$, respectively.

\begin{definition}\label{def:disk_gluing_tangle}
Let $f^0$ and $f^1$ be admissible disks such that
\begin{align*}
f^0&:(\Pi_u,\partial\Pi_u,\vPi_u)\to(S_{xy}^2,\hat\sT, \sS_{\hat\sT}), & \tilde f^0(\bv_0)&=\infty_{i,\ell},\\
f^1&:(\Pi_t,\partial\Pi_t,\vPi_t)\to(\RR^2,\sL,\sS_\sL),& \tilde f^1(\bv_{k})&=\sfv_{i,\ell}.
\end{align*}
Then a {\em replacement} $f^1\tangleReplacementLag f^0$ of the pair $(f^1,f^0)$ is defined by restricting the tangle replacement process to the pair $(f^1,f^0)$.
More preciesly,
\[
f^1\tangleReplacementLag f^0 : (\Pi_{t+u-1},\partial \Pi_{t+u-1},\vPi_{t+u-1}) \to (\RR^2, \sL\tangleReplacementLag\sT,\sS_{\sL\tangleReplacementLag\sT})
\]
is obtained as
\[
(f^1|_{\RR^2\setminus \sU_\sfv}) \coprod_{\partial\bar\phi_\sfv} (f^0|_{S^2_{xy}\setminus \sU_\infty}).
\]
\end{definition}

\begin{lemma}\label{lem:disk_gluing_tangle}
Let $f:(\Pi,\partial\Pi,\vPi)\to(\RR^2,\sL,\sS_\sL)$ be a regular admissible disk with 
\begin{align*}
f^{-1}(\sfv)\cap \vPi&=\{\bv_{k_1},\dots,\bv_{k_n}\},&
\tilde f(\bv_{k_1})&=\sfv_{i_1,\ell_1},\dots,\tilde f(\bv_{k_n})=\sfv_{i_n,\ell_n}.
\end{align*}
For $j=1,\dots,n$, let $g_j:(\Pi_{t_j},\partial\Pi_{t_j},\vPi_{t_j})\to(S_{xy}^2,\hat\sT, \sS_{\hat\sT})$ be admissible disks that satisfy
\[\tilde g_j(\bw_0)=\infty_{i_j,\ell_j}.\]
Then the replacement, $f\tangleReplacementLag(g_1,\dots,g_n)$, 
for $f$ and $(g_1,\dots,g_n)$ along $(\bv_{k_1},\dots,\bv_{k_n})$ is a regular admissible disk in the tangle replacement.
Moreover,
\[
|f\tangleReplacementLag(g_1,\dots,g_n)|_\ZZ=|f|_\ZZ+\sum_{j=1}^n|g_j|_\ZZ.
\]
\end{lemma}

\begin{proof}
It is straightforward from Definition~\ref{def:Regular admissible disks}, the definition of admissible disks for tangles in \S~\ref{sec:Admissible disks hitting the vertex at infinity}, and Definition~\ref{def:disk_gluing_tangle} that the gluing $f\tangleReplacementLag(g_1,\dots,g_n)$ is an admissible disk for $\sL\tangleReplacementLag \sT$.
Since the gluing process preserves the numbers of 2-folding edges and concave vertices,
Propositions~\ref{proposition:degree admissible disk} and \ref{prop:degree hitting the infinity} imply the above degree equality.
\end{proof}

Conversely, when we have a regular admissible disk in the tangle replacement $\sL\tangleReplacementLag\sT$ we decompose the disks in $\sL$ and $\sT$ as follows.
We shall omit the proof, since it is the reverse procedure to that of Lemma~\ref{lem:disk_gluing_tangle}.

\begin{lemma}\label{lem:tangel_decomposition}
Let $f:(\Pi,\partial\Pi,\vPi)\to(\RR^2,\sL\tangleReplacementLag\sT,\sS_{\sL\tangleReplacementLag\sT})$ be a regular admissible disk. The map $f$ can be decomposed along the properly embedded arcs $f^{-1}(\sS_\sfv)$, where $\sS_\sfv$ is the circle in $\RR^2_{xy}$ which is used in Lemma~\ref{lem:tangle_Lagrangian_projection}.
Then a connected component of the decomposition canonically induces a regular admissible disk for $\sL$ and an admissible disk for $\sT$ hitting the vertex at infinity when the component has nontrivial image in $\RR^2\setminus \sU_\sfv$ and $\sU_\sfv$, respectively.
\end{lemma}

Now we are ready to state our van Kampen type theorem for DGAs:
\begin{theorem}\label{thm:van Kampen}
Let $((\cL,v),\cT)$ be a pair as before and $\cL\tangleReplacement\cT$ be its tangle replacement. Then we have the following commutative diagram of DGAs:
\[
\begin{tikzcd}[column sep=3pc]
\cI_m=(I_m,|\cdot|_m,\partial_m)\ar[r,"\bp_\infty"]\ar[d,"\bp_v"',hook]& \cA_\cT=(A_\sT,|\cdot|_{\fP_\sT},\partial_\sT) \ar[d,hook]\\
\cA_\cL=(A_\sL,|\cdot|_{\fP_\sL},\partial_\sL)\ar[r,"\bw_v"]& \cA_{\cL\tangleReplacement\cT}=(A_{\sL\tangleReplacementLag\sT},|\cdot|_{\sL\tangleReplacementLag\sT},\partial_{\sL\tangleReplacementLag\sT})
\end{tikzcd}
\]
Here ${\bw}_v$ is defined by evaluating $\sfv_{i,\ell}$ with $\bp_\infty(\sfv_{i,\ell})$.
\end{theorem}

\begin{proof}
Let us first show that $\bw_v$ is a DGA map. By its construction, it is a degree preserving algebra map.
Since $\bp_\infty$ is a DGA map, we have
\[
\partial_{\sL\tangleReplacementLag\sT} (\bw_v(\sfv_{i,\ell}))=\partial_\sT(\bp_{\infty}(\sfv_{i,\ell}))=\bp_\infty(\partial_m(\sfv_{i,\ell}))=\bw_v(\partial_\sL(\sfv_{i,\ell})).
\]

It remains to show $\partial_{\sL\tangleReplacementLag\sT} (\bw_v(\sfg))=\bw_v(\partial_\sL(\sfg))$ for $a\in \sS(\sL)\setminus\{\sfv_{i,\ell}\}$. Now consider
\[
\bw_v(\partial_\sL(\sfg))=\bw_v\left( \sum_{f\in \cM(\sfg;\sL)_{(1)}} \cP(f) \right)=\sum_{f\in \cM(\sfg;\sL)_{(1)}} \bw_v(\cP(f)).
\]
Then each monomial of the above expression corresponds to an admissible disk $f$ for $\sL$ and $n$-pair of admissible disks $(g_1,\dots, g_n)$ for $\hat\sT$ satisfying
\begin{align*}
f^{-1}(\sfv)\cap \vPi&=\{\bv_{k_1},\dots,\bv_{k_n}\},&
\tilde f(\bv_{k_j})&=\tilde g_j(\bw_0)\quad\text{ for }j=1,\dots,n.
\end{align*}
Then by the construction of Lemma~\ref{lem:disk_gluing_tangle} we recover the chosen monomial
\[
\cP(f\tangleReplacementLag(g_1,\dots,g_n)).
\]
Moreover, $f\tangleReplacementLag(g_1,\dots,g_n)\in \cM(\sfg;\sL\tangleReplacementLag\sT)_{(1)}$.

On the other hand, let us consider 
\[
\partial_{\sL\tangleReplacementLag\sT}(\bw_v(\sfg))=\partial_{\sL\tangleReplacementLag\sT}(\sfg)=\sum_{h\in \cM(\sfg;\sL\tangleReplacementLag\sT)_{(1)}} \cP(h).
\]
Lemma~\ref{lem:tangel_decomposition} implies that every admissible disk $h$ in $\cM(\sfg;\sL\tangleReplacementLag\sT)_{(1)}$ can be decomposed into disks in $\sL$ and $\sT$. This proves $\partial_{\sL\tangleReplacementLag\sT}(\bw_v(\sfg))=\bw_v(\partial_\sL(\sfg))$ and commutativity follows directly from the construction of the diagram.
\end{proof}

\section{Applications}\label{section:Applications}
In this section, we will recover several known DGA invariants for Legendrian links in $\RR^3$ due to Chekanov--Eliashberg \cite{Chekanov2002,Eliashberg2000} and Etnyre--Ng--Sabloff \cite{ENS2002}, for bordered Legendrians due to Sivek \cite{Sivek2011}, and furthermore for Legendrian links in $\#^m (S^2\times S^1)$ due to Ekholm--Ng \cite{EN2015}.

Moreover, we will consider the relationship between our DGA for Legendrian links and partially wrapped Floer homology for certain Weinstein domains.

\subsection{A trivial \texorpdfstring{$(1,1)$}{(1,1)}-tangle and Chekanov--Eliashberg's DGA for Legendrian links}
The simplest example is the trivial $(1,1)$-tangle as shown below
\[
I=\vcenter{\hbox{\includegraphics{RM_0_0_1.pdf}}}
\]
whose DGA is obviously the same as the base ring $\ZZ$ generated by $1$ of degree 0 with zero differential
\[
\cA_I=(\ZZ, 0, 0),
\]
and the peripheral structure at infinity $\bp_\infty:\cI_2\to\cA_I$ is defined as 
\[
\bp_\infty(\infty_{i,\ell})\coloneqq \begin{cases}
1 & \ell=1;\\
0 & \ell>1.
\end{cases}
\]

Let $\cL=(\Lambda,\fP_\Lambda)\in \cL\cG_\fR$ be a Legendrian graph with $\fR$-valued potential. Suppose that $\cL$ has a bivalent vertex $v$. Then the tangle replacement of $\cL$ with $I$ has an effect which removes the designated bivalent vertex and makes two half-edges into a smooth arc. We call this operation a {\em smoothing} at $v$.

\begin{remark}
If we start with a Legendrian graph homeomorphic to a disjoint union of circles, then we eventually get a Legendrian link by replacing all vertices with the trivial $(1,1)$-tangle $I$.
\end{remark}

Suppose that two half-edges of $v$ are {\em opposite} in the sense that they are lying on the opposite side with respect to $y$-axis. We further assume that they have the same potentials. Then the degrees of $\sfv_{1,1}$ and $\sfv_{2,1}$ in $\cA_\cL$ are zero and we have the following push-out diagram
\[
\begin{tikzcd}
\cI_2\ar[r,"\bp_\infty"]\ar[d,"\bp_v"]& I \ar[d]\\
\cA_\cL\ar[r,"\bw_v"]& \cA_{\cL}^{\sm}(v),
\end{tikzcd}
\]
where $\cA_{\cL}^{\sm}(v)\coloneqq \cA_{\cL\amalg_{\Phi_v} I}$ and the induced map $\bw_v$ is a quotient map sending $\sfv_{i,\ell}$ to $1$ if $\ell=1$ or $0$ otherwise.
Here, the superscript `$\sm$' stands for ``smoothing''.

\begin{remark}\label{rmk:potential_smoothing}
For simplicity, suppose that the underlying graph is a circle consisting of one vertex $v$ and one edge.
Then by smoothing at $v$, we obtain a Legendrian knot $\cK$. Initially, it has only one vanishing condition for Maslov potential $\fP$ so that 
\[
|t_v|_{\fP_T}=\fP(\sfh_{\sfv,1}-\sfh_{\sfv,2})=\pm2\rot(\cK)\cdot 1_\fR\in\fR,
\]
where $\rot(\cK)\in\ZZ$ is the rotation number of $\cK$. However, the smoothing condition for the potential implies that 
\[
2\rot(\cK)\cdot 1_\fR = 0\in\fR.
\]

Therefore whenever $\fR=\Zmod{2r'}$ for some $r'$ dividing $\rot(\cK)$, we have a Maslov potential $\fP$. In particular, if $\fR=\Zmod{2\rot(\cK)}$, then this construction perfectly matches with the usual construction of Maslov potentials on Legendrian knots. 
\end{remark}

\begin{theorem}\label{theorem:smoothing}
Let $\cK=(\Lambda,\fP)$ be a Legendrian circle with $(\Zmod{2\rot(\cK)})$-valued potential consisting of one bivalent vertex $v$ and one edge.
Suppose that two half-edges are opposite and have the same potential. Then there is a DGA isomorphism
\[
\cA_\cK^{\sm}(v)\otimes_\ZZ(\Zmod{2})\to\cA^{\CE}_\cK,
\]
where $\cA^{\CE}_\cK$ is the Chekanov--Eliashberg DGA over $\Zmod{2}$ for the Legendrian knot obtained from $\cK$.
\end{theorem}
\begin{proof}
It is obvious that the algebra $\cA_\cK(v)$ is generated only by the crossing~$\sC_{\sL}$.
The grading $|\cdot|_{\fP}$ is the usual grading for the Legendrian knot $\Lambda$ as stated in Remark~\ref{rmk:potential_smoothing}.
The regular admissible disks $\cM_{(1)}$ of degree one in Definition~\ref{def:Regular admissible disks} with Proposition~\ref{proposition:degree admissible disk}(1) recover the {\em admissible immersions} in \cite[\S 3.2]{Chekanov2002}, and hence the differential $\partial$ coincides with the differential of the Chekanov--Eliashberg algebra for Legendrian links.
\end{proof}

\subsection{Bivalent Legendrian graphs and DGAs over \texorpdfstring{$\ZZ[\sft_1^{\pm1},\cdots,\sft_m^{\pm1}]$}{the multivariable Laurent polynomial rings}}
\label{sec:Based Legendrian}
In this section, we recover DGAs for Legendrian $m$-components links defined over the multivariable Laurent polynomial ring $\ZZ[\sft_1^{\pm1},\dots,\sft_m^{\pm1}]$ due to Etnyre--Ng--Sabloff \cite{ENS2002} as follows.

Throughout this section, our Legendrian graph $\Lambda$ is the union of $m$-circles, and each circle $\Lambda_i$ has a bivalent vertex $v^i$.
We consider the group ring $\ZZ[H_1(\Lambda;\RR)]$ which is isomorphic to the Laurent polynomial ring
\[
\ZZ[H_1(\Lambda;\RR)] \simeq \ZZ[H_1(\Lambda,\Lambda\setminus\cV_\Lambda)] \simeq \ZZ[\sft_1^{\pm1},\cdots, \sft_m^{\pm1}].
\]

Now we define a DGA $\cA_\cL^{\sm}=(A_\Lambda^{\sm},|\cdot|_\fP,\partial)$ over $\ZZ[\sft_1^{\pm1},\cdots,\sft_m^{\pm1}]$ as follows: First, take the quotient $A_\Lambda\to \bar{A}_\Lambda^{\sm}$ defined by $\sfc\mapsto\sfc$ and 
\begin{align*}
\sfv_{k,\ell}^i&\mapsto \begin{cases}
\sft_i &  k=1, \ell=1;\\
\sft_i^{-1} & k=2, \ell=1;\\
0 & \ell>1
\end{cases}
\end{align*}
for $v^i\in\cV_\Lambda$. 
The grading and differential for $\bar{\cA}_\cL^{\sm}$ are inherited from $\cA_\cL$ and so $\partial \sft_i=0$.

Next, take another quotient to obtain $A_\Lambda^{\sm}$ so that all $\sft_i$'s are central elements. Namely, $A_\Lambda^{\sm}$ is a quotient of $\bar{A}_\Lambda^{\sm}$ by the subalgebra generated by elements of the form $\sft_i\sfm - \sfm\sft_i$
\begin{align*}
A_\Lambda^{\sm} &\coloneqq \bar{A}_\Lambda^{\sm} / I, &
I&\coloneqq \ZZ\langle
\sft_i \sfc - \sfc \sft_i \mid \sfc\in \bar{A}_\Lambda^{\sm}, 1\le i\le m
\rangle.
\end{align*}
Then one can naturally regard $\cA_\cL^{\sm}$ as a DGA over $\ZZ[\sft_1^{\pm1},\cdots,\sft_m^{\pm 1}]$.

In summary, we have the following theorem.
\begin{theorem}\label{theorem:smoothing 2}
Let $\cL=(\Lambda,\fP)$ be a Legendrian graph with $\ZZ$-valued potential whose underlying graph is a disjoint union of circles. Suppose that each component has only one bivalent vertex whose two half-edges are opposite and have the same potential.
Then there is a DGA isomorphism
\[
\cA_\cL^{\sm}\to\cA^{\Ng}_\cL,
\]
where $\cA^{\Ng}_\cL$ is the $\ZZ$-graded Chekanov-Eliashberg DGA over $\ZZ[\sft_1^{\pm1},\cdots, \sft_m^{\pm1}]$ for the Legendrian $m$-component link $\cL$ generalized by Etnyre, Ng and Sabloff.
\end{theorem}

\subsection{Bordered Legendrians and Sivek's DGAs}
The van Kampen type theorem for Legendrian DGAs is already discussed in \cite{Sivek2011}. Let us briefly recall his construction. 

Let $F\subset \RR^2_{xz}$ be a ({\em simple}) front diagram\footnote{A Legendrian front is {\em simple} if all of its right cusps have the same $x$-coordinate.} of a Legendrian knot $\Lambda$ in $(\RR^3,\xi_0)$. Now consider a vertical line $V=\{x=a\}$ of $\RR^2_{xz}$ avoiding cusps and crossings of $F$. Suppose that $V\cap F$ consists of $m$ points. Let $F^L$ and $F^R$ be the left and right halves of $F$ with respect to $V$.

The type $D$ algebra $D(F^R)$ for $F^R$ is generated by crossings, right cusps, and $\{\rho_{ij}\}_{1\leq i < j \leq m}$, 
while the type $A$ algebra $A(F^L)$ for $F^L$ is generated only by crossings.
The grading for the generators is inherited from the Maslov potential $\mu$ on $F$. 
The differential on $D(F^R)$ and $A(F^L)$ are given by counting ({\em half}) admissible disks of $F^R$ and $F^L$, respectively, see \cite[Definitions 2.5 and 2.9]{Sivek2011}.
Moreover, there is a DGA map \[w:J_m\coloneqq\langle \rho_{ij} \rangle\to A(F^L)\] defined by counting half admissible disks of $F^L$, see \cite[Definition 2.6]{Sivek2011} which makes the following diagram commute: 
\[
\begin{tikzcd}
J_m\ar[r,"w"]\ar[d,hook]& A(F^L) \ar[d,hook]\\
D(F^R)\ar[r,"w'"]& A(F)
\end{tikzcd}
\]
Now we want to interpret the DGAs $J_m, D(F^R), A(F^L),$ and $A(F)$ as DG-subalgebras of $\cI_m, \cA_\cL, \cA_\cT$ and $\cA_{\cL\amalg_{\Phi_v}\cT}$, respectively, for a certain choice of $\cL$ and $\cT$.

Let us construct a Legendrian graph $\cL$ and a Legendrian tangle $\cT$ from $F^R$ and $F^L$, respectively.
By Ng's resolution in \cite{Ng2003} we have corresponding Lagrangian projections $\sL^R$, and $\sL^L$ in $\RR^2_{xy}$ with the induced vertical line $\{x=a\}\subset \RR^2_{xy}$.
We label $\sL^R\cap \{x=a\}$ by $1,2,\dots, m$ from the top to the bottom with respect to the $y$-coordinate.

The Legendrian graph $\cL$ is given by shrinking $\{x\leq a\}$ to a point. The graph $\cL$ has a vertex $v$ of valency $m$ and note that
the infinite region of the vertical line $\{x=a\}\subset \RR^2_{xy}$ corresponds to the sector between $h_{v,m}$ and $h_{v,1}$.
Note that the Maslov potential $\mu$ determines the grading for the crossings and vertex generators.
In a similar manner we consider the Legendrian tangle $\cT$ obtained by removing $\{x \geq a\}$. We also have a vertex at $\infty$ of valency $m$ with one corresponding sector between $h_{\infty,m}$ and $h_{\infty,1}$. 

Now consider the inclusion $\iota_v:J_m\to I_m$ defined by $\rho_{ij}\mapsto \sfv_{i,j-i}$. 
The grading and the differential 
\begin{align*}
|\rho_{ij}|&=\mu(i)-\mu(j)-1\\
\partial \rho_{ij}&=\sum_{i<k<j}\rho_{ik}\rho_{kj}
\end{align*}
are compatible with (\ref{eq:Gr2_1}) and (\ref{eqn:differential_vertex}). Thus the inclusion $\iota_v$ induces a DGA map.
By extending $\iota_v$ we obtain $\iota_R:D(F^R)\to \cA_\cL$. The grading convention (\ref{eq:Gr1}) and degree 1 regular admissible disks in Definition~\ref{def:Regular admissible disks} for $\cL$ recover the grading and differential of crossings and right cusps in $F^R$. 
In other words, $\iota_R$ is a DGA map. A similar statement holds for $\iota_L:A(F^L)\to \cA_\cT$ and $\iota:A(F)\to \cA_{\cL\amalg_{\Phi_v}\cT}$.
Moreover $w:J_m\to A(F^L)$ is compatible with $\bp_{\infty}:\cI_m\to \cA_\cT$ as follows:
\[
\begin{tikzcd}
J_m\ar[r,"w"]\ar[d,"\iota_v"]& A(F^L) \ar[d,"\iota_L"]\\
\cI_m\ar[r,"\bp_\infty"]& \cA_\cT
\end{tikzcd}
\]

In summary we have
\begin{theorem}\label{thm:inclusion of DGA diagram}
There is a canonical inclusion from Sivek's DGA diagram to the DGA diagram in Theorem~\ref{thm:van Kampen} as follows:
\[
\begin{tikzcd}
& J_m \ar[dl,"\iota_v"'] \ar[rr,"w"] \ar[dd,hook] & & A(F^L) \ar[dl,"\iota_L"] \ar[dd,hook] \\
\cI_m \ar[rr,"\bp_\infty" near end, crossing over] \ar[dd,hook] & & \cA_\cT  \\
& D(F^R) \ar[dl,"\iota_R"] \ar[rr,"w'" near start] & & A(F) \ar[dl,"\iota_{L\amalg R}"] \\
\cA_\cL \ar[rr,"\bw_v"] & & \cA_{\cL\amalg\cT}\ar[from=uu,hook, crossing over]
\end{tikzcd}
\]
\end{theorem}

\subsection{Legendrian links in \texorpdfstring{$\#^N(S^1\times S^2)$}{the connected sum of the products of the sphere and the circle}}
In this section, we will discuss how to obtain the Legendrian contact homology algebra associated to a Legendrian link in a connected sum $\#^N (S^1\times S^2)$ as defined by Ekholm and Ng \cite{EN2015}.

Let $\cL=(\Lambda,\fP)$ be a $m$-component Legendrian link with $\ZZ$-valued potential in $\#^N(S^1\times S^2)$, where the ambient space will be regarded as the boundary of the subcritical Weinstein manifold $W_N$ which is the standard $4$-ball with $N$ Weinstein one-handles attached as discussed in \S~\ref{sec:geometric model}.
We assume that each component $\Lambda_i$ has a bivalent vertex $v^i$.

Let us denote the cocore spheres corresponding to the $N$ one-handles by $\{S^2_1,\cdots, S^2_N\}$.
In other words, there is a canonical contactomorphism between the complement of the spheres $\{S^2_i\}$ in $\#^N(S^1\times S^2)$ with the complement $M$ of $(2N)$ three-balls $\{B_i^+, B_i^-\}$ in the standard $S^3$.
\[
\phi:\#^N(S^2\times S^1)\setminus \coprod_{i=1}^N S^2_i \to M\coloneqq S^3\setminus \coprod_{i=1}^N (B_i^+\cup B_i^-)
\]
More precisely, we may assume that the two boundary spheres of $B_i^\pm$ correspond to the sphere $S^2_i$.

Now, we consider the bordered Legendrian $\Lambda\cap M$ in $M$ and obtain the Legendrian graph $\bar\Lambda\subset S^3$, called the {\em completion}, from $\Lambda\cap M$ by filling up the trivial Legendrian $n_i$-tangles for both $B_i^+$ and $B_i^-$ if $\Lambda\cap M$ has $n_i$ ends near $B_i^+$ (or $B_i^-$).
Let us denote the pairs of vertices corresponding to the three-balls $B_i^\pm$ by $v_i^\pm$, respectively.
Then the set of vertices for $\bar\Lambda$ is the union
\begin{align*}
\cV_{\bar\Lambda}&= \cV^{(2)}_{\bar\Lambda}\amalg \cV^\ess_{\bar\Lambda},&
\cV^{(2)}_{\bar\Lambda}&\coloneqq\{v^1,\cdots, v^m\}, &
\cV_{\bar\Lambda}^\ess&\coloneqq\{w_1^+,w_1^-,\cdots,w_N^+,w_N^-\}.
\end{align*}

It is obvious that the potential $\bar\fP$ for $\bar\Lambda$ is inherited from that $\fP$ for $\Lambda$ and so we can define the DGA for the pair $\bar\cL\coloneqq(\bar\Lambda,\bar\fP)$. 
Moreover, for each $i$, the domains of the two canonical peripheral structures $\bp_{\sfw_i^+}$ and $\bp_{\sfw_i^-}$ for $\sfw_i^+$ and $\sfw_i^-$ are isomorphic since the two vertices $w_i^+$ and $w_i^-$ have the same local structures, i.e., valency, potentials, and so on. Indeed, there exists a DGA $\cI_{w_i}=(I_{n_i},|\cdot|,\partial)$ such that
\[
\begin{tikzcd}
\cI_{w_i} \ar[r,"\bp_{\sfw_i^+}",shift left=.5ex]\ar[r,"\bp_{\sfw_i^-}"', shift right=.5ex] & \cA_{\bar\cL}.
\end{tikzcd}
\]

Now we obtain a new DGA $\cA_\cL$ in two steps as follows:
\begin{enumerate}
\item Apply smoothing at each vertex in $\cV^{(2)}_{\bar\Lambda}$ to make the DGA over $\ZZ[\sft_1^{\pm},\cdots, \sft_k^{\pm}]$
\[
\begin{tikzcd}[column sep=4pc]
\cA_{\bar\cL}\ar[r,"\text{smoothing}"]& \cA_{\bar\cL}^\sm\left(\cV^{(2)}_{\bar\Lambda}\right).
\end{tikzcd}
\]
\item Take the coequalizer of $\bp_{\sfw_i^+}$ and $\bp_{\sfw_i^-}$ for each $i$ to define $\cA_\cL$
\begin{align*}
\cA_\cL&\coloneqq \cA_{\bar\cL}^\sm\left(\cV^{(2)}_{\bar\Lambda}\right)\bigg/\sim,&
(\sfw_i^+)_{j,\ell} &\sim (\sfw_i^-)_{j,\ell} \quad\forall j\in\Zmod{n_i}, \ell\ge 1.
\end{align*}
\end{enumerate}

\begin{theorem}\label{theorem:relation to EN}
The DGA $\cA_\cL$ is generalized stable-tame isomorphic to the DGA $\cA_\cL^{\EN}$ associated to the Legendrian $\cL$ in $\#^m(S^1\times S^2)$ defined by Ekholm and Ng.

In other words, we have commutativity of the following diagram.
\[
\begin{tikzcd}[column sep=4pc]
\cL\subset \#^m(S^1\times S^2) \ar[d,"\bar{(\cdot)}"']\ar[rrr,"\cA^{\EN}"] & & & \cA_{\cL}^{\EN}\ar[d,"\stabletameisom"]\\
\bar\cL\subset S^3 \ar[r,"\cA"] & \cA_{\bar\cL} \ar[r,"\text{\rm smoothing}"] &\cA_{\bar\cL}^\sm\left(\cV^{(2)}_{\bar\Lambda}\right) \ar[r,"/\sim"] & \cA_\cL
\end{tikzcd}
\]
\end{theorem}
\begin{proof}
The proof is straightforward from the definition of $\cA^{\EN}_\cL$ and we omit it.
\end{proof}
\begin{remark}
Compared to Ekholm--Ng's construction, ours has the benefit that it works for arbitrary representatives for $\cL$. That is, the Legendrian representative $\bar\cL$ need {\em not} be in {\em Gompf standard form}.
\end{remark}

\section{Relation to partially wrapped Floer homology}\label{section:partially wrapped Floer homology}

For a Legendrian graph $\Lambda$, we will associate a Weinstein domain $\cW_\cR$ whose Liouville vector field is adapted to the Legendrian ribbon $\cR$ of $\Lambda$. 
For Legendrian links and Legendrian tori, the construction and the statement in this section are already discussed in \cite{EL2017,GPS2017} and in the arXiv version of \cite{ENS2018}.

\subsection{Weinstein pairs and domains}\label{sec:Weinstein pair}

\begin{definition}\cite{Eliashberg2017}\label{def:Weinstein hypersurface}
Let $(Y,\xi)$ be a contact manifold. A codimension one submanifold $i:R\hookrightarrow Y$ with boundary is called a {\em Weinstein hypersurface} if 
\begin{enumerate}
\item There exists a contact form $\lambda$ for $\xi$ and a Lyapunov function $\phi_R:R\to \RR$ such that $\cR=(R, i^*\lambda,\phi_R)$ is a Weinstein domain.
\item The Reeb vector field for $\lambda$ is transverse to $R$ and $\partial R$ is a contact submanifold of $(Y,\xi)$. 
\end{enumerate}

A {\em Weinstein pair} $(\cW,\cR)$ consists of a Weinstein domain $\cW=(W,\lambda,\phi)$ and a Weinstein hypersurface $\cR$ in $(\partial W,\ker\lambda)$.
\end{definition}

\begin{remark}\label{rmk:weinstein pair}
Weinstein hypersurfaces are the special cases of the Liouville hypersurfaces introduced by Avdek in \cite{Avdek2012}.
\end{remark}

For each Weinstein pair $(\cW=(W,\lambda_W,\phi_W),\cR=(R,\lambda_R=i^*\lambda_W,\phi_R))$, one can construct a new {\em Weinstein domain} which is relevant to `{\em stops}' of Sylvan \cite{Sylvan2016} and `{\em Liouville sectors}' of Ganatra--Pardon--Shende \cite{GPS2017} $\cW_\cR=(W_R,\lambda_{W_R},\phi_{W_R})$ as follows:
Let us consider a {\em cotangent cone} $\cC_\cR=(C_R,\lambda_{C_R},\phi_{C_R})$ along $\cR$ defined as
\begin{align*}
C_R&\coloneqq R \times D^\epsilon T^*([0,+\infty)),&
\lambda_{C_R}&\coloneqq \lambda_R + f(q_2)dp_2 - p_2 d q_2,&
\phi_{C_R}&\coloneqq \phi_R + \frac{1}{2}f(q_2)^2 + \frac{1}{2}p_2^2,
\end{align*}
where $(q_2,p_2)$ are coordinates for 
\[
D^\epsilon T^*([0,+\infty))=\{(q_2,p_2) \mid q_2\in[0,+\infty), |p_2|\leq \epsilon\},
\]
and $f:[0,+\infty)\to \RR$ is a smooth cut-off function satisfying
\begin{align*}
f(0)&=1, &f(q_2)&=0 \quad\text{ for }q_2\geq 2;\\
f'(0)&=0, &-1< f'(q_2)&\leq 0 \quad\text{ for }q_2\geq 0.
\end{align*}
Then it is easy to check that the Liouville vector field $Z_{\cC_\cR}$ for $\cC_\cR$ is given as
\[
Z_{\cC_\cR} = Z_\cR +\frac{f(q_2)}{f'(q_2)+1}\partial_{q_2}+\frac{p_2}{f'(q_2)+1}\partial_{p_2},
\]
where $Z_\cR$ is the Liouville vector field for $\cR$. Therefore we have
\[
Z_{\cC_\cR}=\begin{cases}
Z_\cR+\partial_{q_2}+p_2\partial_{p_2} & q_2=0;\\
Z_\cR+p_2\partial_{p_2} & q_2\ge 2.
\end{cases}
\]

Let us consider embeddings
\[
\setlength{\arraycolsep}{1pt}
\begin{array}{rcccc}
i_1&:& R \times (-\epsilon,\epsilon) & \hookrightarrow& C_R\\
&&(x,t)&\mapsto& (x,(0,t)),
\end{array}\qquad
\begin{array}{rcccc}
i_2&:& R \times (-\epsilon,\epsilon) &\hookrightarrow& W\\
&&(x,t)&\mapsto& \varphi^t_{Reeb}(x),
\end{array}
\]
where $\varphi^t_{Reeb}$ is a time-$t$ Reeb flow of $\cW$. 
Note that $i_2$ is well-defined because of the condition (2) in Definition~\ref{def:Weinstein hypersurface}.
Then $W_R$ is obtained by gluing $W$ and $C_R$ along $R\times(-\epsilon,\epsilon)$ via $i_1$ and $i_2$.

It is direct to check that the Liouville vector field $Z_{\cC_\cR}$ of $\cC_\cR$ transversely indicates inward direction near $i_1(R\times(-\epsilon,\epsilon))$.
By the contact property of $\partial W$ the Liouville vector field $Z_\cW$ of $\cW$ transversely points outward near $i_2(R \times (-\epsilon,\epsilon))$.
Now we deform the Weinstein structure $\cW=(W,\lambda_W,\phi_W)$ into $\cW'=(W,\lambda'_W,\phi'_W)$ near a neighborhood $U\subset W$ of $i_2(R \times (-\epsilon,\epsilon))$ so that the deformed one smoothly coincides with $\cC$ near the image of $i_2$.
Hence the glued Weinstein structure $\cW_\cR=(W_R,\lambda_{W_R},\phi_{W_R})$ is well-defined up to Weinstein homotopy.

Then the Liouville vector field $Z_{\cW_\cR}$ of $\cW_\cR$ is given by interpolating Liouville vector fields $Z_\cW$ and $Z_{\cC_\cR}$. More precisely,
\[
Z_{\cW_\cR}(w)=
\begin{cases}
Z_\cW& w\in W\setminus U;\\
Z_{\cC_\cR} &w\in C_R.
\end{cases}
\]
The vectors pointing outward in Figure~\ref{fig:new Weinstein} indicate the Liouville vector field $Z_{\cW_\cR}$.

\begin{figure}[ht]\label{fig:new Weinstein}
\[
\vcenter{\hbox{\input{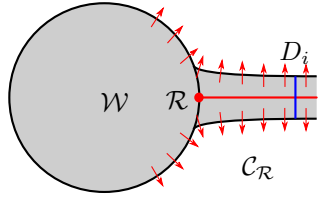}}}
\]
\caption{A schematic picture for the Weinstein domain $\cW_\cR$.}
\end{figure}

We now introduce two important Lagrangian subspaces of the Weinstein domain $\cW_\cR$ which are symplectic dual to each other.
For simplicity, we focus on the 4-dimensional case and assume the Liouville vector field $Z_\cR$ and the corresponding Lyapunov function $\phi_R$ are Morse-Smale which can be achieved generically.

The first one is the {\em Lagrangian skeleton}
\begin{align*}
L_{\cW_\cR}\coloneqq\core(\cW_\cR)&=\bigcap_{t \geq 0}\varphi_{Z}^{-t}(W_R)
\simeq\core(\cW)\cup \bigcup_{t\geq 0}\varphi_Z^{-t}(\core(\cR)),
\end{align*}
where $\varphi_{Z}^t$ is the time $t$ flow of $Z=Z_{\cW_\cR}$.

For the other Lagrangian, let $\{c_1,\cdots,c_m\}$ be a set of index one critical points of the Morse-Smale function $\phi_R$, which is not necessarily exhaustive.
Then we define the stable manifold $I_i\subset R$ corresponding to $c_i$ and their union $I$
\begin{align*}
I_i&\coloneqq\{y\in R\mid \varphi_{Z_\cR}^t(y)\to c_i\text{ as }t\to-\infty\},&
I&\coloneqq\coprod_{i=1}^m I_i.
\end{align*}
Here $\varphi_{Z_\cR}^t$ is the time $t$ flow of the Liouville vector field $Z_\cR$ and note that each $I_i$ intersects $\partial R$ at two points. 
The {\em Lagrangian cocore disk} $D_i$ corresponding to the critical point $c_i$ is 
\begin{align*}
D_i\coloneqq I_i \times D^\epsilon T^*_{\ell_i}([0,+\infty)),
\end{align*}
for some $\ell_i>0$.
We denote the union of all $D_i\coloneqq D_{c_i}$'s by $D_{\cW_\cR}$.
\begin{remark}
The Lagrangian $D_{\cW_\cR}$ does depend on the choice of the set of critical points $\{c_1,\cdots, c_m\}$.
\end{remark}

\begin{definition}
We say that a pair $(\cR,I)$ is a {\em based Weinstein hypersurface} if the complement of $I$ is acyclic and has the same number of components as $R$
\begin{align*}
H_0(R\setminus I;\RR)&\simeq H_0(R),&
H_1(R\setminus I;\RR)&=0,
\end{align*}
and its skeleton will be defined as the pair
\[
\core(\cR,I)\coloneqq (\core(\cR), \{c_1,\cdots,c_m\}).
\]
\end{definition}
For a based Weinstein pair $(\cW,(\cR,I))$, the Lagrangian submanifold $D_{\cW_\cR}$ is what we want to regard as the symplectic dual to the skeleton $L_{\cW_\cR}$.

In the rest of this section, we focus on the endomorphism algebra $CW^*(D_{\cW_\cR},D_{\cW_\cR})$ of $D_{\cW_\cR}$
\begin{align*}
CW^*(D_{\cW_\cR},D_{\cW_\cR})&=\bigoplus_{1\le i,j\le m}CW^*(D_i,D_j).
\end{align*}

Let us briefly recall the construction of $CW^*(D_{\cW_\cR},D_{\cW_\cR})$ from \cite[Appendix B]{EL2017}. 
The algebra is generated by Reeb chords with respect to $(\partial W_R, \lambda_{W_R}|_{\partial W_R})$ starting and ending at $\partial D_{\cW_\cR}=\coprod_i \partial D_i$ and intersections between $D_{\cW_\cR}$ and its push-off. The $A_\infty$ structure is given by counting {\em partial holomorphic buildings} which consist of one primary disk in the symplectization $\partial W_R\times \RR$ with a distinguished negative puncture and several secondary disks of the partial building in $W_R$ on each additional negative punctures of the primary disk.

There are generators $\gamma_\ell$ of $CW^*(D_{\cW_\cR}, D_{\cW_\cR})$, the Reeb chords from $\partial D_{\cW_\cR}$ to itself with respect to $(\partial W_{\cW_\cR}, \lambda_{W_{\cW_\cR}}|_{\partial W_{\cW_\cR}})$, whose trajectories are contained in the cotangent cone $\cC_{\cW_\cR}$.
Then its images under the projection 
\[
\pi_R: C_R=R \times D^\epsilon T^*([0,+\infty))\to R
\] 
travel along the foliation in Figure~\ref{figure:foliation on ribbon}.
In other words, these generators can be identified with (oriented) boundary arcs of $R$ starting and ending at $I$.

\begin{figure}[ht]
\begin{align*}
R&=\vcenter{\hbox{\includegraphics[scale=1.5]{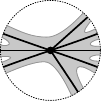}}}&
&\vcenter{\hbox{\includegraphics[scale=1.5]{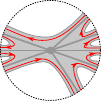}}}
\end{align*}
\caption{A Weinstein hypersurface $R$ and a Reeb foliation on $R$}
\label{figure:foliation on ribbon}
\end{figure}

Moreover, for generators $\gamma_\ell \in CW^*(D_{i_{\ell}},D_{i_{\ell-1}})$, the $A_\infty$ product $m_k(\gamma_k,\cdots, \gamma_1)$ vanishes unless the concatenation
\begin{align*}
\gamma_1\cdot I_{i_1} \cdot\gamma_2\cdot\cdots\cdot I_{i_{k-1}}\cdot\gamma_k\cdot I_{i_k}
\end{align*}
is contractible in the Weinstein hypersurface $R$. See Figure~\ref{figure:m_k_product}.

\begin{figure}[ht]
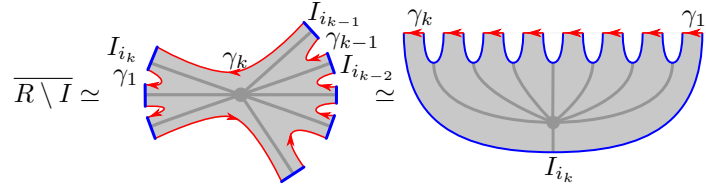

\begin{align*}
\overline{R\setminus I}\simeq\vcenter{\hbox{\def\svgscale{1.5}\input{ribbon_curves_disks_input.tex}}} &\simeq 
\vcenter{\hbox{\input{ribbon_m_k_input.tex}}}
\end{align*}
\caption{The Weinstein hypersurface $R$ contributes to a holomorphic disk.}
\label{figure:m_k_product}
\end{figure}

\begin{remark}\label{rem:quasi-isom}
One expected statement from the above is that there is a quasi-isomorphism between the endomorphism algebra $L_{\cW_\cR}$ in the infinitesimal Fukaya category of $\cW_\cR$ and the endomorphism algebra $D_{\cW_\cR}$ in the partially wrapped Fukaya category of $\cW_\cR$.
See \cite{EL2017} and the references therein for the Legendrian link case.
\end{remark}

\begin{remark}\label{remark:weinstein domain up to deformation}
Up to Weinstein homotopies, the construction of $\cW_\cR$ is well-defined under the deformation of $\cR$.
Moreover, Weinstein homotopy yields the equivalence of the endomorphism algebra and Fukaya categories above.
\end{remark}

\subsection{Based Legendrians and Legendrian bouquets}
Let us consider the standard Weinstein four-ball $\cB$
\[
\cB\coloneqq\left(B^4,\lambda_{B^4}\coloneqq\sum_{i=1,2}(x_i dy_i-y_i dx_i),\phi_{B^4}\coloneqq\frac{1}{2}\sum_{i=1,2}(x_i^2+y_i^2)\right).
\]

From now on, we suppose that $\cW=\cB$ and its contact boundary $S^3_\std$ is equipped with the contact form $\alpha_\std=\lambda_{B^4}|_{S^3}$.

\begin{definition}
A pair $(\Lambda, B)$ of a Legendrian graph with a subset $B\subset \Lambda\setminus\cV_\Lambda$ is called a {\em based Legendrian graph} if the complement of $B$ is a maximal forest of $\Lambda$. That is,
\begin{align*}
H_0\left(\Lambda\setminus B;\RR\right)&\simeq H_0(\Lambda;\RR),&
H_1\left(\Lambda\setminus B;\RR\right)&= 0.
\end{align*}
\end{definition}

For a based Legendrian, each base point $b^i\subset B=\{b^1,\cdots,b^m\}$ with $m=\rk H_1(\Lambda;\RR)$ determines a piecewise smooth Legendrian cycle, say $\Lambda_i$. 
Hence there are $m$-corresponding cycles $\{\Lambda_1,\cdots, \Lambda_m\}$.

Let $(\cR,I)$ be a based Weinstein hypersurface in $S^3_\std$ and $\varphi_{Z_\cR}$ be the flow of the corresponding Liouville vector field. Then the core $\Lambda_\cR\coloneqq\core(\cR)$ becomes a Legendrian graph (possibly with isolated vertices) since $\phi_R$ is Morse-Smale.
Indeed, its vertices are given by the set of index two critical points in $\Lambda$ and each edge contains exactly one critical point of index 1.
Then we regard the set of critical points corresponding to components of $I$ as the set $B_\cR$ of base points.

\begin{remark}
The surface $\cR$ is convex with respect to the Reeb vector field which is a contact vector field, and the characteristic foliation $\cR_\xi$ is the same as the singular foliation given by the Liouville vector field $Z_\cR$. Then the union of singular leaves is precisely $\Lambda_\cR$ so that elliptic singularities are regarded as vertices.
\end{remark}

\begin{figure}[ht]
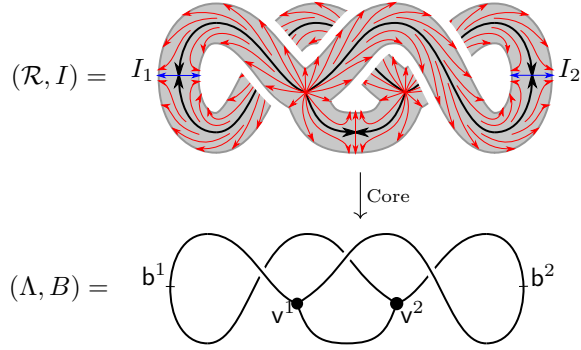

\[
\begin{tikzcd}[column sep=0pc]
(\cR, I)=&\vcenter{\hbox{\input{Theta_ribbon_input.tex}}}\ar[d,"\core"]\\
(\Lambda,B)=&\vcenter{\hbox{\input{Theta_base_input.tex}}}
\end{tikzcd}
\]
\caption{A Liouville vector field on a based Weinstein hypersurface}
\end{figure}

\begin{definition}\cite{Shende2018}\label{def:Legendrian ribbon}
A {\em (based) Legendrian ribbon} of a (based) Legendrian graph $\Lambda$ (or $(\Lambda,B)$) in $S^3_\std$ is a (based) Weinstein hypersurface $\cR$ (or $(\cR, I)$) whose skeleton is $\Lambda$ (or $(\Lambda,B)$).
\end{definition}

Conversely, one can construct a ribbon for a Legendrian graph $\Lambda$ as follows:
Firstly, for each vertex $v \in \cV_\Lambda$, consider a 2-dimensional disk (or a 0-handle) with the standard Weinstein structure 
\[\cR_v=\left(\BB^2,xdy-ydx,\frac{1}{2}(x^2+y^2)\right),\]
and for each edge $e\in \cE_\Lambda$, we attach a Weinstein 1-handle
\[
\cR_e=\left([-2,2]\times \left[-\frac{1}{4},\frac{1}{4}\right], 2xdy+ydx, 2x^2-y^2\right)
\] 
along $\{\cR_v\mid v\in \partial e\}$ respecting the cyclic order at the vertices.
Note that there is a unique critical point $c_e$ of index one for each handle $\cR_e$ and the resulting Weinstein manifold $\cR_\Lambda$ can be realized as a Weinstein hypersurface in $S^3_\std$ whose skeleton is precisely $\Lambda$. Namely, $\cR_\Lambda$ is a ribbon of $\Lambda$.
Moreover, if $\Lambda$ is based at $B$, then we define $I$ as the union of stable manifolds corresponding to critical points of edges which contain a base point.

\begin{remark}\label{rem:ribbon and CW}
For a given Legendrian graph $\Lambda$, there may exist two non-isotopic ribbons $\cR_\Lambda$ and $\cR_\Lambda'$.
However, they are homotopic to each other as Weinstein hypersurfaces in $S^3_\std$, which induce the Weinstein homotopy between Weinstein domains $\cW_{\cR_\Lambda}$ and $\cW_{\cR_\Lambda'}$ as mentioned in Remark~\ref{remark:weinstein domain up to deformation}.

For higher dimensional cases, this may not true but all ribbons still define the objects having no Floer theoretic differences.
See \cite[Lemma 3.8]{Shende2018}.
\end{remark}

For a based Legendrian graph, there is a canonical way to obtain (a union of) Legendrian bouquets---connected Legendrian graphs with one vertex at each component--- by contracting edges without a base point.
See, Figure~\ref{fig:based Legendrian and Legendrian bouquet} for the induced Legendrian bouquet.
Note that each edge in the Legendrian bouquet corresponds to a Legendrian cycle.

\begin{figure}[ht]
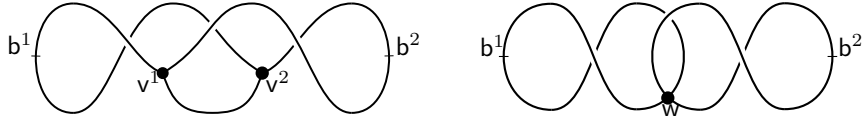

\begin{align*}
\vcenter{\hbox{\input{Theta_base_input.tex}}}
&&
\vcenter{\hbox{\input{Theta_bouquet_input.tex}}}
\end{align*}
\caption{The based Legendrian $\Theta$ and its Legendrian bouquet}
\label{fig:based Legendrian and Legendrian bouquet}
\end{figure}

\begin{lemma}\label{lemma:based ribbon}
Let us consider two based Legendrian $(\Lambda_0,B_0)$ and $(\Lambda_1,B_1)$ which differ by Legendrian Reidemeister moves or contraction of a {\em non-based} edge. 
Then the based Legendrian ribbons $(\cR_i,I_i)$ corresponding to $(\Lambda_i,B_i)$ are homotopic.
\end{lemma}

\begin{proof}
There is nothing to prove for the Legendrian Reidemeister moves. For a non-based edge contraction, the corresponding Weinstein hypersurfaces differ by the handle cancellation between Weinstein $0$- and $1$-handles.
\end{proof}
\begin{remark}
One can rely on the convex surface theory of contact manifolds for the alternative proof. That is, the handle cancellation is nothing but the cancellation of elliptic-hyperbolic pair of singularities in the characteristic foliation on the ribbon which can be realized as a perturbation of the ribbon.
\end{remark}

\begin{definition}For a based Legendrian graph, we mean by a {\em based ribbon isotopy} a sequence of Reidemeister moves, non-based edge contractions and inverse operations.
\end{definition}

Therefore by Remark~\ref{rem:ribbon and CW} and Lemma~\ref{lemma:based ribbon}, the based Legendrian $(\Lambda,B)$ up to based ribbon isotopy defines a based Weinstein hypersurface $(\cR_\Lambda,I_\Lambda)$ uniquely up to homotopy.
Furthermore, by the construction in \S~\ref{sec:Weinstein pair} and Remark~\ref{remark:weinstein domain up to deformation}, this based hypersurface determines uniquely $\cW_{\cR_\Lambda}$ up to Weinstein homotopy and so $CW(D_{\cW_{\cR_\Lambda}},D_{\cW_{\cR_\Lambda}})$ up to equivalence.

\begin{notation}
From now on, we simply use $\cC_\Lambda$, $\cW_\Lambda$ and $D_\Lambda$ instead of $\cC_{\cR_\Lambda}$, $\cW_{\cR_\Lambda}$ and $D_{\cW_{\cR_\Lambda}}$, respectively, to denote the cotangent cone, the Weinstein domain and the distinguished Lagrangian corresponding to the based Legendrian graph $(\Lambda,B)$.
\end{notation}

Without loss of generality, we may assume in the rest of this article that our based Legendrians are Legendrian bouquets.

\subsection{DGAs of composable words}\label{sec:Composable DGAs}
Let $\cL=((\Lambda,B),\fP)$ be a based Legendrian bouquet with potential.
From now on, we will denote each edge $e_i\in\cE_\Lambda$ and its projection $\sfe_i\in\sE_\sL$ simply by $\lambda_i$ unless any ambiguity occurs.

Then each edge $\lambda_i$ of $\Lambda$ has the base point $b^i$ and hence corresponds to the Legendrian cycle $\Lambda_i$. In order to define $\cA_{\cL}(\Lambda,\Lambda)$ we first consider a DGA of composable words as follows:
\begin{definition}
A {\em composable DGA} of the based Legendrian $\cL=((\Lambda,B),\fP)$ is a pair \begin{align*}
\cA_\cL^\co&\coloneqq(A_\Lambda^\co, |\cdot|_{\fP}, \partial^\co),&
\cP_\cL^\co&\coloneqq\{\bp_\lambda^\co \mid \lambda\in\cE_\Lambda\}\amalg\{\bp_v^\co \mid v\in\cV_\Lambda\}
\end{align*}
 of a DGA and an associated collection of extended peripheral structures satisfying
\begin{enumerate}
\item The algebra $A_\Lambda^\co$ is an associative algebra over $\ZZ$ generated by the union $\sG_\sL\amalg\sE_\sL$ 
\[
A_\Lambda^\co\coloneqq \ZZ\langle \sG_\sL\amalg\sE_\sL\rangle \big/ \sim
\]
satisfying the following relations:
\begin{align}\label{eqn:composable relation}
\begin{cases}
\lambda_i \sfg \sim \sfg \sim \sfg \lambda_j &\quad \forall\sfg\in \sG_\sL(\Lambda_i,\Lambda_j);\\
\lambda_i \lambda_j \sim \delta_{ij}\lambda_i &\quad \lambda_i, \lambda_j \in \sE_\sL,
\end{cases}
\end{align}
where $\sG_\sL(\Lambda_i,\Lambda_j)$ is the subset of $\sG_\sL$ consisting of generators whose corresponding capping paths start and end at the cycles $\Lambda_i$ and $\Lambda_j$, respectively.
\item The grading is extended by assigning $|\lambda|_\fP=0$ for all $\lambda\in \sE_\sL$.
\item The differential is defined by counting the admissible disks as before except the following:
Each admissible disk 
\[
f:(\Pi_t,\partial\Pi_t,\vPi_t)\to(\RR^2,\sL,\sS_\sL),
\]
contributes
\[
[f(\be_0)]\tilde f(\bv_1)[f(\be_1)]\cdots [f(\be_{t-1})]\tilde f(\bv_t)[f(\be_{t})]
\]
to $\partial^\co(\tilde f(\bv_0))$. Here we define $[f(\be_r)]$ to be $\lambda_k$ when $f(\be_r) \subset \sL_k=\pi_L(\Lambda_k)$. Let us define 
\begin{align*}
\partial^\co(\lambda)&\coloneqq0,\quad\forall\lambda\in\sE_\sL.
\end{align*}
\item The peripheral structure $\bp_*^\co$ for $*\in\cV_\Lambda\amalg\cE_\Lambda$ is a DGA morphism
\begin{align*}
\bp_*^\co&:\cI_*^\co\to \cA_\cL^\co,
\end{align*}
where $\cI_v^\co$ and $\cI_\lambda^\co$ are non-unital DGAs extended from $\cI_v$ and $\cI_\emptyset$ by adding the generating set $\sE_\sL$ 
and the defining relations coming from (\ref{eqn:composable relation}), and $\bp_*^\co$ maps each $\lambda\in \cI_*$ to $\lambda\in\cA_\cL^\co$ itself. 
\end{enumerate}
\end{definition}

Due to the defining relation (\ref{eqn:composable relation}), all the generators $\lambda_k$ coming from edges will be absorbed into one of the adjacent generators coming from vertices unless $f$ is a monogon. 
In other words, two corresponding monomials in $\partial(\sfg)$ and $\partial^\co(\sfg)$ are the same as words except for the case when $\partial(\sfg)$ contains a monomial $1$ corresponding to a monogon. 
In this case, the corresponding term in $\partial^\co(\sfg)$ is precisely $\lambda_k$, where $\lambda_k$ is the edge containing the boundary of the disk corresponding to $1$ in $\partial(\sfg)$.
For example, we have
\begin{align*}
\partial^\co \sfv_{i,\ell}&\coloneqq \delta_{\ell,\val(\sfv)}\lambda_k+\sum_{\ell_1+\ell_2=\ell}(-1)^{|\sfv_{i,\ell_1}|-1}\sfv_{i,\ell_1}\sfv_{i+\ell_1, \ell_2},& \sfv_{i,\ell}&\in\sG_\sL(\Lambda_k,\Lambda_j).
\end{align*}
Therefore it is straightforward that $\partial^\co \circ \partial^\co =0$ and its generalized-stable tame isomorphism class is invariant under Legendrian isotopies.

The following theorem summarizes the above discussion and we will omit the proof.
\begin{theorem}\label{theorem:non-unital generalization}
For each based Legendrian $\cL$ with potential, there exists a pair $(\cA_\cL^\co, \cP_\cL^\co)$ consisting of an associative, differential graded algebra $\cA_\cL^\co$ over $\ZZ$ generated by $\sG_\sL\amalg \sE_\sL$ and a collection of extended peripheral structures $\cP_\cL^\co$. Its generalized-stable tame isomorphism class is a Legendrian invariant.
\end{theorem}

\begin{definition}
A {\em composable chain complex} $\cA_\cL(\Lambda_i,\Lambda_j)$ is a sub-chain complex of $\cA^{\co}_\cL=(A^\co_\Lambda,|\cdot|_\fP,\partial^\co)$ generated by composable word of the form $\lambda_i \sfm \lambda_j$ for some $\sfm\in A^\co_\Lambda$, or equivalently,
\[
\cA_\cL(\Lambda_i,\Lambda_j)\coloneqq \lambda_i \cdot \cA^\co_\cL \cdot \lambda_j.
\]
In particular, $\cA_\cL(\Lambda_i,\Lambda_i)$ for any $\lambda_i\in\sE_\sL$ becomes a DG-subalgebra of $\cA^{\co}_\cL$.
Let us define
\[
\cA_{\cL}(\Lambda,\Lambda)\coloneqq\bigoplus_{1\le i,j\le m} \cA_{\cL}(\Lambda_i,\Lambda_j).
\]
\end{definition}

Notice that one can regard $\cA_\cL(\Lambda,\Lambda)$ as a DG-category whose objects are $\Lambda_1,\cdots,\Lambda_m$ and morphisms between $\Lambda_i$ and $\Lambda_j$ are given as $\cA_\cL(\Lambda_i,\Lambda_j)$. Then the element $\lambda_i\in
\cA_\cL(\Lambda_i,\Lambda_i)$ for each $\Lambda_i$ plays the role of the identity morphism. One can further regard $\cA_\cL(\Lambda,\Lambda)$ as an $A_\infty$-category with trivial higher compositions.

As an extension of \cite[Theorem2 and Conjecture 3]{EL2017} we state the following conjecture.
\begin{conjecture}\label{conjecture:relation with CW}
There is an $A_\infty$ quasi-isomorphism between $CW^*(D_{\Lambda},D_{\Lambda})$ and $\cA_{\cL}(\Lambda,\Lambda)$ which extends a quasi-isomorphism between the chain complexes $CW^*(D_i,D_j)$ and $\cA_{\cL}(\Lambda_i,\Lambda_j)$.
\end{conjecture}

\subsection{Concrete examples and computation}\label{sec:A concrete example}
In this section, we provide evidence for Conjecture~\ref{conjecture:relation with CW} by computing the homology of the composable Legendrian DGAs for certain Legendrian graphs.

\subsubsection{The splitting case}
Suppose that $\cL^\infty$ is a Legendrian bouquet of two circles whose ribbon is homeomorphic to a pair of pants as depicted in Figure~\ref{figure:example of smoothing}.
\begin{figure}[ht]
\[
\cL^\infty=\vcenter{\hbox{\input{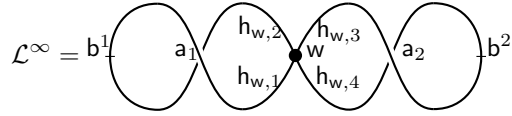}}}
\]
\caption{A based Legendrian bouquet}
\label{figure:example of smoothing}
\end{figure}

There are two base points $\sfb^1$ and $\sfb^2$ marked as small bars, the vertex $\sfw$ corresponding to the dot, two edges $\lambda_1,\lambda_2$ which correspond to cycles $\Lambda_1, \Lambda_2$, and two crossings $\sfa_1$ and $\sfa_2$.

First, let us list $\sG_{\sL^\infty}(\Lambda_i,\Lambda_j)$ for $i,j=1,2$ as follows:
\begin{align*}
\sG_{\sL^\infty}(\Lambda_1,\Lambda_1)&=\{ \lambda_1, \sfa_1,\sfw_{i,4k},\sfw_{1,4k-3},\sfw_{2,4k-1} \mid i=1,2, k\in\NN \};\\
\sG_{\sL^\infty}(\Lambda_2,\Lambda_2)&=\{ \lambda_2, \sfa_2,\sfw_{i,4k},\sfw_{3,4k-3},\sfw_{4,4k-1} \mid i=3,4, k\in\NN \};\\
\sG_{\sL^\infty}(\Lambda_1,\Lambda_2)&=\{ \sfw_{1,4k-2}, \sfw_{1,4k-1},\sfw_{2,4k-3},\sfw_{2,4k-1} \mid k\in \NN\};\\
\sG_{\sL^\infty}(\Lambda_2,\Lambda_1)&=\{ \sfw_{3,4k-2}, \sfw_{3,4k-1},\sfw_{4,4k-3},\sfw_{4,4k-1} \mid k\in \NN\}.
\end{align*}
Note that the chain complexes $\cA_{\cL^\infty}(\Lambda_i,\Lambda_j)=(A_\Lambda(\Lambda_i,\Lambda_j),|\cdot|,\partial^\co)$ are generated by composable words of arbitrary length from $\Lambda_i$ to $\Lambda_j$.
We take a Maslov potential which induces the following degree:
\[
|\lambda_i|=0,\quad |\sfa_i|=1,\quad |\sfw_{k,1}|=0 \text{ for }k=1,3,\quad |\sfw_{k,1}|=-1 \text{ for }k=2,4.
\]
The differentials for $\sfa_i$ are as follows:
\begin{align*}
\partial^\co(\sfa_1)=\lambda_1 + \sfw_{1,1},\quad \partial^\co(\sfa_2)=\lambda_2 + \sfw_{3,1},
\end{align*}
and these imply
\begin{align*}
[\lambda_1]=-[\sfw_{1,1}]&\in H_0(\cA_{\cL^\infty}(\Lambda_1,\Lambda_1));\\
[\lambda_2]=-[\sfw_{3,1}]&\in H_0(\cA_{\cL^\infty}(\Lambda_2,\Lambda_2)).
\end{align*}
Other cycles of a single generator in $H_*(\cA_{\cL^\infty}(\Lambda_i,\Lambda_j))$ are 
\begin{align*}
[\sfw_{2,1}]\in H_{-1}(\cA_{\cL^\infty}(\Lambda_1,\Lambda_2)),\qquad
[\sfw_{4,1}]\in H_{-1}(\cA_{\cL^\infty}(\Lambda_2,\Lambda_1)).
\end{align*}
By the following differentials
\begin{align*}
\partial^\co(\sfa_1 \sfw_{2,1} + \sfw_{1,2}) &= (\lambda_1+\sfw_{1,1})\sfw_{2,1} - \sfw_{1,1} \sfw_{2,1} = \sfw_{2,1};\\
\partial^\co(\sfa_2 \sfw_{4,1} + \sfw_{3,2}) &= (\lambda_2+\sfw_{3,1})\sfw_{4,1} - \sfw_{3,1} \sfw_{4,1} = \sfw_{4,1},
\end{align*}
we conclude that $[\sfw_{2,1}], [\sfw_{4,1}]$ are trivial cycles.

We expect that $[\lambda_1],[\lambda_2]$ are the only nontrivial homology cycles. This expectation coincides with the homology computation of partially wrapped Floer homology using the sheaf theoretic method, which can be seen as a representation of the quiver `$\bullet \quad \bullet$' without any arrow. 
We will discuss it carefully in a future work.

\subsubsection{$A_2$-case}
Let us recall the based Legendrian bouquet of the Legendrian $\Theta$.
The base points, the vertex, and the crossings are marked as before in Figure~\ref{fig:based Legendrian bouquet of theta}.

\begin{figure}[ht]
\begin{align*}
\cL^2=\vcenter{\hbox{\input{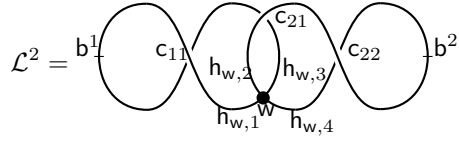}}}
\end{align*}
\caption{The Legendrian bouquet of Legendrian $\Theta$}
\label{fig:based Legendrian bouquet of theta}
\end{figure}

Here we also have two cycles $\Lambda_1$ and $\Lambda_2$ which correspond to $\sfb^1$ and $\sfb^2$, respectively.
Now list $\sG_{\sL^2}(\Lambda_i,\Lambda_j)$ for $i,j=1,2$ as follows:
\begin{align*}
\sG_{\sL^2}(\Lambda_1,\Lambda_1)&=\{ \lambda_1, \sfc_{11},\sfw_{i,4k},\sfw_{1,4k-2},\sfw_{3,4k-2} \mid i=1,3, k\in\NN \};\\
\sG_{\sL^2}(\Lambda_2,\Lambda_2)&=\{ \lambda_2, \sfc_{22},\sfw_{i,4k},\sfw_{2,4k-2},\sfw_{4,4k-2} \mid i=2,4, k\in\NN \};\\
\sG_{\sL^2}(\Lambda_1,\Lambda_2)&=\{ \sfw_{1,4k-3}, \sfw_{1,4k-1},\sfw_{3,4k-3},\sfw_{3,4k-1} \mid k\in \NN\};\\
\sG_{\sL^2}(\Lambda_2,\Lambda_1)&=\{ \sfc_{21}, \sfw_{2,4k-3}, \sfw_{2,4k-1},\sfw_{4,4k-3},\sfw_{4,4k-1} \mid k\in \NN\}.
\end{align*}
The chain complexes $\cA_{\cL^2}(\Lambda_i,\Lambda_j)=(A_\Lambda(\Lambda_i,\Lambda_j),|\cdot|,\partial^\co)$ are generated by composable words of arbitrary length from $\Lambda_i$ to $\Lambda_j$.
Here the degree is given as follows:
\[
|\lambda_i|=1, \quad |\sfc_{11}|=|\sfc_{22}|=1,\quad |\sfc_{21}|=0,\quad |\sfw_{k,1}|=0 \text{ for }k=1,3,\quad |\sfw_{k,1}|=-1 \text{ for }k=2,4.
\]
The differentials for $\sfc_{ij}$ are given by
\begin{align*}
\partial^\co(\sfc_{11})&=\lambda_1 + \sfw_{1,2} + \sfw_{1,1}\sfc_{21};\\
\partial^\co(\sfc_{22})&=\lambda_2 + \sfw_{2,2} - \sfc_{21} \sfw_{3,1};\\
\partial^\co(\sfc_{21})&=-\sfw_{2,1}.
\end{align*}

Here is the list of cycles of a single generator:
\begin{align*}
[\lambda_1]&\in H_0(\cA_{\cL^2}(\Lambda_1,\Lambda_1));\\
[\lambda_2]&\in H_0(\cA_{\cL^2}(\Lambda_2,\Lambda_2));\\
[\sfw_{1,1}],[\sfw_{3,1}]&\in H_{0}(\cA_{\cL^2}(\Lambda_1,\Lambda_2));\\
[\sfw_{2,1}],[\sfw_{4,1}]&\in H_{-1}(\cA_{\cL^2}(\Lambda_2,\Lambda_1)).
\end{align*}
Now examine the following differentials:
\begingroup\allowdisplaybreaks
\begin{align*}
\partial^\co(\sfc_{11}\sfw_{3,1}+\sfw_{1,3}+\sfw_{1,1}\sfc_{22})&=(\sfw_{3,1}+\sfw_{1,2}\sfw_{3,1}+\sfw_{1,1}\sfc_{2,1}\sfw_{3,1})\\
&\mathrel{\phantom{=}}+(-\sfw_{1,2}\sfw_{3,1}-\sfw_{1,1}\sfw_{2,2})\\
&\mathrel{\phantom{=}}+(\sfw_{1,1}+\sfw_{1,1}\sfw_{2,2}-\sfw_{1,1}\sfc_{2,1}\sfw_{3,1})\\
&=\sfw_{1,1}+\sfw_{3,1};\\
\partial^\co(\sfc_{2,2}\sfw_{4,1}+\sfc_{2,1}\sfw_{3,2}+\sfw_{2,3})&=(\sfw_{4,1}+\sfw_{2,2}\sfw_{4,1}-\sfc_{21}\sfw_{3,1}\sfw_{4,1})\\
&\mathrel{\phantom{=}}+(-\sfw_{2,1}\sfw_{3,2}+\sfc_{2,1}\sfw_{3,1}\sfw_{4,1})\\
&\mathrel{\phantom{=}}+(\sfw_{2,1}\sfw_{3,2}-\sfw_{2,2}\sfw_{4,1})\\
&=\sfw_{4,1}.
\end{align*}
\endgroup
In the end, we have
\begin{align*}
[\lambda_1]&\in H_0(\cA_{\cL^2}(\Lambda_1,\Lambda_1));\\
[\lambda_2]&\in H_0(\cA_{\cL^2}(\Lambda_2,\Lambda_2));\\
[\sfw_{1,1}]=-[\sfw_{3,1}]&\in H_{0}(\cA_{\cL^2}(\Lambda_1,\Lambda_2)).
\end{align*}

Even though general cycles of the chain complexes $\cA_{\cL^2}(\Lambda_i,\Lambda_j)$ are complicated,
we guess that the above are the only non-trivial cycles.

\begin{remark}\label{rmk:Legendrian theta and A2}
Let us consider a Weinstein domain $\cW_{\cL^2}$.
Since $\cL^2$ is ribbon isotopic to the Legendrian theta graph $\Theta$, there is a Weinstein isotopy between $\cW_{\cL^2}$ and $\cW_\Theta$.
The skeleton of $\cW_\Theta$ is homeomorphic to the cone of $\Theta$ which is nothing but an arboreal singularity of type $A_2$ in Nadler's list.

Moreover, the expected homology $H_*(\cA_{\cL^2}(\Lambda,\Lambda))$ can be seen as a representation of the $A_2$-quiver `$\bullet \to \bullet$'.
\end{remark}

\subsubsection{$A_3$-case}
Let us consider a based Legendrian 4-complete graph $\cK$ depicted in Figure~\ref{figure:K_4_a}.
This example was suggested by Tobias Ekholm.

It is not hard to see that $\cK$ is ribbon equivalent to a based Legendrian bouquet $\cL^3$ in Figure~\ref{figure:K_4_b}.
The based points $\sfb^1, \sfb^2$ and $\sfb^3$ correspond to Legendrian circles $\Lambda_1,\Lambda_2$ and $\Lambda_3$, respectively.
There are six crossings $\sfc_{ij}$ satisfying $1\leq j \leq i \leq 3$ which are Reeb chords from $\Lambda_j$ to $\Lambda_i$.
The six half edges near the vertex $\sfw$ are labeled $\sfh_{\sfw,1}$, $\sfh_{\sfw,2}$,\dots, $\sfh_{\sfw,6}$ in a clockwise direction.
For simplicity, we only label $\sfh_{\sfw,1}$ in Figure~\ref{figure:K_4_b}.

\begin{figure}[ht]
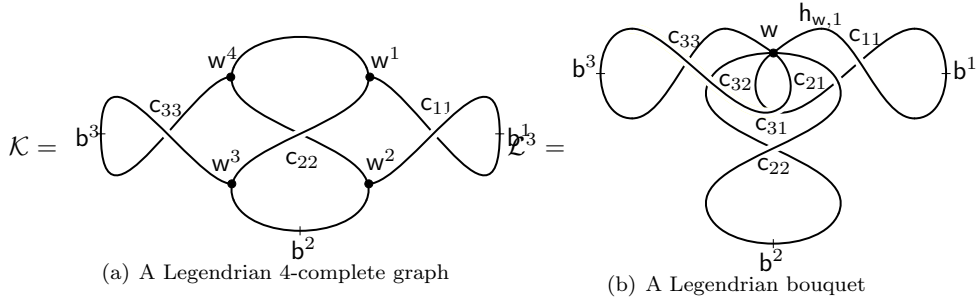

\subfigure[A Legendrian 4-complete graph\label{figure:K_4_a}]{\makebox[0.5\textwidth]{$
\cK=\vcenter{\hbox{\input{K_4_input.tex}}}
$}}
\subfigure[A Legendrian bouquet\label{figure:K_4_b}]{\makebox[0.45\textwidth]{$
\cL^3=\vcenter{\hbox{\input{K_4_2_input.tex}}}
$
}}
\caption{A Legendrian $K_4$ and a Legendrian bouquet $\Lambda$}
\label{figure:K_4}
\end{figure}
The generators are as follows:
\begingroup\allowdisplaybreaks
\begin{align*}
\sG_{\sL^3}(\Lambda_1,\Lambda_1)&=\{ \lambda_1, \sfc_{11},\sfw_{i,6k},\sfw_{1,6k-3},\sfw_{4,6k-3} \mid i=1,4, k\in\NN \};\\
\sG_{\sL^3}(\Lambda_1,\Lambda_2)&=\{ \sfw_{1,6k-5}, \sfw_{1,6k-2},\sfw_{4,6k-5},\sfw_{4,6k-2} \mid k\in \NN\};\\
\sG_{\sL^3}(\Lambda_1,\Lambda_3)&=\{ \sfw_{1,6k-4}, \sfw_{1,6k-1},\sfw_{4,6k-4},\sfw_{4,6k-1} \mid k\in \NN\};\\
\sG_{\sL^3}(\Lambda_2,\Lambda_1)&=\{ \sfc_{21}, \sfw_{2,6k-4}, \sfw_{2,6k-1},\sfw_{5,6k-4},\sfw_{5,6k-1} \mid k\in \NN\};\\
\sG_{\sL^3}(\Lambda_2,\Lambda_2)&=\{ \lambda_2, \sfc_{22},\sfw_{i,6k},\sfw_{2,6k-3},\sfw_{5,6k-3} \mid i=2,5, k\in\NN \};\\
\sG_{\sL^3}(\Lambda_2,\Lambda_3)&=\{ \sfw_{2,6k-5}, \sfw_{2,6k-2},\sfw_{5,6k-5},\sfw_{5,6k-2} \mid k\in \NN\};\\
\sG_{\sL^3}(\Lambda_3,\Lambda_1)&=\{ \sfc_{31}, \sfw_{3,6k-5}, \sfw_{3,6k-2},\sfw_{6,6k-5},\sfw_{6,6k-2} \mid k\in \NN\};\\
\sG_{\sL^3}(\Lambda_3,\Lambda_2)&=\{ \sfc_{32}, \sfw_{3,6k-4}, \sfw_{3,6k-1},\sfw_{6,6k-4},\sfw_{6,6k-1} \mid k\in \NN\};\\
\sG_{\sL^3}(\Lambda_3,\Lambda_3)&=\{ \lambda_3, \sfc_{33},\sfw_{i,6k},\sfw_{3,6k-3},\sfw_{6,6k-3} \mid i=3,6, k\in\NN \}.
\end{align*}
\endgroup
The chain complexes $\cA_{\cL^3}(\Lambda_i,\Lambda_j)$ are $\ZZ$-modules generated by composable words from $\Lambda_i$ to $\Lambda_j$ of arbitrary length.
Here we take a Maslov potential on $\cL^3$ which induces
\[
|\lambda_i|=0,\quad |\sfc_{ij}|=1,\quad |\sfw_{k,1}|=0 \text{ for } k=3,6, \quad |\sfw_{k,1}|=-1 \text{ for } k=1,2,4,5.
\]
The differentials for $\sfc_{ij}$'s are as follows:
\begin{align*}
&\begin{cases}
\partial^\co \sfc_{11}=\lambda_1+\sfw_{1,1}\sfc_{21}+\sfw_{1,2}\sfc_{31}+\sfw_{1,3},\\
\partial^\co \sfc_{22}=\lambda_2+\sfw_{2,1}\sfc_{32}+\sfw_{2,3}+\sfc_{21}\sfw_{4,1},\\
\partial^\co \sfc_{33}=\lambda_3+\sfw_{3,3}+\sfc_{31}\sfw_{4,2}+\sfc_{32}\sfw_{5,1},
\end{cases}&
&\begin{cases}
\partial^\co \sfc_{21}=\sfw_{2,1}\sfc_{31}+\sfw_{2,2},\\
\partial^\co \sfc_{32}=\sfc_{31}\sfw_{4,1}+\sfw_{3,2},\\
\partial^\co \sfc_{31}=\sfw_{3,1}.
\end{cases}
\end{align*}

As in the previous examples, let us focus on cycles of a single generator as follows:
\begin{align*}
[\lambda_i]&\in H_0(\cA_{\cL^3}(\Lambda_i,\Lambda_i)) \text{ for } i=1,2,3;\\
[\sfw_{1,1}],[\sfw_{4,1}]&\in H_{-1}(\cA_{\cL^3}(\Lambda_1,\Lambda_2));\\
[\sfw_{2,1}],[\sfw_{5,1}]&\in H_{-1}(\cA_{\cL^3}(\Lambda_2,\Lambda_3));\\
[\sfw_{3,1}],[\sfw_{6,1}]&\in H_0(\cA_{\cL^3}(\Lambda_3,\Lambda_1)).
\end{align*}

By direct computations, we obtain
\begin{align*}
\partial^\co(\sfw_{1,1}\sfc_{22} + \sfw_{1,2}\sfc_{32} + \sfw_{1,4} + \sfc_{11} \sfw_{4,1})&=\sfw_{4,1} - \sfw_{1,1};\\
\partial^\co(\sfw_{2,1}\sfc_{33} + \sfc_{21}\sfw_{4,2} + \sfc_{22}\sfw_{5,1} + \sfw_{2,4})&=\sfw_{5,1} - \sfw_{2,1};\\
\partial^\co(\sfc_{33}\sfw_{6,1}+\sfc_{32}\sfw_{5,2}+\sfc_{31}\sfw_{4,3}+\sfw_{3,4})&=\sfw_{6,1},
\end{align*}
which induce
\begin{align*}
[\lambda_i]&\in H_0(\cA_{\cL^3}(\Lambda_i,\Lambda_i)) \text{ for } i=1,2,3;\\
[\sfw_{1,1}]=[\sfw_{4,1}]&\in H_{-1}(\cA_{\cL^3}(\Lambda_1,\Lambda_2));\\
[\sfw_{2,1}]=[\sfw_{5,1}]&\in H_{-1}(\cA_{\cL^3}(\Lambda_2,\Lambda_3));\\
0=[\sfw_{3,1}]=[\sfw_{6,1}]&\in H_0(\cA_{\cL^3}(\Lambda_3,\Lambda_1)).
\end{align*}
We expect that the above are the only homology classes for $H_*(\cA_{\cL^3}(\Lambda,\Lambda))$.

\begin{remark}
By the similar argument in Remark~\ref{rmk:Legendrian theta and A2}, the Weinstein domain $\cW_{\cL^3}$ is Weinstein isotopic to $\cW_\cK$ whose skeleton $\cone(\cK)$ is an arboreal singularity of type $A_3$, see \cite{Nadler2017}. The expectation of homologies for $H_*(\cA_{\cL^3}(\Lambda,\Lambda))$ can be interpreted as a representation of the $A_3$-quiver `$\bullet \to \bullet \to \bullet$'.
\end{remark}

\appendix
\section{Admissible pairs and proof of Theorem~\ref{thm:differential}}\label{section:manipulation of disks}
We introduce the notion of {\em admissible pairs} as follows:
\begin{definition}[Admissible pairs]\label{def:admissible pair}
Let $f^1, f^0\in \cM_{(1)}$ be admissible disks on $\Pi_t$ and $\Pi_u$ of degree 1 whose vertices are $\{\bv_0,\cdots,\bv_t\}$ and $\{\bw_0,\cdots,\bw_u\}$, respectively.
\begin{enumerate}
\item An {\em admissible pair $f$} is a triple $(f^0,f^1,r)$ of admissible disks $f^0, f^1\in\cM_{(1)}$ and $1\le r\le t$ such that $\tilde f^0(\bw_0)=\tilde f^1(\bv_r)$, and we write $f\coloneqq(f^0\hybto{r} f^1)$.
\item We define $\cM^\pair(\sfg)$ to be the set of all admissible pairs 
\[
\cM^\pair(\sfg)\coloneqq\left\{(f^0\hybto{r}f^1)\,\middle\vert\,\tilde f^1(\bv_0)=\sfg, |f^0|_\ZZ=|f^1|_\ZZ=1\right\}
\]
and a function\footnote{Here we abuse the notation $\cP$.} $\cP:\cM^\pair(\sfg)\to A_\Lambda$ as 
\begin{align*}
\cP(f)&\coloneqq \sgn(f)\tilde f^1(\bv_1\cdots\bv_{r-1})\tilde f^0(\bw_1\cdots\bw_u) \tilde f^1(\bv_{r+1} \cdots \bv_t)
\in A_\Lambda,
\end{align*}
where 
\[
\sgn(f)\coloneqq\sgn(f^1)\sgn(f^0) (-1)^{|\tilde f^1(\bv_1\cdots\bv_{r-1})|}.
\]
\end{enumerate}
\end{definition}

Accoding to the classification of admissible disks of degree 1, we can separate $\cM^\pair$ into
\begin{align*}
\cM^\pair=\cM^{\reg\hybto{}\reg}\amalg
\cM^{\infinitesimalmonogon\hybto{}\reg}\amalg
\cM^{\infinitesimaltriangle\hybto{}\reg}\amalg
\cM^{\infinitesimaltriangle\hybto{}\infinitesimaltriangle}\amalg
\cM^{\infinitesimalmonogon\hybto{}\infinitesimaltriangle},
\end{align*}
where the meaning of each superscript is obvious
\[
f=(f^0\hybto{r} f^1)\in\cM^{A\hybto{}B}\Longleftrightarrow f^0\in\cM^{A}, f^1\in\cM^{B}.
\]

We introduce several operations on admissible disks and pairs which will be used in the proof of Proposition~\ref{proposition:differential2}.

\subsection{Manipulations of admissible pairs.}\label{sec:Manipulations of admissible pairs}
\subsubsection{Gluing of pairs}\label{sec:Gluing of pairs}
Let $f=(f^0\hybto{r}f^1)\in\cM^\pair$ such that $f^1$ and $f^0$ are admissible disks on $\Pi_t$ and $\Pi_u$ with vertices $\{\bv_0,\cdots,\bv_t\}$ and $\{\bw_0,\cdots,\bw_u\}$, respectively, as before.

\begin{itemize}
\item [($\Gl$-1)] {\em regular ancestor pairs} Suppose that $f^1\in\cM^{\reg}$.  Then $f^0$ is either a regular disk, an infinitesimal monogon or an infinitesimal triangle.
We define 
\begin{align*}
\Gl_+&:\cM^{\reg\hybto{}\reg}\amalg\cM^{\infinitesimalmonogon\hybto{}\reg}
\to\cM^\positivefolding,\\
\Gl_-&:\cM^{\infinitesimaltriangle\hybto{}\reg}\to\cM^\positivefolding
\end{align*}
as follows: If $f\in \cM^{\reg\hybto{}\reg}$, then either 
\begin{itemize}
\item [(1)] $f^1(\bh_{\bv_r}^-)=f^0(\bh_{\bw_0}^+)$ and $f^0(\bh_{\bw_0}^-)\cup f^1(\bh_{\bv_r}^+)$ forms a smooth arc; or
\item [(2)] $f^1(\bh_{\bv_r}^+)=f^0(\bh_{\bw_0}^-)$ and $f^1(\bh_{\bv_r}^-)\cup f^0(\bh_{\bw_0}^+)$ forms a smooth arc.
\end{itemize}

We identify the two half-edges whose images coincide and making one smooth edge from the two half-edges whose image form a smooth arc.
Then we have a glued domain $\Pi$, a $(t+u-1)$-gon,
\[
\Pi\simeq \begin{cases}
\Pi_{t}\amalg\Pi_{u}\big/\left\{\bh_{\bv_r}^-\sim \bh_{\bw_0}^+\right\}, & \bh_{\bw_0}^-\cup\bh_{\bv_r}^+\subset\be_{r+u-1};\\
\Pi_{t}\amalg\Pi_{u}\big/\left\{\bh_{\bv_r}^+\sim \bh_{\bw_0}^-\right\}, &\bh_{\bv_r}^-\cup\bh_{\bw_0}^+\subset\be_{r-1},
\end{cases}
\]
and the disk $g$ is defined as the union of $f^0$ and $f^1$ on this domain. 
\[
g\coloneqq f^0\amalg f^1:(\Pi,\partial\Pi,\vPi)\to(\RR^2, \sL, \sS_\sL).
\]

Obviously, this produces exactly one $2$-folding edge coming from $\be_{r-1}=\bh_{\bv_{r-1}}^+\cup \bh_{\bw_1}^-$ or $\be_{r+u-1}=\bh_{\bw_u}^+\cup \bh_{\bv_{r+1}}^-$. 
Therefore all vertices for $g$ are convex and there is exactly one 2-folding edge, and so 
\begin{align*}
|g|_\ZZ&=1+\#(2\text{-folding edges})+\#(\text{concave vertices})=1+1+0=2.
\end{align*}
\[
\begin{tikzcd}[row sep=-2.5pc, column sep=1pc]
&\hspace{-.7cm}\vcenter{\hbox{\def\svgscale{1}
\begingroup%
  \makeatletter%
  \providecommand\color[2][]{%
    \errmessage{(Inkscape) Color is used for the text in Inkscape, but the package 'color.sty' is not loaded}%
    \renewcommand\color[2][]{}%
  }%
  \providecommand\transparent[1]{%
    \errmessage{(Inkscape) Transparency is used (non-zero) for the text in Inkscape, but the package 'transparent.sty' is not loaded}%
    \renewcommand\transparent[1]{}%
  }%
  \providecommand\rotatebox[2]{#2}%
  \ifx\svgwidth\undefined%
    \setlength{\unitlength}{85.04839463bp}%
    \ifx\svgscale\undefined%
      \relax%
    \else%
      \setlength{\unitlength}{\unitlength * \real{\svgscale}}%
    \fi%
  \else%
    \setlength{\unitlength}{\svgwidth}%
  \fi%
  \global\let\svgwidth\undefined%
  \global\let\svgscale\undefined%
  \makeatother%
  \begin{picture}(1,0.99999991)%
    \put(0,0){\includegraphics[width=\unitlength,page=1]{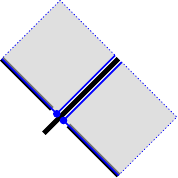}}%
    \put(0.59423391,0.28461113){\color[rgb]{0,0,0}\makebox(0,0)[lb]{\smash{$f^1$}}}%
    \put(0.25857498,0.64406985){\color[rgb]{0,0,0}\makebox(0,0)[lb]{\smash{$f^0$}}}%
    \put(0.42585738,0.2959908){\color[rgb]{0,0,0}\makebox(0,0)[lb]{\smash{$\bv_r$}}}%
  \end{picture}%
\endgroup%
}}\!\!\!\!\ar[rr,"\Gl_+"]&& g=\vcenter{\hbox{\def\svgscale{1}
\begingroup%
  \makeatletter%
  \providecommand\color[2][]{%
    \errmessage{(Inkscape) Color is used for the text in Inkscape, but the package 'color.sty' is not loaded}%
    \renewcommand\color[2][]{}%
  }%
  \providecommand\transparent[1]{%
    \errmessage{(Inkscape) Transparency is used (non-zero) for the text in Inkscape, but the package 'transparent.sty' is not loaded}%
    \renewcommand\transparent[1]{}%
  }%
  \providecommand\rotatebox[2]{#2}%
  \ifx\svgwidth\undefined%
    \setlength{\unitlength}{85.04839463bp}%
    \ifx\svgscale\undefined%
      \relax%
    \else%
      \setlength{\unitlength}{\unitlength * \real{\svgscale}}%
    \fi%
  \else%
    \setlength{\unitlength}{\svgwidth}%
  \fi%
  \global\let\svgwidth\undefined%
  \global\let\svgscale\undefined%
  \makeatother%
  \begin{picture}(1,1)%
    \put(0,0){\includegraphics[width=\unitlength,page=1]{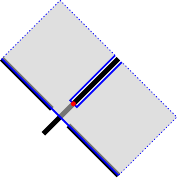}}%
    \put(0.56777839,0.42906645){\color[rgb]{0,0,0}\makebox(0,0)[lb]{\smash{$\be_{r-1}$}}}%
    \put(0.56260354,0.15989181){\color[rgb]{0,0,0}\makebox(0,0)[lb]{\smash{$\be_{r+u-1}$}}}%
  \end{picture}%
\endgroup%
}}\\
f=\vcenter{\hbox{\def\svgscale{1}
\begingroup%
  \makeatletter%
  \providecommand\color[2][]{%
    \errmessage{(Inkscape) Color is used for the text in Inkscape, but the package 'color.sty' is not loaded}%
    \renewcommand\color[2][]{}%
  }%
  \providecommand\transparent[1]{%
    \errmessage{(Inkscape) Transparency is used (non-zero) for the text in Inkscape, but the package 'transparent.sty' is not loaded}%
    \renewcommand\transparent[1]{}%
  }%
  \providecommand\rotatebox[2]{#2}%
  \ifx\svgwidth\undefined%
    \setlength{\unitlength}{132.39973309bp}%
    \ifx\svgscale\undefined%
      \relax%
    \else%
      \setlength{\unitlength}{\unitlength * \real{\svgscale}}%
    \fi%
  \else%
    \setlength{\unitlength}{\svgwidth}%
  \fi%
  \global\let\svgwidth\undefined%
  \global\let\svgscale\undefined%
  \makeatother%
  \begin{picture}(1,0.43577925)%
    \put(0,0){\includegraphics[width=\unitlength,page=1]{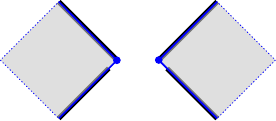}}%
    \put(0.54594389,0.27831281){\color[rgb]{0,0,0}\makebox(0,0)[lb]{\smash{$\bv_r$}}}%
    \put(0.78763622,0.18767819){\color[rgb]{0,0,0}\makebox(0,0)[lb]{\smash{$f^1$}}}%
    \put(0.46870946,0.21033684){\color[rgb]{0,0,0}\makebox(0,0)[lb]{\smash{$\hybto{r}$}}}%
    \put(0.17954187,0.18894462){\color[rgb]{0,0,0}\makebox(0,0)[lb]{\smash{$f^0$}}}%
  \end{picture}%
\endgroup%
}}\hspace{-.5cm}\ar[ru,equal]\ar[rd,equal]\\
&\hspace{-.7cm}\vcenter{\hbox{\def\svgscale{1}
\begingroup%
  \makeatletter%
  \providecommand\color[2][]{%
    \errmessage{(Inkscape) Color is used for the text in Inkscape, but the package 'color.sty' is not loaded}%
    \renewcommand\color[2][]{}%
  }%
  \providecommand\transparent[1]{%
    \errmessage{(Inkscape) Transparency is used (non-zero) for the text in Inkscape, but the package 'transparent.sty' is not loaded}%
    \renewcommand\transparent[1]{}%
  }%
  \providecommand\rotatebox[2]{#2}%
  \ifx\svgwidth\undefined%
    \setlength{\unitlength}{85.04839463bp}%
    \ifx\svgscale\undefined%
      \relax%
    \else%
      \setlength{\unitlength}{\unitlength * \real{\svgscale}}%
    \fi%
  \else%
    \setlength{\unitlength}{\svgwidth}%
  \fi%
  \global\let\svgwidth\undefined%
  \global\let\svgscale\undefined%
  \makeatother%
  \begin{picture}(1,1)%
    \put(0,0){\includegraphics[width=\unitlength,page=1]{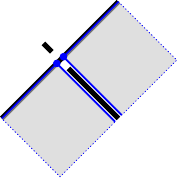}}%
    \put(0.61622765,0.64232188){\color[rgb]{0,0,0}\makebox(0,0)[lb]{\smash{$f^1$}}}%
    \put(0.25941478,0.30994793){\color[rgb]{0,0,0}\makebox(0,0)[lb]{\smash{$f^0$}}}%
    \put(0.43032591,0.64289945){\color[rgb]{0,0,0}\makebox(0,0)[lb]{\smash{$\bv_r$}}}%
  \end{picture}%
\endgroup%
}}\!\!\!\!\ar[rr,"\Gl_+"]&& g=\vcenter{\hbox{\def\svgscale{1}
\begingroup%
  \makeatletter%
  \providecommand\color[2][]{%
    \errmessage{(Inkscape) Color is used for the text in Inkscape, but the package 'color.sty' is not loaded}%
    \renewcommand\color[2][]{}%
  }%
  \providecommand\transparent[1]{%
    \errmessage{(Inkscape) Transparency is used (non-zero) for the text in Inkscape, but the package 'transparent.sty' is not loaded}%
    \renewcommand\transparent[1]{}%
  }%
  \providecommand\rotatebox[2]{#2}%
  \ifx\svgwidth\undefined%
    \setlength{\unitlength}{85.04839463bp}%
    \ifx\svgscale\undefined%
      \relax%
    \else%
      \setlength{\unitlength}{\unitlength * \real{\svgscale}}%
    \fi%
  \else%
    \setlength{\unitlength}{\svgwidth}%
  \fi%
  \global\let\svgwidth\undefined%
  \global\let\svgscale\undefined%
  \makeatother%
  \begin{picture}(1,1)%
    \put(0,0){\includegraphics[width=\unitlength,page=1]{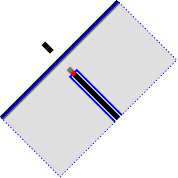}}%
    \put(0.58205593,0.77837895){\color[rgb]{0,0,0}\makebox(0,0)[lb]{\smash{$\be_{r-1}$}}}%
    \put(0.58024056,0.49030767){\color[rgb]{0,0,0}\makebox(0,0)[lb]{\smash{$\be_{r+u-1}$}}}%
  \end{picture}%
\endgroup%
}}
\end{tikzcd}
\]

If $f\in\cM^{\infinitesimalmonogon\hybto{}\reg}$, then $\tilde f^0(\bw_0)=\tilde f^1(\bv_r)=\sfv_{i,\ell}$ for some $i$ with $\ell=\val(\sfv)$. 
We remove neighborhoods $\bU_{\bw_0}$ and $\bU_{\bv_r}$ of $\bw_0$ and $\bv_r$ from $\Pi_{0}$ and $\Pi_{t}$, respectively, and glue the two complementary regions to obtain an admissible disk $g$ on $\Pi_{t-1}$. 
Then gluing the infinitesimal monogon replaces the vertex $\bv_r$ and the two adjacent edges $\be_{r-1}, \be_r$ with a single 2-folding edge $\be$ so that $|g|_\ZZ=2$.
\[
\begin{tikzcd}[column sep=2pc]
f=\vcenter{\hbox{
\begingroup%
  \makeatletter%
  \providecommand\color[2][]{%
    \errmessage{(Inkscape) Color is used for the text in Inkscape, but the package 'color.sty' is not loaded}%
    \renewcommand\color[2][]{}%
  }%
  \providecommand\transparent[1]{%
    \errmessage{(Inkscape) Transparency is used (non-zero) for the text in Inkscape, but the package 'transparent.sty' is not loaded}%
    \renewcommand\transparent[1]{}%
  }%
  \providecommand\rotatebox[2]{#2}%
  \ifx\svgwidth\undefined%
    \setlength{\unitlength}{122.60569651bp}%
    \ifx\svgscale\undefined%
      \relax%
    \else%
      \setlength{\unitlength}{\unitlength * \real{\svgscale}}%
    \fi%
  \else%
    \setlength{\unitlength}{\svgwidth}%
  \fi%
  \global\let\svgwidth\undefined%
  \global\let\svgscale\undefined%
  \makeatother%
  \begin{picture}(1,0.47633728)%
    \put(0,0){\includegraphics[width=\unitlength,page=1]{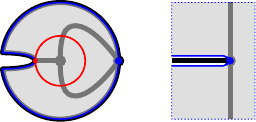}}%
    \put(0.81124664,0.28976872){\color[rgb]{0,0,0}\makebox(0,0)[lb]{\smash{$\bv_r$}}}%
    \put(0.53980953,0.22630156){\color[rgb]{0,0,0}\makebox(0,0)[lb]{\smash{$\hybto{r}$}}}%
    \put(0.73737065,0.07504438){\color[rgb]{0,0,0}\makebox(0,0)[lb]{\smash{$f^1$}}}%
    \put(0.11749734,0.07504438){\color[rgb]{0,0,0}\makebox(0,0)[lb]{\smash{$f^0$}}}%
  \end{picture}%
\endgroup%
}}\ar[r,"\Gl_+"]&
g=\vcenter{\hbox{
\begingroup%
  \makeatletter%
  \providecommand\color[2][]{%
    \errmessage{(Inkscape) Color is used for the text in Inkscape, but the package 'color.sty' is not loaded}%
    \renewcommand\color[2][]{}%
  }%
  \providecommand\transparent[1]{%
    \errmessage{(Inkscape) Transparency is used (non-zero) for the text in Inkscape, but the package 'transparent.sty' is not loaded}%
    \renewcommand\transparent[1]{}%
  }%
  \providecommand\rotatebox[2]{#2}%
  \ifx\svgwidth\undefined%
    \setlength{\unitlength}{74.99985386bp}%
    \ifx\svgscale\undefined%
      \relax%
    \else%
      \setlength{\unitlength}{\unitlength * \real{\svgscale}}%
    \fi%
  \else%
    \setlength{\unitlength}{\svgwidth}%
  \fi%
  \global\let\svgwidth\undefined%
  \global\let\svgscale\undefined%
  \makeatother%
  \begin{picture}(1,0.75199844)%
    \put(0,0){\includegraphics[width=\unitlength,page=1]{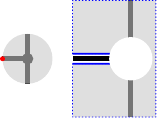}}%
    \put(0.34666734,0.36533429){\color[rgb]{0,0,0}\makebox(0,0)[lb]{\smash{$\sqcup$}}}%
    \put(0.48933429,0.47266791){\color[rgb]{0,0,0}\makebox(0,0)[lb]{\smash{$\be_{r-1}$}}}%
    \put(0.48933429,0.22266677){\color[rgb]{0,0,0}\makebox(0,0)[lb]{\smash{$\be_r$}}}%
  \end{picture}%
\endgroup%
}}=
\vcenter{\hbox{
\begingroup%
  \makeatletter%
  \providecommand\color[2][]{%
    \errmessage{(Inkscape) Color is used for the text in Inkscape, but the package 'color.sty' is not loaded}%
    \renewcommand\color[2][]{}%
  }%
  \providecommand\transparent[1]{%
    \errmessage{(Inkscape) Transparency is used (non-zero) for the text in Inkscape, but the package 'transparent.sty' is not loaded}%
    \renewcommand\transparent[1]{}%
  }%
  \providecommand\rotatebox[2]{#2}%
  \ifx\svgwidth\undefined%
    \setlength{\unitlength}{40.39971789bp}%
    \ifx\svgscale\undefined%
      \relax%
    \else%
      \setlength{\unitlength}{\unitlength * \real{\svgscale}}%
    \fi%
  \else%
    \setlength{\unitlength}{\svgwidth}%
  \fi%
  \global\let\svgwidth\undefined%
  \global\let\svgscale\undefined%
  \makeatother%
  \begin{picture}(1,1.39604373)%
    \put(0,0){\includegraphics[width=\unitlength,page=1]{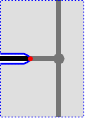}}%
    \put(0.05197719,0.82178851){\color[rgb]{0,0,0}\makebox(0,0)[lb]{\smash{$\be$}}}%
  \end{picture}%
\endgroup%
}}
\end{tikzcd}
\]

If $f\in \cM^{\infinitesimaltriangle\hybto{}\reg}$, then we remove $\bU_{\bw_0}$ and $\bU_{\bv_r}$ from $\Pi_{2}$ and from $\Pi_{t}$ and glue them to obtain $g$ on $\Pi_{t+1}$. 
Then $g$ is very similar to $f^1$ except that $g$ has one 2-folding edge and no concave vertices.
\[
\begin{tikzcd}
f=\vcenter{\hbox{
\begingroup%
  \makeatletter%
  \providecommand\color[2][]{%
    \errmessage{(Inkscape) Color is used for the text in Inkscape, but the package 'color.sty' is not loaded}%
    \renewcommand\color[2][]{}%
  }%
  \providecommand\transparent[1]{%
    \errmessage{(Inkscape) Transparency is used (non-zero) for the text in Inkscape, but the package 'transparent.sty' is not loaded}%
    \renewcommand\transparent[1]{}%
  }%
  \providecommand\rotatebox[2]{#2}%
  \ifx\svgwidth\undefined%
    \setlength{\unitlength}{142.07289079bp}%
    \ifx\svgscale\undefined%
      \relax%
    \else%
      \setlength{\unitlength}{\unitlength * \real{\svgscale}}%
    \fi%
  \else%
    \setlength{\unitlength}{\svgwidth}%
  \fi%
  \global\let\svgwidth\undefined%
  \global\let\svgscale\undefined%
  \makeatother%
  \begin{picture}(1,0.41010865)%
    \put(0,0){\includegraphics[width=\unitlength,page=1]{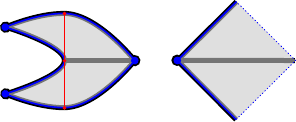}}%
    \put(0.57089741,0.26192902){\color[rgb]{0,0,0}\makebox(0,0)[lb]{\smash{$\bv_r$}}}%
    \put(0.79613391,0.11147808){\color[rgb]{0,0,0}\makebox(0,0)[lb]{\smash{$f^1$}}}%
    \put(0.24470463,0.09300165){\color[rgb]{0,0,0}\makebox(0,0)[lb]{\smash{$f^0$}}}%
    \put(0.49638779,0.19310741){\color[rgb]{0,0,0}\makebox(0,0)[lb]{\smash{$\hybto{r}$}}}%
  \end{picture}%
\endgroup%
}}\ar[r,"\Gl_-"]&
g=\vcenter{\hbox{
\begingroup%
  \makeatletter%
  \providecommand\color[2][]{%
    \errmessage{(Inkscape) Color is used for the text in Inkscape, but the package 'color.sty' is not loaded}%
    \renewcommand\color[2][]{}%
  }%
  \providecommand\transparent[1]{%
    \errmessage{(Inkscape) Transparency is used (non-zero) for the text in Inkscape, but the package 'transparent.sty' is not loaded}%
    \renewcommand\transparent[1]{}%
  }%
  \providecommand\rotatebox[2]{#2}%
  \ifx\svgwidth\undefined%
    \setlength{\unitlength}{72.30041861bp}%
    \ifx\svgscale\undefined%
      \relax%
    \else%
      \setlength{\unitlength}{\unitlength * \real{\svgscale}}%
    \fi%
  \else%
    \setlength{\unitlength}{\svgwidth}%
  \fi%
  \global\let\svgwidth\undefined%
  \global\let\svgscale\undefined%
  \makeatother%
  \begin{picture}(1,0.80565331)%
    \put(0,0){\includegraphics[width=\unitlength,page=1]{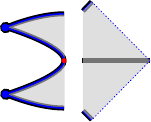}}%
    \put(0.44511254,0.36502219){\color[rgb]{0,0,0}\makebox(0,0)[lb]{\smash{$\sqcup$}}}%
  \end{picture}%
\endgroup%
}}=\vcenter{\hbox{\includegraphics{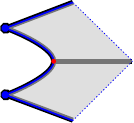}}}
\end{tikzcd}
\]

\item [($\Gl$-2)] {\em infinitesimal ancestor pairs.}
Suppose that $f^1\in\cM^{\infinitesimaltriangle}$. Then $r$ is either $1$ or $2$, and either $f^0\in\cM^{\infinitesimalmonogon}$ or $f^0\in\cM^\infinitesimaltriangle$.
We remove neighborhoods $\bU_{\bw_0}$ and $\bU_{\bv_r}$ of $\bw_0$ and $\bv_r$ from $\Pi_{0}$ and $\Pi_{t}$, respectively, and glue them to obtain an admissible disk $g$ as before. 

According to $r$ which is either 1 or 2 since $f^1$ is a triangle, we denote the gluing map by $\Gl_+$ if $r=1$ and $\Gl_-$ if $r=2$.
Then we have
\begin{align*}
\Gl_+&:\cM^{\infinitesimalmonogon\hybto{1}\infinitesimaltriangle}\to\cM^{\infinitesimalbigonright},&
\Gl_-&:\cM^{\infinitesimalmonogon\hybto{2}\infinitesimaltriangle}\to\cM^{\infinitesimalbigonleft},\\
\Gl_+&:\cM^{\infinitesimaltriangle\hybto{1}\infinitesimaltriangle}\to\cM^{\infinitesimalsquare},&
\Gl_-&:\cM^{\infinitesimaltriangle\hybto{2}\infinitesimaltriangle}\to\cM^{\infinitesimalsquare}.
\end{align*}

If $f^0\in\cM^{\infinitesimalmonogon}$, then $\tilde f^0(\bw_0)=\tilde f^1(\bv_r)=\sfv_{i,\ell}$ for some $i$ with $\ell=\val(\sfv)$. 
\[
\begin{tikzcd}
f=\vcenter{\hbox{\def\svgscale{0.95}
\begingroup%
  \makeatletter%
  \providecommand\color[2][]{%
    \errmessage{(Inkscape) Color is used for the text in Inkscape, but the package 'color.sty' is not loaded}%
    \renewcommand\color[2][]{}%
  }%
  \providecommand\transparent[1]{%
    \errmessage{(Inkscape) Transparency is used (non-zero) for the text in Inkscape, but the package 'transparent.sty' is not loaded}%
    \renewcommand\transparent[1]{}%
  }%
  \providecommand\rotatebox[2]{#2}%
  \ifx\svgwidth\undefined%
    \setlength{\unitlength}{146.35186549bp}%
    \ifx\svgscale\undefined%
      \relax%
    \else%
      \setlength{\unitlength}{\unitlength * \real{\svgscale}}%
    \fi%
  \else%
    \setlength{\unitlength}{\svgwidth}%
  \fi%
  \global\let\svgwidth\undefined%
  \global\let\svgscale\undefined%
  \makeatother%
  \begin{picture}(1,0.47357052)%
    \put(0,0){\includegraphics[width=\unitlength,page=1]{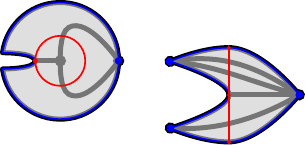}}%
    \put(0.53113702,0.31441802){\color[rgb]{0,0,0}\makebox(0,0)[lb]{\smash{$\bv_1$}}}%
    \put(0.4522233,0.26410408){\color[rgb]{0,0,0}\makebox(0,0)[lb]{\smash{$\hybto{1}$}}}%
    \put(0.09843293,0.13738897){\color[rgb]{0,0,0}\makebox(0,0)[lb]{\smash{$f^0$}}}%
    \put(0.78905433,0.07742515){\color[rgb]{0,0,0}\makebox(0,0)[lb]{\smash{$f^1$}}}%
  \end{picture}%
\endgroup%
}}\ar[r,"\Gl_+"] &
g=\vcenter{\hbox{\def\svgscale{0.95}
\begingroup%
  \makeatletter%
  \providecommand\color[2][]{%
    \errmessage{(Inkscape) Color is used for the text in Inkscape, but the package 'color.sty' is not loaded}%
    \renewcommand\color[2][]{}%
  }%
  \providecommand\transparent[1]{%
    \errmessage{(Inkscape) Transparency is used (non-zero) for the text in Inkscape, but the package 'transparent.sty' is not loaded}%
    \renewcommand\transparent[1]{}%
  }%
  \providecommand\rotatebox[2]{#2}%
  \ifx\svgwidth\undefined%
    \setlength{\unitlength}{71.70258294bp}%
    \ifx\svgscale\undefined%
      \relax%
    \else%
      \setlength{\unitlength}{\unitlength * \real{\svgscale}}%
    \fi%
  \else%
    \setlength{\unitlength}{\svgwidth}%
  \fi%
  \global\let\svgwidth\undefined%
  \global\let\svgscale\undefined%
  \makeatother%
  \begin{picture}(1,0.73512275)%
    \put(0,0){\includegraphics[width=\unitlength,page=1]{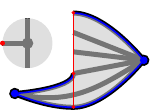}}%
    \put(0.36260897,0.43476433){\color[rgb]{0,0,0}\makebox(0,0)[lb]{\smash{$\sqcup$}}}%
    \put(0.51377867,0.69273847){\color[rgb]{0,0,0}\makebox(0,0)[lb]{\smash{$\be_0$}}}%
    \put(0.20877119,0.17868609){\color[rgb]{0,0,0}\makebox(0,0)[lb]{\smash{$\be_1$}}}%
  \end{picture}%
\endgroup%
}}=
\vcenter{\hbox{\def\svgscale{0.95}
\begingroup%
  \makeatletter%
  \providecommand\color[2][]{%
    \errmessage{(Inkscape) Color is used for the text in Inkscape, but the package 'color.sty' is not loaded}%
    \renewcommand\color[2][]{}%
  }%
  \providecommand\transparent[1]{%
    \errmessage{(Inkscape) Transparency is used (non-zero) for the text in Inkscape, but the package 'transparent.sty' is not loaded}%
    \renewcommand\transparent[1]{}%
  }%
  \providecommand\rotatebox[2]{#2}%
  \ifx\svgwidth\undefined%
    \setlength{\unitlength}{58.40839883bp}%
    \ifx\svgscale\undefined%
      \relax%
    \else%
      \setlength{\unitlength}{\unitlength * \real{\svgscale}}%
    \fi%
  \else%
    \setlength{\unitlength}{\svgwidth}%
  \fi%
  \global\let\svgwidth\undefined%
  \global\let\svgscale\undefined%
  \makeatother%
  \begin{picture}(1,1.23854206)%
    \put(0,0){\includegraphics[width=\unitlength,page=1]{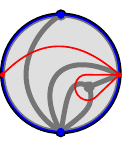}}%
    \put(0.41120877,0.01207152){\color[rgb]{0,0,0}\makebox(0,0)[lb]{\smash{$\bv_0$}}}%
    \put(0.41120877,1.18584202){\color[rgb]{0,0,0}\makebox(0,0)[lb]{\smash{$\bv_1$}}}%
  \end{picture}%
\endgroup%
}}
\end{tikzcd}
\]

If $f^0\in\cM^{\infinitesimaltriangle}$, then $g$ is an infinitesimal quadrilateral. 
\[
\begin{tikzcd}
f=\vcenter{\hbox{\def\svgscale{0.95}
\begingroup%
  \makeatletter%
  \providecommand\color[2][]{%
    \errmessage{(Inkscape) Color is used for the text in Inkscape, but the package 'color.sty' is not loaded}%
    \renewcommand\color[2][]{}%
  }%
  \providecommand\transparent[1]{%
    \errmessage{(Inkscape) Transparency is used (non-zero) for the text in Inkscape, but the package 'transparent.sty' is not loaded}%
    \renewcommand\transparent[1]{}%
  }%
  \providecommand\rotatebox[2]{#2}%
  \ifx\svgwidth\undefined%
    \setlength{\unitlength}{152.51820414bp}%
    \ifx\svgscale\undefined%
      \relax%
    \else%
      \setlength{\unitlength}{\unitlength * \real{\svgscale}}%
    \fi%
  \else%
    \setlength{\unitlength}{\svgwidth}%
  \fi%
  \global\let\svgwidth\undefined%
  \global\let\svgscale\undefined%
  \makeatother%
  \begin{picture}(1,0.41735865)%
    \put(0,0){\includegraphics[width=\unitlength,page=1]{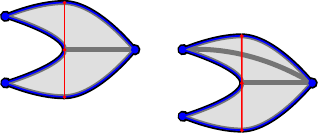}}%
    \put(0.53179911,0.31359203){\color[rgb]{0,0,0}\makebox(0,0)[lb]{\smash{$\bv_1$}}}%
    \put(0.79517516,0.07246143){\color[rgb]{0,0,0}\makebox(0,0)[lb]{\smash{$f^1$}}}%
    \put(0.22794586,0.15623377){\color[rgb]{0,0,0}\makebox(0,0)[lb]{\smash{$f^0$}}}%
    \put(0.46239233,0.24948372){\color[rgb]{0,0,0}\makebox(0,0)[lb]{\smash{$\hybto{1}$}}}%
  \end{picture}%
\endgroup%
}}\ar[r,"\Gl_+"]&
g=\vcenter{\hbox{\def\svgscale{0.95}
\begingroup%
  \makeatletter%
  \providecommand\color[2][]{%
    \errmessage{(Inkscape) Color is used for the text in Inkscape, but the package 'color.sty' is not loaded}%
    \renewcommand\color[2][]{}%
  }%
  \providecommand\transparent[1]{%
    \errmessage{(Inkscape) Transparency is used (non-zero) for the text in Inkscape, but the package 'transparent.sty' is not loaded}%
    \renewcommand\transparent[1]{}%
  }%
  \providecommand\rotatebox[2]{#2}%
  \ifx\svgwidth\undefined%
    \setlength{\unitlength}{78.09959942bp}%
    \ifx\svgscale\undefined%
      \relax%
    \else%
      \setlength{\unitlength}{\unitlength * \real{\svgscale}}%
    \fi%
  \else%
    \setlength{\unitlength}{\svgwidth}%
  \fi%
  \global\let\svgwidth\undefined%
  \global\let\svgscale\undefined%
  \makeatother%
  \begin{picture}(1,0.61017972)%
    \put(0,0){\includegraphics[width=\unitlength,page=1]{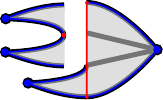}}%
    \put(0.40995765,0.35900665){\color[rgb]{0,0,0}\makebox(0,0)[lb]{\smash{$\sqcup$}}}%
  \end{picture}%
\endgroup%
}}=
\vcenter{\hbox{\def\svgscale{0.95}\includegraphics{quadrilateral_degree_2_2.pdf}}}
\end{tikzcd}
\]
\end{itemize}

By definition of $\sgn(f)$, it is easy to check that $\Gl_\pm$ preserves $\cP$ up to sign
\[
\cP(\Gl_\pm(f))=\pm\cP(f).
\]

\subsubsection{Shrinking the 2-folding edge}\label{section:stretching and shrinking}
Let $f\in \cM^\positivefolding$ be a regular admissible disk of degree 2 with the unique 2-folding edge $\be=\be_{r+1}$ whose ending vertices are $\{\bv_r, \bv_{r+1}\}$.

\begin{itemize}
\item [($\Sh$-1)] {\em Shrink to pairs.}
Suppose that $f(\bv_r)=f(\bv_{r+1})$. Then they should be the same vertex $\sfv\in\sV_\sL$, and by shrinking $\be$, its image eventually converges to the point $f(\bv_r)=f(\bv_{r+1})$.
We define $\Sh_+(f)\in\cM^{\infinitesimaltriangle\hybto{}\reg}$ as the admissible pair $g=(f^0\hybto{r}f^1)$ defined as follows.
The disk $f^1$ is obtained from $f$ by shrinking $\be$ to a point so that it has no folding edges, and $f^0$ is an infinitesimal admissible triangle such that
\begin{align*}
f^0(\be_0)&\subset f(\be_r),&
f^0(\be_1)&\subset f(\be_{r+1}),&
f^0(\be_2)&\subset f(\be_{r+2}).
\end{align*}
\[
\begin{tikzcd}
f=\vcenter{\hbox{\includegraphics{pair_admissible_inf_regular_glued.pdf}}}\ar[r,"Shrink"]&
\vcenter{\hbox{\includegraphics{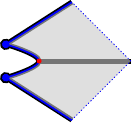}}}\ar[r,"\Sh_-"]&
g=\vcenter{\hbox{}}
\end{tikzcd}
\]

\item [($\Sh$-2)] {\em Shrink to concave vertices.}
Suppose that $f(\bv_r)\neq f(\bv_{r+1})$. Then one of $f(\bv_r)$ or $f(\bv_{r+1})$ is in $\sC_\sL$ and by shrinking $\be$, we define $\Sh_\pm(f)\in\cM^{\concavevertex}$ to be the corresponding admissible disk $g$ having a unique concave (not necessarily negative) vertex $\bv$ without any folding edges. Here we use the subscripts $\pm$ to distinguish two cases as follows:
\[
\begin{tikzcd}[row sep=0pc]
f=\vcenter{\hbox{\includegraphics{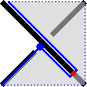}}}\ar[r,"\Sh_-"]&
g=\vcenter{\hbox{\includegraphics{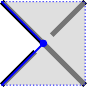}}}&
f=\vcenter{\hbox{\includegraphics{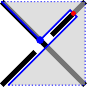}}}\ar[l,"\Sh_+"']
\\
f=\vcenter{\hbox{\includegraphics{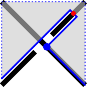}}}\ar[r,"\Sh_-"]&
g=\vcenter{\hbox{\includegraphics{crossing_sign_concave_positive.pdf}}}&
f=\vcenter{\hbox{\includegraphics{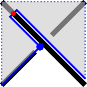}}}\ar[l,"\Sh_+"']
\end{tikzcd}
\]

By the convention of the orientation signs as described in Figure~\ref{figure:convex and concave vertices}, we have
\begin{align*}
\sgn(\Sh_\pm(f))&=\pm\sgn(f),&\cP(\Sh_\pm(f))&=\pm\cP(f).
\end{align*}
\end{itemize}
This is the inverse operation of $\Gl_+$ on $\cM^{\infinitesimaltriangle\hybto{}\reg}$, and therefore $\cP$ is preserved
\[
\cP(\Sh_+(f))=\cP(f).
\]

\subsubsection{Stretching folding edges}\label{section:stretching}
On the other hand, for a regular admissible disk $f$, the stretching along $\be$ is defined by enlarging the image of the folding edge $\be$ as long as it is combinatorially the same as the original disk $f$. 
In this case, the result becomes always an admissible pair. 
\[
\begin{tikzcd}[column sep=3pc]
\vcenter{\hbox{\includegraphics{folding_edge_2_sh_1.pdf}}}\ar[r,"Stretch"]&
\vcenter{\hbox{\includegraphics{folding_edge_2.pdf}}}\ar[r,"Stretch"]&
\vcenter{\hbox{\includegraphics{folding_edge_2_sh_2.pdf}}}\ar[r,"Stretch"]&
\cdots
\end{tikzcd}
\]

We eventually get the following three cases and define $\St(f)$ to be the admissible pair $g$ depicted as follows:
\begin{itemize} 
\item [($\St$-1)] 
The singular point on $\be$ hits $\be$ itself. 
\[
\begin{tikzcd}[row sep=0pc]
f=\vcenter{\hbox{\includegraphics{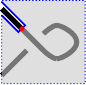}}}\ar[r,"Stretch"]&
\vcenter{\hbox{\includegraphics{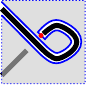}}}\ar[r,"\St"]&
g=\vcenter{\hbox{
\begingroup%
  \makeatletter%
  \providecommand\color[2][]{%
    \errmessage{(Inkscape) Color is used for the text in Inkscape, but the package 'color.sty' is not loaded}%
    \renewcommand\color[2][]{}%
  }%
  \providecommand\transparent[1]{%
    \errmessage{(Inkscape) Transparency is used (non-zero) for the text in Inkscape, but the package 'transparent.sty' is not loaded}%
    \renewcommand\transparent[1]{}%
  }%
  \providecommand\rotatebox[2]{#2}%
  \ifx\svgwidth\undefined%
    \setlength{\unitlength}{41.04840303bp}%
    \ifx\svgscale\undefined%
      \relax%
    \else%
      \setlength{\unitlength}{\unitlength * \real{\svgscale}}%
    \fi%
  \else%
    \setlength{\unitlength}{\svgwidth}%
  \fi%
  \global\let\svgwidth\undefined%
  \global\let\svgscale\undefined%
  \makeatother%
  \begin{picture}(1,0.98419785)%
    \put(0,0){\includegraphics[width=\unitlength,page=1]{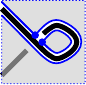}}%
    \put(0.29354177,0.74919985){\color[rgb]{0,0,0}\makebox(0,0)[lb]{\smash{$\bv_r$}}}%
  \end{picture}%
\endgroup%
}}=
\vcenter{\hbox{
\begingroup%
  \makeatletter%
  \providecommand\color[2][]{%
    \errmessage{(Inkscape) Color is used for the text in Inkscape, but the package 'color.sty' is not loaded}%
    \renewcommand\color[2][]{}%
  }%
  \providecommand\transparent[1]{%
    \errmessage{(Inkscape) Transparency is used (non-zero) for the text in Inkscape, but the package 'transparent.sty' is not loaded}%
    \renewcommand\transparent[1]{}%
  }%
  \providecommand\rotatebox[2]{#2}%
  \ifx\svgwidth\undefined%
    \setlength{\unitlength}{80.22829351bp}%
    \ifx\svgscale\undefined%
      \relax%
    \else%
      \setlength{\unitlength}{\unitlength * \real{\svgscale}}%
    \fi%
  \else%
    \setlength{\unitlength}{\svgwidth}%
  \fi%
  \global\let\svgwidth\undefined%
  \global\let\svgscale\undefined%
  \makeatother%
  \begin{picture}(1,0.51973006)%
    \put(0,0){\includegraphics[width=\unitlength,page=1]{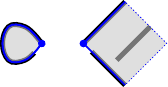}}%
    \put(0.451358,0.35455593){\color[rgb]{0,0,0}\makebox(0,0)[lb]{\smash{$\bv_r$}}}%
    \put(0.31965893,0.24627484){\color[rgb]{0,0,0}\makebox(0,0)[lb]{\smash{$\hybto{r}$}}}%
  \end{picture}%
\endgroup%
}}\\
f=\vcenter{\hbox{\includegraphics{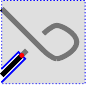}}}\ar[r,"Stretch"]&
\vcenter{\hbox{\includegraphics{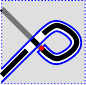}}}\ar[r,"\St"]&
g=\vcenter{\hbox{
\begingroup%
  \makeatletter%
  \providecommand\color[2][]{%
    \errmessage{(Inkscape) Color is used for the text in Inkscape, but the package 'color.sty' is not loaded}%
    \renewcommand\color[2][]{}%
  }%
  \providecommand\transparent[1]{%
    \errmessage{(Inkscape) Transparency is used (non-zero) for the text in Inkscape, but the package 'transparent.sty' is not loaded}%
    \renewcommand\transparent[1]{}%
  }%
  \providecommand\rotatebox[2]{#2}%
  \ifx\svgwidth\undefined%
    \setlength{\unitlength}{41.04838643bp}%
    \ifx\svgscale\undefined%
      \relax%
    \else%
      \setlength{\unitlength}{\unitlength * \real{\svgscale}}%
    \fi%
  \else%
    \setlength{\unitlength}{\svgwidth}%
  \fi%
  \global\let\svgwidth\undefined%
  \global\let\svgscale\undefined%
  \makeatother%
  \begin{picture}(1,0.98419825)%
    \put(0,0){\includegraphics[width=\unitlength,page=1]{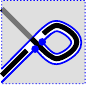}}%
    \put(0.29354303,0.17719009){\color[rgb]{0,0,0}\makebox(0,0)[lb]{\smash{$\bv_r$}}}%
  \end{picture}%
\endgroup%
}}=
\vcenter{\hbox{
\begingroup%
  \makeatletter%
  \providecommand\color[2][]{%
    \errmessage{(Inkscape) Color is used for the text in Inkscape, but the package 'color.sty' is not loaded}%
    \renewcommand\color[2][]{}%
  }%
  \providecommand\transparent[1]{%
    \errmessage{(Inkscape) Transparency is used (non-zero) for the text in Inkscape, but the package 'transparent.sty' is not loaded}%
    \renewcommand\transparent[1]{}%
  }%
  \providecommand\rotatebox[2]{#2}%
  \ifx\svgwidth\undefined%
    \setlength{\unitlength}{80.22829351bp}%
    \ifx\svgscale\undefined%
      \relax%
    \else%
      \setlength{\unitlength}{\unitlength * \real{\svgscale}}%
    \fi%
  \else%
    \setlength{\unitlength}{\svgwidth}%
  \fi%
  \global\let\svgwidth\undefined%
  \global\let\svgscale\undefined%
  \makeatother%
  \begin{picture}(1,0.51973006)%
    \put(0,0){\includegraphics[width=\unitlength,page=1]{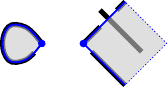}}%
    \put(0.451358,0.35455593){\color[rgb]{0,0,0}\makebox(0,0)[lb]{\smash{$\bv_r$}}}%
    \put(0.31965893,0.24627484){\color[rgb]{0,0,0}\makebox(0,0)[lb]{\smash{$\hybto{r}$}}}%
  \end{picture}%
\endgroup%
}}
\end{tikzcd}
\]

\item [($\St$-2)]
The singular point on $\be$ hits another edge of $\partial\Pi$. Then 
\[
\begin{tikzcd}[row sep=-2.5pc,column sep=1pc]
f=\vcenter{\hbox{}}\hspace{-0.2cm}\ar[r,"\St"]&\hspace{-0.2cm}
\vcenter{\hbox{}}\ar[rd,equal]\\
& &\hspace{-0.4cm}g=\vcenter{\hbox{}}\\
f=\vcenter{\hbox{}}\hspace{-0.2cm}\ar[r,"\St"]&\hspace{-0.2cm}
\vcenter{\hbox{}}\ar[ru,equal]
\end{tikzcd}
\]
\item [($\St$-3)]
The singular point on $\be$ hits a point $\bx\in \rPi$ so that $f(\bx)$ is a vertex.
\[
\begin{tikzcd}
f=\vcenter{\hbox{}}\ar[r,"\St"]&
g=\vcenter{\hbox{}}
\end{tikzcd}
\]
\end{itemize}

The above operations are inverses of gluing operations which preserve labels and therefore
\begin{align*}
\cP(\St(f))=\cP(f).
\end{align*}

\subsubsection{Cutting disks}\label{sec:Cutting concave vertices in two ways}
There are three kinds of cutting disks, which are inverses of certain gluing operations.

\begin{itemize}
\item [($\Cut$-1)] {\em Cutting concave vertices in two ways.} Let us define
\[
\Cut_\pm:\cM^{\concavevertex}\to\cM^\positivefolding.
\]

Let $f\in \cM^{\concavevertex}$ be a non-infinitesimal admissible disk with unique concave vertex $\bv_r$ without any folding edge.
Then $\Cut_{\pm}(f)$ is defined by creating a folding edge at $\bv_r$ in two ways as depicted as follows:

By changing one of edges $\be_r$ and $\be_{r+1}$ with a folding edge, we obtain an admissible disk $g\in\cM^\positivefolding$.
\[
\begin{tikzcd}
g=\vcenter{\hbox{\includegraphics{concave_cut_1.pdf}}}
&\ar[l,"\Cut_-"']f=\vcenter{\hbox{\includegraphics{crossing_sign_concave_negative_2.pdf}}}\ar[r,"\Cut_+"]&
g=\vcenter{\hbox{\includegraphics{concave_cut_2.pdf}}}\\
g=\vcenter{\hbox{\includegraphics{concave_cut_positive_1.pdf}}}
&\ar[l,"\Cut_-"']f=\vcenter{\hbox{\includegraphics{crossing_sign_concave_positive.pdf}}}\ar[r,"\Cut_+"]&
g=\vcenter{\hbox{\includegraphics{concave_cut_positive_2.pdf}}}
\end{tikzcd}
\]

These are inverses of the operations $\Sh_\pm$ on $\cM^{\concavevertex}$ and so we have
\begin{align*}
\cP(\Cut_\pm(f))&=\pm\cP(f).
\end{align*}

\item [($\Cut$-2)] {\em Cutting along the arc in $\bO_f$.}
We define functions
\begin{align*}
\Cut_+&:\cM^{\infinitesimalbigonright}\to\cM^{\infinitesimalmonogon\hybto{1}\infinitesimaltriangle},&
\Cut_-&:\cM^{\infinitesimalbigonleft}\to\cM^{\infinitesimalmonogon\hybto{2}\infinitesimaltriangle}\\
\Cut_+&:\cM^{\infinitesimalsquare}\to\cM^{\infinitesimaltriangle\hybto{1}\infinitesimaltriangle},&
\Cut_-&:\cM^{\infinitesimalsquare}\to\cM^{\infinitesimaltriangle\hybto{2}\infinitesimaltriangle}.
\end{align*}

For $f\in\cM^{\infinitesimalbigonleft}\amalg\cM^{\infinitesimalbigonright}$, we decompose $f$ into two parts cutting along the unique circuit in $\bO_f$.
From the component containing $\bv_0$, there is a unique way of extending to an infinitesimal admissible triangle, say $f^1$.
There is also a unique way to obtain an infinitesimal admissible monogon from the other component, say $f^0$.
Then $\Cut_\pm(f)$ is defined to be the admissible pair $g=(f^0\hybto{r}f^1)$.
\[
\begin{tikzcd}
f=\vcenter{\hbox{}} = 
\vcenter{\hbox{}}\ar[r,"\Cut_+"]&
g=\vcenter{\hbox{}}
\end{tikzcd}
\]

If $f\in\cM^{\infinitesimalsquare}$ is an infinitesimal quadrilateral, then we can decompose $f$ into a pair of infinitesimal triangles in two ways as depicted below.
For each case, there is a unique way of assigning an admissible pair $g=(f^0\hybto{r}f^1)$ and we define them to be $\Cut_\pm(f)$, respectively:
\[
\begin{tikzcd}[row sep=-2pc,column sep=1.5pc]
& \vcenter{\hbox{}}\ar[r,"\Cut_+"] & g=\vcenter{\hbox{}}\\
f=\vcenter{\hbox{\includegraphics{quadrilateral_degree_2_2.pdf}}}\hspace{-.7cm}\ar[ru,equal]\ar[rd,equal]\\
& \vcenter{\hbox{
\begingroup%
  \makeatletter%
  \providecommand\color[2][]{%
    \errmessage{(Inkscape) Color is used for the text in Inkscape, but the package 'color.sty' is not loaded}%
    \renewcommand\color[2][]{}%
  }%
  \providecommand\transparent[1]{%
    \errmessage{(Inkscape) Transparency is used (non-zero) for the text in Inkscape, but the package 'transparent.sty' is not loaded}%
    \renewcommand\transparent[1]{}%
  }%
  \providecommand\rotatebox[2]{#2}%
  \ifx\svgwidth\undefined%
    \setlength{\unitlength}{78.09959942bp}%
    \ifx\svgscale\undefined%
      \relax%
    \else%
      \setlength{\unitlength}{\unitlength * \real{\svgscale}}%
    \fi%
  \else%
    \setlength{\unitlength}{\svgwidth}%
  \fi%
  \global\let\svgwidth\undefined%
  \global\let\svgscale\undefined%
  \makeatother%
  \begin{picture}(1,0.61017972)%
    \put(0,0){\includegraphics[width=\unitlength,page=1]{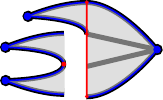}}%
    \put(0.40995765,0.22291366){\color[rgb]{0,0,0}\makebox(0,0)[lb]{\smash{$\sqcup$}}}%
  \end{picture}%
\endgroup%
}}\ar[r,"\Cut_-"] & g=\vcenter{\hbox{
\begingroup%
  \makeatletter%
  \providecommand\color[2][]{%
    \errmessage{(Inkscape) Color is used for the text in Inkscape, but the package 'color.sty' is not loaded}%
    \renewcommand\color[2][]{}%
  }%
  \providecommand\transparent[1]{%
    \errmessage{(Inkscape) Transparency is used (non-zero) for the text in Inkscape, but the package 'transparent.sty' is not loaded}%
    \renewcommand\transparent[1]{}%
  }%
  \providecommand\rotatebox[2]{#2}%
  \ifx\svgwidth\undefined%
    \setlength{\unitlength}{152.51820414bp}%
    \ifx\svgscale\undefined%
      \relax%
    \else%
      \setlength{\unitlength}{\unitlength * \real{\svgscale}}%
    \fi%
  \else%
    \setlength{\unitlength}{\svgwidth}%
  \fi%
  \global\let\svgwidth\undefined%
  \global\let\svgscale\undefined%
  \makeatother%
  \begin{picture}(1,0.41737255)%
    \put(0,0){\includegraphics[width=\unitlength,page=1]{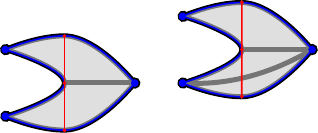}}%
    \put(0.53179911,0.0900397){\color[rgb]{0,0,0}\makebox(0,0)[lb]{\smash{$\bv_2$}}}%
    \put(0.79517516,0.30772824){\color[rgb]{0,0,0}\makebox(0,0)[lb]{\smash{$f^1$}}}%
    \put(0.22794586,0.05132826){\color[rgb]{0,0,0}\makebox(0,0)[lb]{\smash{$f^0$}}}%
    \put(0.46239233,0.14457821){\color[rgb]{0,0,0}\makebox(0,0)[lb]{\smash{$\hybto{2}$}}}%
  \end{picture}%
\endgroup%
}}
\end{tikzcd}
\]

Since these are the inverse operations of gluing maps described in ($\Gl$-2), we have
\begin{align*}
\cP(\Cut_\pm(f))&=\pm\cP(f).
\end{align*}
\end{itemize}

\subsubsection{Rolling the circuit of $\bO_f$}
Finally, we define
\begin{align*}
\Rol^{\sL}&:\cM^{\infinitesimalbigonleft}\amalg\cM^{\infinitesimalbigonmiddle}\to\cM^{\infinitesimalbigonmiddle}\amalg\cM^{\infinitesimalbigonright},&
\Rol^{\sR}&:\cM^{\infinitesimalbigonmiddle}\amalg\cM^{\infinitesimalbigonright}\to\cM^{\infinitesimalbigonleft}\amalg\cM^{\infinitesimalbigonmiddle}.
\end{align*}

\begin{itemize}
\item [($\Rol$)]
Let $f$ be an infinitesimal admissible bigon of degree 2, that is, $f\in\cM^{\infinitesimalbigonleft}\amalg\cM^{\infinitesimalbigonmiddle}\amalg\cM^{\infinitesimalbigonright}$.
Then $\Rol^*(f)$ is defined by rolling the unique circuit in $\bO_f$ to the right and left, respectively, along $\bO_f$ with respect to $\bv_0$. 

\[
\begin{tikzcd}[column sep=2.5pc]
\vcenter{\hbox{
\begingroup%
  \makeatletter%
  \providecommand\color[2][]{%
    \errmessage{(Inkscape) Color is used for the text in Inkscape, but the package 'color.sty' is not loaded}%
    \renewcommand\color[2][]{}%
  }%
  \providecommand\transparent[1]{%
    \errmessage{(Inkscape) Transparency is used (non-zero) for the text in Inkscape, but the package 'transparent.sty' is not loaded}%
    \renewcommand\transparent[1]{}%
  }%
  \providecommand\rotatebox[2]{#2}%
  \ifx\svgwidth\undefined%
    \setlength{\unitlength}{58.40840098bp}%
    \ifx\svgscale\undefined%
      \relax%
    \else%
      \setlength{\unitlength}{\unitlength * \real{\svgscale}}%
    \fi%
  \else%
    \setlength{\unitlength}{\svgwidth}%
  \fi%
  \global\let\svgwidth\undefined%
  \global\let\svgscale\undefined%
  \makeatother%
  \begin{picture}(1,1.23854202)%
    \put(0,0){\includegraphics[width=\unitlength,page=1]{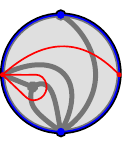}}%
    \put(0.41120876,0.01207152){\color[rgb]{0,0,0}\makebox(0,0)[lb]{\smash{$\bv_0$}}}%
    \put(0.41120876,1.18584198){\color[rgb]{0,0,0}\makebox(0,0)[lb]{\smash{$\bv_1$}}}%
  \end{picture}%
\endgroup%
}}\in\cM^{\infinitesimalbigonleft}\ar[r,"\Rol^\sL",shift left=.5ex]\ &\ 
\vcenter{\hbox{}}\in\cM^{\infinitesimalbigonmiddle}\ar[r,"\Rol^\sL",shift left=.5ex]\ar[l,"\Rol^\sR",shift left=.5ex]\ &\ 
\vcenter{\hbox{}}\in\cM^{\infinitesimalbigonright}\ar[l,"\Rol^\sR",shift left=.5ex]
\end{tikzcd}
\]
\end{itemize}

Since all disks have the same sign, the operations $\Rol^*$ preserve $\cP$
\begin{align*}
\cP(\Rol^\sL(f))&=\cP(f),&
\cP(\Rol^\sR(f))&=\cP(f).
\end{align*}

\subsection{Proof of Theorem~\ref{thm:differential}}\label{sec:The proof of differential}
In summary, we have the following lemma.
\begin{lemma}\label{lemma:operations preserve labels}
All operations preserve $\cP$ up to sign. Moreover, the only operations which reverse the sign of $\cP$ are operations having the subscript $(-)$.
\end{lemma}
\begin{proof}
This is the summary of the above discussion and we omit the proof.
\end{proof}

We define a graph whose vertices are admissible disks and pairs of degree 2, called the {\em moduli graph}. 

\begin{definition}\label{def:moduli graph for non-hybrid disks}
For a given $\sfg\in \sG_\sL$, let us construct a simple, undirected, and edge-labeled graph $\mathbf G_\sfg$ such that
\begin{align*}
\mathbf{VG}_\sfg&\coloneqq\cM(\sfg)_{(2)}\amalg\cM^\pair(\sfg)\\
&=\cM^\concavevertex(\sfg) \amalg \cM^\positivefolding(\sfg) \amalg \cM^{\infinitesimalbigonleft}(\sfg) \amalg \cM^{\infinitesimalbigonmiddle}(\sfg) \amalg \cM^{\infinitesimalbigonright}(\sfg)\amalg \cM^\infinitesimalsquare\amalg \cM^\pair(\sfg).
\end{align*}

For two of vertices $f,g$ in $\bV\bG_\sfg$, we assign an edge $\{f,g\}\in \bE\bG_\sfg$ if there is a map 
\[
\mathbf{Q}\in \{\Gl_\pm, \Sh_\pm, \St, \Cut_\pm, \Rol^\sL,\Rol^\sR\}
\]
such that $g=\mathbf Q(f)$. A label $\mathbf P:\mathbf{EG}_\sfg\to \{+1,-1\}$ is defined when $g=\mathbf Q(f)$ by
\begin{align*}
\mathbf P(\{f,g\})&=\epsilon,&
\cP(f)&=\epsilon\cdot \cP(g).
\end{align*} 
\end{definition}

Notice that the moduli graph $\bG_\sfg$ has no cycles. We write $\bV_\ext\bG_\sfg$ for the set of {\em exterior vertices} of $\mathbf G_\sfg$, i.e., vertices whose valency is one, and write $\bV_\intr\bG_\sfg$ for vertices of valency (at least) two.
As seen in Table~\ref{tab:subsets and operations for nonhybrid}, we have that
\begin{align*}
\bV_\ext\bG_\sfg&=\cM^\pair(\sfg),&
\bV_\intr\bG_\sfg&=\cM_{(2)}(\sfg).
\end{align*}

\begin{table}[ht]
\[\def\arraystretch{1.5}
\begin{array}{c||c|c|c|c|c|c|c||c}
& \Gl_\pm & \Sh_\pm & \St & \Cut_+ & \Cut_- & \Rol_+ & \Rol_- & \val_{\bG_\sfg}\\
\hline\hline
\cM^\pair & \bigcirc & & & & & & & 1\\
\hline
\cM^{\positivefolding} & & \bigcirc & \bigcirc & & & & & 2\\
\hline
\cM^{\concavevertex} & & & & \bigcirc & \bigcirc & & & 2\\
\hline
\cM^{\infinitesimalsquare} & & & & \bigcirc & \bigcirc & & & 2\\
\hline
\cM^{\infinitesimalbigonleft} & & & & \bigcirc & & \bigcirc & & 2\\
\hline
\cM^{\infinitesimalbigonright} & & & & & \bigcirc & & \bigcirc & 2\\
\hline
\cM^{\infinitesimalbigonmiddle} & & & & & & \bigcirc & \bigcirc & 2
\end{array}
\]
\caption{Subsets and operations}
\label{tab:subsets and operations for nonhybrid}
\end{table}

\begin{proposition}\label{proposition:differential2}
For $\sfg\in\sG_\sL$, we have
\[
\sum_{f\in\cM^\pair(\sfg)} \cP(f)=0.
\]
\end{proposition}
\begin{proof}
Since $\bG_\sfg$ has no cycles and no vertices of valency 3 or higher, each connected component of $\bG_\sfg$ is a path joining two elements in $\cM^\pair(\sfg)$ 
\begin{align*}
&\begin{tikzcd}[ampersand replacement=\&]
f_1\ar[r,"\epsilon_1",dash]\&f_2\ar[r,"\epsilon_2",dash]\&\cdots \ar[r,"\epsilon_m",dash]\& f_m\subset\bG_\sfg,
\end{tikzcd}&
\cP(f_1)&=(\epsilon_1\cdots\epsilon_m) \cP(f_m)\in\cA_\Lambda.
\end{align*}
However, it is not hard to check that there exists exactly one edge with label $(-1)$ at every path and therefore the values of $\cP$ for both ends have the opposite sign. This completes the proof.
\end{proof}

\begin{proof}[Proof of Theorem~{\rm\ref{thm:differential}}]
It is easy to see that 
\begin{align}\label{equation:differential}
\partial^2(\sfg)=\sum_{f\in\cM^\pair(\sfg)} \cP(f)
\end{align}
and by Proposition~\ref{proposition:differential2} it is zero as claimed.
\end{proof}

\section{Hybrid admissible pairs and proof of Proposition~\ref{proposition:commutativity}}
\label{appendix:hybrid admissible pairs}
We first define the labels for all types of hybrid disks similarly as follows: for $f\in\scrM_{t,s}$,
\begin{align*}
\scrP(f)&\coloneqq \begin{cases}
0 & f\in \scrM^{\vertexbirigid};\\
\sgn(f)\Psi\left(\tilde f(\bv_1\cdots\bv_s)\right) \tilde f(\bv_{s+1}\cdots \bv_t) & \text{otherwise},
\end{cases}\\
\sgn(f)&\coloneqq\begin{cases}
\sgn(f,\bv_0\cdots\bv_t)(-1)^{|\tilde f(\bv_1\cdots\bv_r)|-1} & r<\separator;\\
-\sgn(f,\bv_0\cdots\bv_t)(-1)^{|\tilde f(\bv_1\cdots\bv_r)|-1} & r>\separator,
\end{cases}
\end{align*}
where the index $r$ is the position of the positive folding edge, the concave vertex if exists, or $0$ otherwise.

\begin{definition}[Hybrid admissible pairs]\label{definition:hybrid admissible pair}
Let $f^1, f^0\in \scrM$ be hybrid admissible disks on $\Pi_t$ and $\Pi_u$ with separator $\separator^i$. Denote the vertices of $\Pi_t$ and $\Pi_u$ by $\{\bv_i\}$ and $\{\bw_i\}$, respectively.

A {\em hybrid admissible pair $f$} is a triple $(f^0,f^1,r)$, denoted by $f\coloneqq(f^0\hybto{r} f^1)$, of hybrid admissible disks $f^0, f^1\in\scrM$ and $1\le r\le t$ such that $\tilde f^0(\bw_0)=\tilde f^1(\bv_r)$ and one of the following holds:
\begin{enumerate}
\item $r<\separator^1$ and $(f^0,f^1)\in\scrM^{\rightmost}\times \scrM^{\negativefolding}$;
\item $r>\separator^1$ and $(f^0,f^1)\in\scrM^{\leftmost}\times \scrM^{\negativefolding}$;
\item $r=\separator^1$.
\end{enumerate}

The {\em $\ZZ$-degree} and {\em $\fR$-degree} of the pair $f$ are the sums of corresponding degrees of $f^0$ and $f^1$
\begin{align*}
|f|_\ZZ&\coloneqq |f^0|_\ZZ+|f^1|_\ZZ,&
|f|&\coloneqq |f^0|+|f^1| = |f|_\ZZ\cdot 1_{\fR}\in\fR.
\end{align*}
\end{definition}

\begin{notation}
We denote the sets of all hybrid admissible pairs by $\scrM^\pair$ of $\ZZ$-degree 1, and 
\begin{align*}
\scrM^\pair(\bar\sfg)&\coloneqq \left\{f\in\scrM^\pair\,\middle\vert\,\tilde f^1(\bv_0)=\bar\sfg\right\},&
\scrM_{(d^0\hybto{}d^1)}^{\pair}&\coloneqq \left\{f\in\scrM^\pair\,\middle\vert\,|f^i|_\ZZ=d^i\right\}.
\end{align*}
\end{notation}

\begin{remark}It is obvious that hybrid admissible pairs are contained in $\scrM^\pair_{(1\hybto{}0)}$ if and only if they satisfy the condition (1) or (2) in Definition~\ref{definition:hybrid admissible pair}. Therefore all hybrid admissible pairs in $\scrM^\pair_{(0\hybto{}1)}$ satisfy the condition (3) and {\it vice versa}.
\end{remark}

\begin{definition}\label{definition:polynomial}
Suppose that a hybrid admissible pair $f=(f^0\hybto{r} f^1)\in\scrM^\pair$ is given for some $f^1\in\scrM_{t,s}$. Let us define functions
\[
\sgn(f)\coloneqq 
\sgn(f^1)\sgn(f^0) (-1)^{|\tilde f^1(\bv_1\cdots\bv_{r-1})|}
\]
and $\scrP:\scrM^\pair\to A$ as follows:
\begin{enumerate}
\item If $r<\separator^1$, then
\[
\scrP(f)\coloneqq \sgn(f)
\Psi\left(\tilde f^1(\bv_1\cdots\bv_{r-1}) \tilde f^0(\bw_1\cdots\bw_u) \tilde f^1(\bv_{r+1}\cdots\bv_s)\right)
\tilde f^1(\bv_{s+1}\cdots\bv_t).
\]
\item if $\separator^1<r$, then
\[
\scrP(f)\coloneqq -\sgn(f)
\Psi\left(\tilde f^1(\bv_1\cdots\bv_s)\right)
\tilde f^1(\bv_{s+1}\cdots\bv_{r-1}) \tilde f^0(\bw_1\cdots\bw_u) \tilde f^1(\bv_{r+1}\cdots\bv_t).
\]
\item if $\separator^1=r$, then
\[
\scrP(f)\coloneqq \sgn(f)
\Psi\left(\tilde f^1(\bv_1\cdots\bv_{s-1})\right)
\tilde f^0(\bw_1\cdots\bw_u)\tilde f^1(\bv_{s+1}\cdots\bv_t).
\]
\end{enumerate}
\end{definition}

With the above definitions, the hybrid disks and pair with the map $\scrP$, we now restate Proposition~\ref{proposition:commutativity}(2).
By the induction hypothesis, the left hand side of (\ref{eqn:Psi commutes differential}) is
\begin{align}
(\partial\circ\Psi_{[i]})(\sfg_i')&=\begin{cases}\label{eqn:(LHS)}
\partial(\sfg_i)+(\partial\circ\Psi^\hyb)(\sfg_i') & \sfg_i'\neq\sfc_i';\\
(\partial\circ\Psi^\hyb)(\sfc_i') & \sfg_i'=\sfc_i',
\end{cases}
\end{align}
where
\begin{align*}
\partial(\sfg_i)&=\sum_{\substack{t\ge0\\f\in\cM_t(\sfg_i)_{(1)}}} \sgn(f)\tilde f(\bv_1\cdots \bv_t);\\
(\partial\circ\Psi^\hyb)(\sfg_i')&=\sum_{\substack {t\ge s\ge 0\\f\in\scrM_{t,s}(\bar\sfg_i)_{(0)}}} \sgn(f)\Psi_{[i-1]}\left(\partial'\left(\tilde f(\bv_1\cdots\bv_s)\right)\right) \tilde f(\bv_{s+1}\cdots\bv_t)\\
&\mathrel{\hphantom{=}}+\sum_{\substack {t\ge s\ge 0\\f\in\scrM_{t,s}(\bar\sfg_i)_{(0)}}} \sgn(f)
(-1)^{|\tilde f(\bv_1\cdots\bv_s)|}
\Psi_{[i-1]}\left(\tilde f(\bv_1\cdots\bv_s)\right)\partial\left(\tilde f(\bv_{s+1}\cdots\bv_t)\right).
\end{align*}

By Lemma~\ref{lemma:AB_into_hybrid}, Definition~\ref{definition:polynomial of hybrid disk}, and Definition~\ref{definition:polynomial}, we have
\begin{align*}
\partial(\sfg_i)&=\sum_{f\in\scrM^{\leftmost}(\bar\sfg_i)}\scrP(f),&
(\partial\circ\Psi^\hyb)(\sfg_i')&=\sum_{f\in\scrM^\pair_{(1\hybto{}0)}(\bar\sfg_i)} \scrP(f).
\end{align*}
Since $\scrM^{\leftmost}(\bar\sfc_i)=\emptyset$ by Remark~\ref{remark:leftmost c_i empty}, we need not separate the cases as in (\ref{eqn:(LHS)}).

On the other hand, the right hand side of (\ref{eqn:Psi commutes differential}) becomes
\begin{align}\label{eqn:(RHS)}
\Psi_{[i]}(\partial' \sfg_i')=\Psi_{[i-1]}(\partial' \sfg_i')=\sum_{\substack{t\ge 0\\f\in\cM(\sfg_i')_{(1)}}} \Psi_{[i-1]}\left(\tilde f(\bv_1\cdots\bv_t)\right)=\sum_{f\in\scrM^{\rightmost}(\bar\sfg_i)}\scrP(f).
\end{align}
The first equaility comes from the fact that $\partial'\sfg_i'\in A_{[i-1]}$.

Now we have a new proposition which is equivalent to Proposition~\ref{proposition:commutativity}(2) as follows:

\begin{proposition}\label{proposition:new commutativity}
Suppose that $\Psi_{[i-1]}$ commutes with differentials. Then
\[
\sum_{f\in\scrM^{\leftmost}(\bar\sfg_i)}\scrP(f)+
\sum_{\substack{f\in\scrM^\pair_{(1\hybto{}0)}(\bar\sfg_i)\\ r<\separator^1}} \scrP(f)
-
\sum_{\substack{f\in\scrM^\pair_{(1\hybto{}0)}(\bar\sfg_i)\\ \separator^1<r}} \scrP(f)
-\sum_{f\in\scrM^{\rightmost}(\bar\sfg_i)}\scrP(f)=0.
\]
\end{proposition}

\subsection{Manipulations of hybrid disks and pairs}\label{sec:Manipulations of hybrid disks and pairs}
Let $f^1, f^0\in \scrM$ be hybrid admissible disks on $\Pi_t$ and $\Pi_u$ with separator $\separator^i$ as before.

\subsubsection{Gluing of pairs}
Let us define 
\[
\Gl_+:\scrM^\pair_{(1\hybto{}0)}\amalg\scrM^\pair_{(0\hybto{}1)} \to \scrM^{\pmfolding}.
\]

For $f=(f^0\hybto{r} f^1)\in\scrM^\pair_{(1\hybto{}0)}$, we define $\Gl_+(f)$ by applying the gluing process ($\Gl$-1) of the admissible pairs, described in \S~\ref{sec:Gluing of pairs}, near $\bv_r\in \Pi_t$ and $\bw_0\in \Pi_u$. The separator of $\Gl_+(f)$ is defined to be $\separator^1$.
\begin{itemize}
\item [($\Gl$-3)] 
If $f=(f^0\hybto{r} f^1)\in\scrM^\pair_{(0\hybto{}1)}$, then $f^0$ is hybrid and we obtain a glued disk $g$ by exactly the same process as ($\Gl$-1) in \S~\ref{sec:Gluing of pairs} again.
But in this case, the separator of $g$ is given by $\separator^0$. Note that $f$ has one more $(+)$-folding region compared to $f^1$ and hence $f\in\scrM^{\pmfolding}$.
\[
\begin{tikzcd}[row sep=-2.5pc, column sep=1pc]
&\hspace{-.7cm}\vcenter{\hbox{\def\svgscale{1}
\begingroup%
  \makeatletter%
  \providecommand\color[2][]{%
    \errmessage{(Inkscape) Color is used for the text in Inkscape, but the package 'color.sty' is not loaded}%
    \renewcommand\color[2][]{}%
  }%
  \providecommand\transparent[1]{%
    \errmessage{(Inkscape) Transparency is used (non-zero) for the text in Inkscape, but the package 'transparent.sty' is not loaded}%
    \renewcommand\transparent[1]{}%
  }%
  \providecommand\rotatebox[2]{#2}%
  \ifx\svgwidth\undefined%
    \setlength{\unitlength}{79.74536192bp}%
    \ifx\svgscale\undefined%
      \relax%
    \else%
      \setlength{\unitlength}{\unitlength * \real{\svgscale}}%
    \fi%
  \else%
    \setlength{\unitlength}{\svgwidth}%
  \fi%
  \global\let\svgwidth\undefined%
  \global\let\svgscale\undefined%
  \makeatother%
  \begin{picture}(1,1.00000009)%
    \put(0,0){\includegraphics[width=\unitlength,page=1]{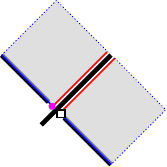}}%
    \put(0.43669211,0.28456712){\color[rgb]{0,0,0}\makebox(0,0)[lb]{\smash{$\bv_r$}}}%
    \put(0.61073407,0.34329975){\color[rgb]{0,0,0}\makebox(0,0)[lb]{\smash{$f^1$}}}%
    \put(0.28156133,0.62544782){\color[rgb]{0,0,0}\makebox(0,0)[lb]{\smash{$f^0$}}}%
  \end{picture}%
\endgroup%
}}\!\!\!\!\ar[r,"\Gl_+"]& g=\vcenter{\hbox{\def\svgscale{1}
\begingroup%
  \makeatletter%
  \providecommand\color[2][]{%
    \errmessage{(Inkscape) Color is used for the text in Inkscape, but the package 'color.sty' is not loaded}%
    \renewcommand\color[2][]{}%
  }%
  \providecommand\transparent[1]{%
    \errmessage{(Inkscape) Transparency is used (non-zero) for the text in Inkscape, but the package 'transparent.sty' is not loaded}%
    \renewcommand\transparent[1]{}%
  }%
  \providecommand\rotatebox[2]{#2}%
  \ifx\svgwidth\undefined%
    \setlength{\unitlength}{85.04839463bp}%
    \ifx\svgscale\undefined%
      \relax%
    \else%
      \setlength{\unitlength}{\unitlength * \real{\svgscale}}%
    \fi%
  \else%
    \setlength{\unitlength}{\svgwidth}%
  \fi%
  \global\let\svgwidth\undefined%
  \global\let\svgscale\undefined%
  \makeatother%
  \begin{picture}(1,1)%
    \put(0,0){\includegraphics[width=\unitlength,page=1]{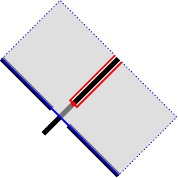}}%
    \put(0.56777839,0.42906645){\color[rgb]{0,0,0}\makebox(0,0)[lb]{\smash{$\be_{r-1}$}}}%
    \put(0.53208323,0.19059125){\color[rgb]{0,0,0}\makebox(0,0)[lb]{\smash{$\be_{r+u-1}$}}}%
  \end{picture}%
\endgroup%
}}\\
f=\vcenter{\hbox{\def\svgscale{1}
\begingroup%
  \makeatletter%
  \providecommand\color[2][]{%
    \errmessage{(Inkscape) Color is used for the text in Inkscape, but the package 'color.sty' is not loaded}%
    \renewcommand\color[2][]{}%
  }%
  \providecommand\transparent[1]{%
    \errmessage{(Inkscape) Transparency is used (non-zero) for the text in Inkscape, but the package 'transparent.sty' is not loaded}%
    \renewcommand\transparent[1]{}%
  }%
  \providecommand\rotatebox[2]{#2}%
  \ifx\svgwidth\undefined%
    \setlength{\unitlength}{132.39973309bp}%
    \ifx\svgscale\undefined%
      \relax%
    \else%
      \setlength{\unitlength}{\unitlength * \real{\svgscale}}%
    \fi%
  \else%
    \setlength{\unitlength}{\svgwidth}%
  \fi%
  \global\let\svgwidth\undefined%
  \global\let\svgscale\undefined%
  \makeatother%
  \begin{picture}(1,0.43577925)%
    \put(0,0){\includegraphics[width=\unitlength,page=1]{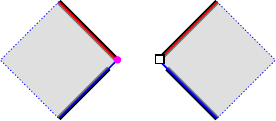}}%
    \put(0.78763622,0.18767819){\color[rgb]{0,0,0}\makebox(0,0)[lb]{\smash{$f^1$}}}%
    \put(0.17954187,0.18894462){\color[rgb]{0,0,0}\makebox(0,0)[lb]{\smash{$f^0$}}}%
    \put(0.46870946,0.21033684){\color[rgb]{0,0,0}\makebox(0,0)[lb]{\smash{$\hybto{r}$}}}%
    \put(0.54594389,0.27831281){\color[rgb]{0,0,0}\makebox(0,0)[lb]{\smash{$\bv_r$}}}%
  \end{picture}%
\endgroup%
}}\hspace{-.5cm}\ar[ru,equal]\ar[rd,equal]\\
&\hspace{-.7cm}\vcenter{\hbox{\def\svgscale{1}
\begingroup%
  \makeatletter%
  \providecommand\color[2][]{%
    \errmessage{(Inkscape) Color is used for the text in Inkscape, but the package 'color.sty' is not loaded}%
    \renewcommand\color[2][]{}%
  }%
  \providecommand\transparent[1]{%
    \errmessage{(Inkscape) Transparency is used (non-zero) for the text in Inkscape, but the package 'transparent.sty' is not loaded}%
    \renewcommand\transparent[1]{}%
  }%
  \providecommand\rotatebox[2]{#2}%
  \ifx\svgwidth\undefined%
    \setlength{\unitlength}{85.04839463bp}%
    \ifx\svgscale\undefined%
      \relax%
    \else%
      \setlength{\unitlength}{\unitlength * \real{\svgscale}}%
    \fi%
  \else%
    \setlength{\unitlength}{\svgwidth}%
  \fi%
  \global\let\svgwidth\undefined%
  \global\let\svgscale\undefined%
  \makeatother%
  \begin{picture}(1,1)%
    \put(0,0){\includegraphics[width=\unitlength,page=1]{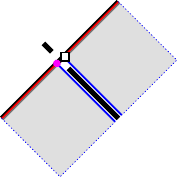}}%
    \put(0.42332278,0.64232188){\color[rgb]{0,0,0}\makebox(0,0)[lb]{\smash{$\bv_r$}}}%
    \put(0.62223354,0.65407989){\color[rgb]{0,0,0}\makebox(0,0)[lb]{\smash{$f^1$}}}%
    \put(0.29823703,0.28860898){\color[rgb]{0,0,0}\makebox(0,0)[lb]{\smash{$f^0$}}}%
  \end{picture}%
\endgroup%
}}\!\!\!\!\ar[r,"\Gl_+"]& g=\vcenter{\hbox{\def\svgscale{1}
\begingroup%
  \makeatletter%
  \providecommand\color[2][]{%
    \errmessage{(Inkscape) Color is used for the text in Inkscape, but the package 'color.sty' is not loaded}%
    \renewcommand\color[2][]{}%
  }%
  \providecommand\transparent[1]{%
    \errmessage{(Inkscape) Transparency is used (non-zero) for the text in Inkscape, but the package 'transparent.sty' is not loaded}%
    \renewcommand\transparent[1]{}%
  }%
  \providecommand\rotatebox[2]{#2}%
  \ifx\svgwidth\undefined%
    \setlength{\unitlength}{85.04839463bp}%
    \ifx\svgscale\undefined%
      \relax%
    \else%
      \setlength{\unitlength}{\unitlength * \real{\svgscale}}%
    \fi%
  \else%
    \setlength{\unitlength}{\svgwidth}%
  \fi%
  \global\let\svgwidth\undefined%
  \global\let\svgscale\undefined%
  \makeatother%
  \begin{picture}(1,1)%
    \put(0,0){\includegraphics[width=\unitlength,page=1]{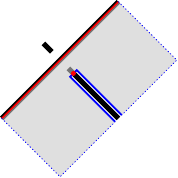}}%
    \put(0.58205593,0.77837895){\color[rgb]{0,0,0}\makebox(0,0)[lb]{\smash{$\be_{r-1}$}}}%
    \put(0.56970887,0.50278219){\color[rgb]{0,0,0}\makebox(0,0)[lb]{\smash{$\be_{r+u-1}$}}}%
  \end{picture}%
\endgroup%
}}
\end{tikzcd}
\]
\end{itemize}

Notice that the gluing map on hybrid admissible pairs always preserves the sign of $\scrP$. That is,
\[
\scrP(\Gl_+(f))=\scrP(f).
\]

\subsubsection{Shrinking of $(+)$-folding edges}
We define 
\[
\Sh_\pm:\scrM^{\pmfolding}\to \scrM^\pair_{(1\hybto{}0)} \amalg \scrM^{\negativefoldingconcave}  \amalg \scrM^{\leftrigid} \amalg \scrM^{\rightrigid}.
\]

For $f\in\scrM^{\pmfolding}$, $f(\separator)$ is $\bar\sfp$ or $\bar\sfq$ in $\bar\sL$ as discussed in Remark~\ref{rmk:separator}.
If the folding edge avoids the separator, then one of $\Sh_\pm(f)$ is defined by applying one of ($\Sh$-1) and ($\Sh$-2) in \S~\ref{section:stretching and shrinking}.
Indeed, we have $\Sh_-(f)=(f^0\hybto{r}f^1)$ when we apply ($\Sh$-1). In particular, $\separator^1$ inherits from $\separator$ and $f^0$ is infinitesimal so we do not need to care about $\separator^0$. Thus $\Sh_-(f) \in \scrM^\pair_{(1\hybto{}0)}$. 

On the other hand, if we apply ($\Sh$-2), then we have $\Sh_\pm(f) \in \scrM^{\negativefoldingconcave}$ by fixing the separator.

\begin{enumerate}
\item[($\Sh$-3)] Suppose that the $(+)$-folding edge $\be_r$ contains the separator $\separator$. If $f(\separator)=\bar\sfq$, then there is another inverse image of $\bar\sfq$ on $\be_r$. Let us define $\Sh(f)$ by eliminating the $(\pm)$-folding region simultaneously and by preserving the separator, i.e., $\Sh_+(f)(\separator)=\bar\sfq$. The similar procedure holds for the case of $f(\separator)=\bar\sfp$. Then $\Sh_+(f)\in \scrM^{\leftrigid}\amalg \scrM^{\rightrigid}$.
\[
\begin{tikzcd}
f=\vcenter{\hbox{
\begingroup%
  \makeatletter%
  \providecommand\color[2][]{%
    \errmessage{(Inkscape) Color is used for the text in Inkscape, but the package 'color.sty' is not loaded}%
    \renewcommand\color[2][]{}%
  }%
  \providecommand\transparent[1]{%
    \errmessage{(Inkscape) Transparency is used (non-zero) for the text in Inkscape, but the package 'transparent.sty' is not loaded}%
    \renewcommand\transparent[1]{}%
  }%
  \providecommand\rotatebox[2]{#2}%
  \ifx\svgwidth\undefined%
    \setlength{\unitlength}{44.40317599bp}%
    \ifx\svgscale\undefined%
      \relax%
    \else%
      \setlength{\unitlength}{\unitlength * \real{\svgscale}}%
    \fi%
  \else%
    \setlength{\unitlength}{\svgwidth}%
  \fi%
  \global\let\svgwidth\undefined%
  \global\let\svgscale\undefined%
  \makeatother%
  \begin{picture}(1,0.94095575)%
    \put(0,0){\includegraphics[width=\unitlength,page=1]{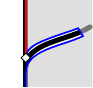}}%
    \put(-0.00528928,0.69276895){\color[rgb]{0,0,0}\makebox(0,0)[lb]{\smash{$\bar\sL_0'$}}}%
    \put(0.05442845,0.16463888){\color[rgb]{0,0,0}\makebox(0,0)[lb]{\smash{$\bar\sfq$}}}%
    \put(0.61746044,0.21497256){\color[rgb]{0,0,0}\makebox(0,0)[lb]{\smash{$\bar\sL_0$}}}%
  \end{picture}%
\endgroup%
}}\qquad\vcenter{\hbox{
\begingroup%
  \makeatletter%
  \providecommand\color[2][]{%
    \errmessage{(Inkscape) Color is used for the text in Inkscape, but the package 'color.sty' is not loaded}%
    \renewcommand\color[2][]{}%
  }%
  \providecommand\transparent[1]{%
    \errmessage{(Inkscape) Transparency is used (non-zero) for the text in Inkscape, but the package 'transparent.sty' is not loaded}%
    \renewcommand\transparent[1]{}%
  }%
  \providecommand\rotatebox[2]{#2}%
  \ifx\svgwidth\undefined%
    \setlength{\unitlength}{44.13531268bp}%
    \ifx\svgscale\undefined%
      \relax%
    \else%
      \setlength{\unitlength}{\unitlength * \real{\svgscale}}%
    \fi%
  \else%
    \setlength{\unitlength}{\svgwidth}%
  \fi%
  \global\let\svgwidth\undefined%
  \global\let\svgscale\undefined%
  \makeatother%
  \begin{picture}(1,0.94666654)%
    \put(0,0){\includegraphics[width=\unitlength,page=1]{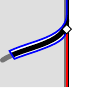}}%
    \put(0.86048151,0.20491033){\color[rgb]{0,0,0}\makebox(0,0)[lb]{\smash{$\bar\sL_0'$}}}%
    \put(0.13669012,0.57332994){\color[rgb]{0,0,0}\makebox(0,0)[lb]{\smash{$\bar\sL_0$}}}%
    \put(0.85920812,0.73388071){\color[rgb]{0,0,0}\makebox(0,0)[lb]{\smash{$\bar\sfp$}}}%
  \end{picture}%
\endgroup%
}}\in\scrM^{\pmfolding}\ar[r,"\Sh_+"]&
g=\vcenter{\hbox{
\begingroup%
  \makeatletter%
  \providecommand\color[2][]{%
    \errmessage{(Inkscape) Color is used for the text in Inkscape, but the package 'color.sty' is not loaded}%
    \renewcommand\color[2][]{}%
  }%
  \providecommand\transparent[1]{%
    \errmessage{(Inkscape) Transparency is used (non-zero) for the text in Inkscape, but the package 'transparent.sty' is not loaded}%
    \renewcommand\transparent[1]{}%
  }%
  \providecommand\rotatebox[2]{#2}%
  \ifx\svgwidth\undefined%
    \setlength{\unitlength}{44.40317599bp}%
    \ifx\svgscale\undefined%
      \relax%
    \else%
      \setlength{\unitlength}{\unitlength * \real{\svgscale}}%
    \fi%
  \else%
    \setlength{\unitlength}{\svgwidth}%
  \fi%
  \global\let\svgwidth\undefined%
  \global\let\svgscale\undefined%
  \makeatother%
  \begin{picture}(1,0.94095575)%
    \put(0,0){\includegraphics[width=\unitlength,page=1]{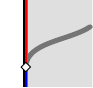}}%
    \put(-0.00528928,0.69276895){\color[rgb]{0,0,0}\makebox(0,0)[lb]{\smash{$\bar\sL_0'$}}}%
    \put(0.05442845,0.16463888){\color[rgb]{0,0,0}\makebox(0,0)[lb]{\smash{$\bar\sfq$}}}%
    \put(0.61746044,0.21497256){\color[rgb]{0,0,0}\makebox(0,0)[lb]{\smash{$\bar\sL_0$}}}%
  \end{picture}%
\endgroup%
}}\qquad\vcenter{\hbox{
\begingroup%
  \makeatletter%
  \providecommand\color[2][]{%
    \errmessage{(Inkscape) Color is used for the text in Inkscape, but the package 'color.sty' is not loaded}%
    \renewcommand\color[2][]{}%
  }%
  \providecommand\transparent[1]{%
    \errmessage{(Inkscape) Transparency is used (non-zero) for the text in Inkscape, but the package 'transparent.sty' is not loaded}%
    \renewcommand\transparent[1]{}%
  }%
  \providecommand\rotatebox[2]{#2}%
  \ifx\svgwidth\undefined%
    \setlength{\unitlength}{43.91974176bp}%
    \ifx\svgscale\undefined%
      \relax%
    \else%
      \setlength{\unitlength}{\unitlength * \real{\svgscale}}%
    \fi%
  \else%
    \setlength{\unitlength}{\svgwidth}%
  \fi%
  \global\let\svgwidth\undefined%
  \global\let\svgscale\undefined%
  \makeatother%
  \begin{picture}(1,0.95131306)%
    \put(0,0){\includegraphics[width=\unitlength,page=1]{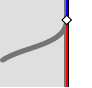}}%
    \put(0.85979671,0.20618284){\color[rgb]{0,0,0}\makebox(0,0)[lb]{\smash{$\bar\sL_0'$}}}%
    \put(0.85851698,0.73775061){\color[rgb]{0,0,0}\makebox(0,0)[lb]{\smash{$\bar\sfp$}}}%
    \put(0.13736101,0.57641178){\color[rgb]{0,0,0}\makebox(0,0)[lb]{\smash{$\bar\sL_0$}}}%
  \end{picture}%
\endgroup%
}}\in\scrM^{\leftrigid}\\
f=\vcenter{\hbox{
\begingroup%
  \makeatletter%
  \providecommand\color[2][]{%
    \errmessage{(Inkscape) Color is used for the text in Inkscape, but the package 'color.sty' is not loaded}%
    \renewcommand\color[2][]{}%
  }%
  \providecommand\transparent[1]{%
    \errmessage{(Inkscape) Transparency is used (non-zero) for the text in Inkscape, but the package 'transparent.sty' is not loaded}%
    \renewcommand\transparent[1]{}%
  }%
  \providecommand\rotatebox[2]{#2}%
  \ifx\svgwidth\undefined%
    \setlength{\unitlength}{52.21544929bp}%
    \ifx\svgscale\undefined%
      \relax%
    \else%
      \setlength{\unitlength}{\unitlength * \real{\svgscale}}%
    \fi%
  \else%
    \setlength{\unitlength}{\svgwidth}%
  \fi%
  \global\let\svgwidth\undefined%
  \global\let\svgscale\undefined%
  \makeatother%
  \begin{picture}(1,1.10793083)%
    \put(0,0){\includegraphics[width=\unitlength,page=1]{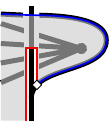}}%
    \put(0.06517453,0.16903534){\color[rgb]{0,0,0}\makebox(0,0)[lb]{\smash{$\bar\sfp$}}}%
    \put(0.49885236,0.23772115){\color[rgb]{0,0,0}\makebox(0,0)[lb]{\smash{$\bar\sL_0$}}}%
    \put(0.24471916,1.08468358){\color[rgb]{0,0,0}\makebox(0,0)[lb]{\smash{$\bar\sL_0'$}}}%
  \end{picture}%
\endgroup%
}}\ \vcenter{\hbox{
\begingroup%
  \makeatletter%
  \providecommand\color[2][]{%
    \errmessage{(Inkscape) Color is used for the text in Inkscape, but the package 'color.sty' is not loaded}%
    \renewcommand\color[2][]{}%
  }%
  \providecommand\transparent[1]{%
    \errmessage{(Inkscape) Transparency is used (non-zero) for the text in Inkscape, but the package 'transparent.sty' is not loaded}%
    \renewcommand\transparent[1]{}%
  }%
  \providecommand\rotatebox[2]{#2}%
  \ifx\svgwidth\undefined%
    \setlength{\unitlength}{51.81444759bp}%
    \ifx\svgscale\undefined%
      \relax%
    \else%
      \setlength{\unitlength}{\unitlength * \real{\svgscale}}%
    \fi%
  \else%
    \setlength{\unitlength}{\svgwidth}%
  \fi%
  \global\let\svgwidth\undefined%
  \global\let\svgscale\undefined%
  \makeatother%
  \begin{picture}(1,1.05889539)%
    \put(0,0){\includegraphics[width=\unitlength,page=1]{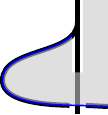}}%
    \put(0.08354196,0.62374524){\color[rgb]{0,0,0}\makebox(0,0)[lb]{\smash{$\bar\sL_0$}}}%
    \put(0.81595554,0.36785288){\color[rgb]{0,0,0}\makebox(0,0)[lb]{\smash{$\bar\sL_0'$}}}%
    \put(0.82336597,0.83052653){\color[rgb]{0,0,0}\makebox(0,0)[lb]{\smash{$\bar\sfp$}}}%
    \put(0,0){\includegraphics[width=\unitlength,page=2]{sh_4.pdf}}%
  \end{picture}%
\endgroup%
}}\in\scrM^{\pmfolding}\ar[r,"\Sh_+"]&
g=\vcenter{\hbox{
\begingroup%
  \makeatletter%
  \providecommand\color[2][]{%
    \errmessage{(Inkscape) Color is used for the text in Inkscape, but the package 'color.sty' is not loaded}%
    \renewcommand\color[2][]{}%
  }%
  \providecommand\transparent[1]{%
    \errmessage{(Inkscape) Transparency is used (non-zero) for the text in Inkscape, but the package 'transparent.sty' is not loaded}%
    \renewcommand\transparent[1]{}%
  }%
  \providecommand\rotatebox[2]{#2}%
  \ifx\svgwidth\undefined%
    \setlength{\unitlength}{52.21545569bp}%
    \ifx\svgscale\undefined%
      \relax%
    \else%
      \setlength{\unitlength}{\unitlength * \real{\svgscale}}%
    \fi%
  \else%
    \setlength{\unitlength}{\svgwidth}%
  \fi%
  \global\let\svgwidth\undefined%
  \global\let\svgscale\undefined%
  \makeatother%
  \begin{picture}(1,1.10793069)%
    \put(0,0){\includegraphics[width=\unitlength,page=1]{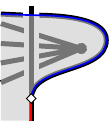}}%
    \put(0.49885207,0.23772112){\color[rgb]{0,0,0}\makebox(0,0)[lb]{\smash{$\bar\sL_0$}}}%
    \put(0.24471913,1.08468344){\color[rgb]{0,0,0}\makebox(0,0)[lb]{\smash{$\bar\sL_0'$}}}%
    \put(0.06517452,0.16903532){\color[rgb]{0,0,0}\makebox(0,0)[lb]{\smash{$\bar\sfp$}}}%
  \end{picture}%
\endgroup%
}}\ \vcenter{\hbox{
\begingroup%
  \makeatletter%
  \providecommand\color[2][]{%
    \errmessage{(Inkscape) Color is used for the text in Inkscape, but the package 'color.sty' is not loaded}%
    \renewcommand\color[2][]{}%
  }%
  \providecommand\transparent[1]{%
    \errmessage{(Inkscape) Transparency is used (non-zero) for the text in Inkscape, but the package 'transparent.sty' is not loaded}%
    \renewcommand\transparent[1]{}%
  }%
  \providecommand\rotatebox[2]{#2}%
  \ifx\svgwidth\undefined%
    \setlength{\unitlength}{51.81444759bp}%
    \ifx\svgscale\undefined%
      \relax%
    \else%
      \setlength{\unitlength}{\unitlength * \real{\svgscale}}%
    \fi%
  \else%
    \setlength{\unitlength}{\svgwidth}%
  \fi%
  \global\let\svgwidth\undefined%
  \global\let\svgscale\undefined%
  \makeatother%
  \begin{picture}(1,1.05889539)%
    \put(0,0){\includegraphics[width=\unitlength,page=1]{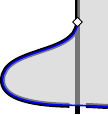}}%
    \put(0.08354196,0.62374524){\color[rgb]{0,0,0}\makebox(0,0)[lb]{\smash{$\bar\sL_0$}}}%
    \put(0.81595554,0.36785288){\color[rgb]{0,0,0}\makebox(0,0)[lb]{\smash{$\bar\sL_0'$}}}%
    \put(0.82336597,0.83052653){\color[rgb]{0,0,0}\makebox(0,0)[lb]{\smash{$\bar\sfp$}}}%
  \end{picture}%
\endgroup%
}}\in\scrM^{\rightrigid}\\
\end{tikzcd}
\]

Recall the sign convention defined in Definition~\ref{definition:polynomial of hybrid disk}. Then $\Sh_+$ preserves the sign and so $\scrP$ as well
\[
\scrP(\Sh_+(f))=\scrP(f).
\]
\end{enumerate}

\subsubsection{Stretching of $(+)$-folding edges}
We define
\[
\St:\scrM^{\pmfolding}\to \scrM^\pair_{(1\hybto{}0)} \amalg \scrM^\pair_{(0\hybto{}1)}.
\]

For $f\in \scrM^{\pmfolding}$, let $\bx\in \be_r$ be a singular point of the $(+)$-folding edge $\be_r$.
As we did in \S~\ref{section:stretching}, let us stretch $\be_r$ along either $f^{-1}(\sL')$ or $f^{-1}(\sL)$ until it hits another point $\by\in\be_s\subset\ePi$.\footnote{The choice of $f^{-1}(\sL)$ or $f^{-1}(\sL')$ is according to the relative position of $r$ and $\separator$. That is, we stretch along $f^{-1}(\sL')$ if and only if $r<\separator$.}
By cutting the domain of $f$ along the stretching trajectory, we have two subdomain $\Pi_t$ and $\Pi_u$ so that $\Pi_t$ contains $\bv_0$. Let $f^1\coloneqq f|_{\Pi_t}$ and $f^0\coloneqq f|_{\Pi_u}$ be the restrictions of $f$.

If both $\bx$ and $\by$ are contained in the same side of $\partial\Pi$ with respect to $\separator$, that is,
\[
\bx,\by\in(0,\separator)\quad\text{ or }\quad\bx,\by\in(\separator,t+1),
\]
then $\St(f)$ is defined by $g=(f^0\hybto{r} f^1)$ which is obtained by applying one of ($\St$-1),($\St$-2) and ($\St$-3) described in \S~\ref{section:stretching},
where $\separator^1$ is inherited from the separator $\separator$ of $f$ and
\[
\separator^0=
\begin{cases}
u+1 &\text{ if }\bx,\by\in(0,\separator);\\
0 &\text{ if }\bx,\by\in(\separator,t+1).
\end{cases}
\]
This yields a hybrid admissible pair $g\in\scrM^{\pair}_{(1\hybto{}0)}$ since the $(-)$-folding edge is contained in $\Pi^1$.

\begin{enumerate}
\item [($\St$-4)] If $\bx<\separator<\by$ or $\by<\separator<\bx$, then $\St(f)$ is defined by $g=(f^0\hybto{r} f^1)$ and by assigning $\separator^1=r=\by$ and $\separator^0$ inherited from $\separator$. It is easy to see that 
$g\in\scrM^{\pair}_{(0\hybto{}1)}$.
\[
\begin{tikzcd}
\vcenter{\hbox{}}\ar[r,"\St"]&
g=\vcenter{\hbox{}}=
\vcenter{\hbox{}}
\end{tikzcd}
\]
\end{enumerate}

\subsubsection{Cutting of concave corners and separators}\label{sec:Cutting of concave corners and the separator}
Now we define
\begin{align*}
\Cut_\pm&:\scrM^{\negativefoldingconcave}\to \scrM^{\pmfolding},& \Cut_+&:\scrM^{\leftrigid}\amalg\scrM^{\rightrigid} \to \scrM^{\pmfolding}.
\end{align*}

For $f\in\scrM^{\negativefoldingconcave}$, $\Cut_{\pm}(f)$ is defined by applying ($\Cut$-1) in \S~\ref{sec:Cutting concave vertices in two ways} near the concave vertex and by preserving the separator. Obviously, $\Cut_\pm(f)\in \scrM^{\pmfolding}$. Notice that ($\Cut$-2) and ($\Cut$-3) are not necessary since they treat only infinitesimal disks.

\begin{enumerate}
\item[($\Cut$-4)]
Suppose $f\in \scrM^{\leftrigid}\amalg\scrM^{\rightrigid}$. Then $\Cut_+(f)$ is defined by creating a $(\pm)$-folding edge near $f(\separator)$. Automatically, the separator $\separator'$ of $\Cut_+(f)$ is determined by $\Cut_+(f)(\separator')=f(\separator)$ and hence $\Cut_+(f)\in\scrM^{\pmfolding}$.
\[
\begin{tikzcd}
f=\vcenter{\hbox{}}\qquad\vcenter{\hbox{}}\in\scrM^{\leftrigid}\ar[r,"\Cut_+"]&
g=\vcenter{\hbox{}}\qquad\vcenter{\hbox{}}\in\scrM^{\pmfolding}\\
f=\vcenter{\hbox{}}\ \vcenter{\hbox{}}\in\scrM^{\rightrigid}\ar[r,"\Cut_+"]&
g=\vcenter{\hbox{}}\ \vcenter{\hbox{}}\in\scrM^{\pmfolding}\\
\end{tikzcd}
\]
\end{enumerate}

\begin{lemma}\label{lem:manipulation label preserving}
The following maps preserve $\scrP:\scrM \amalg \scrM^\pair\to A$ up to sign:
\begin{align*}
\Gl_+&:\scrM^\pair \to \scrM^{\pmfolding};\\
\Sh_\pm&:\scrM^{\pmfolding}\to \scrM^\pair_{(1,0)} \amalg \scrM^{\negativefoldingconcave}  \amalg \scrM^{\leftrigid} \amalg \scrM^{\rightrigid};\\
\St&:\scrM^{\pmfolding}\to \scrM^\pair_{(1,0)} \amalg \scrM^\pair_{(0,1)};\\
\Cut_\pm&:\scrM^{\negativefoldingconcave}\to \scrM^{\pmfolding};\\
\Cut_+&:\scrM^{\leftrigid}\amalg\scrM^{\rightrigid} \to \scrM^{\pmfolding}.
\end{align*}

Moreover, only the functions with negative subscripts reverse the sign of $\scrP$.
\end{lemma}
\begin{proof}
This is not hard to check from the definitions and we omit the proof.
\end{proof}

\subsubsection{Dragging, ancestor, and descendant maps}
\begin{enumerate}
\item [($\Dr$-1)] If an admissible hybrid disk is not of rightmost or right-rigid type, then we define 
\[
\Dr^{\sR}:\scrM^{\leftmost}\amalg\scrM^{\leftrigid}\amalg\scrM^{\flexible}\amalg\scrM^{\vertexleftrigid}\to\scrM^{\rightmost}\amalg\scrM^{\rightrigid}\amalg\scrM^{\flexible}
\]
by dragging the separator in the right direction, i.e. counterclockwise, along $\partial \Pi$ until it becomes vertex-separated or right-rigid.
\item [($\Dr$-2)] If an admissible hybrid disk is not of leftmost or left-rigid type, then we define 
\[
\Dr^{\sL}:\scrM^{\rightmost}\amalg\scrM^{\rightrigid}\amalg\scrM^{\flexible}\to \scrM^{\leftmost}\amalg\scrM^{\leftrigid}\amalg\scrM^{\vertexleftrigid}\amalg\scrM^{\flexible}
\]
by dragging the separator in the left direction, i.e. clockwise, along $\partial \Pi$ until it becomes vertex-separated, leftmost or left-rigid.
\item [($\Dr$-3)] For $f\in \scrM^\pair_{(0\hybto{}1)}$, we define $\tilde{\Dr}^{\sR}(f)$ by the {\em ancestor disk} of the pair
\begin{align*}
\tilde{\Dr}^{\sR}:\scrM^\pair_{(0\hybto{}1)}&\to\scrM^{\flexible}\amalg\scrM^{\vertexleftrigid}\amalg\scrM^{\vertexbirigid}:f=(f^0\hybto{r}f^1)\mapsto f^1.
\end{align*}
\item [($\Dr$-4)] For $f\in\scrM^{\flexible}\amalg\scrM^{\vertexleftrigid}\amalg\scrM^{\vertexbirigid}$ with $\separator=r$, we define $\tilde{\Dr}^{\sL}(f)$ as the set of all hybrid admissible pairs whose ancestor is $f$ with $\separator=r$
\[
\tilde{\Dr}^{\sL}(f)\coloneqq \left\{(f^0\hybto{r}f^1)\in\scrM^\pair\mid f^1=f, r\ge 1\right\}.
\]
\end{enumerate}

\begin{remark}
In general, the map $\Dr^{\sL}$ on $\scrM^{\flexible}$ does not preserve the map $\scrP$ since the map $\Psi$ in Definition~\ref{definition:polynomial of hybrid disk} is applied to different terms.
Instead, by using the equation (\ref{eq:sum of hybrids}), one compensates for the effect of $\Psi$ by considering the set $\tilde{\Dr}^{\sL}$.
\end{remark}

\begin{lemma}\label{lem:label sum at vertex}
The following holds:
\[
\scrP(f)=\begin{cases}
\displaystyle\scrP(\Dr^{\sL}(f))+\sum_{g\in \tilde{\Dr}^{\sL}(f)} \scrP(g)& f\in\scrM^{\flexible};\\
\displaystyle\sum_{g\in \tilde{\Dr}^{\sL}(f)} \scrP(g)& f\in\scrM^{\vertexleftrigid}\amalg\scrM^{\vertexbirigid}.
\end{cases}
\]
\end{lemma}
\begin{proof}
All cases except $f\in\scrM^{\vertexbirigid}$ are clear from the definition of $\scrP(f)$ and (\ref{eq:sum of hybrids}).

As seen earlier, $\tilde{\Dr}^{\sL}(f)$ for $f\in\scrM^{\vertexbirigid}$ consists of two elements
\[
\tilde{\Dr}^{\sL}(f)=\{(f_{\bar\sfx,\bar\sfp}\hybto{r} f), (f_{\bar\sfx,\bar\sfq}\hybto{r} f)\},
\]
and since $\scrP((f_{\bar\sfx,\bar\sfp}\hybto{r}f))=-\scrP((f_{\bar\sfx,\bar\sfq}\hybto{r}f))$, 
\[
\scrP(f)=0=\scrP((f_{\bar\sfx,\bar\sfp}\hybto{r}f))+\scrP((f_{\bar\sfx,\bar\sfq}\hybto{r}f)).\qedhere
\]
\end{proof}

\subsection{Proof of Proposition~\ref{proposition:new commutativity}}
The underlying idea of the following combinatorial arguments is to consider the 1-dimensional moduli space consisting of admissible hybrid disks and pairs, which can be used to extract an algebraic relation from the boundary of the moduli space.

Let $f$ be a hybrid non-infinitesimal admissible disk on $\Pi_t$. For $a\in [0,t+1]$, we denote the length of the image of $[0,a]\subset \partial \Pi_t$ by $\|a\|_f$. For convenience's sake, we set $\|a\|_f=0$ for all $a$ if $f$ is infinitesimal.

\begin{definition}[Separator index]
Let $f\in\scrM(\bar\sfg)$ be a hybrid admissible disk on $\Pi_t$. Then the {\em separator index} $s(f)$ is defined as
\[
s(f)\coloneqq \frac{\|\separator_f\|_f}{\|t+1\|_f}.
\]

Similarly, for a hybrid admissible pair $f=(f^0\hybto{r} f^1)\in\scrM^\pair(\bar\sfg)$ of $f^1$ on $\Pi_t$ and $f^0$ on $\Pi_u$, the separator index $s(f)$ is defined as
\[
s(f)\coloneqq \frac{\|\separator_{f^1}\|_{f^1}+\|\separator_{f^0}\|_{f^0}}{\|t+1\|_{f^1}+\|u+1\|_{f^0}}.
\]
\end{definition}

Notice that if $f$ does not have a positive folding edge, the separator index for the equivalent class $[f]$ is well-defined, or the separator index $s(f)$ will increase or decrease by stretching or shrinking it.

\begin{definition}\label{def:moduli graph}
For a given $\bar\sfg\in \bar\sG_{=\sfc}\amalg \bar\sG_{>\sfc}$,
let us construct a simple, labeled, directed graph $\bG_{\bar\sfg}$ as follows:
\begin{enumerate}
\item The vertices $\bV\bG_{\bar\sfg}$ consist of
\begin{align*}
\bV\bG_{\bar\sfg}\coloneqq\scrM(\bar\sfg)_{(1)} \amalg \scrM^\pair(\bar\sfg)_{(1\hybto{}0)} \amalg \scrM^\pair(\bar\sfg)_{(0\hybto{}1)}.
\end{align*}
\item For two vertices $f,g$ in $\bV\bG_{\bar\sfg}$, we assign an edge $\{f,g\}\in \bE\bG_{\bar\sfg}$ if there is a function 
\[
\bQ\in \{\Gl_\pm, \Sh_\pm, \St, \Cut_\pm, \Dr^{\sR},\tilde{\Dr}^{\sR},\Dr^{\sL}, \tilde{\Dr}^{\sL}\}
\]
such that either $g=\bQ(f)$ or $g \in \bQ(f)$. See Table~\ref{tab:subsets and operations}.
\item 
The orientation for each edge of $\bG_{\bar\sfg}$ is given as follows: Let $\{f,g\}$ be an edge.
\begin{enumerate}
\item if $g=\tilde{\Dr}^{\sR}(f)$ or $f\in\tilde{\Dr}^{\sL}(g)$, we assign an orientation $f\to g$;
\item for the rest of non-oriented edge $\{f,g\}$, we assign an orientation $f\to g$ if $s(f)<s(g)$.
\end{enumerate}
\item
Finally, the label $\bP:\bE\bG_{\bar\sfg}\to A$ is defined by
\begin{align*}
\bP(f\to g)=\begin{cases}
\scrP(g) & g=\bQ(f)\text{ or }f=\bQ(g), \bQ\in\{\Gl_-,\Cut_-\};\\
\scrP(f) & \text{otherwise}.
\end{cases}
\end{align*} 
\end{enumerate}
\end{definition}

\begin{table}[ht]
\[\def\arraystretch{1.5}
\begin{array}{c||c|c|c|c|c|c|c|c|c||c}
& \Gl_\pm & \Sh_\pm & \St & \Cut_+ & \Cut_- & \Dr^{\sR} & \Dr^{\sL} & \tilde{\Dr}^{\sR} & \tilde{\Dr}^{\sL} & \val_{\bG_{\bar\sfg}}\\
\hline\hline
\scrM_{(1\hybto{}0)}^\pair & \bigcirc & & & & & & & & & 1\\
\hline
\scrM_{(0\hybto{}1)}^\pair & \bigcirc & & & & & & & \bigcirc & & 2\\
\hline
\scrM^{\pmfolding} & & \bigcirc & \bigcirc & & & & & & & 2\\
\hline
\scrM^{\negativefoldingconcave} & & & & \bigcirc & \bigcirc & & & & & 2\\
\hline
\scrM^{\leftrigid} & & & & \bigcirc & & \bigcirc & & & & 2\\
\hline
\scrM^{\rightrigid} & & & & \bigcirc & & & \bigcirc & & & 2\\
\hline
\scrM^\leftmost & & & & & & \bigcirc & & & & 1\\
\hline
\scrM^\rightmost & & & & & & & \bigcirc & & & 1\\
\hline
\scrM^{\vertexleftrigid} & & & & & & \bigcirc & & & \bigcirc & \ge2\\
\hline
\scrM^{\vertexbirigid} & & & & & & & & & \bigcirc & 2\\
\hline
\scrM^{\flexible} & & & & & & \bigcirc & \bigcirc & & \bigcirc & \ge2
\end{array}
\]
\caption{Subsets and operations}
\label{tab:subsets and operations}
\end{table}

Notice that if $g=\Gl_-(f)$ or $g=\Cut_-(f)$ then $\scrP(g)=-\scrP(f)$.
Moreover, the set of exterior vertices is
\[
\bV_{\mathsf{ext}}\bG_{\bar\sfg}=\scrM^\leftmost(\bar\sfg)\amalg \scrM^\rightmost(\bar\sfg)\amalg \scrM^\pair_{(1\hybto{}0)}(\bar\sfg),
\]
and it is easy to see that for each interior vertex $f\in\bV_{\mathsf{int}}\bG_{\bar\sfg}$, the number of outgoing edges is
\begin{enumerate}
\item $0$ if $f$ is either in $\scrM^{\negativefoldingconcave}$ with $\separator<r$, or in $\scrM^{\vertexbirigid}$,
\item $2$ if $f$ is in $\scrM^{\negativefoldingconcave}$ with $r<\separator$, or
\item $1$ otherwise.
\end{enumerate}
Here the index $r$ is the position of the concave vertex. Similarly, for each exterior vertex $f\in\bV_{\mathsf{ext}}\bG_{\bar\sfg}$, the number of outgoing edge is $1$ if and only if either
\begin{enumerate}
\item $f\in\scrM^{\leftmost}$,
\item $f\in\scrM^\pair_{(1\hybto{}0)}$ with $f^0\not\in\cM^{\infinitesimaltriangle}$ and $r<\separator^1$, or
\item $f\in\scrM^\pair_{(1\hybto{}0)}$ with $f^0\in\cM^{\infinitesimaltriangle}$ and $\separator^1<r$.
\end{enumerate}

\begin{lemma}\label{lemma:handshaking}
For each interior vertex $f\in\bV_{\mathsf{int}}\bG_{\bar\sfg}$, the label sums of incoming and outgoing edges coincide. That is,
\[
\sum_{\substack{\be\in \bE\bG_{\bar\sfg}\\ \be=(f\to g)}} \bP(\be) - 
\sum_{\substack{\be\in \bE\bG_{\bar\sfg}\\ \be=(g\to f)}} \bP(\be)=0.
\]
\end{lemma}
\begin{proof}
This can be checked case by case as follows:
\begin{enumerate}
\item If $f$ is not bivalent, then this is equivalent to Lemma~\ref{lem:label sum at vertex}.
\item For $f\in\scrM^{\pmfolding}$, there are one incoming and one outgoing edges with the same labels.
\item For $f\in\scrM^{\negativefoldingconcave}$, there are two outgoing or two incoming edges whose signs of labels are opposite as seen in ($\Sh_\pm$-2) or in ($\Cut$-2).
\item For $f\in\scrM^{\vertexbirigid}$, there are two incoming edges whose signs of labels are opposite as seen in Lemma~\ref{lem:label sum at vertex}.\qedhere
\end{enumerate}
\end{proof}

\begin{proof}[Proof of Proposition~{\rm\ref{proposition:new commutativity}}]
We deduce
\begin{align*}
0&=\sum_{\be\in\bE\bG_{\bar\sfg}} \bP(\be)-\bP(\be)
=\sum_{\substack{f\in \bV\bG_{\bar\sfg}\\ \be=(f\to g)}}\bP(\be)-\sum_{\substack{f\in \bV\bG_{\bar\sfg}\\ \be=(g\to f)}}\bP(\be)
=\sum_{\substack{f\in \bV_\ext\bG_{\bar\sfg}\\ \be=(f\to g)}}\bP(\be)-\sum_{\substack{f\in \bV_\ext\bG_{\bar\sfg}\\ \be=(g\to f)}}\bP(\be)\\
&= \left(\sum_{f\in\scrM^{\leftmost}(\bar\sfg_i)}\scrP(f)
+\sum_{\substack{f\in\scrM^\pair_{(1\hybto{}0)}(\bar\sfg_i)\\ f^0\not\in \cM^{\infinitesimaltriangle}, r<\separator^1}} \scrP(f)
+\sum_{\substack{f\in\scrM^\pair_{(1\hybto{}0)}(\bar\sfg_i)\\ f^0\in \cM^{\infinitesimaltriangle}, \separator^1<r}} (-\scrP(f))
\right)\\
&\mathrel{\hphantom{=}}-\left(
\sum_{f\in\scrM^{\rightmost}(\bar\sfg_i)}\scrP(f)
+\sum_{\substack{f\in\scrM^\pair_{(1\hybto{}0)}(\bar\sfg_i)\\ f^0\not\in \cM^{\infinitesimaltriangle}, \separator^1<r}} \scrP(f)
+\sum_{\substack{f\in\scrM^\pair_{(1\hybto{}0)}(\bar\sfg_i)\\ f^0\in \cM^{\infinitesimaltriangle}, r<\separator^1}} (-\scrP(f))
\right)\\
&=\sum_{f\in\scrM^{\leftmost}(\bar\sfg_i)}\scrP(f)+\sum_{\substack{f\in\scrM^\pair_{(1\hybto{}0)}(\bar\sfg_i)\\ r<\separator^1}}\scrP(f)
-\sum_{f\in\scrM^{\rightmost}(\bar\sfg_i)}\scrP(f)
-\sum_{\substack{f\in\scrM^\pair_{(1\hybto{}0)}(\bar\sfg_i)\\ \separator^1<r}} \scrP(f)
\end{align*}
where the second equality is the handshaking lemma for directed graphs, the third and the fourth equalities comes from Lemma~\ref{lemma:handshaking} and the observation above the lemma.
This proves Proposition~\ref{proposition:new commutativity}.
\end{proof}

\section*{Acknowledgement}
The authors are grateful to Gabriel C. Drummond-Cole, Vivek Shende, Otto van Koert, Kaoru Ono, and Tobias Ekholm for detailed and general comments including a meaningful example in the previous draft. The authors also thank anonymous referees for constructive comments. 
The second author thanks Research Institute for Mathematical Sciences, Kyoto University for its warm hospitality. 

The first author was supported by IBS-R003-D1. The second author was partially supported by IBS-R003-D1 and JSPS International Research Fellowship Program.


\begin{thebibliography}{99}
\bibitem{ABK2018}
B.~H. An, Y.~Bae, \and S.~Kim, `Legendrian singular links and singular connected sums', {\em J. Symplectic Geom.}, 16 (2018) no. 4, 885--930.

\bibitem{Avdek2012}
R.~Avdek, `Liouville hypersurfaces and connect sum cobordisms', arXiv:1204.3145.

\bibitem{BI2009}
S.~Baader \and M.~Ishikawa, `Legendrian graphs and quasipositive diagrams', {\em Ann. Fac. Sci. Toulouse Math. (6)}, 18 (2009), no. 2, 285--305.

\bibitem{Chekanov2002}
Y.~Chekanov, `Differential algebra of Legendrian links', {\em Invent. Math.}, 150 (2002), no. 3, 441--483.

\bibitem{EES2005}
T.~Ekholm, J.~B. Etnyre \and M.~G. Sullivan, `The contact homology of Legendrian submanifolds in $\mathbb{R}^{2n+1}$', {\em J. Differential Geom.}, 71 (2005), no. 2, 177--305.

\bibitem{EES2007}
T.~Ekholm, J.~B. Etnyre \and M.~G. Sullivan, `Legendrian contact homology in $P\times\mathbb{R}$', {\em Trans. Amer. Math. Soc.}, 359 (2007), no. 7, 3301--3335.

\bibitem{EL2017}
T.~Ekholm \and Y.~Lekili, `Duality between lagrangian and legendrian invariants', arXiv:1701.01284.

\bibitem{ENS2018}
T.~Ekholm, L.~Ng \and V.~Shende, `A complete knot invariant from contact homology', {\em Invent. Math.}, 211 (2018), no. 3, 1149--1200.

\bibitem{EN2015}
T.~Ekholm \and L.~L. Ng, `Legendrian contact homology in the boundary of a subcritical Weinstein 4-manifold', {\em J. Differential Geom.}, 101 (2015), no. 1, 67--157.

\bibitem{Eliashberg2017}
Y.~Eliashberg, `Weinstein manifolds revisited', arXiv:1707.03442.

\bibitem{Eliashberg2000}
Y.~Eliashberg, `Invariants in contact topology', In {\em Proceedings of the International Congress of Mathematicians, Vol. II (Berlin, 1998)}, {\em Doc. Math.} 1998, Extra Vol. II, 327--338.

\bibitem{ENS2002}
J.~B. Etnyre, L.~L. Ng \and J.~M. Sabloff, `Invariants of Legendrian knots and coherent orientations', {\em J. Symplectic Geom.}, 1 (2002), no. 2, 321--367.

\bibitem{GPS2017}
S.~Ganatra, J.~Pardon \and V.~Shende, `Covariantly functorial wrapped floer theory on liouville sectors', arXiv:1706.03152.

\bibitem{Geiges2008}
H.~Geiges, {\em An Introduction to Contact Topology}, Cambridge Studies in Advanced Mathematics, 109, (Cambridge University Press, 2008).

\bibitem{HS2014}
J.~Harper \and M.~Sullivan, `A bordered legendrian contact algebra', {\em J. Symplectic Geom.}, 12 (2014), no. 2, 237--255.

\bibitem{LS2013}
J.~E. Licata \and J.~M. Sabloff, `Legendrian contact homology in Seifert fibered spaces', {\em Quantum Topol.}, 4 (2013), no. 3, 265--301.

\bibitem{Nadler2017}
D.~Nadler, `Arboreal singularities', {\em Geom. Topol.}, 21 (2017), no. 2, 1231--1274.

\bibitem{Ng2003}
L.~L. Ng, `Computable Legendrian invariants', {\em Topology}, 42 (2003), no. 1, 55--82.

\bibitem{Ng2010}
L.~L. Ng, `Rational symplectic field theory for Legendrian knots', {\em Invent. Math.}, 182 (2010), no. 3, 451--512.

\bibitem{NT2004}
L.~L. Ng \and L.~Traynor, `Legendrian solid-torus links', {\em J. Symplectic Geom.}, 2 (2004), no. 3, 411--443.

\bibitem{OP2014}
D.~O'Donnol \and E.~Pavelescu, `Legendrian $\theta$-graphs', {\em Pacific J. Math.}, 270 (2014), no. 1, 191--210.

\bibitem{Prasolov2014}
M.~Prasolov, `Rectangular diagrams of Legendrian graphs', {\em J. Knot Theory Ramifications}, 23 (2014), no. 13, 1450074, 28 pp.

\bibitem{Sabloff2003}
J.~M. Sabloff, `Invariants of Legendrian Knots in Circle Bundles', {\em Commun. Contemp. Math.}, 5 (2003), no. 4, 569--627.

\bibitem{Shende2018}
V.~Shende, `Arboreal singularities from lefschetz fibrations', arXiv:1809.10359v1.

\bibitem{Sivek2011}
S.~Sivek, `A bordered Chekanov-Eliashberg algebra', {\em J. Topol.}, 4 (2011), no. 1, 73--104.

\bibitem{Sylvan2016}
Z.~Sylvan, `On partially wrapped Fukaya categories', arXiv:1604.02540.

\end{thebibliography}
\end{document}